\documentclass[reqno]{amsart}
\usepackage[utf8]{inputenc}
\usepackage{amsthm,amsfonts,amstext,amssymb,mathrsfs,amsmath,latexsym,mathtools} 
\usepackage{cite} 
\usepackage{enumerate}
\usepackage{bm}
\usepackage[thinlines]{easymat}
\usepackage{mdwlist}
\usepackage[abs]{overpic} 

\usepackage{hyperref} 
\usepackage{comment}
\usepackage{tikz}
\usepackage{caption}
\usepackage{subcaption}
\usepackage{color}
\usepackage{esint} 
\usepackage[colorinlistoftodos]{todonotes}

\definecolor{red}{RGB}{255,0,0}
\definecolor{green}{RGB}{0,150,0}
\definecolor{blue}{RGB}{0,0,255}  

\usepackage{cleveref}
 
\crefname{equation}{equation}{equations}
\crefname{figure}{Figure}{Figures}

\newtheorem{thm}{Theorem}[section]
\newtheorem{prop}[thm]{Proposition}
\newtheorem{lem}[thm]{Lemma}

\theoremstyle{definition}
\newtheorem{definition}[thm]{Definition}

\theoremstyle{remark}
\newtheorem{remark}[thm]{Remark} 
\numberwithin{equation}{section} 

\DeclareMathOperator{\re}{Re}
\DeclareMathOperator{\im}{Im}
\renewcommand{\Re}{\mathop{\rm Re}}
\renewcommand{\Im}{\mathop{\rm Im}}
\DeclareMathOperator{\supp}{supp} 

\DeclareMathOperator{\const}{const}
\DeclareMathOperator{\res}{Res}

\DeclareMathOperator{\diag}{diag}
\DeclareMathOperator{\interior}{int}
\DeclareMathOperator{\Bal}{Bal} 
\newcommand{\isdef}{\stackrel{\text{\tiny def}}{=}}

\renewcommand{\H}{\mathbb{H}}

\newcommand{\N}{\mathbb{N}}
\newcommand{\C}{\mathbb{C}}
\newcommand{\R}{\mathbb{R}}

\newcommand{\Q}{\mathbb{Q}}
\newcommand{\boh}{\mathit{o}}
\newcommand{\Boh}{\mathcal{O}}
\newcommand{\Ai}{\mathop{\rm Ai}}

\newcommand\restr[2]{{
		\left.\kern-\nulldelimiterspace 
		#1 
		\vphantom{\big|} 
		\right|_{#2} 
}}

\title[Critical measures: asymptotics of MOPs]{Critical measures for vector energy: asymptotics of non-diagonal multiple orthogonal polynomials for a cubic weight}

\author[A. Mart\'{\i}nez-Finkelshtein]{Andrei Mart\'{\i}nez-Finkelshtein}

\address[AMF]{Department of Mathematics, Baylor University, TX, USA, and Departamento de Matem\'aticas, Universidad de Almer\'{\i}a, Spain}

\email[AMF]{A\_Martinez-Finkelshtein@baylor.edu}

\email[AMF]{andrei@ual.es}

\author[G.~Silva]{Guilherme L.~F.~Silva}

\address[GS]{University of Michigan, Ann Arbor - MI, USA}

\email{silvag@umich.edu}

\keywords{Logarithmic potential theory, vector energy, equilibrium on the complex plane, critical measures, $S$-property, quadratic differentials trajectories, multiple orthogonal polynomials, asymptotics, Riemann-Hilbert problems.}
 
\subjclass[2010]{Primary: 33C45; Secondary: 30C70, 30E10, 30F30, 31A15.}

\begin{document}

\begin{abstract} 

We consider the type I multiple orthogonal polynomials (MOPs) $(A_{n,m}, B_{n,m})$,  $\deg A_{n,m} \leq n-1$, $ \deg B_{n,m}\leq  m-1$, and type II MOPs $P_{n,m}$, $\deg P_{n,m} = n+m$, satisfying non-hermitian orthogonality with respect to the  weight $e^{-z^3}$ on two unbounded contours $\gamma_1$ and $\gamma_2$ on $\C$, with (in the case of type II MOPs) $n$ conditions on $\gamma_1$ and $m$ on $\gamma_2$. Under the assumption that
$$
n,m \to \infty, \quad \frac{n}{n+m}\to \alpha \in (0, 1)
$$
we find the detailed (rescaled) asymptotics of $A_{n,m}$, $B_{n,m}$ and $P_{n,m}$ on $\C$, and describe the phase transitions of this limit behavior as a function of $\alpha$. This description is given in terms of the vector critical measure $\vec{\mu}_\alpha  = (\mu_1 , \mu_2 , \mu_3 )$, the saddle point of an energy functional comprising both attracting and repelling forces. This critical measure is characterized by a cubic equation (spectral curve), and its components $\mu_j$  live on trajectories of a canonical quadratic differential $\varpi$ on the Riemann surface of this equation.  The structure of these trajectories and their deformations as functions of $\alpha$ was object of study in our previous paper  [Adv. Math. \textbf{302} (2016), 1137--1232], and some of the present results strongly rely on the analysis carried out there.

We conclude that the asymptotic zero distribution of the polynomials $A_{n,m}$  and $P_{n,m}$ are given by appropriate combinations of the components $\mu_j$ of the vector critical measure $\vec{\mu}_\alpha$. However, in the case of the zeros of $B_{n,m}$ the behavior is totally different, and can be described in terms of the balayage of the signed measure $\mu_2 - \mu_3$ onto certain curves on the plane. These curves are constructed with the aid of the canonical quadratic differential $\varpi$, and their topology has three very distinct characters, depending on the value of $\alpha$, and are obtained from the critical graph of $\varpi$.

Once the trajectories and vector critical measures are studied, the main asymptotic technical tool is the Deift-Zhou nonlinear steepest descent analysis of a $3\times 3$ matrix-valued Riemann--Hilbert problem characterizing both $(A_{n,m}, B_{n,m})$ and  $P_{n,m}$, which allows us to obtain the limit behavior of both types of MOPs simultaneously.

We illustrate our findings with results of several numerical experiments, explain the computational methodology, and formulate some conjectures and empirical observations based on these experiments. 

\end{abstract}

\maketitle

\tableofcontents

\section{Introduction}

The theory of Pad\'e approximants, which provide a locally best rational approximation to a power series, was developed  in the 19th--20th centuries intertwined with the theory of continued fractions, see \cite{MR2963451} for some historical remarks. In his 1895 paper \cite{Markov1895} Markov (or Markoff) proved that denominators of these approximants to what we now call Cauchy transform of a positive function on the real line are orthogonal with respect to this function, linking convergence analysis of such approximants to the analytic theory of orthogonal polynomials.

Since Pad\'e approximants allow us, at least potentially, to study the global properties of an analytic function given locally by a power series, they became a very popular tool in applications, which stimulated the rapid development of their theory in the 1970s--80s. During these years the extension to the essentially non-real situation began, first in the works of Gonchar and Nuttall, and later, of Rakhmanov and Stahl. They, together with Mhaskar and Saff, unraveled also the power of logarithmic potential theory tools. 

The main challenge in this extension was that the denominators of Pad\'e approximants for holomorphic germs of algebraic (multivalued) functions satisfy non-Hermitian orthogonality conditions: they do not correspond to a scalar product and the contour of integration can be freely deformed within its homotopy class. It was Stahl \cite{stahl_orthogonal_polynomials_complex_weight_function, stahl_orthogonal_polynomials_complex_measures} who first explained how to select the right contour (and in consequence, to establish weak or $n$-th root asymptotics for these orthogonal polynomials), showing that these contours solve a max-min problem for logarithmic energy. He also characterized them by means of a certain symmetry property (now called \textit{$S$-property}), which  turned out to be a very powerful tool, linking the problem to geometric function theory. In this context, finding such contours with the $S$-property then becomes equivalent to finding quadratic differentials with given critical points and desired topology of trajectories.

Stahl's methods were improved and enhanced in the work of Gonchar and Rakh\-ma\-nov \cite{gonchar_rakhmanov_rato_rational_approximation}, allowing for varying measures (those depending on the degree of the polynomial), giving rise to the \textit{Gonchar-Rakhmanov-Stahl} (or \textit{GRS}) \textit{theory}. This is one of a few general results in the modern analytic theory of Pad\'e approximants and non-hermitian orthogonal polynomials. A crucial ingredient of this method is the existence of contours on $\C$  exhibiting a \textit{weighted $S$-property}: assuming the existence of such contours, it is possible to obtain asymptotics for non-hermitian orthogonal polynomials with varying weights. Stahl was able to prove existence in the non-weighted case, but it took more than 20 years for the existence of the weighted $S$-contour to be established even for polynomial external fields \cite{kuijlaars_silva}, following some ideas of Rakhmanov \cite{rakhmanov_orthogonal_s_curves}. The proof in \cite{kuijlaars_silva} relied on the analysis of the so-called \textit{critical measures} (saddle points of the weighted logarithmic energy), introduced in \cite{martinez_rakhmanov}, which  naturally extend the notion of equilibrium measures. Critical measures turned out to be useful also in the study of polynomial solutions of some linear ODE \cite{martinez_rakhmanov, MR2647571,martinez-saff}.

A construction generalizing continued fractions but for several functions simultaneously was put forward by Hermite \cite{Hermite1873} as the crucial step in his proof of transcendence of the number $e$; it is known today as \textit{Hermite--Pad\'e} or \textit{simultaneous approximation} to a system of functions. There are two types of Hermite--Pad\'e approximants, and as in the case of their one-dimensional ancestors, the corresponding polynomials satisfy a set of orthogonality conditions, but now involving more than one measure (or weight function). These are the \textit{Hermite--Pad\'e} (or \textit{multiple}) \textit{orthogonal polynomials} (shortly MOPs), see their definitions below. 

Although born from the needs of approximation theory and number theory, recently MOPs attracted increasing interest due to their appearance in the analysis of several random matrix and non-intersecting random paths models, among others, as well as in  problems of Statistical Mechanics \cite{alvarez_alonso_medina_1, alvarez_alonso_medina_2, MR2963452, bertola_gekhtman_szmigielski_cauchy_two_matrix_model, bleher_delvaux_kuijlaars_external_source, bleher_kuijlaars_normal_matrix_model, kuijlaars_bleher_external_source_multiple_orthogonal, kuijlaars_bleher_external_source_gaussian_I, duits_geudens_kuijlaars, duits_kuijlaars_two_matrix_model, duits_kuijlaars_mo, kuijlaars_lopez_normal_matrix_model}, to mention  a few. All these applications stimulated also the development of the analytic theory of MOPs, starting with the pioneering works of the Russian school (Aptekarev, Gonchar, Kalyagin, Nikishin, Rakhmanov, and Suetin), which were later on joined by L\'opez-Lagomasino, Nuttall, Stahl, Van Assche, among many others. In the initial stage, the study circumscribed to two somewhat extreme cases on the real line, associated with the names of Angelesco and Nikishin, when existence of all polynomials of their maximal degree is guaranteed (the so-called normality of these approximants, see \cite{fidalgo}). 
Gonchar and Rakhmanov showed \cite{gonchar_rakhmanov_1981_pade, MR807734} that the large degree limit of the zeros of MOPs for Angelesco systems of measures is described by a vector equilibrium problem involving mutually repelling and disjointly supported measures on the real line, while for the Nikishin systems we need to consider vector equilibrium where attractive interactions are present. 

However, extending this methodology to the complex case and for multivalued functions without real symmetry turned out to be a formidable task. The asymptotic analysis of the corresponding Hermite-Pad\'e polynomials satisfying non-Hermitian orthogonality conditions requires considering new non-standard vector equilibria (involving additional measures) with new types of constraints and a new notion of $S$-property. In such settings, the mere existence of solution of the corresponding vector max-min problems is not guaranteed. Hence, despite many efforts and recent contributions \cite{MR2475084, aptekarev_kuijlaars_vanassche_hermite_pade_genus_0, aptekarev_vanassche_yatsselev, suetin13, suetin14, kuijlaars_vanassche_wielonsky_hermite_pade, MR2796829, rakhmanov_hermite_pade,leurs_vanassche}, it looks like a multidimensional extension of the GRS theory, valid for MOPs with arbitrary number of weights and no underlying symmetries, is nowhere in sight. This is why the analysis of even simplest non-trivial cases that do not exhibit any particular symmetry can shed a new light onto the problem. 

In an effort to extend the GRS theory to a general context of multiple non-Hermitian orthogonality, we carried out in our previous work 
\cite{martinez_silva_critical_measures} a systematic study of vector critical measures, and found a characterization valid for systems  comprising both  ``Angelesco'' and  ``Nikishin'' type interactions on the plane. We showed that for our choice of interaction matrices, the condition that a vector of measures is critical is equivalent to their Cauchy transforms being the solutions (in a sense specified in \cite{martinez_silva_critical_measures}) of an algebraic equation, also known as the spectral curve. Furthermore, we demonstrated that the critical measures live on trajectories of a canonical quadratic differential on the compact Riemann surface associated to the spectral curve. This way, we can embed the support of the critical measures on this Riemann surface, in such a way that different pairwise interactions, such as Angelesco and Nikishin, can be put on equal foot. For a cubic potential, we illustrated the power of the theory of quadratic differentials in this setup, by explicitly constructing a one-parametric family of vector critical measures through the precise dynamical description of the trajectories of the underlying quadratic differential. 

The interaction matrix considered in \cite{martinez_silva_critical_measures} appears also in the analysis of random matrix models with external source. For the general quartic potential with even symmetry, and about at the same time, for the general even degree potential with even symmetry, Aptekarev, Lysov and Tulyakov \cite{aptekarev_lysov_tulyakov, aptekarev_lysov_tulyakov_2} and Bleher, Delvaux and Kuijlaars \cite{bleher_delvaux_kuijlaars_external_source}, respectively, addressed such matrix model using the MOPs approach. The key object that allowed them to obtain several asymptotic results was a vector equilibrium problem. The vector equilibrium problems considered in \cite{aptekarev_lysov_tulyakov, aptekarev_lysov_tulyakov_2} and \cite{bleher_delvaux_kuijlaars_external_source} are essentially different, but in both situations their derivation relies on the strong symmetry of the problem. Under this perspective, our choice of interaction matrix in \cite{martinez_silva_critical_measures} generalizes the explicit construction of Aptekarev, Lysov and Tulyakov to potentials without symmetries, and (as an ongoing research shows) can be reduced to the constrained vector equilibrium problem considered by Bleher, Delvaux and Kuijlaars.

\begin{figure}[t]
\begin{subfigure}{.5\textwidth}
\centering
\begin{overpic}[scale=.45]{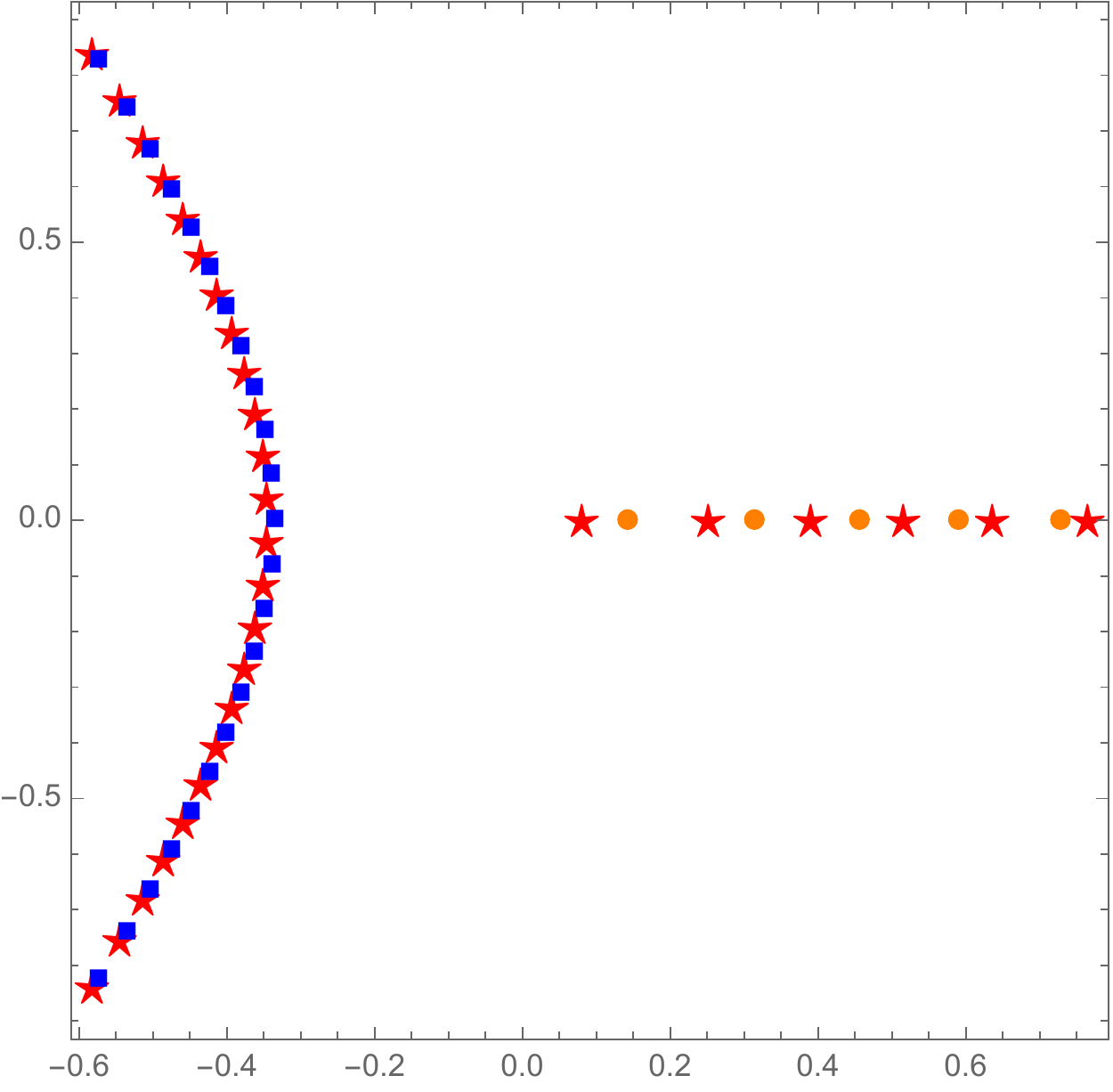}
\end{overpic}
\end{subfigure}%
\begin{subfigure}{.5\textwidth}
\centering
\begin{overpic}[scale=.45]{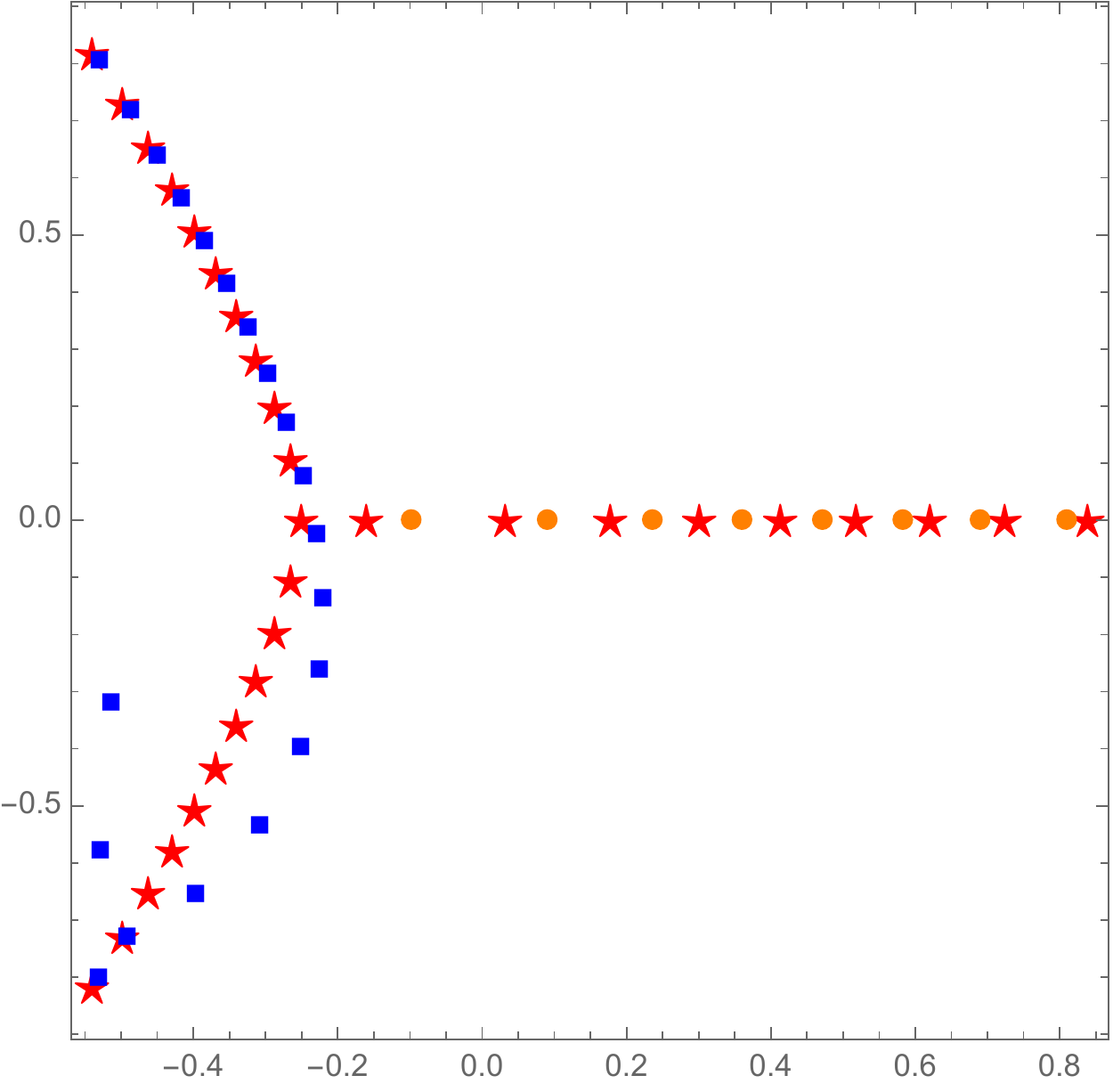}
\end{overpic}
\end{subfigure}\\
\begin{subfigure}{.5\textwidth}
\centering
\begin{overpic}[scale=.45]{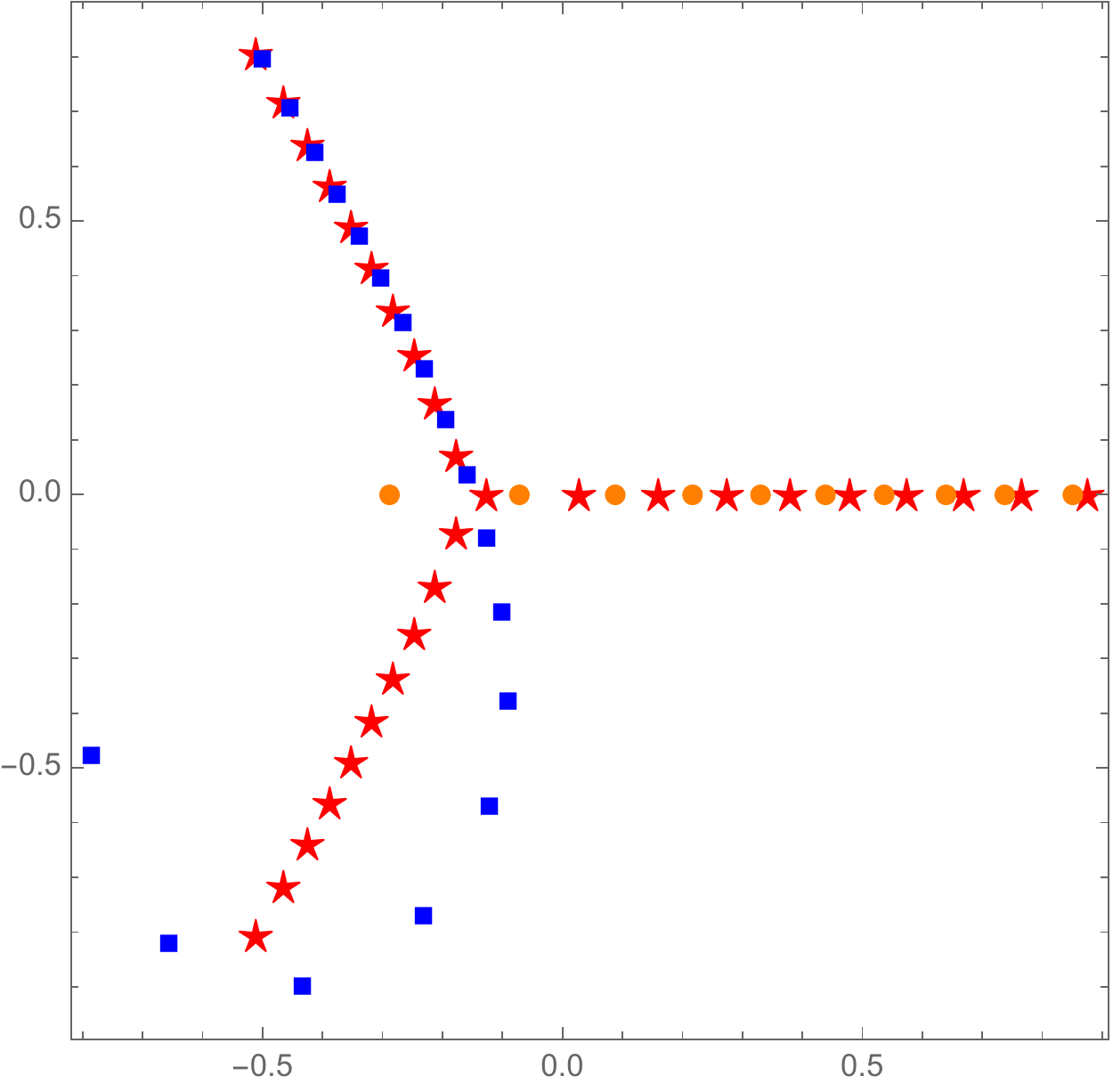}
\end{overpic}
\end{subfigure}%
\caption{Zeros of $P_{n,m}$ (stars), $A_{n,m}$ (dots) and $B_{n,m}$ (squares), for $n+m=30$ fixed and the choices $(n,m)=(6,24)$ (top left), $(n,m)=(9,21)$ (top right) and $(n,m)=(11,19)$ (bottom); see Definitions~\ref{defTypeI} and \ref{defTypeII}.}\label{figure_zeros_N=30}
\end{figure}

The present work is a natural continuation of \cite{martinez_silva_critical_measures}: we analyze a family of non-Hermitian MOPs (both of types I and II) with respect to a cubic weight and find their asymptotic description in terms of the vector critical measures 
constructed in \cite{martinez_silva_critical_measures}. It should be pointed out that the large degree zero distribution of these polynomials is highly non-trivial, exhibiting several phase transitions, as Figures \ref{figure_zeros_N=30} and \ref{figure_zeros_N=50} illustrate. Some previous contributions \cite{MR2475084, aptekarev_vanassche_yatsselev} addressed the large degree asymptotics on the plane in the non-symmetric case and for weights on bounded sets having only finite branch points; see also \cite{suetin13,suetin14} for nice numerical experiments and empirical discussion related to the so-called Nutall's conjecture. To our knowledge, the present work is the first systematic study of zero distribution of MOPs with complex zeros that exhibit non-hermitian orthogonality on unbounded sets (leading to consideration of extremal problems with an external field) with no real symmetry. As we hope it will become clear to the reader, the novelty fact that explains such break of symmetry is the canonical quadratic differential on the spectral curve and its trajectories, that were rigorously described in our previous work \cite{martinez_silva_critical_measures}. Once these trajectories are well understood, the dynamics of the limiting zero distribution can be read off from these trajectories, and the very distinct behaviors portrayed in Figures Figures \ref{figure_zeros_N=30} and \ref{figure_zeros_N=50} are easily understood under the very same framework.

\begin{figure}[t]
\begin{subfigure}{.5\textwidth}
\centering
\begin{overpic}[scale=.45]{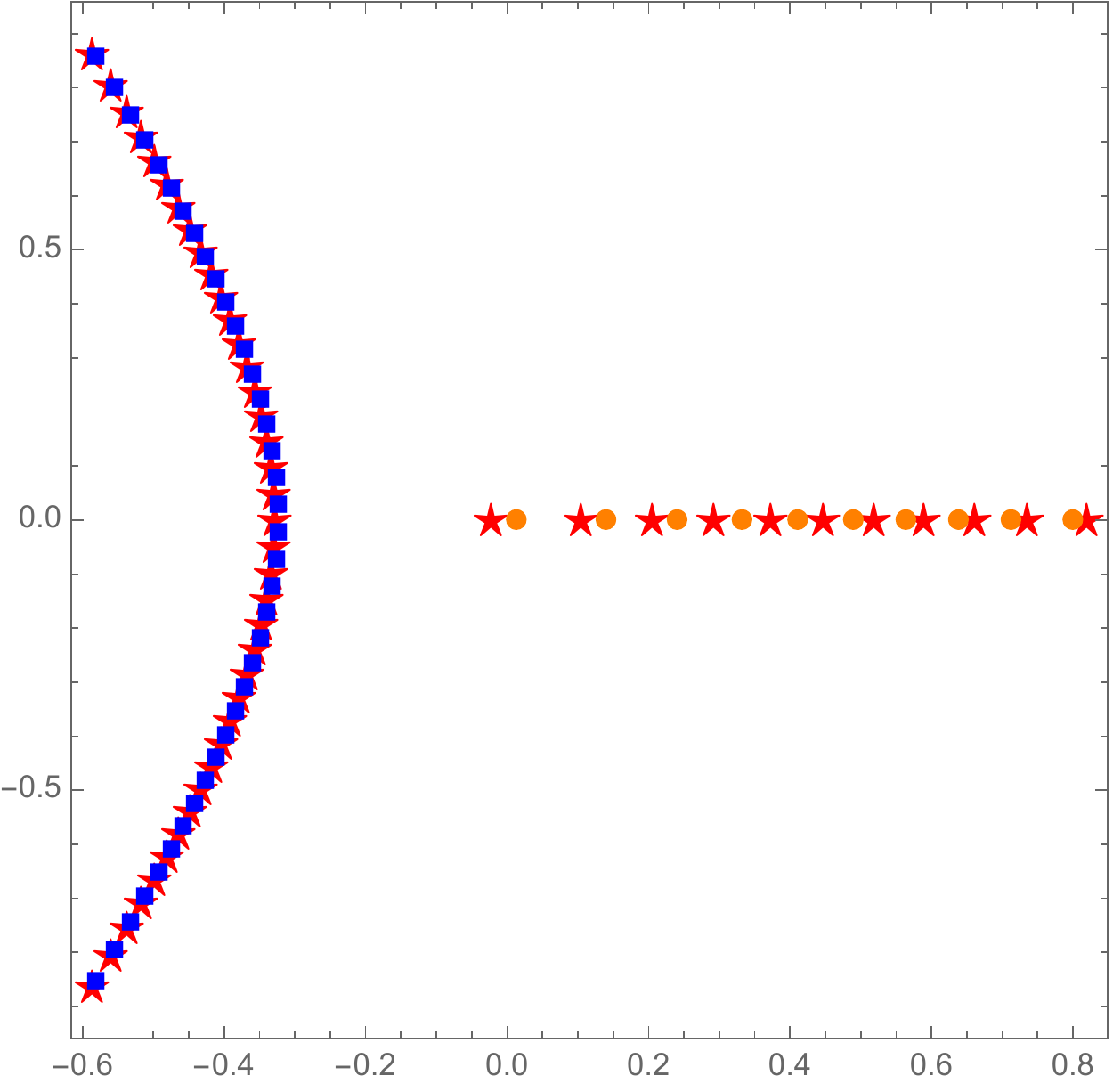}
\end{overpic}
\end{subfigure}%
\begin{subfigure}{.5\textwidth}
\centering
\begin{overpic}[scale=.45]{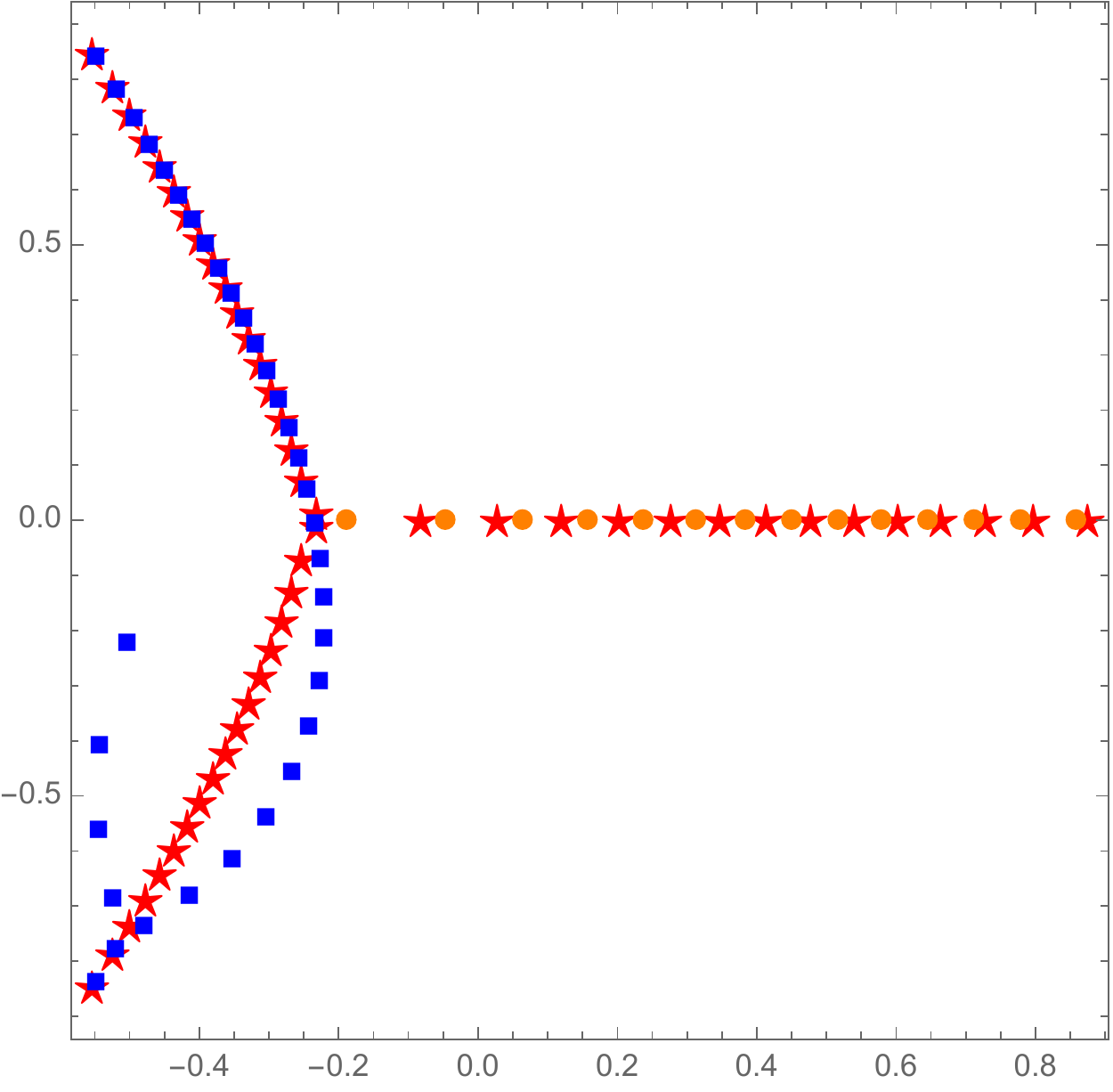}
\end{overpic}
\end{subfigure}\\
\begin{subfigure}{.5\textwidth}
\centering
\begin{overpic}[scale=.45]{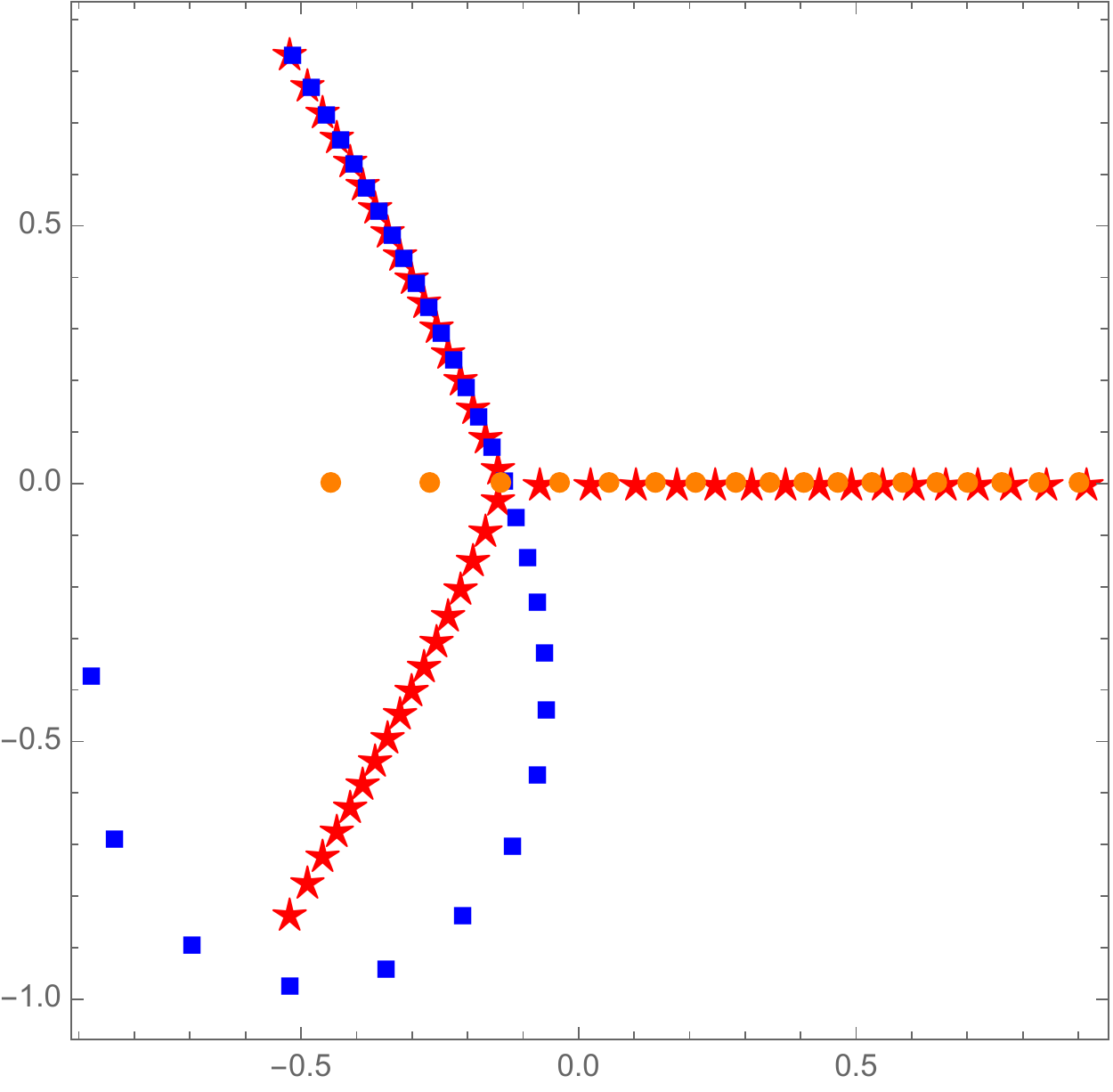}
\end{overpic}
\end{subfigure}%
\caption{Zeros of $P_{n,m}$ (stars), $A_{n,m}$ (dots) and $B_{n,m}$ (squares), for $n+m=50$ fixed and the choices $(n,m)=(11,39)$ (top left), $(n,m)=(15,35)$ (top right) and $(n,m)=(19,31)$ (bottom); see Definitions~\ref{defTypeI} and \ref{defTypeII}.}\label{figure_zeros_N=50}
\end{figure}

Besides the aforementioned quadratic differentials that were previously studied in \cite{martinez_silva_critical_measures}, our main asymptotic tool is the matrix Riemann-Hilbert problem characterization of MOPs \cite{Assche01}  with cubic weights and the development of the Deift-Zhou nonlinear steepest descent method that yields not only the limiting zero distribution but a detailed uniform asymptotics of the polynomial on the whole complex plane. In order to be able to perform the necessary steps of this approach we had to introduce a preliminary transformation that although looked purely technical in the first stage, turned out to be crucially related to the core of the matter.  Moreover, the quadratic differential studied in \cite{martinez_silva_critical_measures} played a fundamental role in the construction of the leading term of the asymptotics, and its critical graph (with all its phase transitions) was an essential guide during the deformations of contours required to complete the steepest descent analysis.
 
 The structure of the paper is as follows. In the next section we state our main findings that explain in particular the numerical outcomes illustrated in Figures \ref{figure_zeros_N=30} and \ref{figure_zeros_N=50}, for which we need to provide a minimum background from \cite{martinez_silva_critical_measures}.  However, the proofs require the use of several technical results from \cite{martinez_silva_critical_measures} that for convenience of the reader we gather in Sections~\ref{sec:spectralcurve} and \ref{sec:trajectories}.  As it was mentioned, our approach to asymptotics is based on the Riemann-Hilbert characterization of the MOPs (Section~ \ref{sec4:RH}) and on the corresponding Deift-Zhou steepest descent analysis. It consists of several transformations, described in Section~\ref{sec:steepestdescent}. We can use the result of this analysis to derive the strong (locally uniform) asymptotics of both type I and type II MOPs, and as a consequence, describe their limit zero behavior (Section~ \ref{sec:asymptotics}), proving in such a way our main assertions made in Section~\ref{sec:statement}. Finally, in Section~\ref{sec:numerics} we explain the computational methodology used for our numerical experiments and formulate some conjectures and empirical observations. 
 
\section{Statement of main results} \label{sec:statement}

Consider two contours $\gamma_1$ and $\gamma_2$ on the complex plane, extending to $\infty$ on their two ends along the directions determined by the angles $-\pi/3$ and $0$, and $-\pi/3$ and $\pi/3$, respectively (see Figure~\ref{figure_contours0}, left). Then expressions
\begin{align*}
\frak f_1(z) &= \int_{\gamma_1} \frac{e^{-t^3}}{z-t}\, dt, \quad z\in \C\setminus \gamma_1,  \\
\frak f_2(z) &= \int_{\gamma_2} \frac{e^{-t^3}}{z-t}\, dt, \quad z\in \C\setminus \gamma_2, 
\end{align*}
define holomorphic functions in the respective domains. Their corresponding asymptotic expansions at infinity are the formal power series (that we again denote by $\frak f_i$),
$$
\frak f_i(z)=\sum_{k=0}^\infty \frac{\frak f_i^{(k)}}{z^{k+1}}, \quad \frak f_i^{(k)}\isdef \int_{\gamma_i} t^k e^{-t^3} \, dt, \quad i=1, 2.
$$
with
\begin{equation} \label{explicitcoeff}
	\begin{split}
\frak f_1^{(k)} & \isdef \int_{\gamma_1} t^k e^{-t^3} \, dt = \frac{1}{3}\, \Gamma \left( \frac{k+1}{3} \right) \left( e^{2\pi i (k+1)/3} -1\right), \\
\frak f_2^{(k)} & \isdef \int_{\gamma_2} t^k e^{-t^3} \, dt = \frac{2i}{3} \, \Gamma \left( \frac{k+1}{3} \right) \sin \left(  2\pi i (k+1)/3  \right).
	\end{split}
\end{equation}

Given integer values $n, m\geq 0$, denote $N\isdef n+m$. \textit{Type I Hermite--Pad\'e} (or \textit{simultaneous}) \textit{approximants} to the pair of functions $(\frak f_1, \frak f_2)$ are constructed by finding polynomials $a_{n,m}$, $b_{n,m}$, with $\deg a_{n,m} \leq n-1$, $ \deg b_{n,m}\leq m-1$, not simultaneously identically zero, such that for some polynomial $d_{n,m}$,
$$
a_{n,m}(z) \frak f_1(z) +  b_{n,m}(z) \frak f_2(z) - d_{n,m}(z) = \mathcal O  \left(\frac{1}{z^N} \right), \quad z\to\infty.
$$
 \textit{Type II Hermite--Pad\'e} (or \textit{simultaneous}) \textit{approximants} to $(\frak f_1, \frak f_2)$ is the pair of rational functions $(q^{(1)}_{n,m}/p_{n,m}, q^{(2)}_{n,m}/p_{n,m})$, where $p_{n,m}\not \equiv 0$ is of degree $\leq N$, and  $q^{(j)}_{n,m}$ are polynomials such that
 \begin{align*}
 	p_{n,m}(z) \frak f_1(z) - q^{(1)}_{n,m}(z) & = \mathcal O  \left(\frac{1}{z^n} \right), \quad z\to\infty, \\
 	p_{n,m}(z) \frak f_2(z) - q^{(2)}_{n,m}(z) & = \mathcal O  \left(\frac{1}{z^m} \right), \quad z\to\infty,
 \end{align*}
 see e.g.~\cite[Chapter 4]{nikishin_sorokin_book}. Standard arguments using Cauchy formula allow to show that $a_{n,m}$, $b_{n,m}$ and $p_{n,m}$ satisfy non-hermitian orthogonality conditions that we formulate next, after appropriate rescaling.
\begin{definition} \label{defTypeI}
	Given $n, m\in \N$ and $N=n+m$, the {\it type I multiple orthogonal polynomials} $A_{n,m}$ and $B_{n,m}$, if they exist, are uniquely defined by the following conditions:
	$$
	\deg A_{n,m} \leq n-1, \quad \deg B_{n,m}\leq  m-1,
	$$
	and
	\begin{equation}\label{mops_conditionsTypeI}
	\begin{aligned}
	\int_{\gamma_1}z^k A_{n,m}(z) e^{-Nz^3}dz + 	\int_{\gamma_2}z^k B_{n,m}(z) e^{-Nz^3}dz =0, & \quad k=0,\hdots, N-2, \\
	\int_{\gamma_1}z^k A_{n,m}(z) e^{-Nz^3}dz + 	\int_{\gamma_2}z^k B_{n,m}(z) e^{-Nz^3}dz =1, & \quad k=N-1.
	\end{aligned}
	\end{equation}
\end{definition}
The dual notion is
\begin{definition} \label{defTypeII}
	Given $n, m\in \N$ and $N=n+m$, we define the {\it type II multiple orthogonal polynomial} $P_{n,m}$, if it exists, to be the unique monic polynomial of degree $N$ ($\deg P_{n,m} = N$) that fulfills the conditions
	\begin{equation}\label{mops_conditions}
	\begin{aligned}
	\int_{\gamma_1}z^k P_{n,m}(z) e^{-Nz^3}dz=0, & \quad k=0,\hdots, n-1, \\
	\int_{\gamma_2}z^k P_{n,m}(z) e^{-Nz^3}dz=0, & \quad k=0,\hdots, m-1.
	\end{aligned}
	\end{equation}
\end{definition}

Both  \eqref{mops_conditionsTypeI}   and \eqref{mops_conditions}   give rise to a non-homogeneous linear system of $n+m=N$ equations on the $N$ (unknown) coefficients of $P_{n,m}$ (in the case of \eqref{mops_conditions}) or of the two polynomials $A_{n,m}, B_{n,m}$ otherwise. Hence if the solution exists, it is unique. However, in contrast to the standard orthogonality on the real line, the existence of $P_{n,m}$ or of $A_{n,m}, B_{n,m}$is a non-trivial matter (and actually not true for general systems of weights).

The main object of our study is the asymptotic analysis of the type I and type II polynomials defined above, in the limit when $N=n+m\to \infty$ in such a way that
$$
\frac{n}{N} \longrightarrow \alpha \in [0,1].
$$
It should be noted that for \eqref{mops_conditions} the extremal cases $n=0$ and $n=m$ (for type II polynomials) have been studied previously by Deaño, Huybrechs and Kuijlaars \cite{deano_kuijlaars_huybrechs_complex_orthogonal_polynomials} and Filipuk, Van Assche and Zhang \cite{filipuk_vanassche_zhang}, respectively; in fact, these studies were the original motivation for \cite{martinez_silva_critical_measures} and the present work. 

We can restrict our analysis to the case
$$
\frac{n}{N} \longrightarrow \alpha \in (0,1/2),
$$
since we can easily derive conclusions for $\alpha\in (1/2,1)$ by mapping $\alpha \mapsto 1-\alpha$, swapping the role of $\mathfrak f_1$ and $\mathfrak f_2$ and rotating the complex plane by $2\pi/3$.   

As it was mentioned, we make use of the Riemann-Hilbert formulation of the multiple orthogonality that characterizes both $P_{n,m}$ and $A_{n,m}, B_{n,m}$, see \cite{Assche01} and Section~\ref{sec4:RH} below. The advantage of this approach is that it gives the uniform asymptotics of these polynomials on every subdomain of the complex plane and that it allows to perform the asymptotic analysis valid simultaneously for both families of polynomials. The limiting behavior is described in terms of the solution of a vector equilibrium problem and of some abelian integrals on the associated Riemann surface, whose construction relies on the analysis carried out in \cite{martinez_silva_critical_measures}. 

For a compactly supported Borel measure $\mu$, its logarithmic potential and Cauchy transform are defined by
$$
U^{\mu}(z)=\int\log\frac{1}{|s-z|}d\mu(s),\quad C^{\mu}(z)=\lim_{\varepsilon\to 0}\int_{|s-z|>\varepsilon}\frac{d\mu(s)}{s-z},\quad z\in \C,
$$
respectively. $U^\mu$ is subharmonic on $\C$ and harmonic on $\C\setminus\supp\mu$, whereas $C^{\mu}$ is well defined and finite a.e. with respect to planar Lebesgue measure, and it is analytic on $\C\setminus \supp\mu$. For
$$
\partial_z=\frac{1}{2}\left( \frac{\partial}{\partial x}-i\frac{\partial}{\partial y} \right),
$$
these functions are related through the equality
$$
2\partial_z U^{\mu}= C^\mu,
$$
which should be understood in the strong sense for $z\in \C\setminus \supp\mu$ and in the distributional sense on $\supp\mu$.

For a given vector of three non-negative measures $\vec\mu=(\mu_1,\mu_2,\mu_3)$, the non-negative definite interaction matrix
\begin{equation*}
A=(a_{jk})=
\begin{pmatrix}
 1 & \frac{1}{2} & \frac{1}{2} \\
 \frac{1}{2} & 1 & -\frac{1}{2} \\
 \frac{1}{2} & -\frac{1}{2} & 1
\end{pmatrix},
\end{equation*}
and the external field $\phi(z)=\re (z^3)$, we consider the energy functional
\begin{equation}\label{vector_energy}
E(\vec\mu)=\sum_{j,k=1}^3 a_{jk}I(\mu_j,\mu_k) +\int \phi(z)d\mu_1(z) + \int \phi(z)d\mu_2(z),
\end{equation}
where
$$
I(\mu,\nu)=\iint \log\frac{1}{|x-y|}d\mu(x)d\nu(y)
$$
is the logarithmic interaction between the two measures $\mu,\nu$.

For a fixed number $\alpha\in [0,1/2)$, we restrict the energy $E(\cdot)$ on the class $\mathcal M_\alpha$ of vectors of measures $\vec \mu$ satisfying the following assumptions.

\begin{itemize}
\item The components $\mu_1$, $\mu_2$ and $\mu_3$ of $\vec\mu \in \mathcal M_\alpha$ are compactly supported non-negative Borel measures on $\C$, their supports have zero planar Lebesgue measure and furthermore the energy $E(\vec\mu)$ is finite. 

\item The total masses of $\mu_1$, $\mu_2$ and $\mu_3$ are related through
\begin{equation}\label{massconstraints}
|\mu_1|+|\mu_2|=1,\quad |\mu_1|+|\mu_3|=\alpha,\quad |\mu_2|-|\mu_3|=1-\alpha.
\end{equation}

\item The set
$$
\bigcup_{1\leq j <k \leq 3} (\supp\mu_j\cap \supp\mu_k)
$$
is finite.
\end{itemize}

In \cite{martinez_silva_critical_measures} we extended the notion of critical measures introduced by the first author and Rakhmanov to the vector energy setting. In the present context, this extension reads as follows. For $t\in \C$ and $h\in C^2(\C)$, denote by $\mu^t$ the pushforward of the scalar measure $\mu$ induced by the transformation $z\mapsto z+th(z)$, $z\in \C$. For $\mu\in \mathcal M_\alpha$, we denote $\vec\mu^t=(\mu_1^t,\mu_2^t,\mu_3^t)$ and say that $\vec \mu$ is {\it critical} if
$$
\lim_{t\to 0} \frac{E(\vec \mu^t)-E(\mu)}{t}=0,
$$ 
for every $h\in C^2(\C)$.

A summary of the results proven in \cite{martinez_silva_critical_measures} that we need here is given by the next theorem.

\begin{thm}[{\cite[Theorems 1.12 and 1.14]{martinez_silva_critical_measures}}]\label{existence_critical_measure_cubic}
For each $\alpha \in (0,1/2)$ there exists a vector critical measure $\vec\mu_\alpha\in \mathcal M_\alpha$ for the energy $E(\cdot)$ defined in \eqref{vector_energy}. Furthermore, the combinations of Cauchy transforms
 \begin{equation}\label{definition_xi_functions}
\begin{aligned}
\xi_1(z) & =2z^2+C^{\mu_1}(z)+C^{\mu_2}(z),\\
\xi_2(z) & =-z^2-C^{\mu_1}(z)-C^{\mu_3}(z),\\
\xi_3(z) & =-z^2-C^{\mu_2}(z)+C^{\mu_3}(z)
\end{aligned}
\end{equation}
satisfy a.e. the algebraic equation (also called spectral curve)
\begin{equation}\label{spectral_curve}
\xi^3-R(z)\xi+D(z)=0
\end{equation}
where
\begin{equation}\label{RD}
R(z)= 3z^4-3z-c,\quad D(z)=-2z^6+3z^3+cz^2-3\alpha(1-\alpha),
\end{equation}
with $c=c(\alpha)$ the real parameter given by
\begin{equation}\label{c}
c=-\left(\frac{243}{64}(1-4\alpha(1-\alpha))^2\right)^{\frac{1}{3}}.
\end{equation}

The support of each of the measures $\mu_1$ and $\mu_3$ is a single interval of the real line, whereas $\supp\mu_2$ is a simple piece-wise analytic arc, symmetric with respect to complex conjugation, that connects two complex conjugate points.

Finally, there exists a critical value $\alpha_c \approx 0.2578357$ for which
\begin{enumerate}[(i)]
\item (subcritical regime) if $\alpha<\alpha_c$, then $\supp\mu_3=\emptyset$ and $\mu_1$ and $\mu_2$ are disjointly supported non-trivial measures;
\item (supercritical regime) if $\alpha>\alpha_c$, then the three measures $\mu_1,\mu_2$ and $\mu_3$ are non-trivial and their supports have one common intersection point, $a_*$.
\end{enumerate} 
\end{thm}

The measures $\mu_1,\mu_2$ and $\mu_3$ obviously depend on $\alpha$, so whenever necessary we write $\mu_j=\mu_{\alpha,j}$ to emphasize this dependency.

We denote the endpoints of $\supp\mu_2$ by $a_2$ and $b_2$, with the convention that $\im a_2<0$. Furthermore, in the subcritical case, $\supp\mu_1=[a_1,b_1]$ and in the supercritical case, $\supp\mu_1=[a_*,b_1]$ and $\supp\mu_3=[a_1,a_*]$. In either cases, the points $a_1$, $a_2$, $b_1$ and $b_2$ are branch points of the spectral curve \eqref{spectral_curve}, and $a_*$ is also the unique intersection point of $\supp\mu_2$ with $\R$.

The structure of $\Delta_j=\supp\mu_j$, $j=1, 2, 3$ (based on actual numerical evaluation), is displayed in Figure \ref{figure_numerics_supports}.

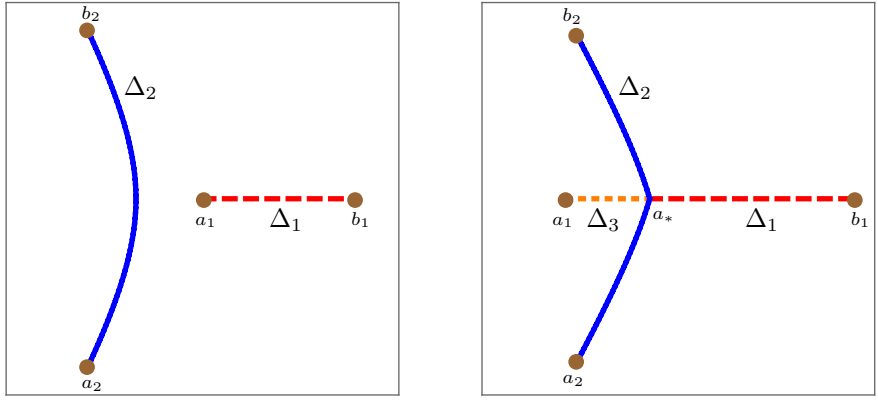
\begin{figure}
\begin{subfigure}{.5\textwidth}
\centering
\begin{overpic}[scale=1]{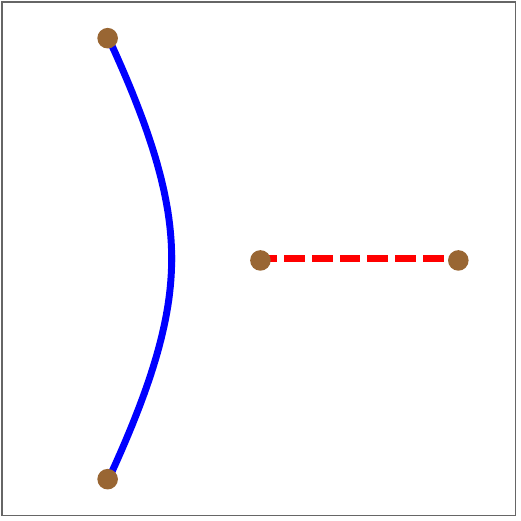}
	\put(100,65){$\Delta_1$}
	\put(45,115){$\Delta_2$}
	\put(72,66){\scriptsize $a_1$}
		\put(131,66){\scriptsize $b_1$}
		\put(29,4){\scriptsize $a_2$}
	\put(29,144){\scriptsize $b_2$}
\end{overpic}
\end{subfigure}%
\begin{subfigure}{.5\textwidth}
\centering
\begin{overpic}[scale=1]{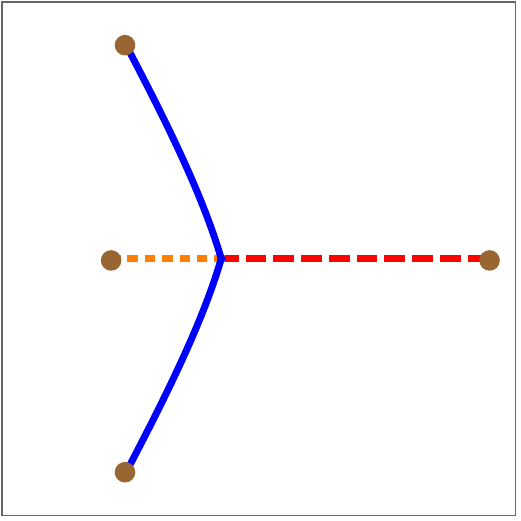}
	\put(100,65){$\Delta_1$}
	\put(52,115){$\Delta_2$}
	\put(40,65){$\Delta_3$}
		\put(27,66){\scriptsize $a_1$}
	\put(140,66){\scriptsize $b_1$}
	\put(31,6){\scriptsize $a_2$}
	\put(31,143){\scriptsize $b_2$}
			\put(65,68){\scriptsize $a_*$}
\end{overpic}
\end{subfigure}
\caption{(Extracted from \cite{martinez_silva_critical_measures}) For $\alpha \approx 0.14786<\alpha_c \approx 0.2578357$ (left panel) and for $\alpha\approx 0.348342>\alpha_c$ (right panel), the output of the numerical evaluation of $\Delta_1\isdef \supp\mu_1$ (long dashed 
line), $\Delta_2\isdef \supp\mu_2$ (continuous line) and $\Delta_3\isdef \supp\mu_3$ (short dashed line - only on the right panel). }\label{figure_numerics_supports}
\end{figure}

The existence of the vector critical measure $\vec\mu$ claimed by Theorem \ref{existence_critical_measure_cubic} is highly non-trivial, and a great portion of the work \cite{martinez_silva_critical_measures} is dedicated to this proof. In \cite{martinez_silva_critical_measures} the measures $\mu_1$, $\mu_2$ and $\mu_3$ are described rather explicitly in terms of critical trajectories of a quadratic differential defined on the Riemann surface associated to the spectral curve \eqref{spectral_curve}. This quadratic differential also plays a fundamental role in the analysis of the $g$-functions carried out in Section \ref{sec:steepestdescent} below.

Denote by
$$
\nu(Q)\isdef \frac{1}{N}\sum_{Q(w)=0} \delta_w,
$$
where $Q$ is one of the polynomials $P_{n,m}$, $A_{n,m}$ and $B_{n,m}$, if it exist. Notice that  each zero is counted in the sum above according to its multiplicity.

As a first result, we relate the weak limit, which we denote by $\stackrel{*}{\to}$, of the  zero counting measures for the type I ($A_{n,m}$) and type II ($P_{n,m}$)  polynomials with the vector critical measure given by Theorem \ref{existence_critical_measure_cubic}: 
\begin{thm}\label{theorem_zero_counting_measure}
Let $\alpha\in (0,1/2)\setminus \{\alpha_c\}$. If $\vec\mu_\alpha=(\mu_1, \mu_2, \mu_3)$ is the corresponding vector critical measure given by Theorem \ref{existence_critical_measure_cubic}, then for all sufficiently large $n$ and $m$ such that
$$
\frac{n}{N} \longrightarrow \alpha 
$$
polynomials $P_{n,m}$ and $A_{n,m}$ exist, and  
$$
\nu(P_{n,m})\stackrel{*}{\to} \mu_1+\mu_2, \quad \nu(A_{n,m})\stackrel{*}{\to} \mu_1+\mu_3.
$$
\end{thm}

For illustration, the support of $\mu_1+\mu_2$ together with the zeros of $P_{n,m}$ for various choices of $n,m$ are shown in Figure \ref{figure_zeros_traj_P}, whereas the zeros of $A_{n,m}$ are shown in Figures \ref{figure_zeros_traj_AB} and \ref{figure_zeros_traj_AB_2}.

\begin{figure}
\begin{subfigure}{.5\textwidth}
\centering
\begin{overpic}[scale=.45]{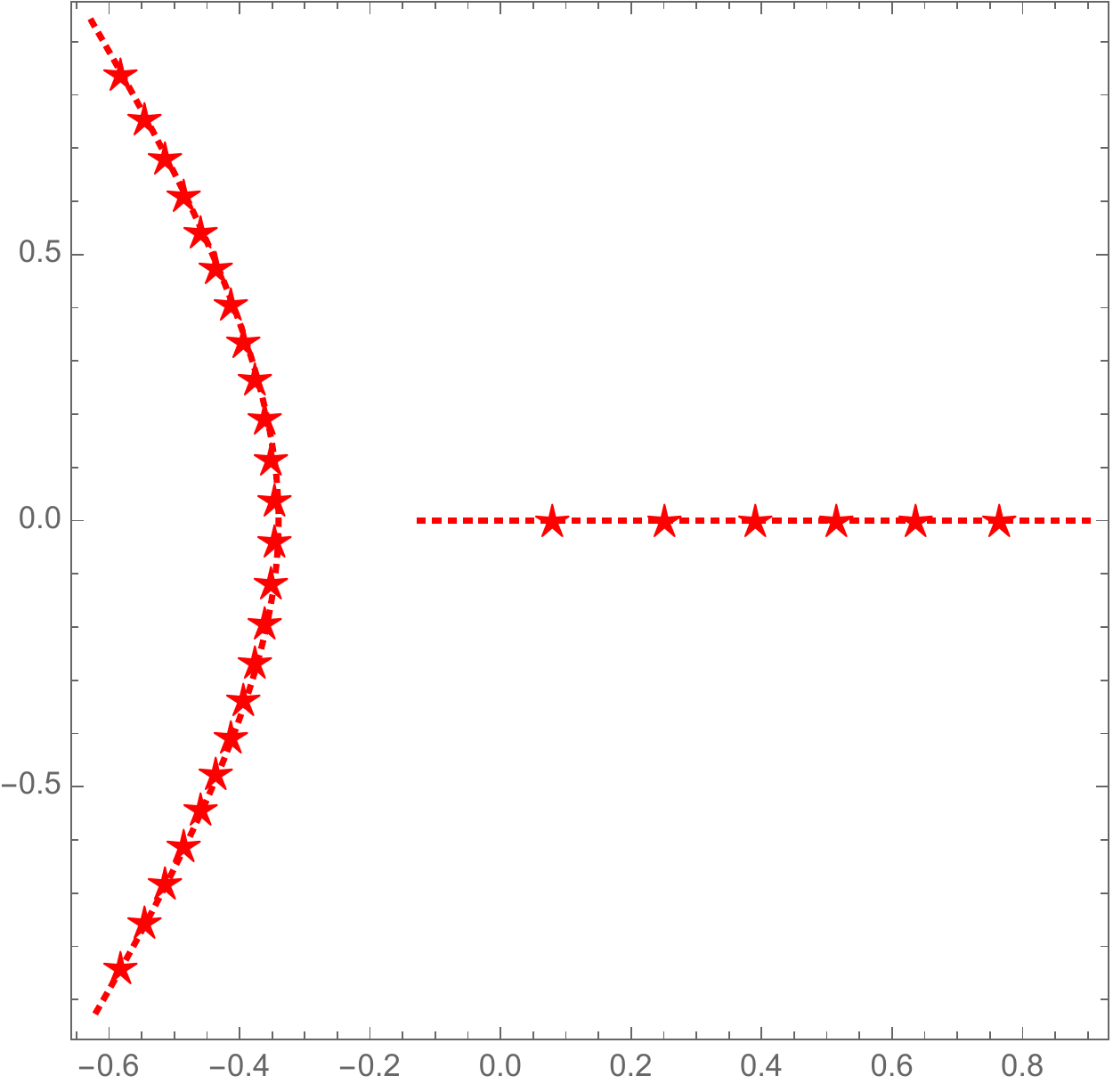}
\end{overpic}
\end{subfigure}%
\begin{subfigure}{.5\textwidth}
\centering
\begin{overpic}[scale=.45]{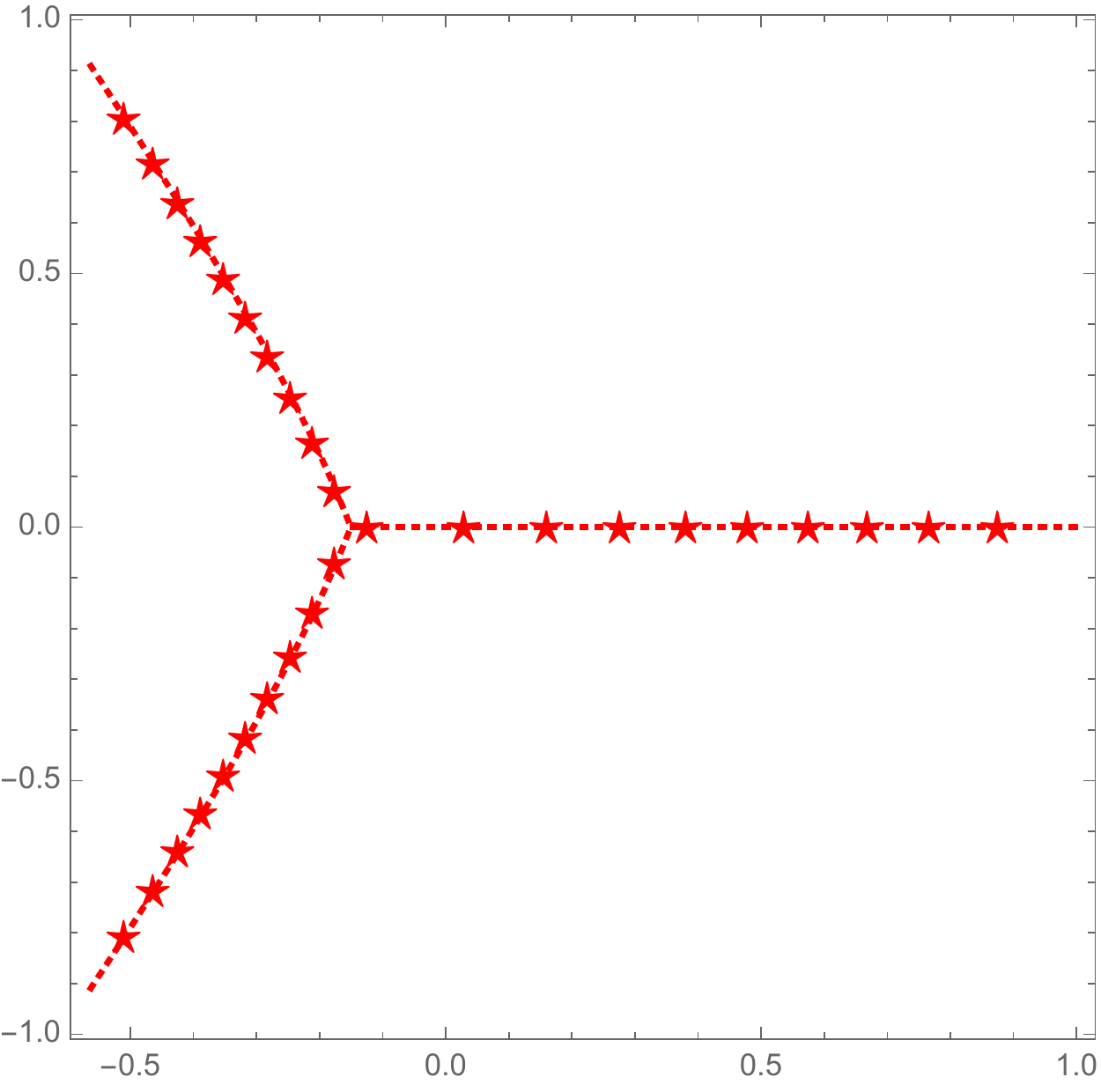}
\end{overpic}
\end{subfigure}
\caption{Zeros of $P_{6,24}$ ($\alpha=0.2<\alpha_c$, left) and $P_{11,19}$ ($\alpha\approx 0.366>\alpha_c$, right), along with the support of $\mu_1+\mu_2$.}\label{figure_zeros_traj_P}
\end{figure}

Now we switch to the description of the weak limit of the zero-counting measure for polynomials $B_{n,m}$.

Recall \cite[Ch. II]{saff_totik_book} that if $\Gamma$ is a closed set of positive capacity on $\C$ and $\mu$ is a positive measure on $\C$, the \textit{balayage} of $\mu$ onto $\Gamma$ gives us a unique measure, denoted by $\nu= \Bal( \mu;\Gamma)$, such that  $\supp \nu\subset \Gamma$,
$$
|\nu|=|\mu|, \qquad U^{\nu}(z)=U^{\mu}(z)+\kappa ,\quad  \text{ quasi-everywhere on  $\Gamma$}, 
$$
for some constant $\kappa $. For a signed measure $\mu=\mu_{+}-\mu_{-}$ we define its balayage by $\Bal( \mu;\Gamma)\isdef \Bal( \mu_+;\Gamma)- \Bal( \mu_-;\Gamma)$. 

The following theorem shows that for $\alpha\in (0,1/2)$, among all connected sets containing $b_2$ and at least one of the two points $a_1$ and $a_2$ there is a unique set $E_\alpha$, made of analytic arcs, such that  $\Bal( \mu_2-\mu_3 ; E_\alpha)$ is a positive measure, which will be precisely the asymptotic zero distribution of $B_{n,m}$ when $n/N\to \alpha$. 

We will denote
\begin{equation}\label{upper_lower_half_plane_def}
\H_+ \isdef \{ z\in \C \; \mid \; \im z>0 \}, \quad \H_- \isdef \{ z\in \C \; \mid \; \im z<0 \}.
\end{equation}

\begin{thm}\label{theorem_zero_counting_measureB}
	Let $\alpha\in (0,1/2)\setminus \{\alpha_c\}$. If $\vec\mu_\alpha=(\mu_1, \mu_2, \mu_3)$ is the corresponding vector critical measure given by Theorem \ref{existence_critical_measure_cubic}, then for all sufficiently large $n$ and $m$ such that
	$$
	\frac{n}{N} \longrightarrow \alpha, 
	$$
	the polynomials $B_{n,m}$ exist and there exists the weak limit
	$$
	\quad \nu(B_{n,m})\stackrel{*}{\to} \mu_B.
	$$
	The measure $\mu_B$ is the balayage of $\mu_2-\mu_3$ onto a set $E_\alpha$, a finite union of analytic arcs with the following properties:
	\begin{enumerate}[(i)]
		\item (subcritical regime) if $0<\alpha<\alpha_c\approx 0.2578357$, then $E_\alpha=\Delta_2=\supp\mu_2$ and $\mu_B=\mu_2$.
		\item (intermediate regime) if $\alpha_c< \alpha< \alpha_2\approx 0.354933$, then $E_\alpha $ is the union of three analytic arcs starting from $a_1$, $a_2$ and $b_2$, and ending at one common point $a_B \in \Delta_2$ on the lower half plane $\H_-$. Furthermore, 
		\begin{equation}\label{UpperHalfPlane}
	E_\alpha \cap \H_+=\Delta_2\cap \H_+.
		\end{equation}
		
		\item (supercritical regime) for $\alpha>\alpha_2$, the set $E_\alpha $ is a single analytic arc from $b_2$ to $a_1$, and \eqref{UpperHalfPlane} still holds.  
	\end{enumerate} 
\end{thm}

\begin{figure}
\begin{subfigure}{.5\textwidth}
\centering
\begin{overpic}[scale=.45]{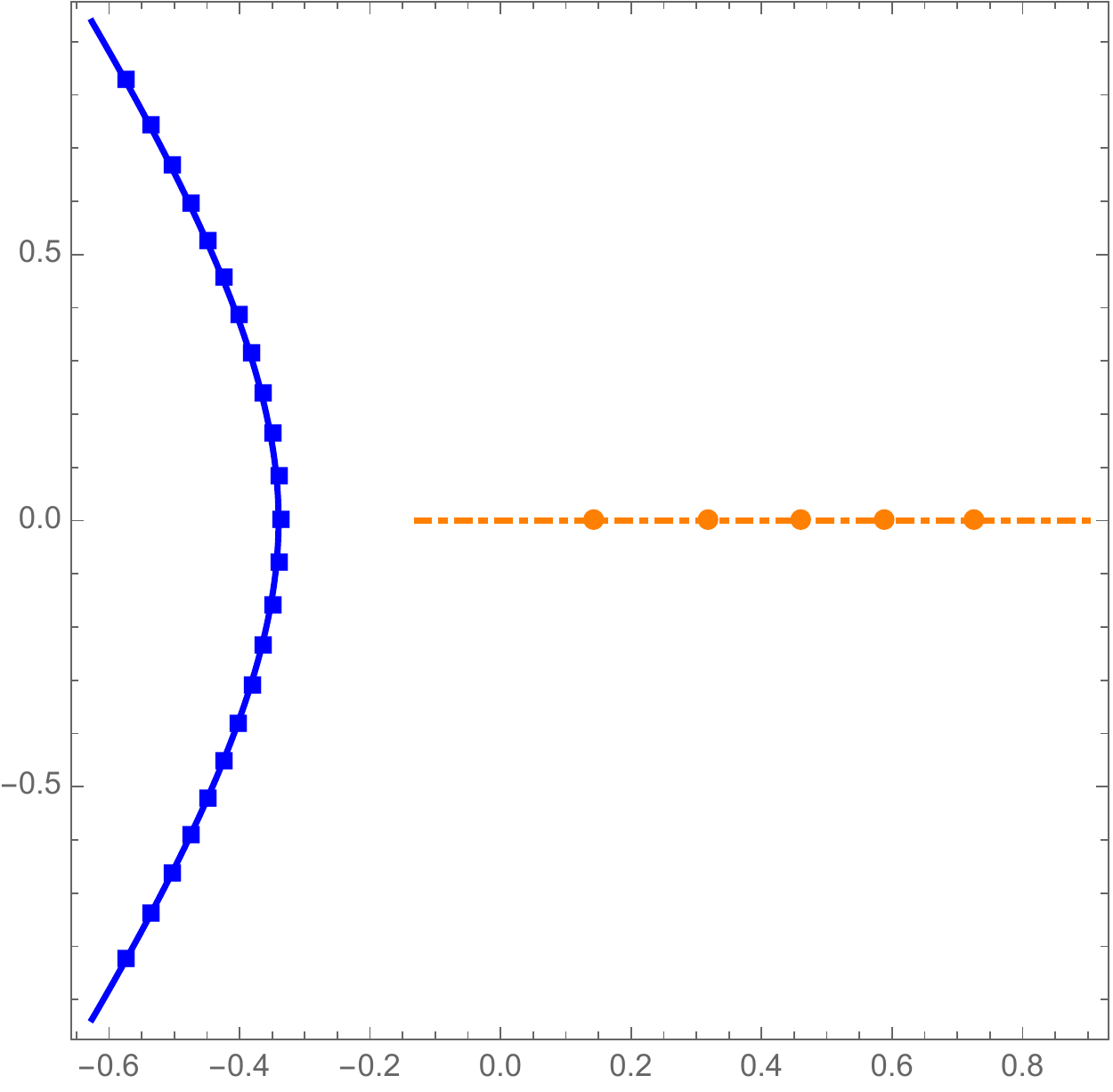}
\end{overpic}
\end{subfigure}%
\begin{subfigure}{.5\textwidth}
\centering
\begin{overpic}[scale=.45]{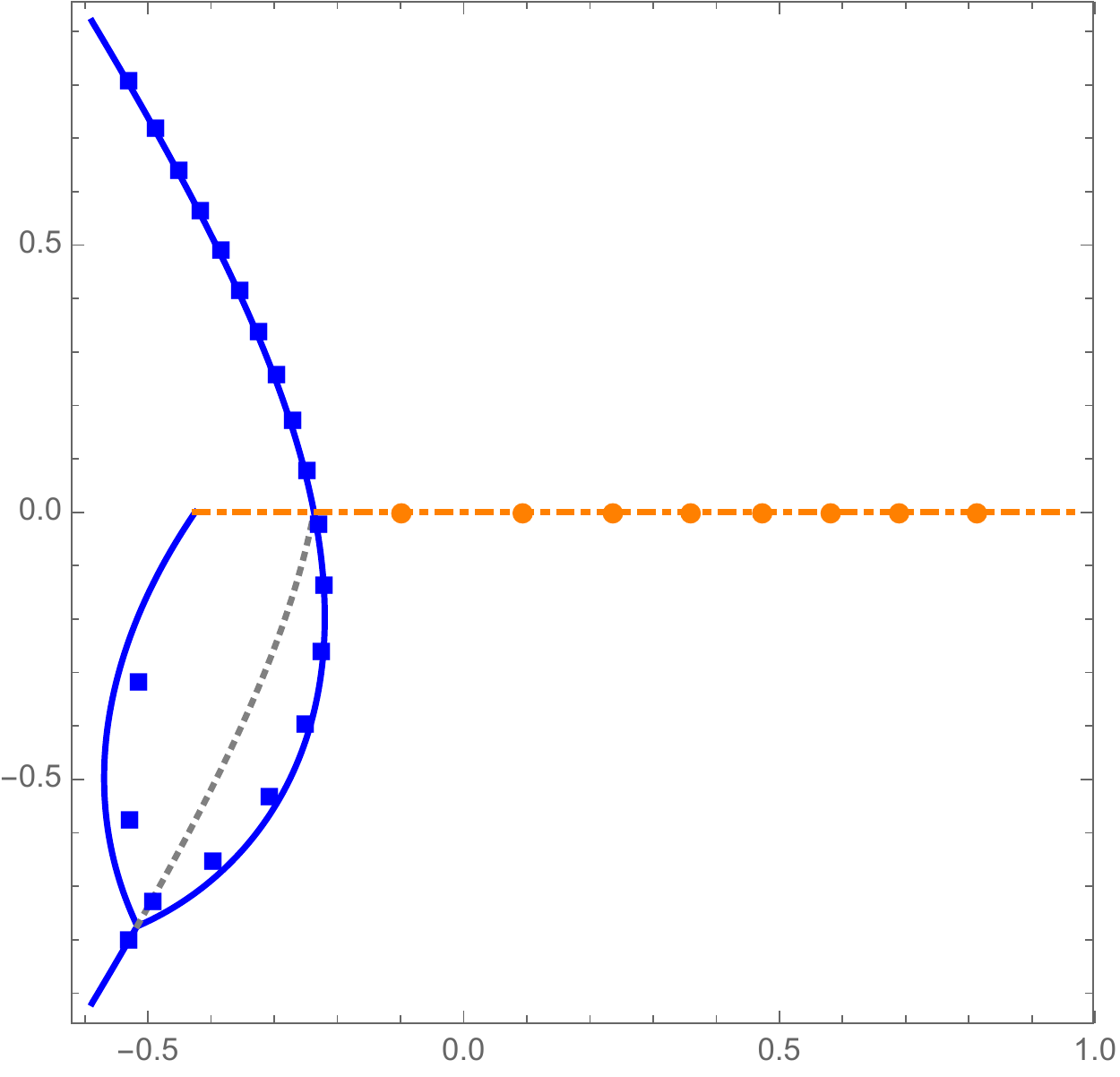}
\end{overpic}
\end{subfigure}\\
\begin{subfigure}{1\textwidth}
\centering
\begin{overpic}[scale=.45]{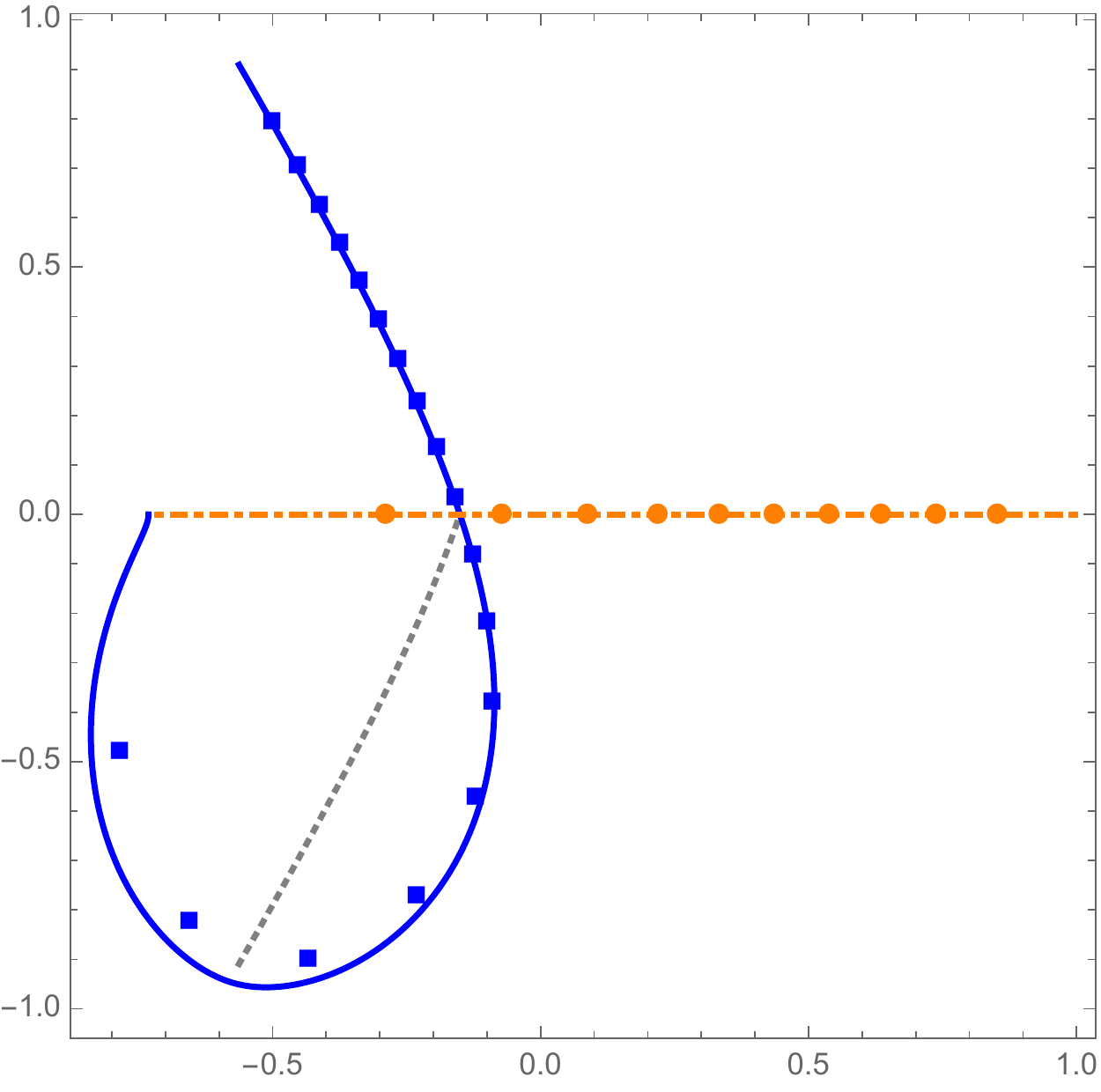}
\end{overpic}
\end{subfigure}
\caption{For $n+m=30$, the zeros of $A_{n,m}$ (dots) and $B_{n,m}$ (squares), along with the supports of $\mu_1+\mu_3$ (dot-dashed lines) and $\mu_B$ (solid lines). They are numerical outcomes for the values $(n,m)=(6,24)$ ($\alpha=0.2<\alpha_c$, top left), $(n,m)=(9,21)$ ($\alpha=0.3\in (\alpha_c,\alpha_2)$, top right) and $(n,m)=(11,19)$ ($\alpha \approx 0.366 >\alpha_2$, bottom). In the latter two, the dashed curve on the lower half plane is the part of $\supp\mu_2$ that does not coincide with $\supp\mu_B$.  }\label{figure_zeros_traj_AB}
\end{figure}

\begin{figure}
\begin{subfigure}{.5\textwidth}
\centering
\begin{overpic}[scale=.45]{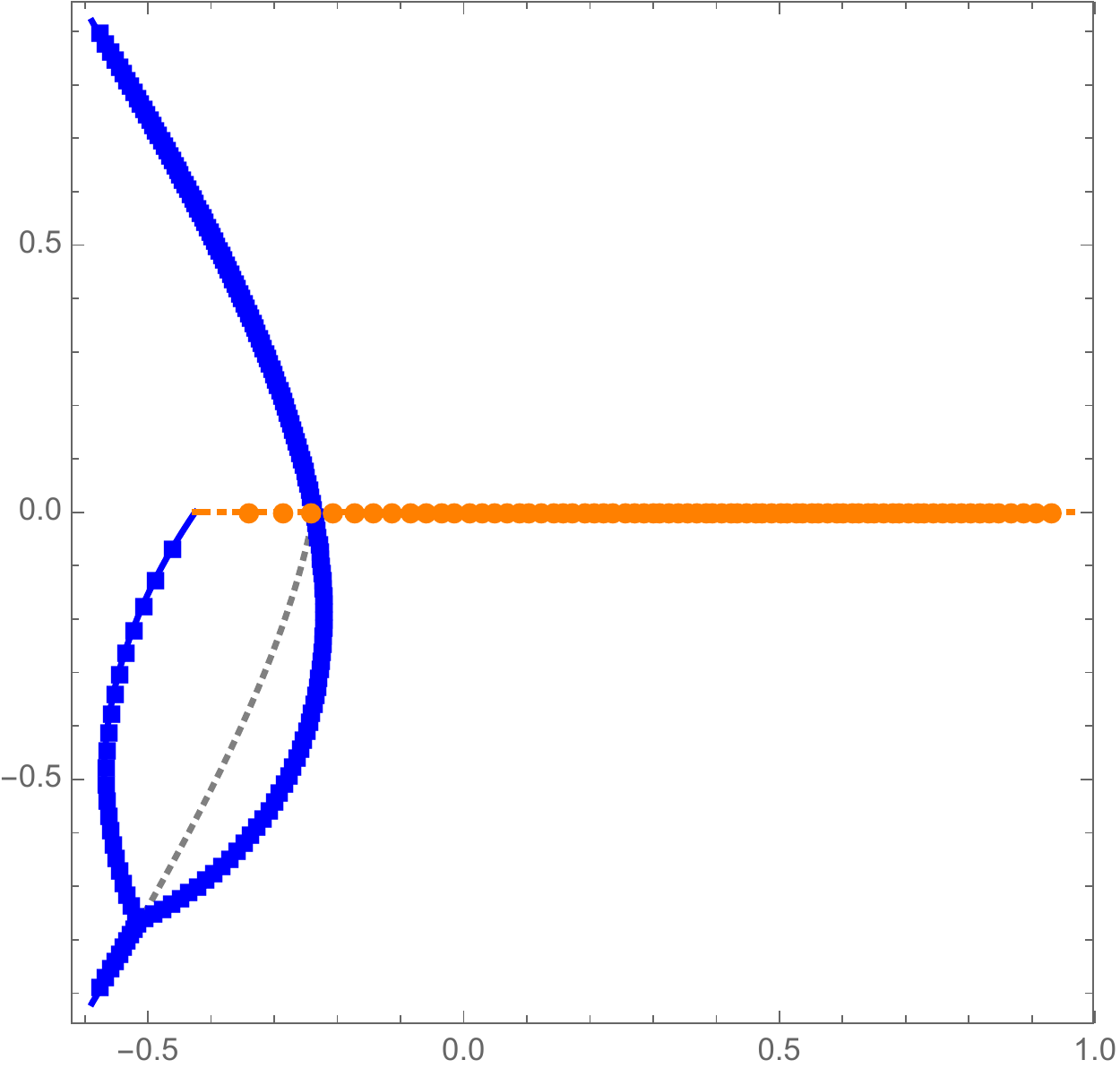}
\end{overpic}
\end{subfigure}%
\begin{subfigure}{.5\textwidth}
\centering
\begin{overpic}[scale=.45]{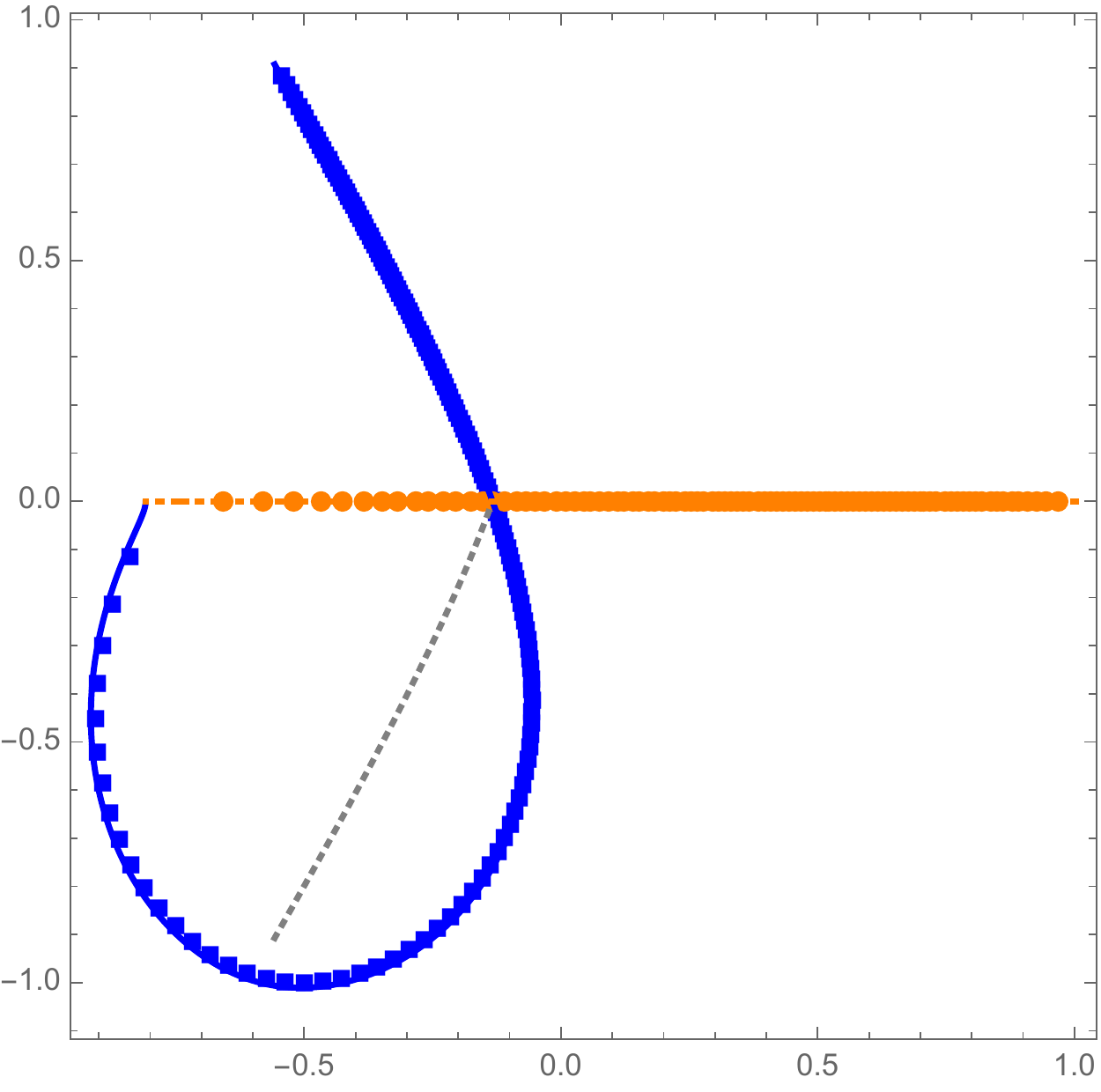}
\end{overpic}
\end{subfigure}
\caption{For $n+m=250$, the zeros of $A_{n,m}$ (dots) and $B_{n,m}$ (squares), along with the supports of $\mu_1+\mu_3$ (dot-dashed lines), $\mu_B$ (solid lines) and the part of $\supp\mu_2$ on the lower half plane that does not coincide with $\supp\mu_B$. These figures are numerical outcomes for the values $(n,m)=(75,175)$ ($\alpha=0.3\in (\alpha_c,\alpha_2)$, left) and $(n,m)=(95,155)$ ($\alpha=0.38>\alpha_2$, right).  }\label{figure_zeros_traj_AB_2}
\end{figure}

For illustration, the support of $\mu_B$ together with the zeros of $B_{n,m}$ for various choices of $\alpha$ are shown in Figures \ref{figure_zeros_traj_AB} and \ref{figure_zeros_traj_AB_2}.

\begin{remark}
The restrictions $\alpha\notin \{ 0, 1/2\}$ in Theorems \ref{theorem_zero_counting_measure}--\ref{theorem_zero_counting_measure} are made only to be consistent with our previous work \cite{martinez_silva_critical_measures}. Nevertheless, the case $n=0$ (corresponding to $\alpha=0$) was previously studied by Dea\~no, Huybrechs and Kuijlaars \cite{deano_kuijlaars_huybrechs_complex_orthogonal_polynomials}. In this case, $P_{0,m}=B_{0,m}$, $A_{0,m}=0$, $\mu_1=\mu_3=0$ and $|\mu_2|=1$. Filipuk, Van Assche and Zhang \cite{filipuk_vanassche_zhang} studied the polynomial $P_{n,m}$ for arbitrary, but finite, $n$ and $m$. When  $n=m$ (which corresponds to $\alpha=1/2$), the polynomial $P_{n,n}$ is invariant under the rotation $z\mapsto e^{2\pi i/3}z$, and using this symmetry they could also obtain asymptotics of $P_{n,n}$ as $n\to\infty$. In our language, their results say that 
\begin{equation}\label{symmetry_measures}
d\mu_2\restr{(z)}{\arg z=\pm\frac{2\pi}{3}}=d\mu_1(e^{\mp 2\pi i/3}z),
\end{equation}
and the convergence $\nu(P_{n,n})\stackrel{*}{\to}\mu_1+\mu_2$ still takes place. However, in this situation $\supp\mu_3=(-\infty,0]$, so the vector $(\mu_1,\mu_2,\mu_3)$ is not critical in the strict sense of the definition used for Theorem \ref{existence_critical_measure_cubic}, where unbounded components of the support were not allowed.  Polynomials of type I $A_{n,n}$ and $B_{n,n}$ were not considered in \cite{filipuk_vanassche_zhang}; however, their limiting zero distribution still exist and satisfy Theorems \ref{theorem_zero_counting_measure} and \ref{theorem_zero_counting_measureB}, with the feature of the symmetry  \eqref{symmetry_measures} and that
$$
\restr{\mu_B}{\H_+}=\restr{\mu_2}{\H_+},
$$
and
$$
\supp\mu_B\cap \H_-=[0,e^{-\pi i/3}\infty),\quad  d\mu_B\restr{(z)}{\H_-}=d\mu_3(e^{-2\pi i/3}z).
$$

We also strongly believe that Theorems \ref{theorem_zero_counting_measure} and \ref{theorem_zero_counting_measureB} still hold true for $\alpha=\alpha_c$. For a rigorous proof in the exceptional cases just described we need to modify the Riemann-Hilbert analysis that we carry out below. Such modifications have no additional difficulty, but are lengthy, which motivated our decision to omit them for the sake of brevity.
\end{remark}

We need to introduce further notation in order to formulate the strong asymptotic results, from which Theorems \ref{theorem_zero_counting_measure} and \ref{theorem_zero_counting_measureB} readily follow.

For $\xi_j$ as in \eqref{definition_xi_functions}, we denote by $g_j$ their primitive functions,
$$
g_j(z)=\int ^z \xi_j(s)ds,\quad j=1,2, 3. 
$$
They have different domains of definition, $g_1$ is defined in $\C\setminus (\supp\mu_1\cup\supp\mu_2)$, $g_2$ in $\C\setminus (\supp\mu_1\cup\supp\mu_3)$, and $g_3$, in $\C\setminus (\supp\mu_2\cup\supp\mu_3)$. Furthermore, $g_1$ and $g_2$ are well defined modulo $2\pi i$, so that functions $\exp(g_1)$ and $\exp(g_2)$ are single-valued in their domain of definition. These $g$-functions do depend on $\alpha$, so when necessary we will emphasize it by $g_j=g_{\alpha,j}$.

Moreover, let $f_1=(f_{11}, f_{12}, f_{13})$ be a vector of holomorphic non-vanishing functions, where $f_{1j}$ has the same domain of definition than $g_j$, $j=1, 2, 3$. These functions will be uniquely characterized by a Riemann-Hilbert (boundary value) problem and constructed explicitly in terms of certain abelian integral on the Riemann surface of the spectral curve \eqref{RD}, see Section~\ref{sec:globalparam} for details. Furthermore, 
$$
f_{11}(z)= 1 + \mathcal O(1/z), \quad f_{12}(z)= \mathcal O(1/z), \quad f_{13}(z)= \mathcal O(1/z), \qquad z\to \infty.
$$

\begin{thm} \label{thm:main}
For $n,m$ large enough such that $\alpha_N\isdef n/N\to \alpha\in (0,1/2)\setminus \{\alpha_c\}$, $N=n+m$, the polynomial $P_{n,m}$ exists, and for an appropriate constant $r_1=r_{\alpha_N,1}$, 
\begin{equation*} 
P_{n,m}(z)=f_{11}(z)e^{-N(g_{\alpha_N,1}(z)-\frac{2z^3}{3}-r_1)}(1+\Boh(N^{-1})),\quad N\to\infty,
\end{equation*}
holds locally uniformly in $\C\setminus (\supp\mu_{\alpha,1}\cup\supp\mu_{\alpha,2})$.
\end{thm}
Since we take $P_{n,m}$ monic, constant $r_1$ can be deduced from the normalization
$$
\lim_{z\to \infty} \left( g_{\alpha_N,1}(z)-\frac{2z^3}{3}+\log z  \right) =r_1.
$$

\begin{thm} \label{thm:main2}
	For $n,m$ large enough such that $\alpha_N\isdef n/N\to \alpha\in (0,1/2)\setminus \{\alpha_c\}$, $N=n+m$, the polynomial $A_{n,m}$ exists, and for an appropriate constant $r_2=r_{\alpha_N,2}$, 
	\begin{equation*} 
2\pi i 	A_{n,m}(z)= f_{12}(z)e^{N(g_{\alpha_N,2}(z)+\frac{z^3}{3}-r_2)}(1+\Boh(N^{-1})),\quad N\to\infty,
	\end{equation*}
	locally uniformly in  $\C\setminus (\supp\mu_{\alpha,1}\cup\supp\mu_{\alpha,3})$.
\end{thm}

As we can infer from Theorem~\ref{theorem_zero_counting_measureB}, the asymptotic formula for $B_{n,m}$  is a little bit more involved, and requires introducing a sub-domain of the plane, that we denote by $\Omega_\alpha$:
\begin{enumerate}[(i)]
	\item (subcritical regime) if $0<\alpha<\alpha_c$, then $\Omega_\alpha=\emptyset$;
	\item (supercritical regime) for $\alpha>\alpha_c$, then $\Omega_\alpha$ is the bounded connected component of $\H_- \setminus E_\alpha$.
\end{enumerate}

\begin{thm} \label{thm:main3}
	For $n,m$ large enough such that $\alpha_N\isdef n/N\to \alpha\in (0,1/2)\setminus \{\alpha_c\}$, $N=n+m$, the polynomial $B_{n,m}$ exists, and for an appropriate constant $r_3=r_{\alpha_N,3}$, 
	\begin{multline*}\label{asymptotic_formula_mopsII}
	2\pi i  B_{n,m}(z) = \\ \begin{cases}
	f_{13}(z)   e^{ N(g_{\alpha_N,3}(z)+\frac{1}{3}z^3-r_3)}
	\left( 1 + \mathcal O(1/N)  \right) , & \text{loc. uniformly in } \C\setminus (\Omega_\alpha \cup \Delta_2),\\
	- f_{12}(z)   e^{ N(g_{\alpha_N,2}(z)+\frac{1}{3}z^3-r_3)}  
	\left( 1 + \mathcal O(1/N)  \right), & \text{loc. uniformly in } \Omega_\alpha.  
	\end{cases}
	\end{multline*}
\end{thm}

\begin{remark}
In Theorems \ref{thm:main}--\ref{thm:main3}, the functions $f_{1j}$ are defined on the Riemann surface $\mathcal R$ of the spectral curve \eqref{spectral_curve} for the limiting value $\alpha$, so they do not depend on $\alpha_N$. As it was mentioned earlier, they can be constructed  in terms of certain abelian integral on $\mathcal R$. Taking advantage of the fact that $\mathcal R$ is of genus $0$, we could \textit{in theory} use the rational parametrization of $\mathcal R$ to give a more explicit representation of $f_{1j}$'s. However attractive, a practical implementation of this program requires finding this rational parametrization, and in virtue of the relatively high total degree of \eqref{spectral_curve} this turned out to be a formidable task that we were not able to complete. 
\end{remark}

\section{Spectral curve and auxiliary functions} \label{sec:spectralcurve}

In this section, for convenience of the reader, we summarize some properties of the solutions of the spectral curve \eqref{spectral_curve}--\eqref{c} for $\alpha\in [0,1/2]$, or equivalently, for 
$$
\tau=\alpha(1-\alpha)\in [0, 1/4],
$$
and introduce some related auxiliary functions. It was shown in \cite{martinez_silva_critical_measures} that the algebraic equation \eqref{spectral_curve}  with coefficients given by \eqref{RD} and \eqref{c}, has four branch points and a double point. Two of the branch points, $a_1<b_1$, are 
real and the other two form a complex conjugate pair: $b_2=\overline{a_2}$, $\Im a_2<0$. These are the same $a_j$'s and $b_j$'s used in the previous section to denote the end points of the support of the components $\mu_k$ of the vector critical measure $\vec \mu_\alpha$. For $\tau \neq 1/12$ all these points are distinct, while when
$\tau= 1/12$, the double point and one of the real branch points of \eqref{spectral_curve} coalesce.

The three solutions $\xi_1,\xi_2,\xi_3$ of \eqref{spectral_curve} are determined by their asymptotic expansion at infinity,
\begin{equation}\label{asymptotics_xi_functions}
\begin{split}
& \xi_1(z)=2z^2-\frac{1}{z}+\Boh(z^{-2}),\\
& \xi_2(z)=-z^2+\frac{\alpha}{z}+\Boh(z^{-2}),\\
& \xi_3(z)=-z^2+\frac{1-\alpha}{z}+\Boh(z^{-2}),
\end{split}
\end{equation}
and by their domain of definition: $\xi_1$ is defined on $\C\setminus (\Delta_1\cup \Delta_2)$, $\xi_2$ on $\C\setminus (\Delta_1\cup \Delta_3)$, and $\xi_3$ on $\C\setminus  (\Delta_2\cup \Delta_3)$, images of the respective sheets $\mathcal R_j$ onto $\C$ by canonical projection, see Figure~\ref{figure_sheet_structure}. Each $\xi_j$ is a meromorphic function in its domain of definition, with the only pole at infinity.

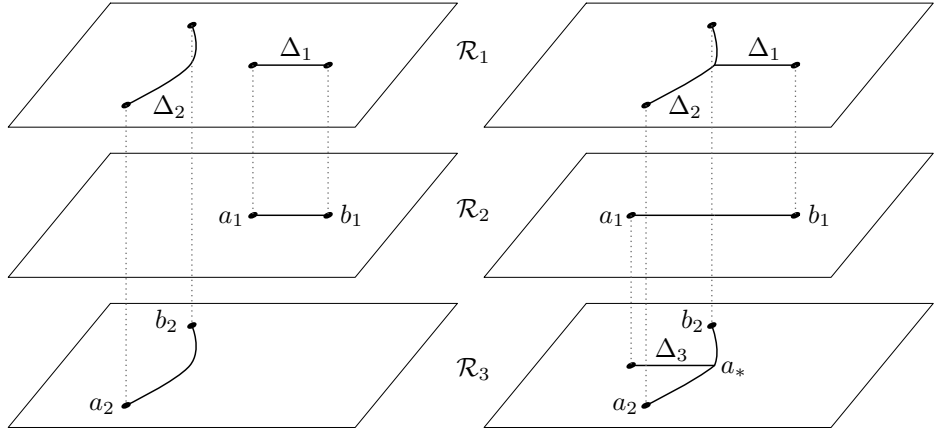
\begin{figure}[t]
	\begin{center}
		\begin{minipage}[c]{0.5\textwidth}
			\begin{overpic}[scale=0.4,grid=false]{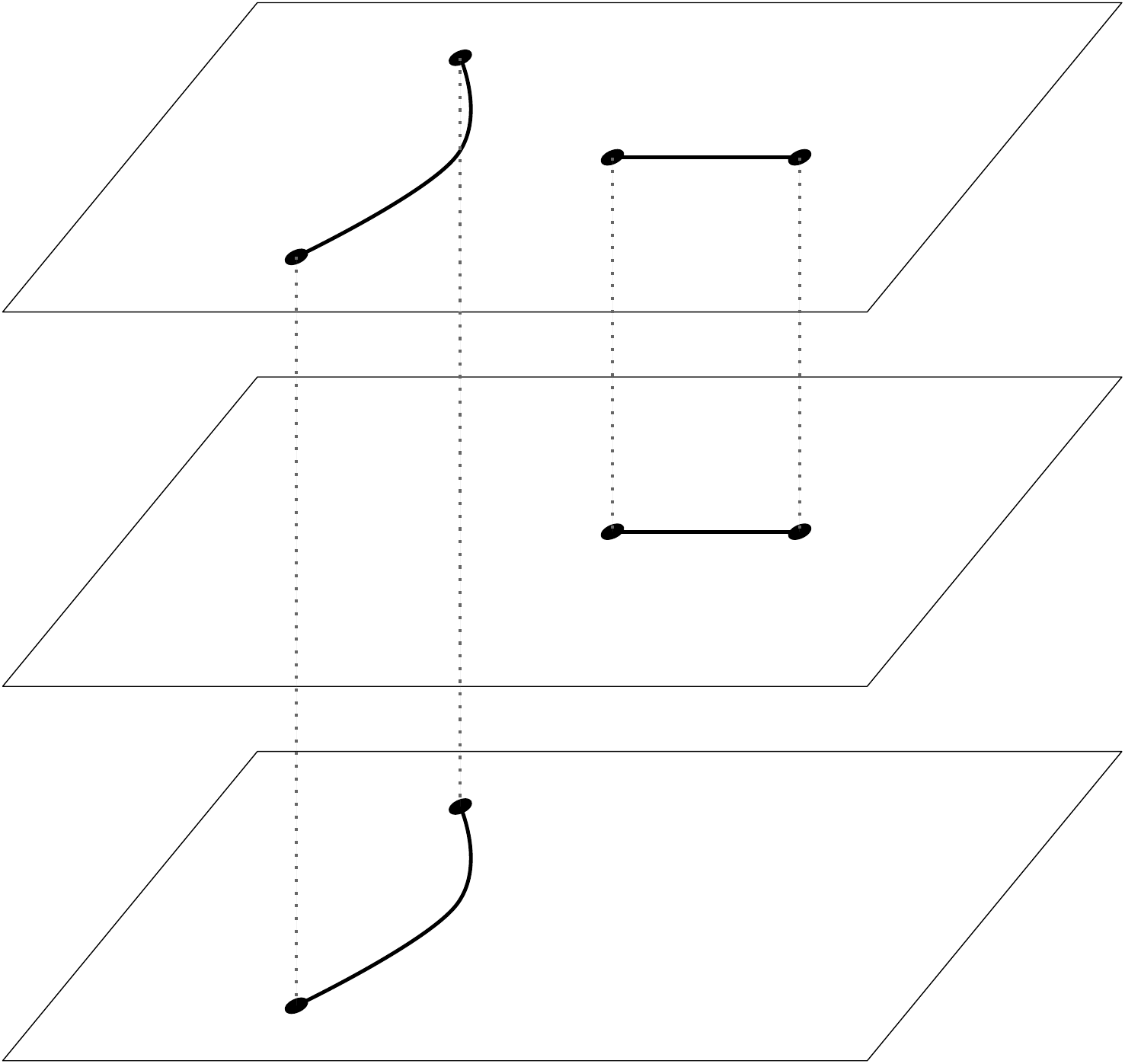}
				\put(170,140){$ \mathcal R_1$}
				\put(170,80){$ \mathcal R_2$}
				\put(170,20){$ \mathcal R_3$}
				\put(103,141){$ \Delta_1$}
				\put(55,119){$ \Delta_2$}
				\put(80,78){$ a_1$}
				\put(126,78){$ b_1$}
				\put(56,38){$ b_2$}
				\put(31,7){$ a_2$}
			\end{overpic}
		\end{minipage}%
		\begin{minipage}[c]{0.5\textwidth}
			\begin{overpic}[scale=0.4]{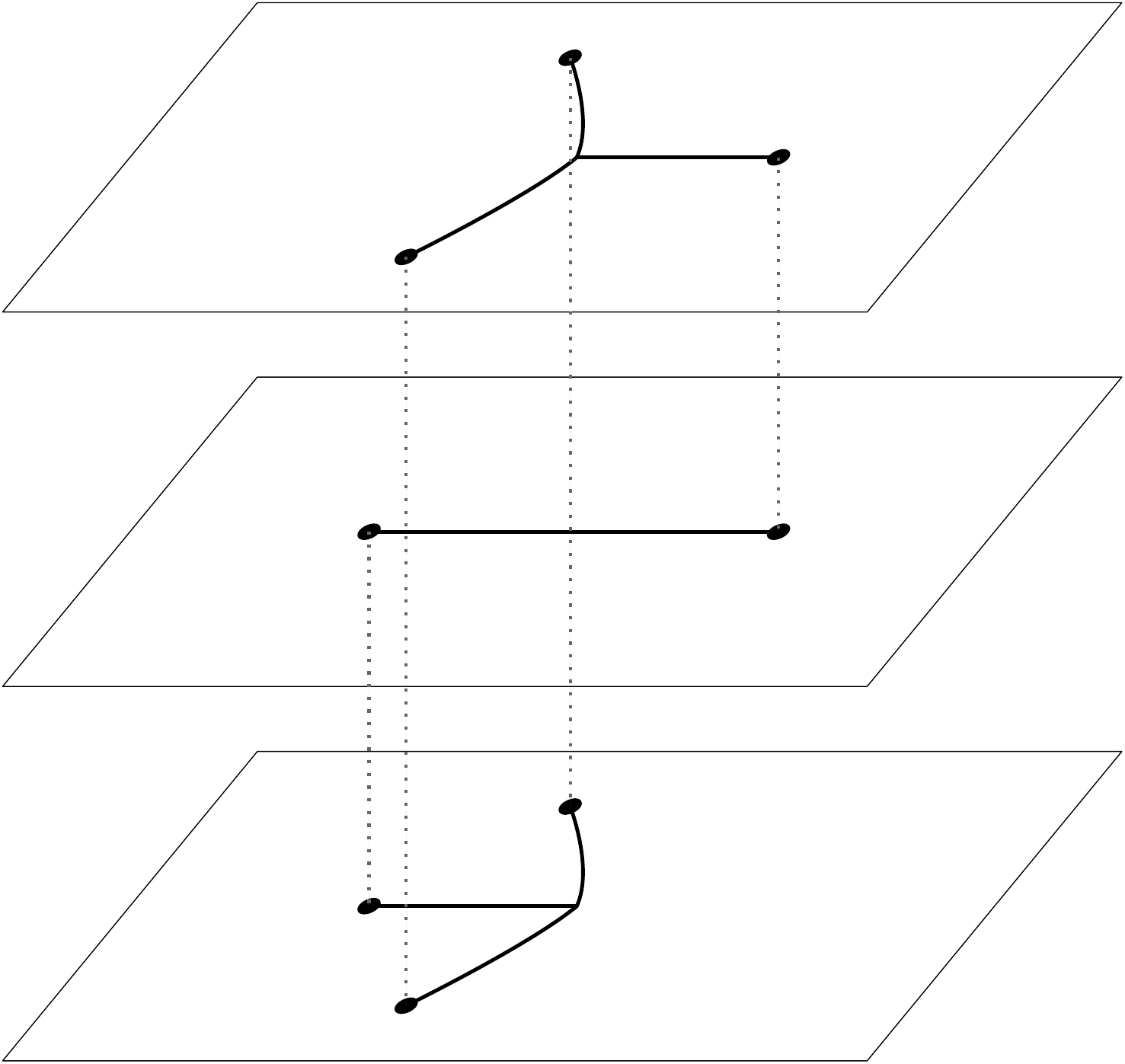}
				\put(100,141){$ \Delta_1$}
				\put(70,119){$ \Delta_2$}
				\put(65,27){$ \Delta_3$}
				\put(44,78){$ a_1$}
				\put(123,78){$ b_1$}
				\put(75,38){$ b_2$}
				\put(49,7){$ a_2$}
				\put(90,21){$ a_*$}
			\end{overpic}
		\end{minipage}%
		\caption{Sheet structure in the \emph{subcritical} $0<\tau<\tau_c$ (left) and \emph{supercritical} $\tau_c<\tau<1/4$ (right) regimes (reproduced from \cite{martinez_silva_critical_measures}).}\label{figure_sheet_structure}
	\end{center}
\end{figure}

Let us discuss some notation that we will use frequently. According to a  standard convention, any orientation of a piece-wise smooth arc on $\C$ induces the left (denoted by the subscript ``$+$'') and right (with the subscript ``$-$'') boundary values of functions defined in a neighborhood of such an arc.

For what comes next, recall the definition of the upper and lower half planes $\H_+$ and $\H_-$ in \eqref{upper_lower_half_plane_def}. In addition, for $\varepsilon>0$ and $a\in \C$, set $B(a,\varepsilon)\isdef \{ z\in \C:\, |z-a|<\varepsilon \}$. If $\mathcal U$ is a domain on $\C$, then $\partial \mathcal U$ stands for its boundary.

Moreover, if $\Gamma$ is a finite union of smooth arcs, then $\overset{\circ}{\Gamma}$ the union of all $a\in \Gamma$ for which there exists a small $\varepsilon>0$ such that $B(a,\varepsilon)\cap \Gamma$ is a smooth Jordan arc. Loosely speaking, $\overset{\circ}{\Gamma}$ is $\Gamma$ with all its end points, points of self-intersection and discontinuities of the derivative removed.

Now we can formulate the result that appears as Proposition 4.3 in \cite{martinez_silva_critical_measures}, which will be the source of the needed identities and properties of $\xi_j$'s: 
\begin{prop}  \label{proposition_distribution_branchpoints} 
	Let $\tau \in (0,1/4)$, $\tau \neq \tau_c\isdef \alpha_c(1-\alpha_c)\approx 0.1913565$. Then
	\begin{enumerate}
		\item[(i)] for $x\in \R\setminus (\Delta_1\cup \Delta_2\cup \Delta_3)$,
		\begin{align}
			& \xi_3(x)<\xi_2(x)<\xi_1(x), && x<\min(a_*,a_1),  \label{inequality_xi_1}\\
				& \xi_1(x)<\xi_2(x)<\xi_3(x), && a_* < x< a_1, \quad \text{if } 0<\tau <\tau_c,  \label{inequality_xi_3} \\
			& \xi_2(x)<\xi_3(x)<\xi_1(x), && x>b_*,  \label{inequality_xi_2} \\
			& \xi_2(x)  <\xi_1(x)<\xi_3(x), && b_1<x<b_*, \quad \text{if} \quad  0<\tau<1/12, \label{inequality_xi_4} \\
			& \xi_3(x)  <\xi_2(x)<\xi_1(x), &&  b_1<x<b_*, \quad \text{if} \quad  1/12<\tau<1/4. \label{inequality_xi_5}
		\end{align}
		Additionally,
		\begin{equation}\label{coinciding_nodes}
			\begin{split}
				\xi_2(b_*)<\xi_3(b_*)& =\xi_1(b_*), \quad \text{for} \quad  0<\tau<1/12,\\
				\xi_2(b_*)=\xi_3(b_*) & <\xi_1(b_*), \quad \text{for} \quad  1/12<\tau<1/4.
			\end{split}
		\end{equation}
		\item[(ii)] on $  \Delta_1\cup \Delta_2\cup \Delta_3$,
		\begin{itemize}
			\item  for  $x \in \overset{\circ}{\Delta}_1:=\Delta_1\setminus \{\max(a_1,a_*), b_1 \}$, 
			\begin{equation}
				\xi_2(x)=\overline{\xi_1(x)}\in \C\setminus\R, \quad \xi_3(x)\in \R,   \label{equality_xi_1}
			\end{equation}
			and
			\begin{equation}
				\xi_{1\pm}(x)  = \xi_{2\mp}(x), \quad  \xi_{3+}(x)  = \xi_{3-}(x).  \label{equality_xi_2}
			\end{equation}
			\item for  $z \in \overset{\circ}{\Delta}_2:=\Delta_2\setminus \{a_2, b_2 ,\max(a_1,a_*)\}$,
			\begin{equation}
				\xi_{1\pm}(z)= \xi_{3\mp}(z),\quad  \xi_{2+}(z) = \xi_{2-}(z).  \label{equality_xi_3}
			\end{equation}
			\item for  $x \in \overset{\circ}{\Delta}_3:=\Delta_3\setminus \{a_1, a_* \}$ (when $\tau >\tau_c$),
			\begin{equation}
				\xi_2(x)=\overline{\xi_3(x)}\in \C\setminus\R, \quad \xi_1(x)\in \R,   \label{equality_xi_4}
			\end{equation}
			and
			\begin{equation}
				\xi_{2\pm}(z)= \xi_{3\mp}(z),\quad  \xi_{1+}(z) = \xi_{1-}(z).  \label{equality_xi_5}
			\end{equation}
			
			Moreover,
			\begin{align}
				\xi_1(a_1)  &= \xi_2( a_1 ),\quad \text{if } \tau < \tau_c, \label{xi_at_branchpoints1} \\
				\xi_3(a_1)  &= \xi_2( a_1 ),\quad \text{if } \tau > \tau_c, \label{xi_at_branchpoints2} \\
				\xi_1(b_1)  &= \xi_2( b_1 ),  \label{xi_at_branchpoints3} \\
				\xi_1(a_2) & = \xi_3(a_2) \quad \text{and} \quad \xi_1(b_2)  =\xi_3(b_2). \label{xi_at_branchpoints4}
			\end{align}
			
		\end{itemize}

	\end{enumerate}
	
\end{prop}

Also, it follows from \eqref{asymptotics_xi_functions} and the structure of cuts $\Delta_j$ that $\xi_j$ have  the following periods:
\begin{equation}\label{periods}
 \ointctrclockwise_{\Delta_1\cup\Delta_2}\xi_1   = -2\pi i  ,\quad \ointctrclockwise_{\Delta_1\cup\Delta_3}\xi_2  = 2\pi i\alpha, \quad  \ointctrclockwise_{\Delta_2\cup\Delta_3}\xi_3  = 2\pi i (1-\alpha).
\end{equation}

We define measures $\mu_1$, $\mu_2$, $\mu_3$ on $\Delta_1$, $\Delta_2$, $\Delta_3$, respectively, through the formulas
\begin{equation}\label{definition_measures_mu}
\begin{split}
d\mu_1(s) &\isdef \frac{1}{2\pi i}(\xi_{1+}(s)-\xi_{2+}(s))ds, \quad s\in \Delta_1,\\
d\mu_2(s) & \isdef \frac{1}{2\pi i}(\xi_{1+}(s)-\xi_{3+}(s))ds, \quad s\in \Delta_2,\\[2mm]
\text{and } & \\
d\mu_3(s) & \isdef \begin{cases}
0, & \text{if } 0<\tau<\tau_c, \\[1mm]
\dfrac{1}{2\pi i}(\xi_{3+}(s)-\xi_{2+}(s))ds, \quad s\in \Delta_3, & \text{if } \tau_c<\tau<1/4, \\
\end{cases}
\end{split}
\end{equation}
where $ds$ denotes the complex line element on the respective arc. As it was shown in \cite[Proposition 3.4]{martinez_silva_critical_measures}, these are bona fide positive measures, whose total masses satisfy the constraints \eqref{massconstraints}.

Furthermore, we define
\begin{equation}\label{defGfunctions}
\begin{split}
& g_1(z)\isdef \int_{b_1}^z \xi_1(s)ds + c_1,\quad z\in \C\setminus ((-\infty,b_1)\cup \Delta_2), \\
& g_2(z) \isdef \int_{b_1}^z \xi_2(s)ds + c_2,\quad z\in \C\setminus (-\infty,b_1), \\
& g_3(z) \isdef  \int_{a_*}^z \xi_3(s)ds + c_3,\quad z\in \C\setminus ((-\infty,a_*)\cup \Delta_2), 
\end{split}
\end{equation}
where for $g_3$ the path of integration emanates from $a_*$ in the direction $+\infty$. 
The constants are 
\begin{equation}\label{g_functions_constants}
c_1=c_2\isdef \int_{a_*}^{a_2}(\xi_{3-}(s)-\xi_{3+}(s))ds,\quad c_3\isdef -\int_{a_*}^{b_1}\xi_{1-}(s)ds.
\end{equation}
Notice that $c_1, c_2\in i\R_-$, which follows from an alternative representation 
$$
c_1=-\pi i |\mu_2|=\int_{b_2}^{a_*}(\xi_{1+}(s)-\xi_{3+}(s))ds, 
$$
where $\mu_2$ is second component of the critical measure $\vec \mu$ (see \eqref{definition_measures_mu}). 

In their domains of definition, $\xi_i(\overline{z})=\overline{\xi_i(z)}$, $i=1, 2, 3$, and thus
\begin{equation}\label{symmetryforG}
g_i(\overline{z})-c_i=\overline{g_i(z)-c_i}, \quad i=1, 2, 3,
\end{equation}
were we again consider $z$ in the appropriate domain of definitions of the function.

Notice that from the asymptotic expansions in \eqref{asymptotics_xi_functions}, we get the existence of constants $r_j\in \C$ such that
\begin{equation}\label{defRforGfunctions}
\begin{split}
& g_1(z)=\frac{2}{3}z^3+r_1-\log z +\Boh(z^{-1}), \\
& g_2(z)=-\frac{1}{3}z^3+r_2+\alpha \log z +\Boh(z^{-1}),\quad z\to\infty. \\
& g_3(z)=-\frac{1}{3}z^3+r_3+(1-\alpha)\log z +\Boh(z^{-1}).
\end{split}
\end{equation}
Because the functions $\xi_1$ and $\xi_2$ are real on the interval $(b_1,\infty)$ and $\xi_3$ is real on $(a_*,\infty)$, we surely have $r_1,r_2,r_3\in \R$.

Finally it will be convenient for the future to introduce additional functions, $\Phi_j$, closely related to $g_j$'s. Namely,  for $z\in \C\setminus ((-\infty,b_1]\cup\Delta_2)$, let
\begin{align}
\Phi_1(z) & \isdef \int_{b_1}^z(\xi_1(s)-\xi_2(s))ds,   
\label{def_phi_1} \\
\Phi_2(z) & \isdef \int_{b_2}^z(\xi_1(s)-\xi_3(s))ds,   
\label{def_phi_2} \\
\Phi_3(z) & \isdef \int^z_{\min(a_1,a_*)}(\xi_3(s)-\xi_2(s))ds .  \label{def_phi_3b}
\end{align}
The path of integration in the definition  of $\Phi_3$ starts in $\H_-$, and if $a_*<a_1$ (i.e., if $0 <\tau <  \tau_c$), in the sector determined by $(a_*,+\infty)$ and $\Delta_2$.
 
Using \eqref{periods} and \eqref{g_functions_constants}, we can see that  for $z\in \C\setminus ((-\infty,b_1]\cup\Delta_2)$,
\begin{align}
\Phi_1(z) & =g_1(z)-g_2(z),  
\label{value_phi_1} \\
\Phi_2(z) & = g_1(z)-g_3(z) + 2\pi i ,   
\label{value_phi_2} \\
\Phi_3(z) & =  g_3(z)-g_2(z) - 2\pi i \alpha+
\begin{cases}
\displaystyle  \int_{a_*}^{a_1} (\xi_{1}-\xi_2)ds  +c_2,  & 0 <\tau < \tau_c, \\
  0,  & \tau_c <\tau < 1/4.
\end{cases}  
\label{value_phi_3}
\end{align}

\section{The canonical quadratic differential} \label{sec:trajectories}

The study of the behavior of functions $\Phi_j$'s, needed in what follows, relies heavily on the results and detailed analysis performed in our previous paper \cite{martinez_silva_critical_measures} of the trajectories of a canonical quadratic differential that we describe next.

As it was established there (see Theorem 1.9), if
\begin{equation*}
Q(z) \isdef
\begin{cases}
\xi_2(z)-\xi_3(z)  , & \text{ on } \mathcal R_1, \\
\xi_1(z)-\xi_3(z) , & \text{ on } \mathcal R_2, \\
\xi_1(z)-\xi_2(z) , & \text{ on } \mathcal R_3,
\end{cases}
\end{equation*}
then $Q^2$ extends to a meromorphic function on the Riemann surface $\mathcal R$ of the algebraic equation \eqref{spectral_curve}, and 
\begin{equation}
\label{defQqd}
\varpi =-Q^2 (z)dz^2
\end{equation}
is a meromorphic quadratic differential on $\mathcal R$ with  poles only at the points at $\infty$. Finally, for $j=1,2,3$, each $\Delta_j=\supp\mu_j$ is an arc 
of trajectory of $\varpi$, satisfying
$$
\Re \int^z \sqrt{-\varpi}=\Re \int_p^z Q(s)ds =\const .
$$
A trajectory $\gamma$ extending to a zero of $Q$ along at least one of its directions is called {\it critical}; in the case when it happens in both directions, we call 
this trajectory {\it bounded} (also {\it finite} or \textit{short}), and {\it unbounded} (or {\it infinite}) otherwise.  The union of finite and infinite critical trajectories is the \textit{critical graph} of $\varpi$.

The parameter $\tau=\alpha(1-\alpha)\in [0, 1/4]$ has transition values $\tau_0=1/12$, $\tau_1$, $\tau_c$ and $\tau_2$, which satisfy
$$
0<\tau_0<\tau_1<\tau_c<\tau_2<\frac{1}{4},
$$
and determine the global structure of trajectories of the  quadratic differential $\varpi$ defined in \eqref{defQqd}. 
According to \cite{martinez_silva_critical_measures}, at $\tau=\tau_0$ the double point (node) of $\mathcal R$ collides with the branch point $b_1$ transitioning from $\mathcal R_2$ to $\mathcal R_1$; at $\tau=\tau_1$,  a critical trajectory connecting $a_2^{(1)}$ and $b_2^{(1)}$ on $\mathcal R_1$ hits the branch point $a_1^{(1)}$, while at $\tau=\tau_c$ the same happens on $\mathcal R_2$, now for $a_2^{(2)}$, $b_2^{(2)}$ and $a_1^{(2)}$. Finally, for $\tau=\tau_2$, the point $a_*$ collides with the critical trajectory connecting $a_2^{(1)}$ and $b_2^{(1)}$ on $\mathcal R_1$; at this moment, there are short trajectories connecting $a_1^{(1)}$ with $a_2^{(1)}$ and $b_2^{(1)}$.

As described in \cite{martinez_silva_critical_measures}, the critical graph of $\varpi$ is topologically the same for $\tau$ in each of the intervals determined by consecutive transition values $0,\tau_0,\tau_1,\tau_c,\tau_2,1/4$, and it undergoes a phase transition when $\tau$ moves from one of these intervals to the next one.

Of particular relevance for us here are the trajectories of $\varpi$ on $\mathcal R_1$, so in Figures~\ref{traj_top_sheet_1}--\ref{traj_top_sheet_2} we reproduce these relevant illustrations as provided in \cite[Figures~10,12,15,19,21]{martinez_silva_critical_measures}

\begin{figure}
\begin{subfigure}{.5\textwidth}
\centering
\begin{overpic}[scale=1]{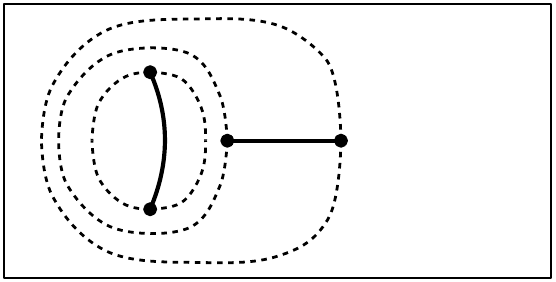}
\end{overpic}
\vspace{0.5cm}
\end{subfigure}%
\begin{subfigure}{.5\textwidth}
\centering
\begin{overpic}[scale=1]{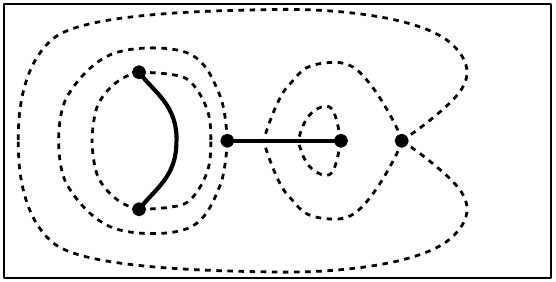}
\end{overpic}
\vspace{0.5cm}
\end{subfigure}
\begin{subfigure}{1\textwidth}
\centering
\begin{overpic}[scale=1]{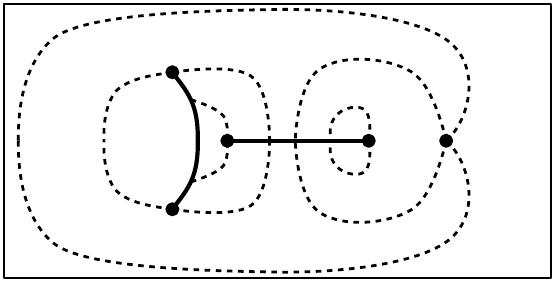}
\end{overpic}
\end{subfigure}
\caption{Critical graph of $\varpi$ on the first sheet $\mathcal R_1$, corresponding to the intervals $0<\tau<\tau_0$ (top left frame), $\tau_0<\tau<\tau_1$ (top right frame) and $\tau_1<\tau<\tau_c$ (bottom frame). The dashed lines represent the critical trajectories, whereas the solid lines are the branch cuts connecting $\mathcal R_1$ to the remaining sheets, which are not displayed.}\label{traj_top_sheet_1}
\end{figure}

\begin{figure}
\begin{subfigure}{.5\textwidth}
\centering
\begin{overpic}[scale=1]{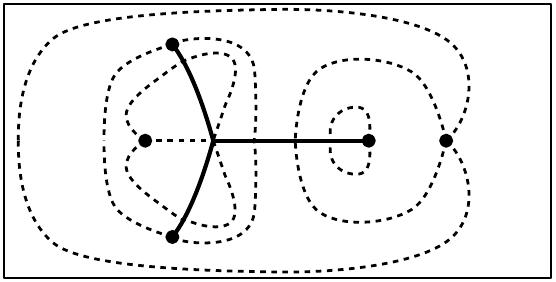}
\end{overpic}
\end{subfigure}%
\begin{subfigure}{.5\textwidth}
\centering
\begin{overpic}[scale=1]{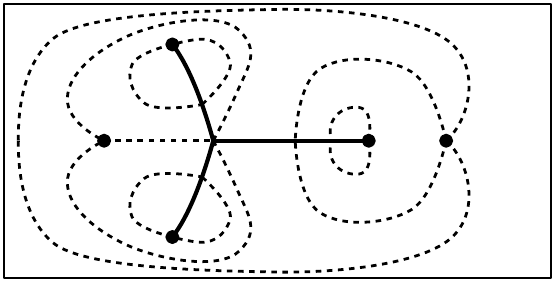}
\end{overpic}
\end{subfigure}
\caption{Critical graph of $\varpi$ on the first sheet $\mathcal R_1$, corresponding to the intervals $\tau_c<\tau<\tau_2$ (left frame) and $\tau_2<\tau<1/4$ (right frame). The dashed lines represent the critical trajectories, whereas the solid lines are the branch cuts connecting $\mathcal R_1$ to the remaining sheets, which are not displayed.}\label{traj_top_sheet_2}
\end{figure}

We need to describe some domains of positivity of functions $\re \Phi_j$. Notice that by \eqref{def_phi_1}--\eqref{def_phi_3b}, the level curves of $\re \Phi_j$ coincide with trajectories of the quadratic differential $\varpi$ on the sheet $\mathcal R_{4-j}$, $j=1, 2, 3$, which establishes the connection with the topology of the critical graph of $\varpi$.

As it follows from the detailed analysis of the global structure of the trajectories performed in  \cite[Section 4.5]{martinez_silva_critical_measures}, for $1/12 <\tau< 1/4$ there exists an arc of trajectory $\gamma_0$ emanating from $b_1^{(1)}$ on the lower half-plane (on the sheet $\mathcal R_1$), and crossing the real line at a point $x_0\in \Delta_1$. Let us define
\begin{equation}\label{defXstar}
x_* \isdef \begin{cases}
b_1, & 0<\tau \leq 1/12, \\
x_0, & 1/12 < \tau < 1/4.
\end{cases}
\end{equation}

The next lemma is complementary to Proposition ~\ref{proposition_distribution_branchpoints}.

\begin{lem} \label{lemma_tech}
If $\tau<1/12$, then
\begin{equation}\label{inequalityXi0}
\re\xi_{2\pm}(x)<\xi_3(x),\quad x>a_*.
\end{equation}

If $\tau>1/12$, then there exists a value $\zeta_*\in (x_*,b_1)$ such that
\begin{equation}\label{inequalityXi1}
\re\xi_{2\pm}(x)<\xi_3(x), \quad  x\in (a_*,\zeta_*)\cup (b_*,+\infty) 
\end{equation}
and
\begin{equation}\label{inequalityXi2}
\re\xi_{2\pm}(x)>\xi_3(x), \quad  x\in (\zeta_*,b_*).
\end{equation}

\end{lem}
\begin{proof}
	By \eqref{inequality_xi_3}--\eqref{inequality_xi_5}, we only need to establish the inequalities on $\Delta_1$.
	
	Recall that $\xi_1,\xi_2$ and $\xi_3$ are the solutions of the cubic equation $\xi^3-R(z)\xi+D(z)=0$ in \eqref{spectral_curve} 
	and that $\xi_{2\pm}=\xi_{1\mp}$ on $\Delta_1$, whereas $\xi_3$ is analytic across $\Delta_1$. Consequently, by \eqref{equality_xi_1}--\eqref{equality_xi_2}, if $\xi_3(\zeta_*)=\re \xi_{2\pm}(\zeta_*)$ for some value $\zeta_*\in \Delta_1$, then
	$$
	0=\xi_{1+}(\zeta_*)+\xi_{2+}(\zeta_*)+\xi_{3}(\zeta_*)=3\xi_{3}(\zeta_*),
	$$
	and consequently $D(\zeta_*)=\xi_{1+}(\zeta_*)\xi_{2+}(\zeta_*)\xi_3(\zeta_*)=0$. From \cite[Proposition~4.2]{martinez_silva_critical_measures} we know that $D$ has no zeros on $\Delta_1$ if $\tau < 1/12$, and has exactly one zero  on $\Delta_1$ otherwise, which thus has to be $\zeta_*$. In other words, $\xi_{2\pm}-\xi_3$ does not change sign on $\Delta_1$ (if $\tau < 1/12$) or otherwise,  it can change sign at most once, and it has to be at the zero $\zeta_*$ of $D$.
	
	Case $\tau < 1/12$ then trivially follows from \eqref{inequality_xi_3} and \eqref{xi_at_branchpoints1}; so, let us consider $1/12< \tau <\tau_c $. 
	
	By the definition of the trajectory $\gamma_0$ emanating from $b_1$ on the lower half plane in the first sheet,  
	$$
	0=\re \int_{\gamma_0} \sqrt{-\varpi} = \int_{b_1}^{x_*}(\re\xi_{2-}(x)-\xi_3(x))dx.
	$$
	Hence, $\re\xi_{2-}(x)-\xi_3(x)$ has to change sign on $(x_*,b_1)$, thus $\zeta_*>x_*$, so consequently $\re\xi_{2-}(x)-\xi_3(x)$ has constant sign on $\Delta_1\setminus [x_*,b_1]$.
	
	It remains to use \eqref{inequality_xi_3} and \eqref{xi_at_branchpoints1}: when combined these equations give us that for $\tau<\tau_c$ we have $\xi_2<\xi_3$ immediately to the right of $a_1$, so that by continuity of the $\xi_j$'s in both $z$ and $\tau$ we conclude the proof.
\end{proof}

As a preparatory step for the study of positivity of $\re \Phi_3$, let us examine  the boundary values of $\Phi_3$ on subsets of $\R$ from the lower half plane. 

Assume first $\tau < \tau_c$. Combining \eqref{equality_xi_3}, \eqref{g_functions_constants} and \eqref{def_phi_3b},  for $x<a_*$, 
\begin{align*}
\Phi_{3-}(x) &=\int^{a_2}_{a_*}(\xi_{3-}(s)-\xi_{3+}(s))ds +  \int^x_{a_*}(\xi_3(s)-\xi_2(s))ds \\
& = c_2 +  \int^x_{a_*}(\xi_3(s)-\xi_2(s))ds ,
\end{align*}
so that
$$
\Re \Phi_{3-}(x) = \int^x_{a_*}(\xi_3(s)-\xi_2(s))ds >0, \quad x<a_*,
$$
where we have used \eqref{inequality_xi_1}. 

On the other hand, by \eqref{inequalityXi0}--\eqref{inequalityXi1},  
\begin{equation}\label{ineq3}
\Re \Phi_{3-}(x) = \int^x_{a_*}(\xi_3(s)-\xi_{2-}(s))ds >0, \quad a_*< x\leq \zeta_*,
\end{equation}
where as we have seen, $a_1<x_*< \zeta_*\leq b_1$. 

Let now $\tau > \tau_c$, so that
$$
\Phi_3(z) = \int^z_{ a_1}(\xi_3(s)-\xi_2(s))ds , 
$$
with the path of integration starting in $\H_-$. In particular, using \eqref{inequality_xi_1}
$$
\Phi_{3-}(x)= \int^x_{ a_1}(\xi_3(s)-\xi_{2}(s))ds >0 , \quad x< a_1.
$$
Furthermore, for $x\in (a_1, a_*)=\overset{\circ}{\Delta}_3$, we now use \eqref{inequality_xi_5} and get
$$
\Phi_{3-}(x)= \int^x_{ a_1}(\xi_{3-}(s)-\xi_{2-}(s))ds = \int^x_{ a_1}(\xi_{3-}(s)-\overline{\xi_{3-}(s)})ds  ,
$$
so that $\Re \Phi_{3-}(x)=0$ for $x\in [a_1, a_*]$.

On the other hand, for $x\in (a_*, \zeta_*]$,
\begin{align}\label{valuePhi3}
\Phi_{3-}(x) &
\!\begin{multlined}[t]
= \int^{a_*}_{ a_1}(\xi_{3-}(s)-\overline{\xi_{3-}(s)})ds  + \int^{a_2}_{a_*}(\xi_{3+}(s)-\xi_{3-}(s))ds \\ +  \int^x_{a_*}(\xi_{3}(s)-\xi_{2-}(s))ds 
\end{multlined}\nonumber
\\
& = \int^{a_*}_{ a_1}(\xi_{3-}(s)-\overline{\xi_{3-}(s)})ds  - c_2 +  \int^x_{a_*}(\xi_3(s)-\xi_{2-}(s))ds ,
\end{align}
so that \eqref{ineq3} holds also in this case. 

Our findings are summarized as
\begin{equation}\label{setofinequalities}
\re\Phi_{3-}(x) \begin{cases}
> 0, & x< \min(a_1, a_*) \text{ or } a_* < x \leq \zeta_*, \\
=0, & a_1 < x < a_* \text{ and } \tau > \tau_c.
\end{cases}
\end{equation}

Now we are ready to prove a statement that will be important for our asymptotic analysis. We will be interested in what happens in the lower half plane $\H_-$, so we introduce
\begin{equation*} 
	\Omega_-  \isdef \{ z \in \H_- \setminus \Delta_2:\, \re \Phi_3(z)<0\}, \quad 
	\Omega_+   \isdef \{ z \in \H_- \setminus \Delta_2:\, \re \Phi_3(z)>0 \}.
\end{equation*}

\begin{prop}\label{proposition_signs_Phi}
The structure of $\Omega_-$ is as follows:
	\begin{enumerate}
		\item[(i)] For $0<\tau<\tau_c$, $\Omega_-=\emptyset$.
		\item[(ii)] For $\tau_c<\tau<\tau_2$, $\Omega_-$ is an open set with two connected components, whose boundaries intersect at the same subarc $(a_*, z_B)$ of $\Delta_2$, where $z_B\in \Delta_2$, $z_B\neq a_2$. One of these two components contains also the interval $(a_1, a_*)$ on its boundary, and the boundary of the other component intersects the real axis only at $a_*$.
		\item[(iii)] For $\tau_2<\tau<1/4$, $\Omega_-$ is an open connected set whose boundary contains both the arc $\Delta_2 \cap \H_-$ and the interval $(a_1, a_*)$.
	\end{enumerate}
Furthermore, there exists an unbounded domain $\mathcal V \subset \H_-$, with $a_2$ on its boundary and extending to $\infty$ along the direction determined by the angle $-2\pi/3$, such that 
\begin{equation} \label{valuesPhi}
\Re \Phi_1(z)> 0, \quad \Re \Phi_2(z) > 0 , \quad z \in  \mathcal V.
\end{equation}
\end{prop}
\begin{proof}
We start with some basic considerations. Because $\re\Phi_3$ is harmonic on $\H_-\setminus \Delta_2$, the boundary of any connected component of $\Omega_{\pm}$ can only emanate from $\R$, $\Delta_2$ or $\infty$. Furthermore, any level line of $\re \Phi_3$ is an arc of trajectory of the quadratic differential \eqref{defQqd} on $\mathcal R_1$. A common feature of such trajectories is that they are closed arcs near $\infty$ (see Figures \ref{traj_top_sheet_1}--\ref{traj_top_sheet_2}), so combined with the inequality \eqref{setofinequalities} as $x\to -\infty$, we conclude that a neighborhood of $\infty$ belongs to $\Omega_+$ or, in other words, $\Omega_-$ is bounded.

Also, observe that the boundary of $\Omega_-$ does not contain points on the interval $(-\infty,\min\{a_1,a_*\})$. This is true because if $\tilde b$ is the smallest such point, then $\Re\Phi_{3}$ is harmonic immediately below to $\tilde b$, so we must have $\tilde b\in \partial \Omega_-\cap \partial \Omega_+$ and thus $\Re\Phi_{3-}(\tilde b)=0$, which is in contradiction with \eqref{setofinequalities}.

We claim that the boundary of $\Omega_-$ does not contain points on $(a_*,\infty)$ either. Indeed, to the contrary, suppose that $b$ is such a point. We can assume $b$ to be the right-most of such points. Consequently $\Phi_{3-}(b)=0$ and any $x>b$ belongs to $\overline \Omega_+$. Take the smallest possible $a\geq a_*$ such that $[a,b]\subset \partial \Omega_-$. If $a=a_*$ then $\re \Phi_{3-}(a)=0$ (the value understood when we approach $a_*$ from the sector determined by $\Delta_2\cap \H_-$ and $(a_*,+\infty)$), whereas if $a>a_*$ then again $\re \Phi_{3-}(a_*)=0$, but now because $\Re\Phi_3$ is harmonic in a neighborhood immediately below $a$. In either case, we can certainly write
$$
0=\re\Phi_{3-}(b)-\re \Phi_{3-}(a)=\re\int_a^b (\xi_{3-}(s)-\xi_{2-}(s))ds,
$$ 
so $\xi_{3-}-\xi_{2-}$ must change sign on $(a,b)$. In virtue of \eqref{inequalityXi0}--\eqref{inequalityXi2}, we must have at least one of the points $b_*$ and $\zeta_*$ in $(a,b)$, so certainly $b\geq \zeta_*$. Now, by construction, $\overline{\Omega}_+$ is to the right-hand side of $b$, so $\xi_{3-}-\xi_{2-}$ must be positive for values slightly larger than $b$. Again using \eqref{inequalityXi1}--\eqref{inequalityXi2}, we thus get that actually $b>b_*$.

Let $\gamma$ be the arc of $\partial \Omega_-$ that emanates from $b$, oriented outwards of $b$. With such orientation, $\Omega_\pm$ is on the $\pm$-side of $\gamma$. We follow the large behavior of $\gamma$, having in mind that $\gamma$ must be an arc of trajectory of $\varpi$ emanating from $(b_*,+\infty)$. A quick inspection of the critical graphs in Figures \ref{traj_top_sheet_1}--\ref{traj_top_sheet_2} then shows that $\gamma$ must end up at a point $\tilde b<\min\{a_*,a_1\}$. Thus, keeping track of the orientation of $\gamma$, we then see that $\tilde b\in \partial \Omega_- \cap (-\infty,\min\{a_*,a_1\})$, but we already know that this last intersection is empty.

As a conclusion of the observations above, the only possible real points on $\partial \Omega_-$ are on the interval $[\min\{a_1,a_*\},a_*]$. 

To verify (i)--(iii), we again use the structure of trajectories shown in Figures \ref{traj_top_sheet_1}--\ref{traj_top_sheet_2}. 

For (i), notice first that $[\min\{a_1,a_*\},a_*]=\{a_*\}$. Observing the structure displayed in Figure \ref{traj_top_sheet_1}, we get that $\re \Phi_3(z)$ does not change sign on the immediate vicinity of $a_*$, so $a_*\notin \partial \Omega_-$, and also that any possible boundary component of $\Omega_-$ in $\H_-$ must intersect the real axis, which we showed that cannot occur away from $a_*$.

For (ii) and (iii), we observe that $[\min\{a_1,a_*\},a_*]$ and the only regions on $\H_-$ that contain trajectories which do not intersect $\R\setminus [\min\{a_1,a_*\},a_*]$ are the ones marked in gray in Figure~\ref{omega1}, so $\Omega_-$ must be contained in these regions. A simple analysis of the sign of $\Re\Phi_3$, making use of \eqref{setofinequalities}, then shows that $\Omega_-$ actually coincides with these marked sectors. This concludes the proof of (ii) and (iii).

Finally, for every value of $\tau$ there is a half-plane canonical domain on $\mathcal R_2$ and a half-plane canonical domain on $\mathcal R_3$, both extending to $\infty$ in the asymptotic direction $e^{-2\pi /3}$, containing $a_2$ on its boundary. Take as $\mathcal V$ the (non-empty) intersection of the projection of these domains onto $\C$. Functions $\Re \Phi_2$ and $\Re \Phi_1$ do not change sign in $\mathcal V$, and by \eqref{defRforGfunctions}, \eqref{value_phi_1} and \eqref{value_phi_2},
$$
\Phi_1(z)= z^3 +\Boh(1), \quad \Phi_2(z)= z^3+ \Boh(1),  \qquad z \in \mathcal V, \; z\to \infty, 
$$
from where \eqref{valuesPhi} follows.
\end{proof}

\begin{figure}
\begin{subfigure}{.5\textwidth}
\centering
\begin{overpic}[scale=1]{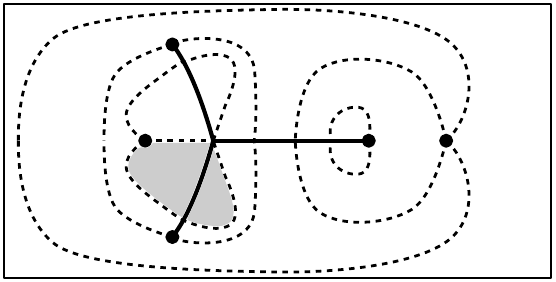}
\end{overpic}
\end{subfigure}%
\begin{subfigure}{.5\textwidth}
\centering
\begin{overpic}[scale=1]{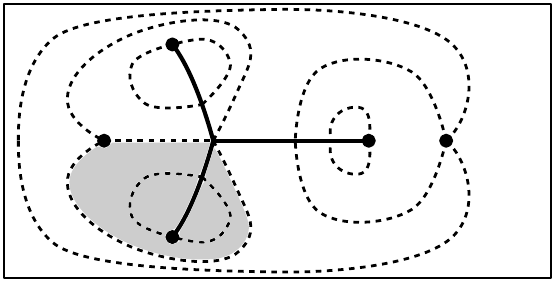}
\end{overpic}
\end{subfigure}
\caption{The set $\Omega_-$ is displayed in gray, together with the trajectories of $\varpi$ on the first sheet $\mathcal R_1$. The left frame corresponds to $\tau_c<\tau<\tau_2$ and the right frame corresponds to $\tau_2<\tau<1/2$.}\label{omega1}
\end{figure}

For $\tau \in (\tau_c,\tau_2)$, the boundary of $\Omega_-$ consists of the interval $[a_1,a_*]$, an arc of trajectory $\gamma_L$ from $a_1$ to a point $a_B\in \Delta_2$, a second arc of trajectory $\gamma_R$ from $a_B$ to $a_*$ and a third arc $(a_B,a_*)\subset \Delta_2$ connecting $a_B$ to $a_*$. It is also a consequence of the analysis in \cite{martinez_silva_critical_measures} that the arc of trajectory $\gamma_R$ is the analytic extension of $\Delta_2\cap \H_+$ to the lower half plane. These quantities are displayed in Figure \ref{omega_boundary}, left frame.

When $\tau>\tau_c$, the boundary of $\Omega_-$ consists of $\Delta_2\cap \H_-$, the interval $[a_1,a_*]$ and an arc $\gamma_R$ joining $a_1$ and $a_*$, the latter being the analytic extension of $\Delta_2\cap \H_+$ to the lower half plane. These are displayed in Figure \ref{omega_boundary}, right frame.

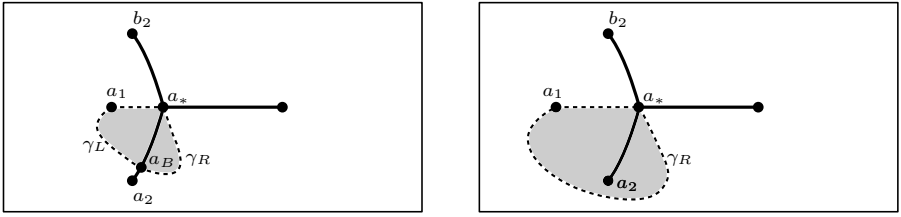
\begin{figure}
\begin{subfigure}{.5\textwidth}
\centering
\begin{overpic}[scale=1]{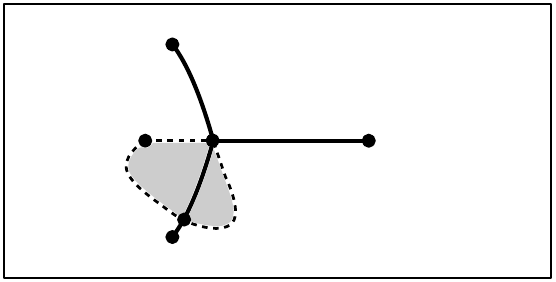}
\put(50,72){\scriptsize $b_2$}
\put(40,45){\scriptsize $a_1$}
\put(56,19){\scriptsize $a_B$}
\put(70,20){\scriptsize $\gamma_R$}
\put(31,25){\scriptsize $\gamma_L$}
\put(50,5){\scriptsize $a_2$}
\put(63,43){\scriptsize $a_*$}
\end{overpic}
\end{subfigure}%
\begin{subfigure}{.5\textwidth}
\centering
\begin{overpic}[scale=1]{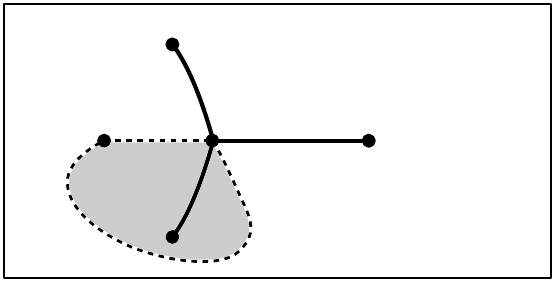}
\put(50,72){\scriptsize $b_2$}
\put(25,45){\scriptsize $a_1$}
\put(53,10){\scriptsize $a_2$}
\put(72,20){\scriptsize $\gamma_R$}
\put(53,10){\scriptsize $a_2$}
\put(63,43){\scriptsize $a_*$}
\end{overpic}
\end{subfigure}
\caption{For the cases $\tau_c<\tau<\tau_2$ (left frame) and $\tau_2<\tau<1/4$ (right frame) the set $\Omega_-$ is displayed in gray, together with the newly introduced arcs $\gamma_L$ and $\gamma_B$ and the point $a_B$.}\label{omega_boundary}
\end{figure}

Using the notation just introduced, we define the open set $\Omega_\alpha \subset \H_-$ and the continuum $E_\alpha   \subset \C$  that will play an important  role in the asymptotic analysis of $B_{n,m}$:
\begin{definition} \label{defE}
\begin{enumerate}[(i)]
	\item If $\tau \in (0, \tau_c)$, then $\Omega_\alpha \isdef \emptyset$ and $E_\alpha \isdef \Delta_2  $. Notice that on $E_\alpha$,
	\begin{equation} \label{bdvaluesPhi}
		\Re \left(\Phi_{3+} -\Phi_{3-} \right) (z)=0 .
	\end{equation}

	\item If $\tau \in (\tau_c,\tau_2)$, then 
	$$
	E_\alpha \isdef \gamma_L \cup \gamma_R \cup (a_2, a_B) \cup (\Delta_2 \cap \H_+),
	$$ 
	where $(a_2, a_B)$ denotes the arc of $\Delta_2$ connecting $a_2$ with $a_B$.  
On $(a_2, a_B)$, identity \eqref{bdvaluesPhi} holds, while on $ \gamma_L \cup \gamma_R$, $\Re \Phi_3(z)=0$. Notice also that $\overline{\gamma_R\cup (\Delta_2 \cap \H_+)}$   is an analytic  arc joining $a_B$ and $b_2$.
 
		Furthermore, $\Omega_\alpha $ is the bounded component of  $\C \setminus (\R \cup \gamma_L \cup \gamma_R) $.
	
	\item For $\tau\in (\tau_2, 1/4)$, the set $E_\alpha\isdef  \gamma_L \cup (\Delta_2 \cap \H_+)$   is a single analytic arc from $b_2$ to $a_1$, passing through $a_*$, and   $\Omega_\alpha $ is the connected domain bounded by $\gamma_L  \cup (a_1, a_*) \cup (\Delta_2 \cap \H_-) $, that is, $\Omega_\alpha=\Omega_-\cup (a_b,a_*)$.
\end{enumerate}
\end{definition}

In any of the cases (i)--(iii) above, note also that
\begin{equation}\label{boundary_omegaplus}
E_\alpha\cap \H_-=\partial \Omega_+\cap \H_-.
\end{equation}

As the last step, we define the positive measure $\mu_B$ on the set $E_\alpha$ that will ultimately describe the zero distribution of $B_{n,m}$ as in Theorem \ref{theorem_zero_counting_measureB}.
\begin{prop} \label{PropdefmuB}
Let
\begin{equation}\label{defH}
	H(z) \isdef  \begin{cases}
	\Re \left( g_2(z)+\frac{1}{3}z^3- r_1 \right)  , &  z\in \Omega_\alpha,  
	\\
	\Re \left( g_3(z)+\frac{1}{3}z^3- r_1 \right)  , & z \in  \C\setminus (\Omega_\alpha \cup E_\alpha).
	\end{cases}
\end{equation} 
Then $H$ is harmonic in $\C \setminus E_\alpha$ and extends to a subharmonic function in $\C$. Moreover, there exists a positive measure $\mu_B$, $\supp \mu_B=E_\alpha$, such that
$$
H(z) = -U^{\mu_B}(z) +\Re \left( r_3- r_1 \right), \quad z\in \C,
$$
with constants $r_1$, $r_3$ defined in \eqref{defRforGfunctions}. Measure $\mu_B$ satisfies $|\mu_B|=1-\alpha$, is absolutely continuous with respect to the arc-length measure, and
$$
\mu_B'(s) =  \begin{cases}
\mu_2'(s), & \text{for } s\in  (\Delta_2 \cap \H_+ )\cup  (\Delta_2 \setminus  \overline{\Omega_\alpha }),  \\[1mm]
\dfrac{1}{2\pi i} \left( \xi_2(s) - \xi_3(s)\right), & \text{for } s\in \gamma_L\cup \gamma_R.
	\end{cases}
$$ 
Finally,
\begin{equation}\label{balayage}
\mu_B=\Bal( \mu_2-\mu_3 ; E_\alpha).
\end{equation}
\end{prop}
\begin{proof}

The function $g_3$ is holomorphic on $\C\setminus ((-\infty,a_*)\cup \Delta_2)$, so we get that $H$ is harmonic on $\C\setminus ((-\infty,a_*)\cup \Delta_2\cup \Omega_\alpha)$. Also, $g_{3+}-g_{3-}$ is purely imaginary along $(-\infty,\min\{ a_1,a_* \})$ (see \eqref{identityG3new} below), so $H$ is harmonic across this interval as well. Furthermore, $g_{2+}=g_{3-}$ across $\Delta_3$ (see \eqref{G2plusminusG3minus}), and this means that $H$ is harmonic across $\Delta_3$ as well. To conclude the harmonicity of $H$ on $\C\setminus E_\alpha$, it only remains to observe that, according to \eqref{defGfunctions}, the function $g_2$ is holomorphic on the lower half plane.

Subharmonicity of $H$ in a neighborhood of $E_\alpha\cap \H_-$ follows from the alternative representation
$$
H(z) = \max \left\{ 	\Re \left( g_2(z)+\frac{1}{3}z^3- r_1 \right) , 	\Re \left( g_3(z)+\frac{1}{3}z^3- r_1 \right)  \right\},
$$
valid in a neighborhood of $E_\alpha\cap \H_-$, as assured by Proposition \ref{proposition_signs_Phi}. Furthermore, using now \eqref{defGfunctions} and \eqref{definition_xi_functions} to write $H=-U^{\mu_2}+ \text{harmonic}$, we conclude that $H$ is subharmonic on a neighborhood of $E\cap \H_+$ as well, and thus on the whole plane $\C$.

As a consequence of \cite[Theorem~II.3.3]{saff_totik_book} there exists a positive measure $\mu_B$ on $E_\alpha$ and a harmonic function $h$ such that 
$$
H(z) = -U^{\mu_B}(z) +h(z), \quad z\in \C.
$$
By \eqref{defH},
$$
-C^{\mu_B}(z)+u'(z)= \begin{cases}
\xi_2(z)+ z^2   , &  z\in \Omega_\alpha,  
\\
\xi_3(z)+ z^2  , & z \in  \C\setminus (\Omega_\alpha \cup E_\alpha),
\end{cases}
$$
where $u$ is an analytic function such that $\Re u=h$. Since by \eqref{asymptotics_xi_functions},
$$
\xi_3(z)+ z^2 =  \frac{1-\alpha}{z}+\Boh(z^{-2}),
$$
we conclude that $|\mu_B|=1-\alpha$ and $h\equiv \const$. 

The expression for $\mu'_B$ is recovered using the Sokhotsky-Plemelj formula. 

Finally, by \eqref{definition_xi_functions},
$$
-C^{\mu_2-\mu_3}(z)=\xi_3(z)+ z^2 =-C^{\mu_B}(z), \quad z \in E_\alpha,
$$
which is equivalent to \eqref{balayage}.
\end{proof}

\section{The Riemann-Hilbert formulation} \label{sec4:RH}

In this section, we characterize the polynomials defined by \eqref{mops_conditionsTypeI} and \eqref{mops_conditions}  in terms of a $3\times 3$ non-commutative boundary value problem. It will be convenient to use the following matrix-related notation: for $a, b, c\in \C$, 
$$
\diag(a, b, c) \isdef \begin{pmatrix}
a & 0 & 0 \\
0 & b & 0 \\
0 & 0 & c
\end{pmatrix};
$$
 $\bm e_1$, $\bm e_2$, $\bm e_3$ are the column vectors of the $3\times 3$ identity  matrix $\bm I=\diag(1,1,1)$, and 
\begin{equation*} 
\bm M_{ij}\isdef \bm e_i \bm e_j^T.
\end{equation*}
Thus,  $\bm M_{ij}$ is the $3\times 3$ matrix whose only non-zero element is $1$ in the position $(i,j)$. A straightforward consequence of this definition is that
$$
\bm M_{ij}\bm M_{ab}=\delta_{ja}\bm M_{ib},
$$
so that if $f$, $g$ are scalars, then for $j\neq a$,
$$
(\bm I+ f \bm M_{ij})(I+ g \bm M_{ab})=\bm I + f\bm M_{ij}+ g\bm M_{ab};
$$
in particular, for $i\neq j$, 
$$
(\bm I+ f \bm M_{ij})^{-1}=I-f\bm M_{ij}.
$$

Let two unbounded contours $\gamma_1$ and $\gamma_2$ extended to $\infty$ on their two ends along the directions determined by the angles $-2\pi/3$ and $0$, and $-2\pi/3$ and $2\pi/3$, respectively (see Figure~\ref{figure_contours0}, left; we assume them oriented as depicted there). 
For $n, m \in \N\cup \{0 \}$, $N=n+m$, consider 
the following non-commutative Riemann-Hilbert problem (RHP): find a matrix-valued function $\bm Y:\, \C\setminus (\gamma_1\cup \gamma_2)\to \C^{3\times 3}$, such that 
\begin{enumerate}
	\item[$\bullet$] $\bm Y:\C\setminus (\gamma_1\cup \gamma_2)\to \C^{3\times 3}$ is analytic;
	
	\item[$\bullet$] $\bm Y$ has continuous boundary values $\bm Y_{\pm}$ on $\gamma_1\cup \gamma_2$, and $\bm Y_+(z)=\bm Y_-(z)\bm J_Y(z)$ for $ z\in \gamma_1\cup \gamma_2$, where  
	$$
	\bm J_Y(z)\isdef 
	\begin{cases}
	I + e^{-Nz^3} M_{12},  
	&  z\in \gamma_1   \\
	I + e^{-Nz^3} M_{13},  
	& z\in \gamma_2.
	\end{cases}
	$$
	
	\item[$\bullet$] 
	$\displaystyle{ 
		\bm Y(z)=(I+\Boh(z^{-1})) \diag \left(z^N,z^{-n},z^{-m}  \right)
	}$, $z\to \infty$.
\end{enumerate}
Obviously, matrix $\bm Y$ depends on $m$ and $n$ through the asymptotic condition at infinity, but we omit the explicit reference to $(m,n)$ from the notation whenever it cannot lead us into confusion. 

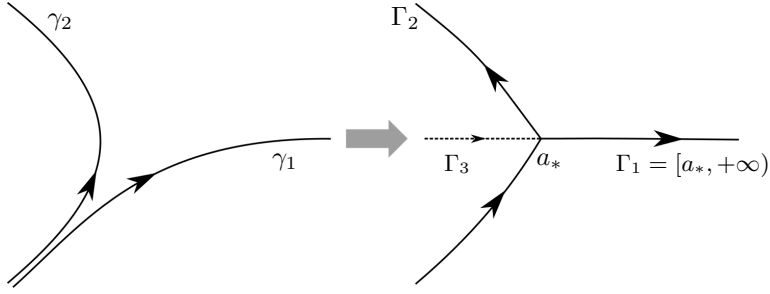
\begin{figure}[t]
	\centering
	\begin{overpic}[scale=0.8]
		{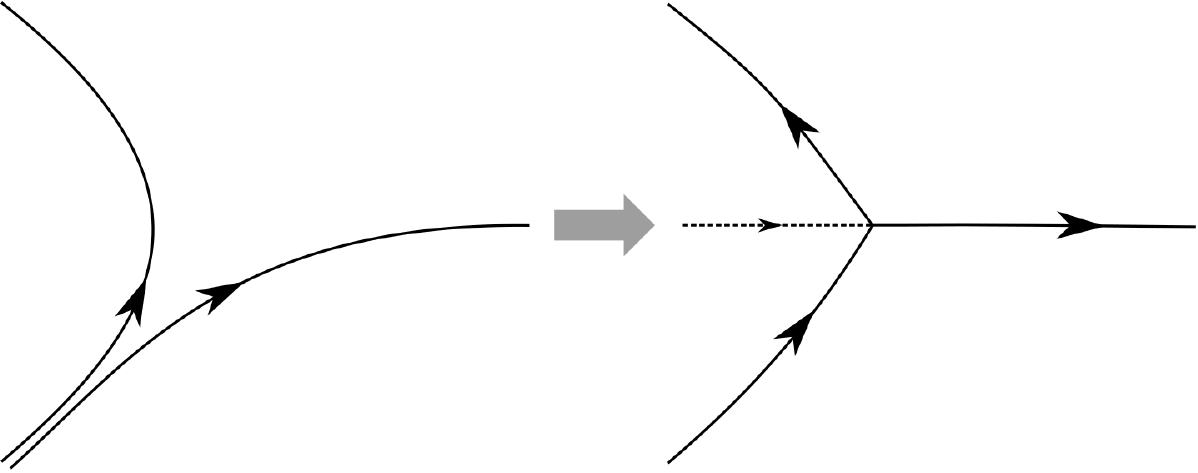}
		\put(15,100){$\gamma_2$}
 \put(100,47){$\gamma_1$}
\put(145,100){$\Gamma_2$}
 	\put(230,45){\small $\Gamma_1=[a_*,+\infty)$}
 		\put(200,47){$a_*$}
\put(165,45){\small $\Gamma_3$}
	\end{overpic}
	\caption{Orthogonality contours $\gamma_1$, $\gamma_2$ (left), and their deformation into a connected set $\Gamma_1\cup \Gamma_2$.}\label{figure_contours0}
\end{figure}

Recall that for $\alpha\in [0,1/2]$, the typical structure of the support of the components of the critical vector measure $\vec \mu_\alpha=(\mu_1, \mu_2, \mu_3)$, described above, is depicted on Figure~\ref{figure_numerics_supports}. More precisely, we consider three oriented sets, $\Delta_j=\supp \mu_j$, $j=1, 2, 3$, (which depend on $\alpha\in [0,1/2]$). Then 
$\Delta_2$ is a piece-wise analytic arc joining the complex-conjugate branch points $a_2,b_2$, $\Im a_2 < 0$, and oriented from $a_2$ to $b_2$. We denote $a_*=\Delta_2 \cap \R$.

In the subcritical regime ($0<\tau <\tau_c$), $\Delta_1=[a_1, b_1]\subset \R$ oriented from $a_1$ to $b_1$, and set $\Delta_3=\emptyset$. In the supercritical regime ($\tau_c<\tau<1/4$), $\Delta_1=[a_*,b_1]\subset \R$ and $\Delta_3=[a_1,a_*]\subset \R$, both with the natural orientation.

Furthermore, let 
$\Gamma_1=[a_*,+\infty)\supset \Delta_1$,  $\Gamma_3=(-\infty,a_*]\supset \Delta_3$, and $\Gamma_2$ a piece-wise analytic curve extended to $\infty$ on its both ends along the directions determined by the angles $-2\pi/3$ and $2\pi/3$, and containing $\Delta_2$. We choose the orientation of $\Gamma_j$'s consistent with those of $\Delta_j$; observe that $\Gamma_1 \cup \Gamma_2\cup\Gamma_3$ is a connected set, see  Figures~\ref{figure_geometry_support_precritical} and \ref{figure_geometry_support_supercritical}.

\begin{figure}[t]
	\centering
	\begin{overpic}[scale=0.6]
		{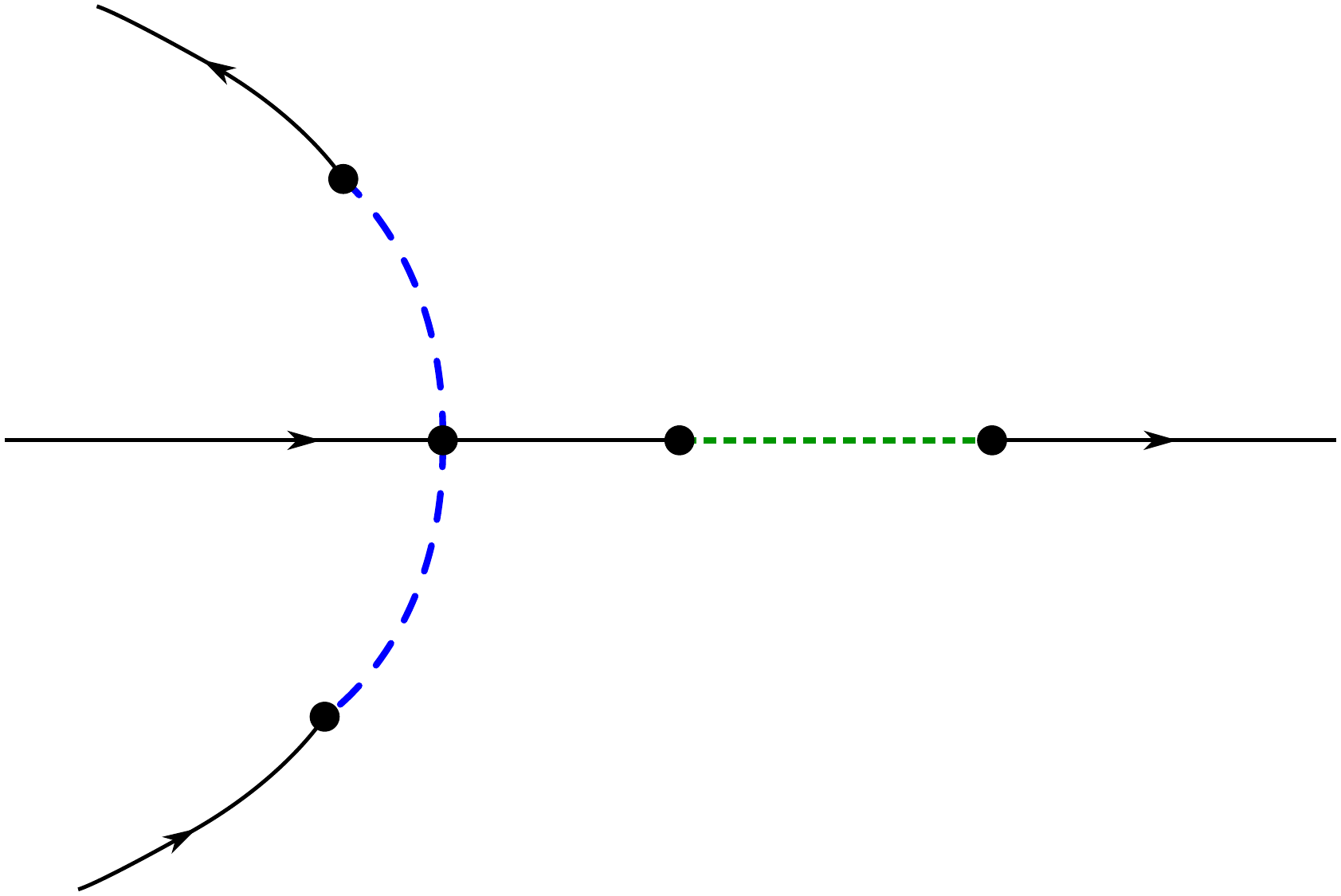}
		\put(150,104){$a_1$}
		\put(85,104){$a_*$}
		\put(10,104){$\Gamma_3=(-\infty,a_*]$}
		\put(220,104){$b_1$}
		\put(240,104){ $\Gamma_1=[a_*,+\infty)$}
		\put(72,30){$a_2$}
		\put(80,155){$b_2$}
		\put(98,125){\textcolor{blue}{$\Delta_2$}}
		\put(180,104){\textcolor{green}{$\Delta_1$}}
		\put(61,171){$\Gamma_2$}
	\end{overpic}
	\caption{Pictorial representation of the sets $\Delta_1$, $\Delta_2$, $\Gamma_1$, $\Gamma_2$ and $\Gamma_3$ in the subcritical regime.}\label{figure_geometry_support_precritical}
\end{figure}

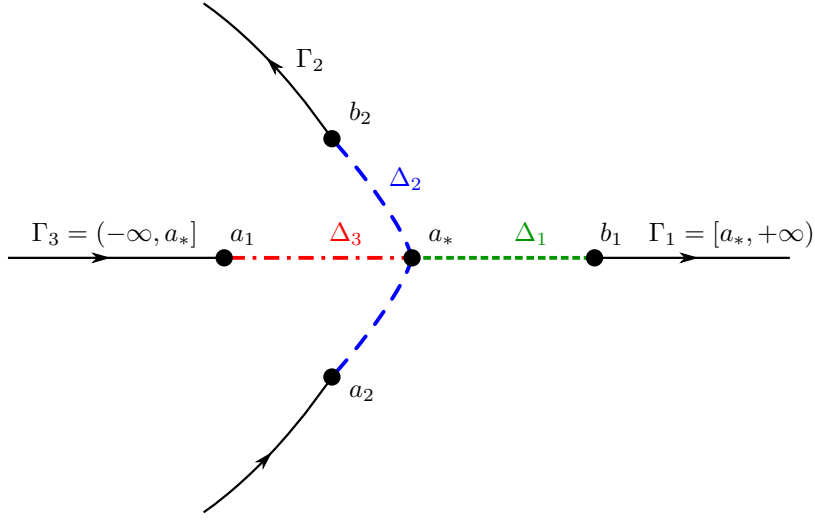
\begin{figure}[t]
	\centering
	\begin{overpic}[scale=0.6]
		{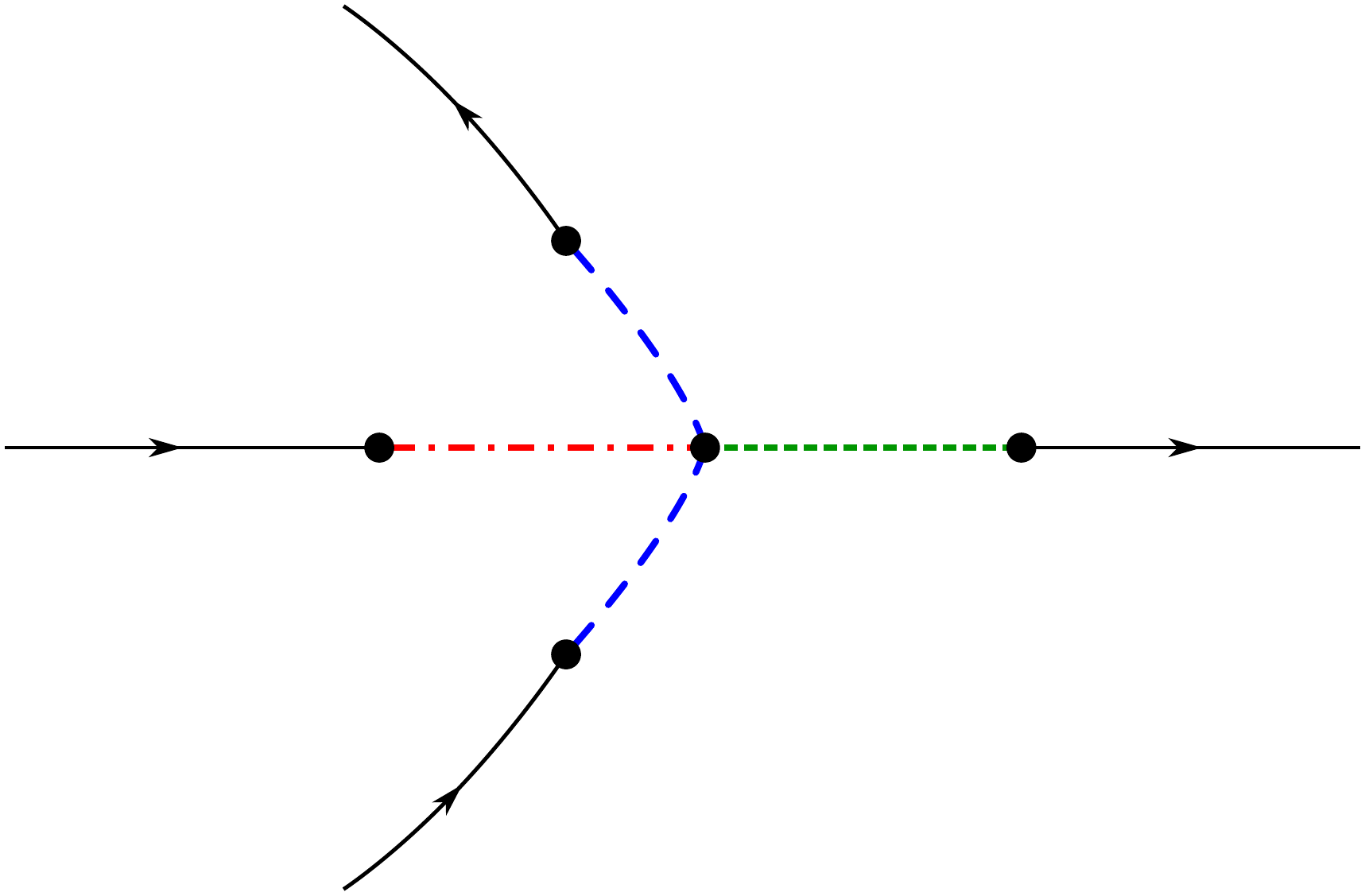}
		\put(160,104){$a_*$}
		\put(85,104){$a_1$}
		\put(10,104){$\Gamma_3=(-\infty,a_*]$}
		\put(225,104){$b_1$}
		\put(240,104){ $\Gamma_1=[a_*,+\infty)$}
		\put(130,150){$b_2$}
		\put(130,45){$a_2$}
		\put(145,125){\textcolor{blue}{$\Delta_2$}}
		\put(192.5,104){\textcolor{green}{$\Delta_1$}}
		\put(122.5,104){\textcolor{red}{$\Delta_3$}}
		\put(110,170){$\Gamma_2$}
	\end{overpic}
	\caption{Pictorial representation of the sets $\Delta_1$, $\Delta_2$, $\Delta_3$, $\Gamma_1$, $\Gamma_2$ and $\Gamma_3$ in the supercritical regime.}\label{figure_geometry_support_supercritical}
\end{figure}

Since the orthogonality conditions in \eqref{mops_conditions} are non-hermitian, we can deform the above mentioned contours $\gamma_1$ and $\gamma_3$ freely, preserving their asymptotic directions. 
In particular, we can make $\gamma_2$ coincide with $\Gamma_2$, while $\gamma_1$ will follow $\Gamma_2\cap \H_-$ and $\Gamma_1$ (see Figure~\ref{figure_contours0}, right). Thus, $\bm Y$ can be alternatively characterized by the following RHP:
\begin{enumerate}
	\item[$\bullet$] $\bm Y:\C\setminus \Gamma_Y\to \C^{3\times 3}$ is analytic, where $\Gamma_Y=\Gamma_1\cup \Gamma_2$;
	
	\item[$\bullet$] $\bm Y$ has continuous boundary values $\bm Y_{\pm}$ on $\Gamma_Y$, and $\bm Y_+(z)=\bm Y_-(z)\bm J_Y(z)$ for $z\in \Gamma_Y$, with\footnote{ Here and in what follows we will allow a slight abuse of notation for the sake of simplicity: whenever we speak about boundary values of a function on a set $\Gamma$, we refer to its values on $\overset{\circ}{\Gamma}$, where they are well defined, although frequently dropping the superscript $\circ$.}
	\begin{equation}\label{RHPY1}
	\bm J_Y(z)\isdef 
	\begin{cases}
	\bm I + e^{-Nz^3} \bm M_{12},   
	&  z\in \Gamma_1 \\
\bm 	I + e^{-Nz^3} \left( \bm M_{12}+  \bm M_{13}\right),   
	& z\in \Gamma_2\cap \mathbb H_-, \\
\bm 	I + e^{-Nz^3} \bm M_{13},   
	& z\in \Gamma_2\cap \mathbb H_+.
	\end{cases}
	\end{equation}

	\item[$\bullet$] 
	$\displaystyle{ 
	\bm 	Y(z)=(\bm I+\Boh(z^{-1})) 
		\diag \left(z^N,z^{-n},z^{-m}  \right)
	}$, $z\to \infty$.
\item[$\bullet$] $\bm Y(z)$ is bounded as $z\to a_*$.
\end{enumerate}

Standard arguments imply that $\det \bm Y\equiv 1$, so that $\bm Y$ is invertible everywhere on the plane. As shown in \cite{Assche01}, the polynomial $P_{n,m}$ coincides with the $(1,1)$ entry of the matrix $\bm Y$, while polynomials $A_{n,m}$ and $B_{n,m}$ are (up to a factor $2\pi i$) the $(2,1)$ and $(3,1)$ entries of the inverse matrix $\bm Y^{-1}$, respectively.

\section{Steepest Descent Analysis} \label{sec:steepestdescent}

We now use the Riemann-Hilbert problem characterization, discussed in the previous section, as the starting point for our steepest descent analysis whose goal is to establish the detailed asymptotics of the solution $\bm Y$ above under the regime
$$
N=n+m\to \infty,\quad \frac{n}{N}\to \alpha\in (0,1/2)\setminus\{\alpha_c\},\quad \frac{m}{N}\to 1-\alpha.
$$

For the moment we actually assume and additional constraint,
\begin{equation}\label{assumption_rational}
\frac{n}{N}=\alpha,\quad \frac{m}{N}=1-\alpha ,
\end{equation}
so that, in particular, we are restricted to $\alpha \in \Q$. This restriction is only for convenience, and in Section \ref{sec:generalparameter} we extend our analysis to the general situation.

Recall that the  Deift-Zhou nonlinear steepest descent method  consists in a number of transformations of  our Riemann-Hilbert characterization above in order to bring it to an equivalent ``close-to-identity'' problem. We will perform this analysis for the subcritical and supercritical cases in parallel.

\subsection{Preliminary transformation}\label{prelim}

Our first goal is to modify slightly the structure of the jump matrix $\bm J_Y$ in \eqref{RHPY1} in a neighborhood of the curve $\Delta_2$ in the lower half plane\footnote{ At this stage, it might appear a technical step, but it will turn out to be linked to the essence of the problem.}.  

To this end, we consider a closed Jordan contour $\gamma\subset \overline{\mathbb H}_-$, oriented clockwise, starting and ending at $a_*$ and encircling $\Delta_2\cap \mathbb H_-$. By Proposition~\ref{proposition_signs_Phi}, we can always take $\gamma\cap \H_-$ lying in the domain $\Omega_+$, where $\Re \Phi_3(z)>0$, and choose  $\gamma$ arbitrarily close to the set $E_\alpha\cap \H_-$, see Definition~\ref{defE}. We assume also that for some $\varepsilon>0$,  
\begin{equation}\label{initial_assumptions_gamma}
\gamma\cap \R = 
\begin{cases}
\{a_* \}, & \tau <\tau_c, \\
[a_1-\varepsilon,a_*+\varepsilon], & \tau>\tau_c, 
\end{cases}
\end{equation} 
see Figures \ref{figure_first_transformation_precritical} and \ref{figure_first_transformation_supercritical}.  
We denote by $\mathcal U \supset (\Delta_2 \cap  \mathbb H_-)$ the bounded domain in $\H_-$ encircled by $\gamma$.

\begin{figure}[t]
	\centering
	\begin{overpic}[scale=0.6]
		{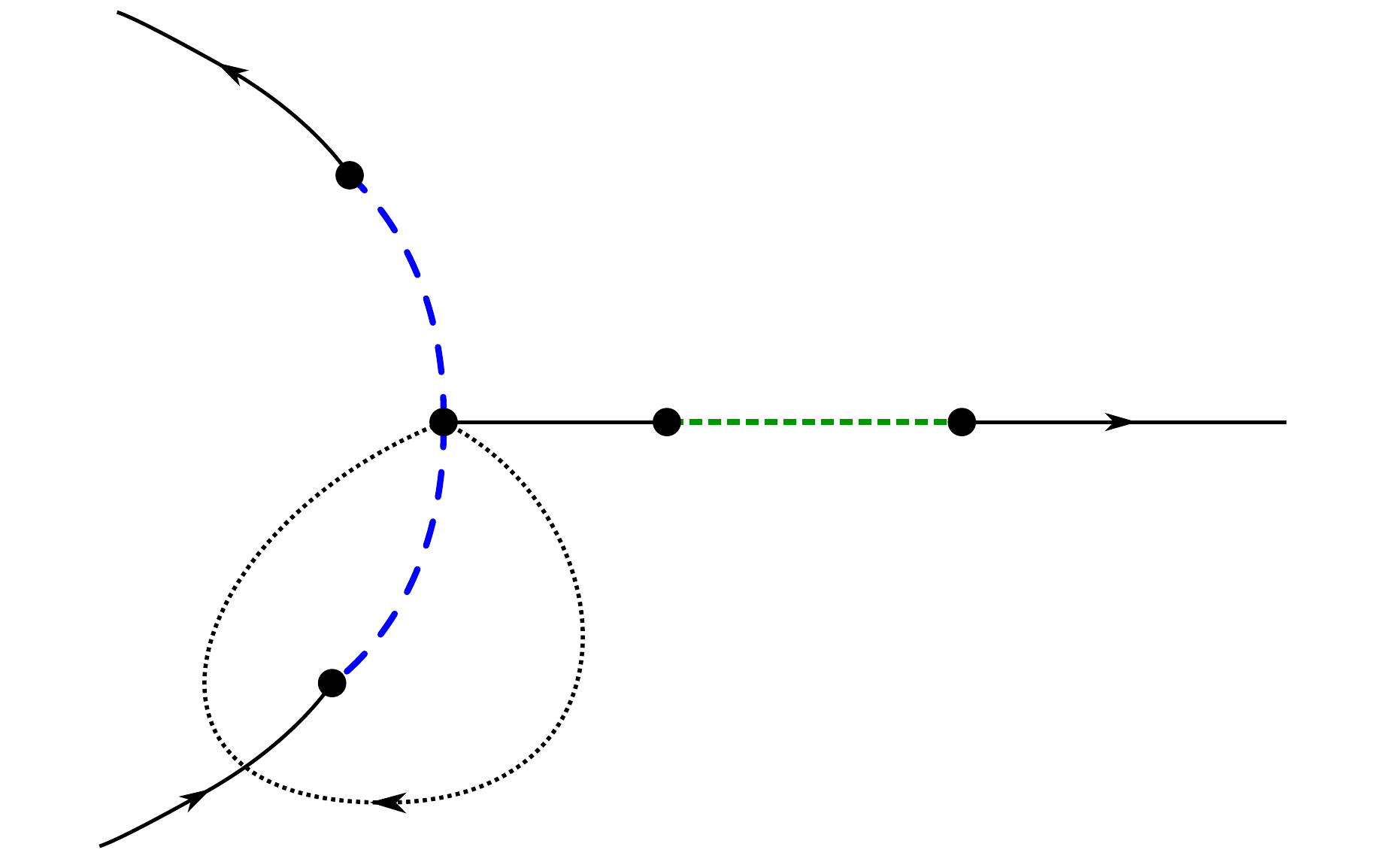}
		\put(141,107){$a_1$}
		\put(88,102){$a_*$}
		\put(118,75){$\gamma$}
		\put(110,50){$\mathcal U$}
		\put(210,107){$b_1$}
		\put(240,107){ $\Gamma_1=[a_*,+\infty)$}
		\put(62,39){$a_2$}
		\put(68,155){$b_2$}
		\put(85,125){\textcolor{blue}{$\Delta_2$}}
		\put(180,105){\textcolor{green}{$\Delta_1$}}
		\put(39,179){$\Gamma_2$}
	\end{overpic}
	\caption{Pictorial representation of the contour $\Gamma_X=\Gamma_1\cup\Gamma_2\cup \gamma$ and the region $\mathcal U$ in the subcritical  regime.}\label{figure_first_transformation_precritical}
\end{figure}

\begin{figure}[t]
	\centering
	\begin{overpic}[scale=0.6]
		{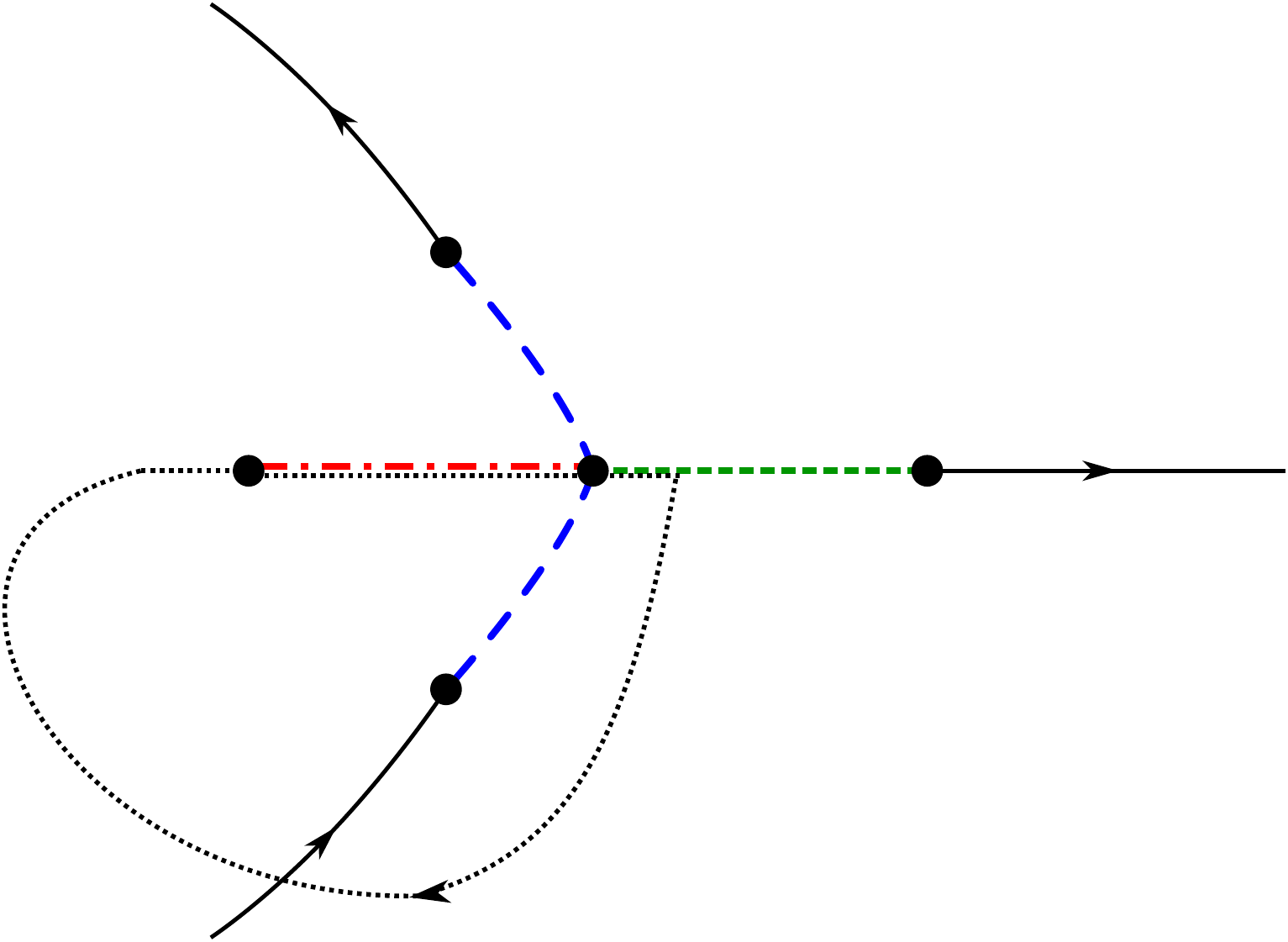}
		\put(50,104){$a_1$}
		\put(125,102){$a_*$}
		\put(112,15){$\gamma$}
		\put(50,60){$\mathcal U$}
		\put(195,102){$b_1$}
		\put(220,104){ $\Gamma_1=[a_*,+\infty)$}
		\put(95,45){$a_2$}
		\put(98,140){$b_2$}
		\put(110,125){\textcolor{blue}{$\Delta_2$}}
		\put(155,102){\textcolor{green}{$\Delta_1$}}
		\put(80,102){\textcolor{red}{$\Delta_3$}}
		\put(73,170){$\Gamma_2$}
	\end{overpic}
	\caption{Pictorial representation of the contour $\Gamma_X=\Gamma_1\cup\Gamma_2\cup \gamma$ and the region $\mathcal U$ in the supercritical 
		regime.}\label{figure_first_transformation_supercritical}
\end{figure}

We start with the  transformation
\begin{equation}
\bm X(z)\isdef
\begin{cases}
\bm Y(z)(\bm I-\bm M_{32}),& z\in \mathcal U, \\
\bm Y(z),& \mbox{otherwise }, 
\end{cases}
\end{equation}\label{def:transformationX}
so that  $\bm X$ satisfies the following RHP:
\begin{enumerate}
	\item[$\bullet$] $\bm X:\C\setminus \Gamma_X\to \C^{3\times 3}$ is analytic, where $\Gamma_X=\Gamma_Y\cup\gamma$;
	
	\item[$\bullet$] $\bm X$ has continuous boundary values $\bm X_{\pm}$ on $\Gamma_X$, and  $\bm X_+(z)=\bm X_-(z)\bm J_X(z)$ for $z\in \Gamma_X$, with
	$$
	\bm J_X(z)\isdef 
	\begin{cases}
   \bm  I+e^{-Nz^3}\bm M_{12},
	& z\in \Gamma_1\setminus\gamma \\
\bm 	I+e^{Nz^3}\bm M_{12} +\bm M_{32}, 
	& z\in \Delta_1 \cap \gamma   \text{ (if } \tau>\tau_c \text{)}, \\
\bm 	I+e^{-Nz^3}\bm M_{13}, 
	&  
	z\in \Gamma_2\cap (\mathcal U\cup   \mathbb H_+), \\
\bm I+e^{-Nz^3} \left( \bm M_{12}+ \bm M_{13}\right), 
	& z\in (\Gamma_2\cap \mathbb H_-)\setminus \mathcal U, \\
\bm 	I+\bm M_{32}, 
	& z\in \gamma\setminus \Delta_1, \\
	\end{cases}
	$$
	
	\item[$\bullet$] 
	$\displaystyle{ 
	\bm 	X(z)=(\bm I+\Boh(z^{-1})) 
	\diag \left(z^N,z^{-n},z^{-m}  \right),
		\quad \mbox{ as } z\to\infty.
	}
	$
\item[$\bullet$] $\bm X(z)$ is bounded at all boundary points of the jump contours.
\end{enumerate}

\subsection{First transformation} \label{sec:firstransf}

Now we normalize the behavior of the RHP at infinity, using functions $g_j$'s and constants $r_j$'s, defined in~\eqref{defGfunctions} and \eqref{defRforGfunctions}, respectively. 
Let $\Gamma_T=\Gamma_X$, and set
\begin{multline}\label{def:transformationT}
\bm T(z)\isdef  \diag \left(e^{-Nr_1}, e^{-Nr_2}, e^{-Nr_3}  \right)
\bm X(z) 
\\
\times \diag \left(e^{N(g_1(z)-\frac{2}{3}z^3)}, e^{N(g_2(z)+\frac{1}{3}z^3)}, e^{N(g_3(z)+\frac{1}{3}z^3)} \right),
\quad z\in \C\setminus \Gamma_T.
\end{multline}
A priori, $\bm T$ is analytic on $\C\setminus (\Gamma_T \cup (-\infty, a_*] )$, and by the assumption \eqref{assumption_rational},  
$$
\bm T(z)=\bm I+\Boh(z^{-1}),\quad \mbox{ as } z\to\infty,
$$
which was the primary goal of this transformation.

We analyze the jumps of $\bm T$ on the curves comprising $\Gamma_T \cup (-\infty, a_*)$; it is convenient to use here the notation
$$
\diag e^{N\Delta g}=\diag (e^{N\Delta g_1},e^{N\Delta g_2},e^{N\Delta g_3}), \quad \Delta g_j = g_{j+}-g_{j-}.
$$

Using the value of the periods \eqref{periods}, we see that on $ \R\setminus\Gamma_1 =  (-\infty, a_*)$,
\begin{equation} \label{identityG1anew}
 g_{1+}-g_{1-}=\ointctrclockwise_{\Delta_1\cup\Delta_2}\xi_1 = -2\pi i  .  
\end{equation}
In the same vein, on $ \R\setminus(\Gamma_1 \cup \Delta_3)$,
\begin{equation}
\begin{split}
& g_{2+}-g_{2-}=\ointctrclockwise_{\Delta_1\cup\Delta_3}\xi_2 = 2\pi i \alpha,   \\	
& g_{3+}-g_{3-}=\ointctrclockwise_{\Delta_2\cup\Delta_3}\xi_3 = 2\pi i (1-\alpha),
\end{split} 
\label{identityG3new}
\end{equation}
and again, due to the crucial assumption \eqref{assumption_rational}, in \eqref{identityG1anew}--\eqref{identityG3new}, all expressions of the form $N  \Delta g_{i}$ are integer multiples of $2\pi i$. In particular, this shows that $\bm T$ is actually analytic on $\C\setminus \Gamma_T $.
From the RHP for $\bm X$ we get that $\bm T_+(z)=\bm T_-(z)\bm J_T(z)$, $ z\in \Gamma_T$, with
\begin{multline*}
\bm J_T= \\
\begin{cases}
\diag e^{N\Delta g} + e^{N(g_{2+}-g_{1-})}M_{12}, & z\in \Gamma_1\setminus \gamma, \\
\diag e^{N\Delta g} + e^{N(g_{2+}-g_{1-})}M_{12} + e^{N(g_{2+}-g_{3-})}M_{32}, & z\in \gamma \cap \Delta_1, \text{ (if $\tau>\tau_c$)}, \\
\diag e^{N\Delta g} + e^{N(g_{3+}-g_{1-})}M_{13}, & z\in \Gamma_2\cap (\mathcal U\cup   \mathbb H_+), \\
\diag e^{N\Delta g} + e^{N(g_{2+}-g_{1-})}M_{12}+ e^{N(g_{3+}-g_{1-})}M_{13}, & z\in (\Gamma_2\cap \mathbb H_-)\setminus \mathcal U, \\
\diag e^{N\Delta g} + e^{N(g_{2+ }-g_{3 -})}M_{32}, & z\in \gamma\setminus \Delta_1, 
\end{cases}
\end{multline*}
where we have used the structure of cuts in the definition in \eqref{defGfunctions}.

The expression for some entries of the jump matrix $\bm J_T$ can be simplified, at least on subsets of the jump contours:
\begin{enumerate}[\itshape (i)]
	\item On $\Gamma_1$, by \eqref{g_functions_constants},
	\begin{equation*}
	 g_{2+}(z)-g_{1-}(z)=\int_{b_1}^z(\xi_{2+}-\xi_{1-})ds +c_2-c_1=\int_{b_1}^z(\xi_{2+}-\xi_{1-})ds.
	\end{equation*}
	Since $\xi_{2+}=\xi_{1-}$ on $\overset{\circ}{\Delta}_1:=\Delta_1\setminus \{\max(a_1,a_*), b_1 \}$ (see \eqref{equality_xi_2}), we  conclude
	in particular that $g_{2+}(z)-g_{1-}(z)=0$, $z\in \Delta_1$.

		\item On $\Delta_2\cap \H_-$, by \eqref{equality_xi_3} and \eqref{g_functions_constants},
		\begin{equation} \label{G3pusminusG1minus}
		\begin{split}
		g_{3+}(z)-g_{1-}(z) & =\int_{a_2}^z(\xi_{3+}-\xi_{1-})ds+\int_{a_*}^{a_2}(\xi_{3-} -\xi_{1-})ds\\ &  -\int_{b_1}^{a_*}\xi_{1-}ds +c_3-c_1 
	  = \int_{a_2}^{z}(\xi_{3+}-\xi_{1-})ds =0.
		\end{split}
		\end{equation}
		On the other hand, by \eqref{symmetryforG} we have on $\Delta_2 \cap \mathbb H_+$:
		$$
		g_{3+}(z)-g_{1-}(z)=  c_3- \overline{ c_3}-c_1+\overline{ c_1}.
		$$
		Using \eqref{periods} and \eqref{g_functions_constants} we get that
		\begin{equation}\label{identity_periods1}
		g_{3+}(z)-g_{1-}(z) =  \ointclockwise_{\Delta_1\cup\Delta_2}\xi_1 = 
		2\pi i   \quad \text{on $\Delta_2 \cap \mathbb H_+$}.
		\end{equation}
	
\item On $ \Delta_3 $   (for $\tau>\tau_c$), by \eqref{equality_xi_2}, \eqref{equality_xi_5} and \eqref{g_functions_constants},
\begin{equation}\label{G2plusminusG3minus}
 g_{2+}(z)-g_{3-}(z)=\int_{a_*}^z \left(\xi_{2+}-\xi_{3-} \right)+ \int_{b_1}^{a_*} \left(\xi_{2+}-\xi_{1-} \right)=0,
\end{equation}
and analogously,
\begin{equation}\label{G3plusminusG2minus}
g_{3+}(z)-g_{2-}(z)= \varointclockwise_{\Delta_1 \cup \Delta_2}\xi_1 = 2 \pi i.
\end{equation}
	\item On $(a_*,a_1)$, for $\tau<\tau_c$,
\begin{equation} \label{identityG1new}
\begin{split}
& g_{1+}-g_{1-}=\ointctrclockwise_{\Delta_1}\xi_1=\varointclockwise_{\Delta_1}\xi_2 = -2\pi i \alpha ,  \\
& g_{2+}-g_{2-}=\ointctrclockwise_{\Delta_1}\xi_2 = 2\pi i \alpha .  
\end{split}
\end{equation}

\end{enumerate}

Gathering these identities, we conclude that $\bm T$ satisfies the following RHP (see Figures~\ref{figure_first_transformation_precritical} and \ref{figure_first_transformation_supercritical}):
\begin{enumerate}
	\item[$\bullet$] $\bm T:\C\setminus \Gamma_T\to \C^{3\times 3}$ is analytic;
	
	\item[$\bullet$] $\bm T$ has continuous boundary values $\bm T_{\pm}$ on $\Gamma_T$, and  $\bm T_+(z)=\bm T_-(z)\bm J_T(z)$, $ z\in \Gamma_T$, with
\begin{equation}\label{JTnew}
\begin{split}
\bm J_T=
\begin{cases}
\diag\left(e^{N(g_{1+}-g_{1-})}, e^{N(g_{2+}-g_{2-})}, 1 \right)+\bm M_{12},
&
z\in \Delta_1\setminus\gamma,
\\
%
\begin{multlined}[b]
\diag\left(e^{N(g_{1+}-g_{1-})}, e^{N(g_{2+}-g_{2-})}, 1 \right) \\[-8pt] +\bm M_{12} + e^{N(g_{2+}-g_{3-})} \bm M_{32},
\end{multlined}
&
z\in \Delta_1\cap\gamma \; (\text{if } \tau >\tau_c),
\\
\bm I+ e^{N(g_{2+}-g_{1-})}\bm M_{12}, & z\in \Gamma_1\setminus \Delta_1, \\
\diag\left(e^{N(g_{1+}-g_{1-})},  1 , e^{N(g_{3+}-g_{3-})} \right)+\bm M_{13} ,
&
z \in \Delta_2,
\\
\bm I+ e^{N(g_{3}-g_{1})}\bm M_{13}, & z\in \Gamma_2\cap (\mathcal U \cup \H_+) \setminus \Delta_2, \\
\bm I+ e^{N(g_{2}-g_{1})}\bm M_{12}+e^{N(g_{3}-g_{1})}\bm M_{13}, & z\in (\Gamma_2\cap \H_-)\setminus \mathcal U, \\
\diag\left(1, e^{N(g_{2+}-g_{2-})},    e^{N(g_{3+}-g_{3-})} \right)+\bm M_{32} ,
&
z \in \Delta_3  \; (\text{if } \tau >\tau_c),
\\
\bm I+ e^{N(g_{2}-g_{3})}\bm M_{32}, & z\in \gamma\setminus \R.
\end{cases}
\end{split}
\end{equation}

\item[$\bullet$] 
$\bm 	T(z)=\bm I+\Boh(z^{-1})$  as $ z\to\infty$.

\item[$\bullet$] $\bm T(z)$ is bounded at all end points of the analytic arcs comprising $\Gamma_T$.
\end{enumerate}

\begin{remark}
Notice that the jump on $(\gamma\cap \Gamma_3)\setminus \Delta_3$ has the form $\bm I+ e^{N(g_{2+}-g_{3-})}\bm M_{32}$. However, due to \eqref{identityG3new} we can rewrite it as $\bm I+ e^{N(g_{2-}-g_{3-})}\bm M_{32}$, which justifies stating that $\bm J_T=\bm I+ e^{N(g_{2}-g_{3})}\bm M_{32}$ on $\gamma\setminus (\Delta_1\cup\Delta_3)$, understanding that on the portion of $\gamma$ along $\R$ we use the ``$-$'' boundary values of the entries. 
\end{remark}

An important fact about $\bm J_T$ is that not all its components are asymptotically relevant, since some of the off-diagonal entries of $\bm J_T$ exhibit an exponential decay. Indeed,
\begin{itemize}
\item On $\Gamma_1\setminus \Delta_1$, as we have seen,
\begin{equation} \label{ineqG2G1}
g_{2+}-g_{1-}=\int_{b_1}^z(\xi_{2+}-\xi_{1-})ds
\begin{cases}
<0, & z>b_1, \\
<0, & a_*<z<a_1 \; (\mbox{when } \tau<\tau_c),
\end{cases}
\end{equation}
where the inequalities are a direct consequence of \eqref{inequality_xi_3}--\eqref{inequality_xi_5}.

\item On $\gamma\setminus \R$, by \eqref{value_phi_3},
$$
g_2(z)-g_3(z)    =  -\Phi_3(z) - 2\pi i \alpha+
\begin{cases}
\displaystyle  \int_{a_*}^{a_1} (\xi_{1}-\xi_2)ds  +c_2,  & 0 <\tau < \tau_c, \\
0,  & \tau_c <\tau < 1/4,
\end{cases}  
$$
with $\Phi_3$ defined in \eqref{def_phi_3b}. It remains to use that $c_2\in i\R$, Proposition~\ref{proposition_signs_Phi}, and the fact that by \eqref{inequality_xi_3},  $\xi_1<\xi_2<\xi_3$ on $[a_*,a_1)$ for $0<\tau <\tau_c$, in order to conclude that  $\re(g_2-g_3)<0$ on $\gamma\setminus \R$ (but uniformly up to  $a_*$ for $\tau<\tau_c$).  

\item Finally, by \eqref{value_phi_1}, \eqref{value_phi_2} and Proposition~\ref{proposition_signs_Phi}, we can always take $\Gamma_2$ in such a way that on $\Gamma_2\cap \H_{-}\setminus\Delta_2 $,  
\begin{align*}
\Re \left( g_{2}-g_{1} \right)(z)  =- \Re \Phi _1(z)  
		&  <0, \\
\Re \left( g_{3}-g_{1} \right)(z)  =- \Re \Phi _2(z)  
&  <0.
\end{align*}
\end{itemize}

In summary, on some contours of $\Gamma_T$, the jump matrix $\bm J_T$ is exponentially close to $\bm I$, while on others, it has entries with oscillatory behavior. 
The (now) standard technique to deal with these oscillations is to split the jump contour there in several new contours using  the classical factorization
\begin{equation}\label{factorizationLCR}
\begin{pmatrix}
e^{-A} & 1 \\
0     & e^{A}
\end{pmatrix}
=
\begin{pmatrix}
1 & 0 \\
e^{A} & 1
\end{pmatrix}
\begin{pmatrix}
0 & 1 \\
-1 & 0
\end{pmatrix}
\begin{pmatrix}
1 & 0 \\
e^{-A} & 1
\end{pmatrix}
=\bm L\bm C\bm R,
\end{equation}
with the obvious choices for the notation $\bm L,\bm C,\bm R$. 
This is the next step, in which it is convenient to treat the subcritical and supercritical cases separately.

Motivated by \eqref{factorizationLCR}, we will  make use of a special notation for the following $3 \times 3$ matrices:
\begin{equation} \label{defSigmas}
\bm \sigma_{3j}\isdef \begin{cases}
\bm M_{12}-\bm M_{21}+\bm M_{33}, & \text{if } j=1, \\
\bm M_{13}-\bm M_{31}+\bm M_{22}, & \text{if } j=2, \\
\bm M_{32}-\bm M_{23}+\bm M_{11}, & \text{if } j=3. \\
\end{cases}
\end{equation}
For instance, 
$$
\bm \sigma_{31}= \begin{pmatrix}
0			& 1 				& 0 \\
-1 		 	& 0				& 0 \\
0			& 0				& 1
\end{pmatrix},
$$
etcetera. A feature of these matrices, useful in what follows, is that $\bm \sigma_{3j}^ T= \bm \sigma_{3j}^{-1}$, $j=1, 2,3 $.

\subsection{Opening of lenses in the precritical case} \label{sec:precritical}

Let us start with the case $0<\tau<\tau_c$, when  \eqref{JTnew} simplifies to
$$
J_T=
\begin{cases}
\diag\left(e^{N(g_{1+}-g_{1-})}, e^{N(g_{2+}-g_{2-})}, 1 \right)+\bm M_{12},
&
z\in \Delta_1,
\\
\diag\left(e^{N(g_{1+}-g_{1-})},  1 , e^{N(g_{3+}-g_{3-})} \right)+\bm M_{13} ,
&
z \in \Delta_2,
\\
\bm I+ e^{N(g_{2+}-g_{1-})}\bm M_{12}, & z\in \Gamma_1\setminus \Delta_1, \\
\bm I+ e^{N(g_{3}-g_{1})}\bm M_{13}, & z\in \Gamma_2\cap (\mathcal U \cup \H_+) \setminus \Delta_2, \\
\bm I+ e^{N(g_{2}-g_{1})}\bm M_{12}+e^{N(g_{3}-g_{1})}\bm M_{13}, & z\in (\Gamma_2\cap \H_-)\setminus \mathcal U, \\
\bm I+ e^{N(g_{2}-g_{3})}\bm M_{32}, & z\in \gamma.
\end{cases}
$$

It is convenient to rewrite the jumps on $\Delta_1$ and $\Delta_2$ in terms of the functions $\Phi_j$ introduced in \eqref{def_phi_1}--\eqref{def_phi_3b}.
First, by \eqref{equality_xi_2}, $g_{1\pm}(z)=g_{2\mp}(z)$ for $z\in \Delta_1$, and this gives us the identities below:
\begin{equation}\label{compatibility_for_lenses_delta1}
\begin{aligned}
g_{1+}(z)-g_{1-}(z) 		& =\Phi_{1+}(z), 			&& z\in \Delta_1,\\
g_{2+}(z)-g_{2-}(z) 		& =\Phi_{1-}(z), 			&& z\in \Delta_1, \\
\Phi_{1+}(z)+\Phi_{1-}(z) 	& =0, 					&& z\in \Delta_1. 
\end{aligned}
\end{equation}

In consequence, using \eqref{factorizationLCR} and \eqref{compatibility_for_lenses_delta1}, we can rewrite the jump on $\Delta_1$ as 
\begin{align*}
\bm J_T & =
\begin{pmatrix}
e^{N\Phi_{1+}} 	& 1 					& 0 \\
0   			   	& e^{N\Phi_{1-}}		& 0 \\
0				& 0					& 1
\end{pmatrix} \\
& =
\begin{pmatrix}
1 				& 0 				& 0 \\
e^{N\Phi_{1-}}	& 1				& 0 \\
0				& 0				& 1
\end{pmatrix} 
\begin{pmatrix}
0			& 1 				& 0 \\
-1 		 	& 0				& 0 \\
0			& 0				& 1
\end{pmatrix}
\begin{pmatrix}
1 				& 0 				& 0 \\
e^{N\Phi_{1+}}	& 1				& 0 \\
0				& 0				& 1
\end{pmatrix} \\
&=(\bm I+e^{N\Phi_{1-}}\bm M_{21})\bm \sigma_{31}(\bm I+e^{N\Phi_{1+}}\bm M_{21}), 
\end{align*}
with $\bm \sigma_{31}$ defined in \eqref{defSigmas}. In the same vein, by \eqref{equality_xi_3}, \eqref{G3pusminusG1minus} and \eqref{identity_periods1}, 
\begin{equation}\label{compatibility_for_lenses_delta2}
\begin{aligned}
g_{1+}(z)-g_{1-}(z) 		& \equiv \Phi_{2+}(z) \mod (2\pi i), 	&&  z\in \Delta_2, \\
g_{3+}(z)-g_{3-}(z) 		& \equiv\Phi_{2-}(z) \mod (2\pi i), 	&&  z\in \Delta_2, \\
\Phi_{2+}(z)+\Phi_{2-}(z)	& \equiv 0 \mod (2\pi i), 			&& z\in \Delta_2, 
\end{aligned}
\end{equation}
and again \eqref{factorizationLCR} and now \eqref{compatibility_for_lenses_delta2} yield that $J_T$ on $\Delta_2$ has the form
\begin{equation*}
\bm J_T=(\bm I+e^{N\Phi_{2-}}\bm M_{31})\bm \sigma_{32}(\bm I+e^{N\Phi_{2+}}\bm M_{31}), 
\end{equation*}
where $\bm \sigma_{32}$ was defined in \eqref{defSigmas}.

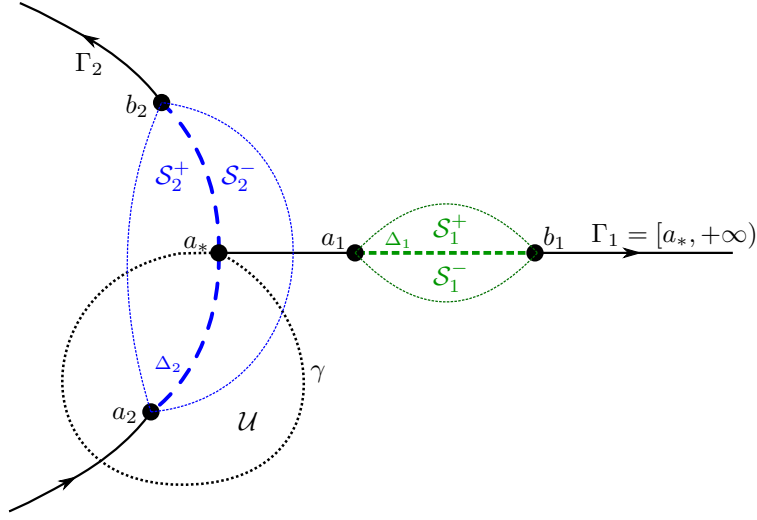
\begin{figure}[t]
	\centering
	\begin{overpic}[scale=0.6]
		{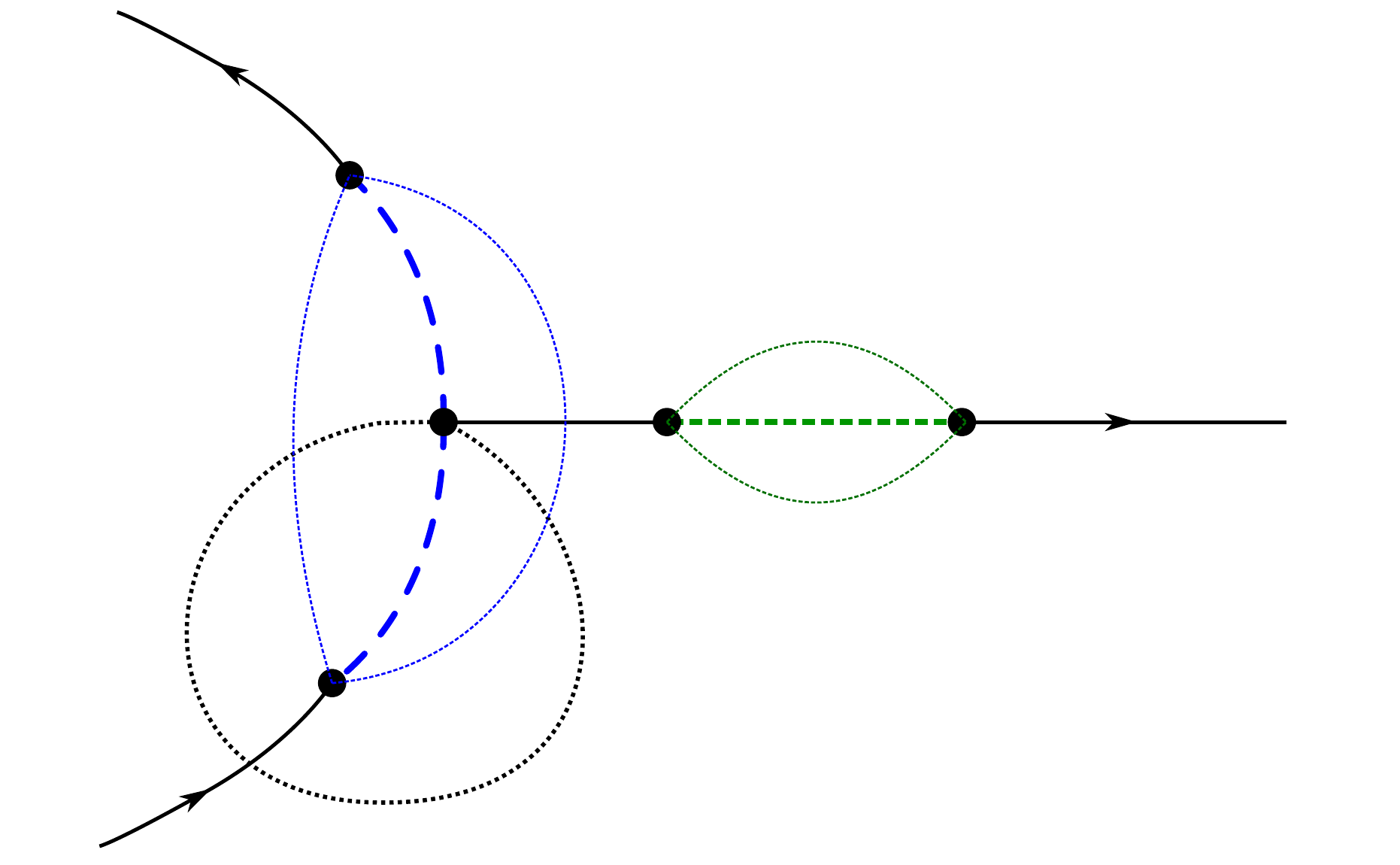}
		\put(141,105){$a_1$}
		\put(89,104){$a_*$}
		\put(137,55){$\gamma$}
		\put(110,35){$\mathcal U$}
		\put(225,105){$b_1$}
	\put(240,106){ $\Gamma_1=[a_*,+\infty)$}
		\put(62,39){$a_2$}
		\put(67,155){$b_2$}
		\put(78,58){\scriptsize{\textcolor{blue}{$\Delta_2$}}}
		\put(75,128){ \textcolor{blue}{$\mathcal S_2^+$}}
		\put(100,128){ \textcolor{blue}{$\mathcal S_2^-$}}
		\put(165,105){\scriptsize{\textcolor{green}{$\Delta_1$}}}
		\put(180,108){ \textcolor{green}{$\mathcal S_1^+$}}
		\put(180,90){ \textcolor{green}{$\mathcal S_1^-$}}
		\put(48,171){$\Gamma_2$}
	\end{overpic}
	\caption{Pictorial representation of the contour $\Gamma_S=\Gamma_X\cup\partial\mathcal S_1\cup \partial\mathcal S_2$ in the subcritical regime.}\label{figure_lenses_precritical}
\end{figure}

These factorizations of the jumps across $\Delta_1$ and $\Delta_2$ motivate to introduce new domains (or ``to open lenses'') $\mathcal S_j=\mathcal S_j^+\cup\mathcal S_j^-$ around $\Delta_j$, $j=1, 2$ (see Figure~\ref{figure_lenses_precritical}), and define
\begin{equation}\label{defS}
\bm S(z)\isdef \bm T(z)\times
\begin{cases}
\bm I, 						& z\in \C\setminus \mathcal S_1\cup\mathcal S_2, \\
\bm I \mp e^{N\Phi_{1}}\bm M_{21}, 	& z\in \mathcal S_1^\pm \\
\bm I \mp e^{N\Phi_{2}}\bm M_{31},	& z\in \mathcal S_2^\pm .
\end{cases}
\end{equation}

Observe that on $(a_*, a_1) $ by \eqref{identityG1new},
\begin{equation}\label{compatibility_for_lenses_delta3}
\Phi_{2+}(z)-\Phi_{2-}(z) =g_{1+}(z)-g_{1-}(z)	 = -2\pi i\alpha, 	\quad  z\in (a_*,a_1),
\end{equation}
so that, in particular, $e^{N\Phi_{2+}(z)}=e^{N\Phi_{2-}(z)}$ on $(a_*,a_1)$. Moreover, as a consequence of \eqref{value_phi_1} and \eqref{identityG1new},
$$
g_{2+}(z)-g_{1-}(z)		  = -\Phi_{1+}(z)-2\pi i \alpha,		\quad  z\in (a_*,a_1). 
$$

Hence, $\bm S$ satisfies the following RHP (see Figure~\ref{figure_lenses_precritical}):
\begin{enumerate}
	\item[$\bullet$] $\bm S:\C\setminus \Gamma_S\to \C^{3\times 3}$ is analytic, where $\Gamma_S=\Gamma_T\cup\partial\mathcal S_1\cup \partial\mathcal S_2$;
	
	\item[$\bullet$] $\bm S$ has continuous boundary values $\bm S_{\pm}$ on $\Gamma_S$, and  $\bm S_+(z)=\bm S_-(z)\bm J_S(z),\ z\in \Gamma_S$, with
	\begin{equation}\label{jumpSpre}
	\bm J_S=
	\begin{cases}
\bm 	I+e^{N\Phi_j}\bm M_{j+1,1}, 			& z\in \partial \mathcal S_j^{\pm}, \quad j=1,2, 	\\
\bm 	\sigma_{3,j},					& z\in \Delta_j,\quad j=1,2,				\\
\bm 	I+e^{-N\Phi_{1+}}(\bm M_{12}-e^{N\Phi_{2-}}\bm M_{32}),	& z\in (a_*, a_1)\cap \mathcal S_2,			\\
	\bm J_T,						& \mbox{elsewhere on } \Gamma_S.
	\end{cases}
	\end{equation}
	
	\item[$\bullet$] 
	$\bm 	S(z)=I+\Boh(z^{-1})$  as $ z\to\infty$.
	
	\item[$\bullet$] $\bm S(z)$ is bounded at all end points of the analytic arcs comprising $\Gamma_S$.
\end{enumerate}
Observe that the jumps on $\Delta_1$ and $\Delta_2$ now are independent of $N$. 

We claim that we can select the domains $\mathcal S_j$ in such a way that the rest of the jumps are exponentially close to $\bm I$. Indeed, by \eqref{inequality_xi_3},  $\re \Phi_{1+}(x)>0$ for $x\in [a_*,a_1)$. Furthermore, recall that $\mu_j$ from \eqref{definition_measures_mu} are positive measures, and we can conclude that
\begin{align*}
-d\Phi_{1+}(s)=(\xi_{1+}(s)-\xi_{2+}(s))ds\in i\R_+, & \quad s\in \Delta_1, \\
-d\Phi_{2+}(s)=(\xi_{1+}(s)-\xi_{3+}(s))ds\in i\R_+, & \quad s\in \Delta_2.
\end{align*}
Thus,
\begin{equation}\label{signonlenses}
\Phi_{j\pm}(z)\in i\R_{\mp},\quad z\in \Delta_j, \quad j=1, 2.
\end{equation}
This is enough to assure that the domains $\mathcal S_j$ can be taken such that 
\begin{equation}\label{correct_decay_lenses_2}
\re \Phi_j(z)<0,\quad z\in \overline{ \mathcal S_j^\pm} \setminus   \Delta_j ,\quad j=1,2;
\end{equation}
we assume that in what follows the inequalities \eqref{correct_decay_lenses_2} hold. 

An important consequence of \eqref{ineqG2G1}, \eqref{compatibility_for_lenses_delta1}, \eqref{signonlenses},  and \eqref{correct_decay_lenses_2}, as well as of our analysis at the end of Section~\ref{sec:firstransf}, is that away from $\Delta_1$ and $\Delta_2$, the jump matrix $\bm J_S$ is  exponentially close to $\bm I$. Moreover, this is true uniformly all the way up to $\Delta_j$ as long as we stay away from the end points $a_j$, $b_j$.

\subsection{Opening of lenses in the supercritical case}

By \eqref{initial_assumptions_gamma} and \eqref{JTnew}, for $ \tau_c<\tau<1/4$,
$$
\bm J_T=
\begin{cases}
\diag\left(e^{N(g_{1+}-g_{1-})}, e^{N(g_{2+}-g_{2-})}, 1 \right)+\bm M_{12},
&
z\in \Delta_1\setminus\gamma= (a_*+\varepsilon, b_1),
\\
\begin{multlined}[b]
\diag\left(e^{N(g_{1+}-g_{1-})}, e^{N(g_{2+}-g_{2-})}, 1 \right) \\[-10pt] +\bm M_{12} + e^{N(g_{2+}-g_{3-})} \bm M_{32}, 
\end{multlined}
&
z\in \Delta_1\cap\gamma=(a_*, a_*+\varepsilon),
\\
\bm I+ e^{N(g_{2 }-g_{1})}\bm M_{12}, & z\in \Gamma_1\setminus \Delta_1=(b_1,+\infty), \\
\diag\left(e^{N(g_{1+}-g_{1-})},  1 , e^{N(g_{3+}-g_{3-})} \right)+\bm M_{13} ,
&
z \in \Delta_2,
\\
\bm I+ e^{N(g_{3}-g_{1})}\bm M_{13}, & z\in \Gamma_2\cap (\mathcal U \cup \H_+) \setminus \Delta_2, \\
\bm I+ e^{N(g_{2}-g_{1})}\bm M_{12}+e^{N(g_{3}-g_{1})}\bm M_{13}, & z\in (\Gamma_2\cap \H_-)\setminus \mathcal U, \\
\diag\left(1, e^{N(g_{2+}-g_{2-})},    e^{N(g_{3+}-g_{3-})} \right)+\bm M_{32} ,
&
z \in \Delta_3 ,
\\
\bm I+ e^{N(g_{2}-g_{3})}\bm M_{32}, & z\in \gamma\setminus (\Delta_1\cup\Delta_3).
\end{cases}
$$
Again, we rewrite these jumps in terms of the functions $\Phi_j$, introduced in \eqref{def_phi_1}--\eqref{def_phi_3b}. On $\Delta_3$, by \eqref{identityG1anew}, \eqref{G2plusminusG3minus} and \eqref{G3plusminusG2minus} we now have
\begin{equation}\label{compatibility_for_lenses_delta3_supercritical}
\begin{aligned}
g_{2+}(z)-g_{2-}(z) 		& =\Phi_{3-}(z) + 2\pi i \alpha,		&&  z\in \Delta_3, \\
g_{3+}(z)-g_{3-}(z) 		& =\Phi_{3+}(z)+ 2\pi i \alpha,		&&  z\in \Delta_3, \\
\Phi_{3+}(z)+\Phi_{3-}(z)	& =2\pi i (1-2\alpha),		&& z\in \Delta_3.
\end{aligned}
\end{equation}
As for $\Delta_1$ and $\Delta_2$, along with the relations \eqref{compatibility_for_lenses_delta1} and \eqref{compatibility_for_lenses_delta2}, we get (using now \eqref{equality_xi_2} and \eqref{value_phi_2}),
\begin{equation}\label{compatibility_for_lenses_delta1_supercritical}
\begin{aligned}
g_{2+}(z)-g_{3-}(z)			& = \Phi_{2-}(z)-2\pi i,				&& z\in \Delta_1, 
\end{aligned}
\end{equation}
and
\begin{equation}\label{compatibility_for_lenses_delta2_supercritical}
\begin{aligned}
g_{1+}(z)-g_{1-}(z) 			& =\Phi_{2+}(z), 				&&  z\in \Delta_2, \\
g_{3+}(z)-g_{3-}(z) 			& =\Phi_{2-}(z), 				&&  z\in \Delta_2, \\
\Phi_{2+}(z)+\Phi_{2-}(z)	& =0, 							&& z\in \Delta_2, \\
\Phi_{3+}(z)-\Phi_{1-}(z)	& = 2\pi i (3\alpha - 1), 	&& z\in \Delta_2\cap \H_+, \\
\Phi_{3+}(z)-\Phi_{1-}(z)	& = 2\pi i \alpha, 				&& z\in \Delta_2\cap \H_-;
\end{aligned}
\end{equation}
furthermore,
\begin{equation}\label{compatibility_for_lenses_delta2_supercriticalBis}
\begin{aligned}
\Phi_{2+}(z)-\Phi_{2-}(z) 	& = -2\pi i(2-\alpha), 			&& z\in \gamma\cap \R, \\ 
\Phi_{2+}(z)-\Phi_{2-}(z) 	& = -2\pi i\alpha, 				&& z\in (a_*,a_1). \\
\end{aligned}
\end{equation}
All these identities are simple consequences of \eqref{value_phi_1}--\eqref{value_phi_3} and of the periods~\eqref{periods}.

Like in Section~\ref{sec:precritical}, the identities above imply that
\begin{equation*}
\bm J_T=
\begin{cases}
(\bm I+e^{N\Phi_{1-}}\bm M_{21})\bm \sigma_{31}(\bm I+e^{N\Phi_{1+}}\bm M_{21}), & z\in \Delta_1\setminus \gamma,  \\
(\bm I+e^{N\Phi_{1-}}\bm M_{21})\bm \sigma_{31}(\bm I+e^{N\Phi_{1+}}\bm M_{21})(I+e^{N\Phi_{2-}}\bm M_{32}) ,  
& z\in \Delta_1\cap \gamma,\\
(\bm I+e^{N\Phi_{2-}}\bm M_{31})\bm \sigma_{32}(\bm I+e^{N\Phi_{2+}}\bm M_{31}), & z\in \Delta_2, \\
(\bm I+e^{N\Phi_{3-}}\bm M_{23})\bm \sigma_{33}(\bm I+e^{N\Phi_{3+}}\bm M_{23}), & z\in \Delta_3, \\
\end{cases}
\end{equation*}
where $\bm \sigma_{33}$ was defined in \eqref{defSigmas}. 
Moreover, from the fact that the measures in \eqref{definition_measures_mu} are positive we conclude that
\begin{align*}
(\xi_{1+}(s)-\xi_{2+}(s))ds\in i\R_+, & \quad s\in \Delta_1, \\
(\xi_{1+}(s)-\xi_{3+}(s))ds\in i\R_+, & \quad s\in \Delta_2, \\
(\xi_{3+}(s)-\xi_{2+}(s))ds\in i\R_+, & \quad s\in \Delta_3.
\end{align*}
As a consequence,
$$
\Phi_{j\pm}(z)\in i\R_{\mp},\quad z\in \Delta_j, \quad j=1,2,3,
$$
which shows again that the inequalities
\begin{equation}\label{correct_decay_lenses_supercritical}
\re\Phi_j(z)<0,
\end{equation}
hold true in a neighborhood of $\interior\Delta_j$, $j=1,2,3$.

These considerations motivate to open the lens $\mathcal S_j = \mathcal S_j^+ \cup \mathcal S_j^-$ around $\Delta_j$ so that \eqref{correct_decay_lenses_supercritical} holds true on $\partial \mathcal S_j^\pm\setminus \{a_1,a_2,b_1,b_2\}$. The 
geometry of the lens is shown in Figure \ref{figure_lenses_supercritical}. We can impose an additional requirement, 
$$
\partial \mathcal S_2^- \cap S_1^- = \gamma\cap S_1^-.
$$

\begin{figure}[t]
	\centering
	\begin{overpic}[scale=.9]
		{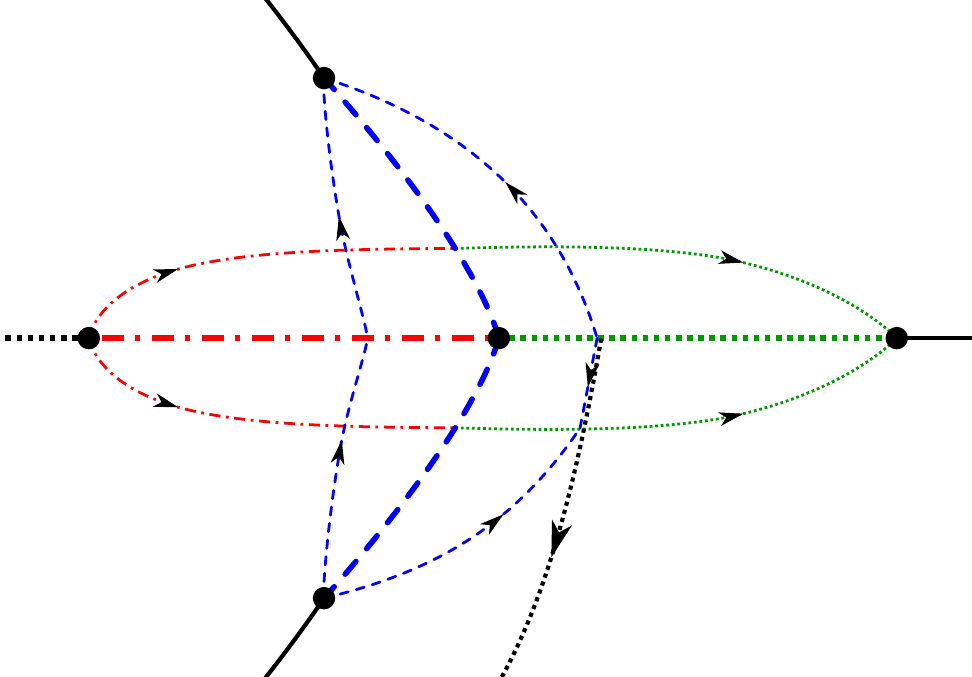}
		\put(170,95){{\textcolor{green}{$\scriptstyle\mathcal S_1^+$}}}
		\put(170,75){{\textcolor{green}{$\scriptstyle\mathcal S_1^-$}}}
		\put(115,45){{\textcolor{blue}{$\scriptstyle\mathcal S_2^-$}}}
		\put(90,45){{\textcolor{blue}{$\scriptstyle\mathcal S_2^+$}}}
		\put(118,120){{\textcolor{blue}{$\scriptstyle\mathcal S_2^-$}}}
		\put(90,120){{\textcolor{blue}{$\scriptstyle\mathcal S_2^+$}}}
		\put(50,95){{\textcolor{red}{$\scriptstyle\mathcal S_3^+$}}}
		\put(50,75){{\textcolor{red}{$\scriptstyle\mathcal S_3^-$}}}
		\put(143,20){$\gamma$}
	\end{overpic}
	\caption{The opening of lenses in the supercritical case.}\label{figure_lenses_supercritical}
\end{figure}
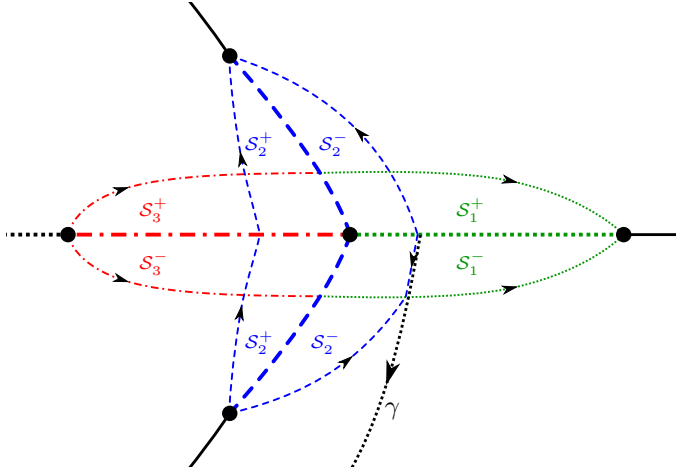
In other words, we define
\begin{equation}\label{defSsuper}
\bm S(z)=\bm T(z)\times
\begin{cases}
\bm I, 								& z\in \C\setminus \mathcal S_1\cup\mathcal S_2\cup \mathcal S_3, \\
\bm I\pm e^{N\Phi_{1}}\bm M_{21}, 					& z\in \mathcal S_1^\mp \setminus \mathcal S_2, \\
\bm I\pm e^{N\Phi_{2}}\bm M_{31},					& z\in \mathcal S_2^\mp \setminus (\mathcal S_1\cup\mathcal S_3), \\
\bm I\pm e^{N\Phi_{3}}\bm M_{23},					& z\in \mathcal S_3^\mp \setminus \mathcal S_2. \\
(\bm I+e^{N\Phi_2}\bm M_{31})(\bm I-e^{N\Phi_1}\bm M_{21}), 			& z\in \mathcal S_1^+\cap \mathcal S_2^-, \\
(\bm I-e^{N\Phi_2}\bm M_{31})(\bm I-e^{N\Phi_3}\bm M_{23}), 			& z\in \mathcal S_2^+\cap \mathcal S_3^+, \\
(\bm I-e^{N\Phi_2}\bm M_{31})(\bm I+e^{N\Phi_3}\bm M_{23}), 			& z\in \mathcal S_2^+\cap \mathcal S_3^-, \\
(\bm I+e^{N\Phi_2}\bm M_{31})(\bm I+e^{N\Phi_1}\bm M_{21}),			& z\in \mathcal S_1^-\cap \mathcal S_2^+. \\
\end{cases}
\end{equation}

With this choice, the jump matrix $\bm J_S$ for $\bm S$ becomes
\begin{equation} \label{defJS}
\bm J_S=
\begin{cases}
\bm \sigma_{3j}, 	&  z\in \Delta_j, \; j=1, 2, 3 \\
\bm I+e^{N\Phi_1}\bm M_{21}, &  z\in \partial \mathcal S_1 \\
\bm I+e^{N\Phi_2}\bm M_{31}, &  z\in \partial \mathcal S_2 \setminus (\gamma\cup \mathcal S_3)\\
\bm I+e^{N\Phi_3}\bm M_{23}, &  z\in \partial \mathcal S_3 \\
\bm I+e^{N\Phi_2}(\bm M_{31}+e^{N\Phi_3}\bm M_{21}), & z\in \partial \mathcal S_2 \cap \mathcal S_3^+  \\
\bm I+e^{N\Phi_2}(\bm M_{31}-e^{N\Phi_3}\bm M_{21}), & z\in \partial \mathcal S_2 \cap \mathcal S_3^- \\
\bm I+e^{N\Phi_2}(\bm M_{32}+(e^{N\Phi_1}-1)\bm M_{31}), & z\in \partial \mathcal S_2 \cap \gamma \\
\bm J_T, & \mbox{elsewhere}.
\end{cases} 
\end{equation}
Still, 	$	\bm S(z)=I+\Boh(z^{-1})$  as $ z\to\infty$, and $\bm S(z)$ is bounded at all end points of the analytic arcs comprising $\Gamma_S=\Gamma_T\cup\partial\mathcal S_1\cup \partial\mathcal S_2 \cup \partial\mathcal S_3$.

As before, we conclude that away from $\Delta_j$'s the jump matrix $\bm J_S$ is  exponentially close to $\bm I$, and   this is true uniformly all the way up to $\Delta_j$'s as long as we stay away from the end points $a_j$, $b_j$.

\subsection{Construction of the global parametrix} \label{sec:globalparam}

In order to deal with the jumps $\bm J_S=\bm \sigma_{3j}$ on $\Delta_j$, $j=1, 2, 3$, see \eqref{defJS},  we will use a model matrix (``global parametrix'') $\bm F$ solving the following RHP\footnote{ Recall that $\Delta_1$ and $\Delta_3$ have the orientation of $\mathbb R$, while $\Delta_2$ is oriented upwards, from $a_2$ to $b_2$.}:
\begin{itemize}
	\item[$\bullet$] $\bm F:\C\setminus \Gamma_F \isdef \Delta_1 \cup \Delta_2 \cup \Delta_3  \to \C^{3\times 3}$ is analytic;
	\item[$\bullet$] $\bm F_+(z)=\bm F_-(z)\bm J_F(z)$, $z\in \Gamma_F$, where $\bm J_F=\bm \sigma_{3,j}$ on $\Delta_j$;
	\item[$\bullet$] $\bm F(z)=\Boh((z-q)^{-1/4})$, as $z\to q\in \{a_1,a_2,b_1,b_2\}$.
	\item[$\bullet$] $\bm F(z)=\bm I+\Boh(z^{-1})$ as $z\to \infty$.
\end{itemize}

The uniqueness of the solution of this problem follows from standard arguments. We construct the global parametrix using the general approach with meromorphic differentials and abelian integrals on the Riemann surface $\mathcal R$, see e.g.~ \cite{kuijlaars_mo}. Namely, for $i=1, 2, 3$, let $\eta_i$ be the meromorphic differential on $\mathcal R$ uniquely defined by the following conditions:
\begin{itemize}
\item its only poles (all simple) are the branch points $a_1^{(1)}$, $a_2^{(1)}$, $b_1^{(1)}$, $b_2^{(1)}$ and the points at infinity $\infty^{(j)}$, $j\in \{1,2, 3 \}$, $j\neq i$; 
\item the residues at these poles are
\begin{align}
& \res(\eta_i,a_j^{(1)})=\res(\eta_i,b_j^{(1)})=-\frac{1}{2},  \quad j=1,2. \label{residues_eta_1}\\
& \res(\eta_i,\infty^{(j)})=1, \quad j\in \{1,2, 3 \}, \quad j\neq i.   \label{residues_eta_2}
\end{align}
\end{itemize}
Since the sum of residues is zero and the genus of $\mathcal R$ is also zero, conditions above determine each meromorphic differential $\eta_i$ uniquely. Moreover, in virtue of \eqref{residues_eta_1}--\eqref{residues_eta_2} we get that for any loop $\zeta\subset \mathcal R_j$, $j=1,2,3$, the differential $\eta$ satisfies
\begin{equation}\label{condition_integral_loop}
\int_\zeta \eta_i \equiv 0 \mod (2\pi i) . 
\end{equation}

We define the corresponding abelian integrals
$$
u_i(p)\isdef \int_{\infty^{(i)}}^p \eta_i,\quad p\in \mathcal R, \quad i=1, 2,3,
$$
specifying in each case the path of integration is as follows\footnote{ For the sheet structure of $\mathcal R$, see Figure~ \ref{figure_sheet_structure}.}:
\begin{enumerate}[(i)]
	\item For $i=1,2,3$, if $p \in \mathcal R_i$, then for $u_i(p)$ we integrate along any path entirely contained in $\mathcal R_i$;
	\item For $u_1(p)$, \begin{itemize}
	\item if $p\in \mathcal R_2$, then the path starting in $\mathcal R_1$, enters $\mathcal R_2$ through the cut $\Delta_{1-}^{(1)}$ and stays in $\mathcal R_2$;
	\item  if $p\in \mathcal R_3$, then the path starting in $\mathcal R_1$,  enters $\mathcal R_3$ through the cut $\Delta_{2-}^{(1)}$ and stays in $\mathcal R_3$;
	\end{itemize} 
\item For $u_2(p)$, \begin{itemize}
	\item if $p\in \mathcal R_1$, then the path starting in $\mathcal R_2$, enters $\mathcal R_1$ through the cut $\Delta_{1+}^{(2)}$ and stays in $\mathcal R_1$;
	\item  if $p\in \mathcal R_3$, then the path used to define $u_2(p)$ on  $\mathcal R_1$,  is continued, enters $\mathcal R_3$ from $\mathcal R_1$ through the cut $\Delta_{2-}^{(1)}$ and stays in $\mathcal R_3$;
\end{itemize} 
\item For $u_3(p)$, \begin{itemize}
	\item if $p\in \mathcal R_1$, then the path starting in $\mathcal R_3$, enters $\mathcal R_1$ through the cut $\Delta_{2+}^{(3)}$ and stays in $\mathcal R_1$;
	\item  if $p\in \mathcal R_2$, then the path  used to define $u_3(p)$ on  $\mathcal R_1$,  is continued, enters $\mathcal R_2$ from $\mathcal R_1$ the cut $\Delta_{1-}^{(1)}$ and stays in $\mathcal R_2$.
\end{itemize} 
\end{enumerate}
The selection of paths is determined by the structure of the jump matrix $\bm J_F$. Observe that in virtue of \eqref{condition_integral_loop}   functions $u_i$ are well defined modulo $2\pi i$. Furthermore, they satisfy, for instance, for  $p\in \Delta_{j}^{(1)}$  and $ j=1,2$,
\begin{align}
u_{1+}(p)-u_{1-}(p) 	& \equiv 2\pi i \res(\eta_1,b_j^{(1)}) \mod (2  \pi i )\nonumber \\
& \equiv \pi i  \mod (2  \pi i ), \quad i=1,2,3,  \label{jump_condition_u_function}
\end{align}
with analogous additive jump conditions for $u_2$ on $\Delta_1 \cap \Delta_3$ and $u_3$ on $\Delta_2\cap \Delta_3$. Moreover, from \eqref{residues_eta_1} we get that they admit the (multivalued) expansion
\begin{equation}\label{local_behavior_u}
u_i(z)=-\frac{1}{4}\log(z-q)+\Boh(1),\quad z\to q\in \{a_1,a_2,b_1,b_2\}.
\end{equation}

Now we define
$$
f_i(p)\isdef e^{u_i(p)},\quad p\in \mathcal R,
$$
and  
$$
f_{ij}\isdef \restr{f_i}{\mathcal R_j},\quad i,j=1,2,3.
$$
Notice that by construction, $f_{i1}$, $f_{i2}$ and $f_{i3}$ are holomorphic on $\C\setminus(\Delta_1\cup\Delta_2)$, $\C\setminus(\Delta_1\cup\Delta_3) $ and $\C\setminus (\Delta_2 \cup\Delta_3)$, 
respectively. Furthermore, in virtue of \eqref{jump_condition_u_function} they satisfy the relations
\begin{equation}\label{conditions_m_1}
\begin{split}
(f_{i1})_+(z) & =-(f_{i, j+1})_-(z),\quad z\in \Delta_j, \; j=1,2,   \\
(f_{i3})_+(z) & =(f_{i, 2})_-(z),\quad z\in \Delta_3, \\
\end{split}
\end{equation}
and from their definition it follows that  also  
\begin{equation}\label{conditions_m_2}
\begin{split}
(f_{i1})_-(z) & =(f_{i, j+1})_+(z),\quad z\in \Delta_j, \; j=1,2, \\
(f_{i3})_+(z) & =(f_{i, 3})_-(z),\quad z\in \Delta_1, \\
(f_{i2})_+(z) & =(f_{i, 2})_-(z),\quad z\in \Delta_2, \\
(f_{i1})_+(z) & =(f_{i, 1})_-(z),\quad z\in \Delta_3, \\
(f_{i3})_-(z) & =(f_{i, 2})_+(z),\quad z\in \Delta_3.
\end{split}
\end{equation}
Moreover, \eqref{local_behavior_u} gives us
\begin{equation*} 
f_{ij}(z)=(z-q)^{-1/4}(1+\boh(1)),\quad z\to q,
\end{equation*}
whenever $q\in \{a_1,a_2,b_1,b_2\}$ is a singularity of $f_{ij}$, $j=1,2,3$. Finally, from the definition of $u_i$'s we conclude that
\begin{equation*} 
f_{ii}(z)=1+\Boh(z^{-1}),\quad z\to\infty, \quad i=1,2,3,
\end{equation*}
while from \eqref{residues_eta_2} we get also that
\begin{equation}\label{conditions_m_5}
f_{ij}(z)=\frac{1}{z}+\Boh(z^{-2}),\quad z\to\infty, \quad i, j\in \{1,2,3 \}, \quad i\neq j.
\end{equation}

Summarizing, conditions \eqref{conditions_m_1}--\eqref{conditions_m_5}, show that 
$$
\bm F\isdef \left( f_{ij}\right)_{i,j=1}^3 
$$
is the unique solution of the RHP stated at the beginning of this section.

\begin{remark} \label{remarkonFinv}
Notice that the RH problem for $\bm F$ easily implies that $\det(\bm F(z))\equiv 1$ on $\C$, so that $\bm F$ is invertible everywhere  on the plane (which is used in the standard argument proving the uniqueness of the solution $\bm F$). Moreover, by \eqref{defSigmas}, 
$$
\sigma_{3,j}^{-1}=\sigma_{3,j}^T, \quad j=1, 2, 3,
$$
so that $ \left( \bm F^{-1}\right)^T$ satisfies the same RHP than $\bm F$. By uniqueness, we conclude that
\begin{equation}\label{invequaltransp}
	\bm F^{-1}(z) = \bm F^{T}(z), \quad z\in \C.
\end{equation}
\end{remark}

\subsection{Parametrices near branch points}\label{sec:localparam}

Our construction shows that the matrix-valued function $\bm F$ defined in the previous section is invertible on $\mathbb C \setminus \Gamma_S$, and $ \bm F^{-1}\bm S$ is analytic in $\mathbb C \setminus \Gamma_S$ and has jumps exponentially close to $\bm I$ everywhere on $ \Gamma_S$ except at the branch points $a_1,a_2,b_1,b_2$. In order to get this proximity to $\bm I$ uniform on all contours we have to replace $\bm F$ by a different model, a local parametrix. Its construction, nowadays standard, is done in terms of Airy functions.

We describe here only the construction at the end point $b_1$, since for the rest of branch points it is similar.

Consider a small fixed disk, $B_\delta=B_\delta(b_1) \isdef \{z\in \mathbb C:\, |z-b_1|< \delta \}$, of center at $b_1$ and radius $ \delta>0$ small enough so that $B_\delta$ does not contain any other branch point of $\mathcal R$ or $a_*$ (see Figure~\ref{fig:localanalysisB1}).  
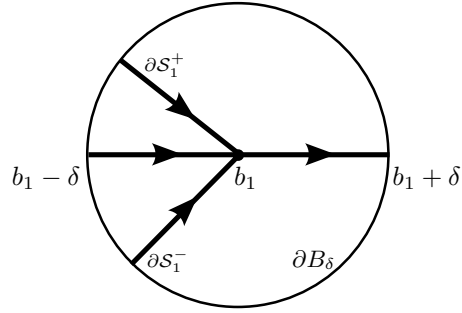
\begin{figure}[t]
	\centering \begin{overpic}[scale=1.2]%
		{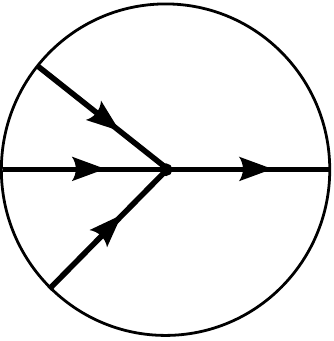}%
		\put(56, 48){$b_1 $}
		\put(116, 48){$b_1  +\delta$}
		\put(-28, 48){$b_1  -\delta$}
		\put(22, 91){$\scriptstyle\partial \mathcal S_1^+$}
		\put(23,17){$\scriptstyle\partial \mathcal S_1^-$}
		\put(78,17){\small $\partial B_\delta $}
	\end{overpic}
	\caption{Local analysis at $b_1$.}
	\label{fig:localanalysisB1}
\end{figure}
We look for $\bm P(\cdot;b_1)$   holomorphic in $ B_\delta \setminus (\R\cup  \partial \mathcal S_1^\pm )$, such that $ \bm P_+(z;b_1)=\bm P_-(z;b_1) \bm J_{S}(z)$, where, by \eqref{value_phi_1} and \eqref{defJS},
\begin{align*}
\bm J_{S}(x)& = \begin{pmatrix}
0			& 1 				& 0 \\
-1 		 	& 0				& 0 \\
0			& 0				& 1
\end{pmatrix},	 \quad x\in B_\delta \cap \Delta_1=(b_1-\delta, b_1),
\\
\bm J_{S}(x)& =
\begin{pmatrix}
1  &  e^{- N\Phi_1 (x)} & 0 \\
0& 1  &0 \\
0& 0& 1
\end{pmatrix}, \quad  x \in (b_1, b_1+\delta),
\\
\bm J_{S}(z)& = \begin{pmatrix}
1   & 0&0 \\
e^{N\Phi_1(z)} & 1 & 0 \\
0& 0& 1
\end{pmatrix} ,	 \quad z\in B_\delta\cap \partial \mathcal S_1^\pm,
\end{align*}
and $\bm P(\cdot;b_1)$ is bounded as $z\to b_1$, $z\in   \C \setminus (\R\cup  \partial \mathcal S_1^\pm  )$.

With these conditions, by \eqref{compatibility_for_lenses_delta1}, matrix
$$
 \widehat{\bm P}(z;b_1) \isdef \bm P(z;b_1) \diag\left(e^{-N\Phi_1(z)/2} ,  e^{N\Phi_1(z)/2} , 1\right)
$$
has constant jumps, $ \widehat{\bm P}_+(z;b_1)=\widehat{\bm P}_-(z;b_1) \bm J_{\widehat{P}}(z)$, with 
\begin{align*}
\bm J_{\widehat{P}}(x)& = \begin{pmatrix}
0			& 1 				& 0 \\
-1 		 	& 0				& 0 \\
0			& 0				& 1
\end{pmatrix},	 \quad x\in B_\delta \cap \Delta_1=(b_1-\delta, b_1),
\\
\bm J_{\widehat{P}}(x)& =
\begin{pmatrix}
1  &  1 & 0 \\
0& 1  & 0\\
0& 0& 1
\end{pmatrix}, \quad  x \in (b_1, b_1+\delta),
\\
\bm J_{\widehat{P}}(z)& = \begin{pmatrix}
1   & 0&0 \\
1 & 1 & 0 \\
0& 0& 1
\end{pmatrix} ,	 \quad z\in B_\delta\cap \partial \mathcal S_1^\pm,
\end{align*}
and again $\widehat{\bm P}$ is bounded as $z\to b_1$, $z\in   \C \setminus (\R\cup  \partial \mathcal S_1 ^\pm )$.

Since we additionally need that, as $n\to \infty$,  
\begin{equation} \label{matchingcondition}
\bm P(z;b_1)=\left( \bm I + \mathcal O(1/N) \right) \bm F(z)  \quad z\in \partial B_\delta \setminus (\R \cup \partial \mathcal S_1^\pm),
\end{equation}
we follow a well-known scheme, and build $\widehat{\bm P}(\cdot;b_1)$ in the form
\begin{equation} \label{parametrixP}	
\widehat{\bm P}(z;b_1) = \bm E(z)
\bm    \Psi\left(N^{2/3}\varphi (z)\right),
\end{equation}
where  
\begin{equation} \label{eq:defE}
\bm   E(z) =   \bm F(z)\,    \begin{pmatrix}
\sqrt{\pi} & -\sqrt{\pi} & 0 \\
-i \sqrt{\pi} & -i \sqrt{\pi} & 0 \\
0 & 0 & 1
\end{pmatrix} \,
\begin{pmatrix}
n^{1/6}\varphi^{1/4}(z) &   0 & 0 \\
0 &   n^{-1/6}\varphi ^{-1/4}(z) & 0 \\
0 &   0 & 1
\end{pmatrix},
\end{equation}
and
\begin{align} \label{fqdef}
\varphi(z) =
\left[\frac{3}{2}\Phi_1(z)\right]^{2/3}
\end{align}
is a biholomorphic (conformal) map of a neighborhood of $b_1$ onto a neighborhood of the origin such that $\varphi(z)$ is real and positive
for $z>b_1$, see \eqref{inequality_xi_2}--\eqref{inequality_xi_5} and \eqref{def_phi_1}.

We may deform the contours $\mathcal S_1^\pm$ near $b_1$ in such a way that $\varphi$ maps
$\mathcal S_1^\pm \cap B_\delta $ onto the rays with angles $\frac{2\pi}{3}$ and $-\frac{2\pi}{3}$, respectively.

Matrix $\bm \Psi$ is built using a Riemann-Hilbert problem for the Airy functions, as described for instance in \cite[page 253]{kuijlaars_martinez-finkelshtein_wielonsky_bessel_paths}. We recall the main steps for later reference.

First consider a $2\times 2$ matrix $\bm K$ solving the following RHP.
\begin{itemize}
\item $\bm K:\C\setminus \Gamma_K\to \C^{2\times 2}$ is analytic, where $\Gamma_K=\R \cup [0,e^{i2\pi i/3}\infty)\cup [0,e^{-2\pi i /3}\infty)$;
\item $\bm K$ has continuous boundary values $\bm \Psi_\pm$ on $\Gamma_\Psi$, and $\bm \Psi_+(z)=\bm \Psi_-(z)\bm J_\Psi(z)$, $z\in \Gamma_\Psi$, with
\begin{equation}\label{jumpsK}
\bm J_K(z)=
\begin{cases}
\begin{pmatrix}
1 & 0 \\
1 & 1 
\end{pmatrix}, & z\in [0,e^{2\pi i/3}\infty)\cup [0,e^{-2\pi i /3}\infty), \\
\begin{pmatrix}
0 & 1 \\
-1 & 0
\end{pmatrix}, & z\in (-\infty,0), \\
\begin{pmatrix}
1 & 1 \\
0 & 1
\end{pmatrix}, & z\in (0,+\infty).
\end{cases}
\end{equation}

\item As $z\to \infty$ with $-\pi < \arg z < \pi$,
$$
\bm K(z)=(I+\Boh(z^{-1}))\diag(z^{-1/4},z^{1/4})\frac{1}{\sqrt{2}}
\begin{pmatrix}
1 & i \\
i & 1
\end{pmatrix}
\diag(e^{-\frac{2}{3}z^{3/2}},e^{\frac{2}{3}z^{3/2}});
$$
\item $\bm K$ remains bounded as $z\to \infty$ along $\C\setminus \Gamma_\Psi$.
\end{itemize}

We put
\begin{equation*}
y_0(s)=\Ai(s), \quad  y_1(s) = \omega \Ai(\omega s),
\quad y_2(s) = \omega^{ 2 } \Ai(\omega^{ 2 } s), \quad \omega=e^{2\pi i/3}\,,
\end{equation*}
where $\Ai$ is the usual Airy function. Then $\bm   K$ is obtained explicitly as
\begin{equation}\label{formulaK}
\bm     K(s) =
\begin{cases}
\begin{pmatrix}
y_0(s) &  -y_2(s)   \\
y_0'(s) &  -y_2'(s)     \end{pmatrix}, 
& \arg s \in (0, 2 \pi/3),
\\
\begin{pmatrix}
-y_1(s) &  -y_2(s)   \\
-y_1'(s) &  -y_2'(s)     \end{pmatrix}, 
& \arg s \in (2\pi/3,\pi),
\\
\begin{pmatrix}
-y_2(s) &   y_1(s)   \\
-y_2'(s) &  y_1'(s)
\end{pmatrix},
& \arg s \in (-\pi,-2\pi/3),
\\
\begin{pmatrix}
y_0(s) & y_1(s)   \\
y_0'(s) & y_1'(s)
\end{pmatrix},
& \arg s \in (-2\pi/3,0).
\end{cases}
\end{equation}
Then we take the $3\times 3$ matrix $\bm   \Psi$ as
\begin{equation}\label{formulaPsiK}
\bm   \Psi(s) = \left(\begin{MAT}{c.c}
\bm K (s) & \bm 0 \\.
\bm  0 & 1 \\
\end{MAT} \right).
\end{equation}

As we said, the construction of the rest of the local parametrices $\bm P(\cdot;q)$, $q\in \{ a_1, a_2, b_2\}$, follows the same scheme, nowadays standard.

\subsection{Final transformation}\label{sec:finaltransformation}

We define the matrix valued function $\bm R$ as
\begin{equation} \label{Rdef}
\bm R(z) \isdef \begin{cases}
\bm S(z) \bm P^{-1}(z;q), & \text{in the  neighborhoods $B_{\delta}(q)$ around $q\in \{ a_1, a_2 , b_1, b_2\}$,   }   \\
\bm S(z) \bm F^{-1}(z), & \text{elsewhere.}
\end{cases}
\end{equation}
The jump matrices of $\bm S$ and $\bm F$ coincide on $\Delta_1 \cup \Delta_2\cup \Delta_3$ 
and the jump matrices of $\bm S$ and $\bm P$ coincide inside the four disks $B_\delta(q)$, $q\in \{ a_1, a_2, b_1, b_2\}$.
It follows that $\bm R$ has an analytic continuation to the complex plane
minus the contours $\Gamma_R$ shown in  Figure~\ref{figure_final_precritical}.  
The matching conditions \eqref{matchingcondition} shows that $\bm P \bm F^{-1} = \bm I + \mathcal  O(1/N)$ as $N\to \infty$,  uniformly on $\bigcup_q \partial B_\delta(q)$, Thus
$$
\bm R_+(z)=\bm R_-(z) \left( \bm I + \mathcal O(1/N)\right), \quad \text{as $N\to \infty$}, \quad \text{uniformly on } \bigcup_{q \in \{ a_1, a_2 , b_1, b_2\}} \partial B_\delta(q),
$$
 while on the rest of the contours of $\Gamma_R$, for some $c>0$, 
 $$
 \bm R_+(z)=\bm R_-(z) \left( \bm I + \mathcal O(e^{-cN})\right), \quad \text{as $N\to \infty$}.
 $$
 Since $\bm R(z) \to \bm I$ as $z\to \infty$, standard arguments yield that 
 $$
 \bm R(z) = \bm I + \mathcal O(1/N)  \quad \text{as $N\to \infty$}, 
 $$
 uniformly in $z\in \mathbb C\setminus \Gamma_R$. 
 
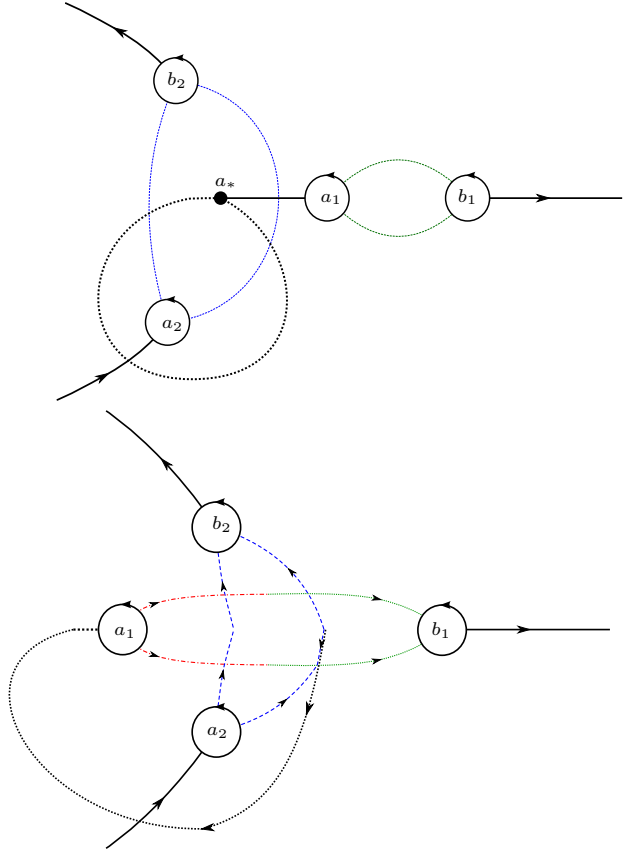
\begin{figure}[t]
	\begin{center}
		\begin{minipage}[c]{0.6\textwidth}
			\begin{overpic}[scale=0.5,grid=false]{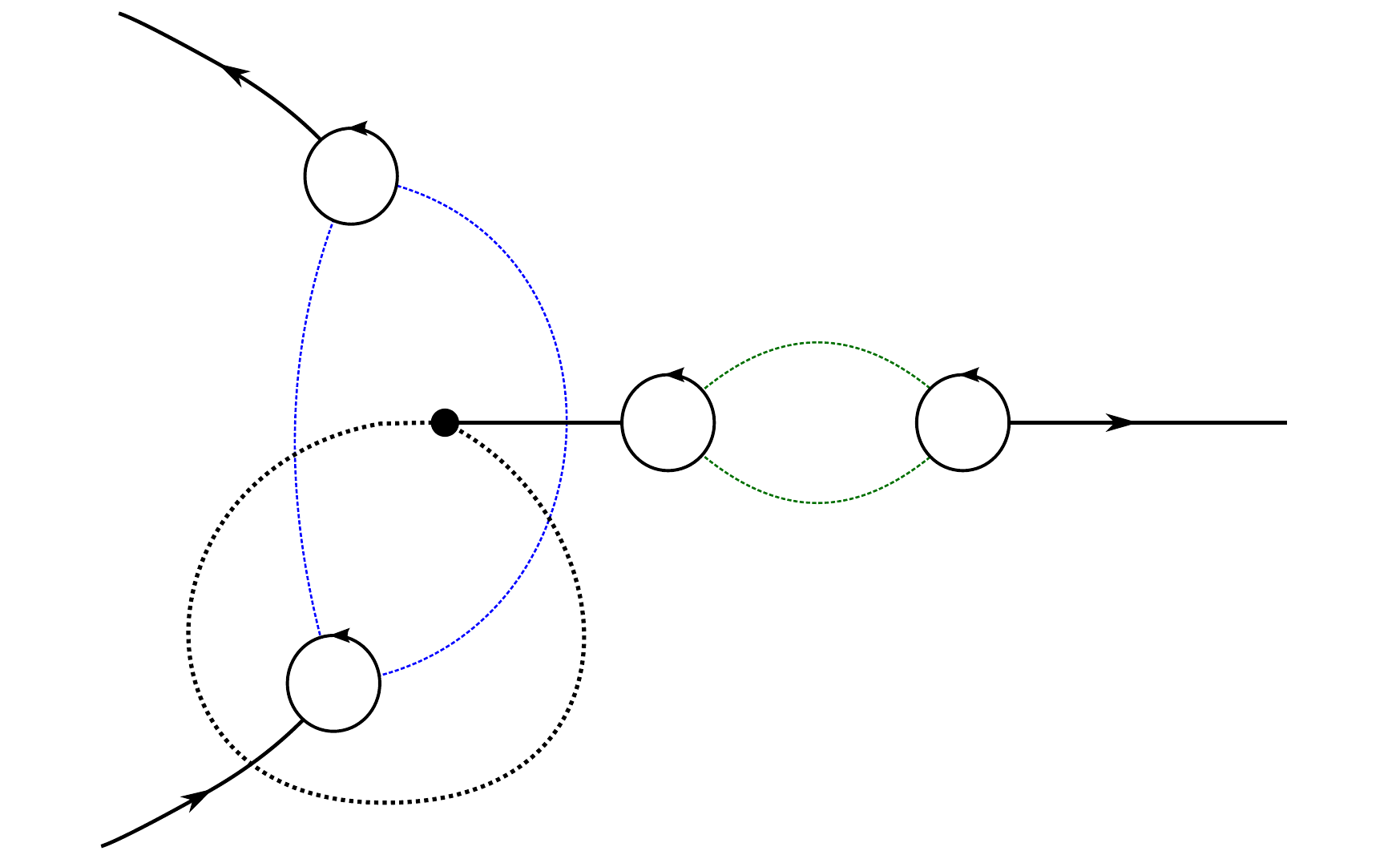}
				\put(118,78){\scriptsize $ a_1$}
				\put(170,78){\scriptsize $ b_1$}
				\put(60,122){\scriptsize $ b_2$}
				\put(58,30){\scriptsize $ a_2$}
				\put(78,84){\scriptsize $ a_*$}
			\end{overpic}
		\end{minipage}%
	\\
		\begin{minipage}[c]{0.6\textwidth}
			\begin{overpic}[scale=0.55,grid=false]{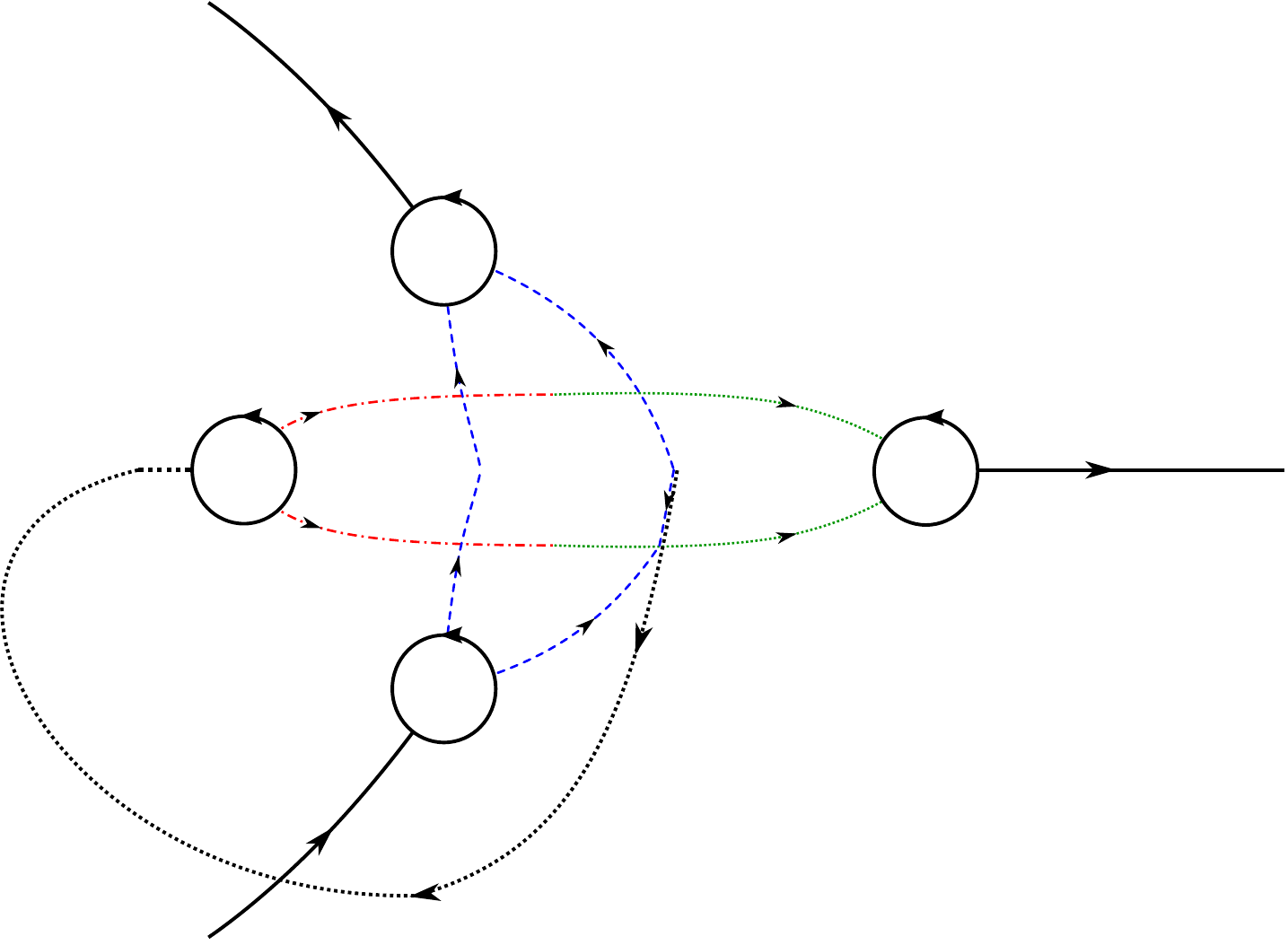}
				\put(40,82){\scriptsize $ a_1$}
				\put(160,82){\scriptsize $ b_1$}
				\put(77,122){\scriptsize $ b_2$}
				\put(75,44){\scriptsize $ a_2$}
			\end{overpic}
		\end{minipage}%
		\caption{Pictorial representation of the contours $\Gamma_R$ in the sub- (top) and super-critical (bottom) regimes.}\label{figure_final_precritical}
	\end{center}
\end{figure}

\section{Steepest descent analysis for general $\alpha$}\label{sec:generalparameter}

All transformation described in Section~\ref{sec:steepestdescent} were carried out under the assumption \eqref{assumption_rational}. Now we extend our conclusions to arbitrary sequences of $\alpha_N$ satisfying
$$
\alpha_N\to \alpha_0 \in (0,1/2)\setminus \{\alpha_c\}.
$$

We also assume that $N$ is chosen large enough so that $\alpha_N$ and $\alpha_0$ belong to the same connected component of $(0,1/2)\setminus \{\alpha_c\}$.

Starting with the first Riemann-Hilbert problem $\bm Y=\bm Y_{\alpha_N}$, we make the first transformation $\bm Y_{\alpha_N}\mapsto \bm X_{\alpha_N}$ as in \eqref{def:transformationX}. 

The second transformation $\bm X_{\alpha_N}\mapsto \bm T_{\alpha_N}$ is also as before, that is, as in \eqref{def:transformationT}, taking into account that all required quantities will depend on $\alpha_N$. That is, for the given $\alpha_N$, we define $\xi_k=\xi_{k,\alpha_N}$ as the solutions of the spectral curve \eqref{spectral_curve} and asymptotics \eqref{asymptotics_xi_functions} with $\alpha=\alpha_N$, then we define the corresponding $g$-functions $g_k=g_{k,\alpha_N}$ and $r$-constants $r_k=r_{k,\alpha_N}$ as in \eqref{defGfunctions} and \eqref{defRforGfunctions}, respectively. 

With such choices, all the identities used for the $g$-functions remain valid, taking into account that $\alpha=\alpha_N$ and also that the $\Phi$-functions in \eqref{def_phi_1}--\eqref{def_phi_3b} will be $\alpha_N$-dependent. 

At this stage, the jump matrix $\bm J_T$ remains as in \eqref{JTnew}, stressing once again that the $g$-functions now depend on $\alpha_N$. Thus, we can open lenses as in \eqref{defS} (in the subcritical case) or as in \eqref{defSsuper} (in the supercritical case).  At this stage, we also fix sufficiently small disks $B_\delta(a_k),B_\delta(b_k)$ for $a_k=a_k(\alpha_0)$ and $b_k=b_k(\alpha_0)$ and denote their union by $B$. The precise choice of $\delta$ is  determined later in the construction of the local parametrix, but we anticipate that it will be possible to choose $\delta$ to be depending on $\alpha_0$ but not on $\alpha_N$.

The shape of the lenses $\mathcal S_j^{\pm}$ does also depend on $\alpha_N$. However, continuity and convergence of all the $\alpha_N$-dependent quantities guarantee that we can choose the lenses so that 
\begin{itemize}
\item the endpoints $a_1=a_1(\alpha_N)$ (in the subcritical case),$a_2=a_2(\alpha_N), b_1=b_1(\alpha_N)$ and $b_2=b_2(\alpha_N)$ all belong to $B$,
\item the sets $\mathcal S_j^\pm\setminus B$ are independent of $\alpha_N$, 
\end{itemize}
and in such a way that the uniform version of the inequalities \eqref{correct_decay_lenses_2} (in the subcritical case) or \eqref{correct_decay_lenses_supercritical} (in the supercritical case) hold true on the lenses away from the endpoints, namely
\begin{equation}\label{expdecay_varying}
\re\Phi_{j,\alpha_0}(z)\leq -\rho<0 ,\quad \re \Phi_{j,\alpha_N}(z)\leq -\rho <0,  \quad z\in \partial\mathcal S_j^\pm\setminus B,
\end{equation}
for every $n,N$ sufficiently large and some fixed $\rho>0$ that does not depend on the parameters $\alpha_N$ but only on $\alpha_0$. We stress that the sets $\mathcal S_j^\pm\setminus B$ appearing above were chosen to be $\alpha_N$-independent.

After the opening of lenses, the obtained matrix $\bf S$ has jumps as in \eqref{jumpSpre} (in the subcritical case) or \eqref{defJS} (in the supercritical case). In virtue of \eqref{expdecay_varying}, at this stage we can be sure that ${\bm J}_S$ decays exponentially to the identity uniformly away from $B$ and from the sets $\Delta_k=\Delta_k(\alpha_N)$.

For the construction of the parametrix, we observe that ${\bm J}_S$ has constant jumps on $\Delta_k(\alpha_N)$. Thus, for any $\alpha_N$ the global matrix $\bm F$ we consider is the one constructed in Section \ref{sec:globalparam}, but for the choice $\alpha=\alpha_0$. Although $\bm J_F$ does not coincide with $\bm J_S$ close to the varying endpoints $a_k(\alpha_N)$ and $b_k(\alpha_N)$, these points are assumed to be on $B$ (by the construction of $B$), so this local inconsistency of $\bm J_F$ with $\bm J_S$ will anyway be solved by the local parametrix.

In each of the connected components of $B$, we need to construct the local parametrix $\bm P$ that will depend on $\alpha_N$. We focus on the construction connected component $B(b_1)$, recalling that $b_1=b_1(\alpha_0)$. The overall method is the same as the one explained in Section \ref{sec:localparam}, but carried out for $\alpha=\alpha_N$. The main differences are explained next.

The neighborhood $B(b_1)$ can be chosen to be a disc centered at $b_1=b_1(\alpha_0)$ but not at $b_1(\alpha_N)$. This imposes no extra difficulty.

The conformal map $\varphi=\varphi_{\alpha_N}$ in \eqref{fqdef} is now varying with $\alpha_N$. At any rate, it is easy to see that
$$
\varphi_{\alpha_N}\to \varphi_{\alpha_0}
$$
locally uniformly. The function $\Phi_1=\Phi_{1,\alpha_N}$ satisfies
$$
\Phi_{1,\alpha_N}(z)=z^3+\Boh(1),\quad z\to\infty
$$
as can be seen from \eqref{defRforGfunctions} and \eqref{value_phi_1}. 

From the expansion above, the following claim follows. For each $\alpha_N$, there exists an arc $\Gamma(\alpha_N)$ starting at the origin, extending to $\infty$ with angle $2\pi /3$, and such that $\varphi_{\alpha_N}$ maps $\partial \mathcal S^+\cap B(b_1)$ onto $\Gamma_N$, and $\Gamma(\alpha_0)$ reduces to a half line ray. 

Having the construction above in mind, we consider the RHP for $\bm K$, replacing in \eqref{jumpsK} the arcs $[0,e^{2\pi i/3}\infty)$ and $[0,e^{-2\pi i/3}\infty)$ by $\Gamma(\alpha_N)$ and its symmetric $\Gamma(\alpha_N)^*$ under complex conjugation. The solution $\bm K=\bm K_{\alpha_N}$ to this RHP can be constructed in a formula very much like \eqref{formulaK}, replacing the angular sectors by the sectors determined by $\Gamma(\alpha_N)$, its complex conjugate and the real axis. In particular, as it follows from their explicit formulas in terms of Airy functions, the convergence
$$
\bm K_{\alpha_N}\to \bm K_{\alpha_0}
$$
holds true uniformly in compacts of $\C\setminus (\R\cup [0,e^{2\pi i/3}\infty)\cup [0,e^{-2\pi i/3}\infty))$.

Once the matrix $K_{\alpha_N}$ is constructed, we then obtain $\bm \Phi_{\alpha_N}$ as in \eqref{formulaPsiK}, consider the parametrix $\widehat{\bm P}_{\alpha_N}$ as in \eqref{parametrixP}, and then obtain $\bm P_{\alpha_N}$ through \eqref{matchingcondition}, where we stress that $\bm F=\bm F_{\alpha_0}$ does not depend on $\alpha_N$.

Once the parametrices $\bm P_{\alpha_N}$ and $\bm F_{\alpha_0}$ are constructed, we define $R=R_{\alpha_N}$ in a similar way as in \eqref{Rdef}, taking now the form
\begin{equation} \label{Rdefvarying}
\bm R_{\alpha_N}(z) \isdef 
\begin{cases}
\bm S_{\alpha_N}(z) \bm P_{\alpha_N}^{-1}(z), & z\in B,   \\
\bm S_{\alpha_N}(z) \bm F_{\alpha_0}^{-1}(z), & \mbox{elsewhere.}
\end{cases}
\end{equation}
and the analysis follows exactly as in Section \ref{sec:finaltransformation}, in particular leading to 
$$
\bm R_{\alpha_N}(z)=I+\Boh(1/N) \quad \mbox{ as } N\to \infty,
$$
uniformly for $z\in \C\setminus \Gamma_R$, where we also call the reader's attention to the fact that $\Gamma_R$ does not depend on $\alpha_N$.

\section{Asymptotics} \label{sec:asymptotics}

Now we can unravel all transformations and find the desired asymptotic expressions.

\begin{proof}[Proof of Theorem \ref{thm:main}]
For $z \in \C\setminus \left( \mathcal S_1^\pm \cup \mathcal S_2^\pm \cup \mathcal U\right)$, away from the contours $\Gamma_R$ depicted on Figure~\ref{figure_final_precritical}, we have 
\begin{multline*}
\bm Y(z)=  \diag \left(e^{Nr_1}, e^{Nr_2}, e^{Nr_3}  \right)
 \left( \bm I + \mathcal O(1/N)  \right)
\bm F(z) 
\\
\times \diag \left(e^{- N(g_1(z)-\frac{2}{3}z^3)}, e^{- N(g_2(z)+\frac{1}{3}z^3)}, e^{- N(g_3(z)+\frac{1}{3}z^3)} \right).
\end{multline*}
Since the  polynomial $P_{n,m}$ coincides with the $(1,1)$ entry of the matrix $\bm Y$, we have
\begin{align}
P_{n,m}(z) &= \begin{pmatrix}
1 & 0 & 0
\end{pmatrix}\bm Y(z) \begin{pmatrix}
1 \\ 0 \\ 0
\end{pmatrix}  \nonumber \\
& =   e^{N(-g_1(z)+\frac{2}{3}z^3+r_1)}
\begin{pmatrix}
1 & 0 & 0
\end{pmatrix} \left( \bm I + \mathcal O(1/N)  \right)
\bm F(z)    \begin{pmatrix}
1 \\ 0 \\ 0
\end{pmatrix}
\nonumber \\
& =   f_{11}(z)e^{N(-g_1(z)+\frac{2}{3}z^3+r_1)}
 \left( 1 + \mathcal O(1/N)  \right), \label{asympPN1}
\end{align}
which gives us the expression of Theorem~\ref{thm:main} for $z \in \C\setminus \left( \mathcal S_1^\pm \cup \mathcal S_2^\pm \cup \mathcal U\right)$. Note that in the varying case $\alpha_N=n/N$, the formula above (and all the calculations below) are still the same with $g_{1}=g_{\alpha_N,1}$, and the global parametrix is constructed with $f_{1j}$'s for the limiting value $\alpha=\lim \alpha_N$.

Furthermore, for $z \in \mathcal U \setminus \left( \mathcal S_1^\pm \cup \mathcal S_2^\pm   \right)$,  
\begin{multline*}
\bm Y(z)=  \diag \left(e^{Nr_1}, e^{Nr_2}, e^{Nr_3}  \right)
\left( \bm I + \mathcal O(1/N)  \right)
\bm F(z) 
\\
\times \diag \left(e^{- N(g_1(z)-\frac{2}{3}z^3)}, e^{- N(g_2(z)+\frac{1}{3}z^3)}, e^{- N(g_3(z)+\frac{1}{3}z^3)} \right)
 \left( \bm I + \bm M_{32}\right),
\end{multline*}
and we clearly obtain the same asymptotic expression for $P_{n,m}$ as before. Because the lenses $\mathcal S_1^\pm$ and $\mathcal S_2^\pm$ can be taken arbitrarily close to $\Delta_1$ and $\Delta_2$, this is enough to conclude Theorem \ref{thm:main}.
\end{proof}

\begin{proof}[Proof of Theorem \ref{theorem_zero_counting_measure} for $P_{n,m}$]
Since $f_{11}$ does not vanish in $\C\setminus \left( \Delta_1\cup \Delta_2\right)$, we conclude from Theorem~\ref{thm:main} that the zeros of $P_{n,m}$ asymptotically belong to $\Delta_1\cup \Delta_2$. Moreover, since domains $\mathcal S^\pm_1$ and $\mathcal S^\pm_2$ can be taken arbitrarily small, we conclude from \eqref{asympPN1} that 
\begin{equation}\label{asymptLogPot}
\lim_{N\to \infty} \left|P_{n,m}(z)\right|^{1/N}=e^{ -\Re \left( g_{1}(z)-\frac{2}{3}z^3-r_1 \right)  }
\end{equation}
locally uniformly on any compact subset of $\C\setminus \left( \Delta_1\cup \Delta_2\right)$, where on the right-hand side above we use $g_{1}=g_{\alpha,1}= \lim_N g_{\alpha_N,1}$, the latter limit taken pointwise. Let us show that
\begin{equation}
\label{identityPotentials}
\Re \left( g_1(z)-\frac{2}{3}z^3-r_1 \right) = U^{\mu_1+\mu_2}(z), \quad  \C\setminus \left( \Delta_1\cup \Delta_2\right).
\end{equation}
First of all, it is easy to observe that by definition \eqref{defGfunctions} and asymptotics \eqref{defRforGfunctions},
$$
g'_1(z)-2z^2=  \xi_1(z)-2z^2= -\frac{1}{z}+ \mathcal O\left( 1/z^2\right), \quad z\to \infty. 
$$
On the other hand, by \eqref{equality_xi_2}--\eqref{equality_xi_3} and the definition of $\mu_j$'s in \eqref{definition_measures_mu}, 
$$
g'_{1+}(z)-g'_{1-}(z)=\xi_{1+}(z)-\xi_{1-}(z)=\begin{cases}
2\pi i \mu_1'(z), & \text{if } z\in \Delta_1,  \\ 2\pi i  \mu_2'(z), & \text{if } z\in \Delta_2,
\end{cases}
$$
from which we conclude that
$$
g'_1(z)-2z^2=   C^{\mu_1+\mu_2}(z), \quad  \C\setminus \left( \Delta_1\cup \Delta_2\right).
$$
Finally, \eqref{identityPotentials} follows from the asymptotic formula \eqref{defRforGfunctions}.

In particular, \eqref{asymptLogPot} can be written as
\begin{equation*} 
\lim_{N\to \infty} \left|P_{n,m}(z)\right|^{1/N}=e^{ -U^{\mu_1+\mu_2}(z) },  \quad  \C\setminus \left( \Delta_1\cup \Delta_2\right).
\end{equation*}
Since $\Delta_1 \cup \Delta_2$ have empty interior and connected complement, this proves (see e.g.~\cite[Theorem 4.1]{saff_totik_book}) the statement of Theorem~\ref{theorem_zero_counting_measure}  for the weak limit of $\nu(P_{n,m})$.

\end{proof}

Next we turn to the asymptotic analysis of the type I multiple orthogonal polynomials $A_{n,m}$ and $B_{n,m}$, satisfying the Definition~\ref{defTypeI}. Recall that 
\begin{equation}\label{ABfromY}
\begin{split}
A_{n,m}(z) &= \frac{1}{2\pi i}\begin{pmatrix}
0 & 1 & 0
\end{pmatrix}\bm{Y}^{-1}(z) \begin{pmatrix}
1 \\ 0 \\ 0
\end{pmatrix}, \\
B_{n,m}(z) &= \frac{1}{2\pi i}\begin{pmatrix}
0 & 0 & 1
\end{pmatrix}\bm{Y}^{-1}(z) \begin{pmatrix}
1 \\ 0 \\ 0
\end{pmatrix},
\end{split}
\end{equation}
check out the comments after \eqref{RHPY1}.

\begin{proof}[Proof of Theorem \ref{thm:main2}]
For $z \in \C\setminus \left( \mathcal S_1^\pm \cup \mathcal S_2^\pm \cup \mathcal S_2^\pm\cup  \mathcal U\right)$, away from the contours $\Gamma_R$ depicted on Figure~\ref{figure_final_precritical}, we have 
\begin{multline} 
\bm{Y}^{-1}(z)= \diag \left(e^{N(g_1(z)-\frac{2}{3}z^3)}, e^{ N(g_2(z)+\frac{1}{3}z^3)}, e^{ N(g_3(z)+\frac{1}{3}z^3)} \right)  \bm{F}^{T}(z) 
\\
\times \left( \bm I + \mathcal O(1/N)  \right) \diag \left(e^{-Nr_1}, e^{-Nr_2}, e^{-Nr_3}  \right), \label{asymptotics_inverse_Y}
\end{multline}
where we have used \eqref{invequaltransp}, so that
\begin{align}
2\pi i A_{n,m}(z) &= e^{ N(g_2(z)+\frac{1}{3}z^3-r_1)}
\begin{pmatrix}
0 & 1 & 0
\end{pmatrix} \bm{F}^{T}(z)  \left( \bm I + \mathcal O(1/N)  \right)
   \begin{pmatrix}
1 \\ 0 \\ 0
\end{pmatrix}
\nonumber
\\
& =    f_{12}(z)   e^{ N(g_2(z)+\frac{1}{3}z^3-r_1)}
\left( 1 + \mathcal O(1/N)  \right).
\label{asymptoticsAout}
\end{align}

On the other hand, for $z \in \mathcal U \setminus \left( \mathcal S_1^\pm \cup \mathcal S_2^\pm  \cup \mathcal S_3^\pm  \right)$, 
\begin{multline}\label{asymptotics_inverse_Y_2}
\bm{Y}^{-1}(z)= \left( \bm I - \bm M_{32}\right) \diag \left(e^{N(g_1(z)-\frac{2}{3}z^3)}, e^{ N(g_2(z)+\frac{1}{3}z^3)}, e^{ N(g_3(z)+\frac{1}{3}z^3)} \right)  \bm{F}^{T}(z) 
\\
\times \left( \bm I + \mathcal O(1/N)  \right) \diag \left(e^{-Nr_1}, e^{-Nr_2}, e^{-Nr_3}  \right),
\end{multline}
from where
\begin{align}\label{asymptoticsAin}
2\pi i A_{n,m}(z) &= e^{ N(g_2(z)+\frac{1}{3}z^3-r_1)}
\begin{pmatrix}
0 & 1 & 0
\end{pmatrix} \bm{F}^{T}(z)  \left( \bm I + \mathcal O(1/N)  \right)
\begin{pmatrix}
1 \\ 0 \\ 0
\end{pmatrix} \nonumber
\\
& =   f_{12}(z)   e^{ N(g_2(z)+\frac{1}{3}z^3-r_1)}
\left( 1 + \mathcal O(1/N)  \right),
\end{align}
which coincides with the formula in \eqref{asymptoticsAout}. Again, as $\mathcal S_1^\pm,\mathcal S_2^\pm,\mathcal S_3^\pm$ can be made arbitrarily close to $\Delta_1$, $\Delta_2$ and $\Delta_3$, we conclude Theorem \ref{thm:main2} for $z\in \C\setminus (\Delta_1\cup\Delta_2\cup\Delta_3)$.

The functions
$$
\frac{A_{n,m}(z) }{f_{11}(z) }e^{ -N(g_2(z)+\frac{1}{3}z^3-r_1)}
$$
are holomorphic on compact subsets of $\C\setminus (\Delta_1\cup \Delta_3)$, and by \eqref{asymptoticsAout}, \eqref{asymptoticsAin} and Montel's Theorem, they form a normal family. This proves that 
$$
\lim_{N\to \infty} \frac{A_{n,m}(z) }{f_{11}(z) }e^{ -N(g_2(z)+\frac{1}{3}z^3-r_1)}=1
$$
also on $\Delta_2\setminus \R$, concluding the proof of Theorem \ref{thm:main2}.

\end{proof}

\begin{proof}[Proof of Theorem \ref{theorem_zero_counting_measure} for $A_{n,m}$]
By construction $ f_{12}$ does not vanish in $\C\setminus \left( \Delta_1\cup \Delta_3\right)$, so by Theorem \ref{thm:main2} we conclude that all zeros of $A_{n,m}$ asymptotically belong to $\Delta_1\cup \Delta_3$, and  
\begin{equation*} 
\lim_{N\to \infty} \left|A_{n,m}(z)\right|^{1/N}=e^{ -\Re \left( -g_2(z)-\frac{1}{3}z^3+r_1 \right)  }
\end{equation*}
locally uniformly on any compact subset of $\C\setminus \left( \Delta_1\cup \Delta_3\right)$. 
But by definition \eqref{defGfunctions} and asymptotics \eqref{defRforGfunctions},
$$
-g'_2(z)-z^2=  -\xi_2(z)-z^2= -\frac{\alpha}{z}+ \mathcal O\left( 1/z^2\right), \quad z\to \infty. 
$$
On the other hand, by \eqref{equality_xi_2}--\eqref{equality_xi_3} and the definition of $\mu_j$'s in \eqref{definition_measures_mu}, 
$$
-g'_{2+}(z)+g'_{2-}(z)=\xi_{2-}(z)-\xi_{2+}(z)=\begin{cases}
2\pi i\,  \mu_1'(z), & \text{if } z\in \Delta_1,  \\ 2\pi i \, \mu_3'(z), & \text{if } z\in \Delta_3 \text{ and } \tau >\tau_c,
\end{cases}
$$
from which  
$$
-g'_2(z)-z^2=   C^{\mu_1+\mu_3}(z), \quad  \C\setminus \left( \Delta_1\cup \Delta_3\right).
$$
Thus, by the asymptotic formula \eqref{defRforGfunctions},
\begin{equation*}
\Re \left( -g_2(z)-\frac{1}{3}z^3+r_1 \right) = r_1-r_2+U^{\mu_1+\mu_3}(z), \quad  \C\setminus \left( \Delta_1\cup \Delta_3\right),
\end{equation*}
and
\begin{equation*}
\lim_{N\to \infty} \log\left|A_{n,m}(z) \right|^{1/N}=r_2-r_1 -U^{\mu_1+\mu_3}(z),  \quad  \C\setminus \left( \Delta_1\cup \Delta_3\right).
\end{equation*}
Since $\Delta_1 \cup \Delta_3$ has empty interior and connected complement, this proves the statement of Theorem~\ref{theorem_zero_counting_measure}  for the weak limit of $\nu(A_{n,m})$.

\end{proof}

We now turn our attention to the asymptotic results for $B_{n,m}$, which require more care. 

\begin{proof}[Proof of Theorem \ref{thm:main3}]

For $z \in \C\setminus \left( \mathcal S_1^\pm \cup \mathcal S_2^\pm \cup \mathcal S_3^\pm\cup  \mathcal U\right)$, away from the contours $\Gamma_R$ depicted on Figure~\ref{figure_final_precritical}, we use  \eqref{ABfromY} and \eqref{asymptotics_inverse_Y} to arrive at
\begin{align}
2\pi i B_{n,m}(z) &= e^{ N(g_3(z)+\frac{1}{3}z^3-r_1)}
\begin{pmatrix}
0 & 0 & 1
\end{pmatrix} \bm{F}^{T}(z)  \left( \bm I + \mathcal O(1/N)  \right)
\begin{pmatrix}
1 \\ 0 \\ 0
\end{pmatrix}
\nonumber
\\
& =   f_{13}(z)   e^{ N(g_3(z)+\frac{1}{3}z^3-r_1)}
\left( 1 + \mathcal O(1/N)  \right). \label{asymptoticsBout}
\end{align}

Recall that according to the construction at the beginning of Section \ref{prelim}, the contour $\gamma=\partial \mathcal U$ can be chosen arbitrarily within $\Omega_+$, so by \eqref{boundary_omegaplus} we can take it as close to $E_\alpha\cap \H_-$ as we want. Also, the lenses $\mathcal S_j^\pm$ can be chosen arbitrarily close to $\Delta_j$. Thus, the asymptotic formula \eqref{asymptoticsBout} is actually valid uniformly on compacts of $\C\setminus (\Delta_1\cup\Delta_2\cup \Delta_3\cup \Omega_-)$.

In a similar way, for $z \in \mathcal U \setminus \left( \mathcal S_1^\pm \cup \mathcal S_2^\pm  \cup \mathcal S_3^\pm  \right)$ we use \eqref{asymptotics_inverse_Y_2} to get
\begin{align*}
2\pi i B_{n,m}(z) &=  
\begin{pmatrix}
0 & -1 & 1
\end{pmatrix}  
\diag \left(e^{N(g_1(z)-\frac{2}{3}z^3)}, e^{ N(g_2(z)+\frac{1}{3}z^3)}, e^{ N(g_3(z)+\frac{1}{3}z^3)} \right) \\ & 
\times \bm{F}^{-1}(z) 
 \left( \bm I + \mathcal O(1/N)  \right) \diag \left(e^{-Nr_1}, e^{-Nr_2}, e^{-Nr_3}  \right)
\begin{pmatrix}
1 \\ 0 \\ 0
\end{pmatrix}
\\ & = \begin{pmatrix}
0 & -e^{ N(g_2(z)+\frac{1}{3}z^3-r_1)} &  e^{ N(g_3(z)+\frac{1}{3}z^3-r_1)} 
\end{pmatrix}   \bm{F}^{T}(z)    \begin{pmatrix}
1+ \mathcal O(1/N) \\  \mathcal O(1/N) \\  \mathcal O(1/N)
\end{pmatrix},
\end{align*}
and thus,
\begin{equation}\label{asymptoticsBinU}
\begin{split}
2\pi i B_{n,m}(z) = &     f_{13}(z)   e^{ N(g_3(z)+\frac{1}{3}z^3-r_1)}  
\left( 1 + \mathcal O(1/N)  \right)\\
& - f_{12}(z)   e^{ N(g_2(z)+\frac{1}{3}z^3-r_1)}  
\left( 1 + \mathcal O(1/N)  \right),
\end{split}
\end{equation}
uniformly on compacts of $\mathcal U\setminus \Delta_2$

Obviously, which one  of these two terms is dominant depends on whether we are in the domain $\Omega_\alpha$, see Definition~ \ref{defE}, or not. Indeed, we can summarize \eqref{asymptoticsBout} and \eqref{asymptoticsBinU} by
\begin{equation}\label{asymptoticsBsummary}
2\pi i B_{n,m}(z) =  \begin{cases}
 f_{13}(z)   e^{ N(g_3(z)+\frac{1}{3}z^3-r_1)}
\left( 1 + \mathcal O(1/N)  \right) , & z \in   \C\setminus \left( \mathcal S_j^\pm \cup  \overline{\Omega_\alpha}  \right) ,  \\
- f_{12}(z)   e^{ N(g_2(z)+\frac{1}{3}z^3-r_1)}  
\left( 1 + \mathcal O(1/N)  \right), & z \in \Omega_\alpha \setminus \Delta_2=\Omega_-.
\end{cases}
\end{equation}
Due to analiticity of $g_3$ and $f_{13}$ across $\Delta_1$ (see \eqref{defGfunctions} and \eqref{conditions_m_2}), and of $g_2$ and $f_{12}$ across $\Delta_2$ (see \eqref{equality_xi_3} and \eqref{conditions_m_2}), we can extend the first of these formulas to compacts of $\C\setminus (\Omega_\alpha\cup \Delta_2)$, and the second of these formulas to the whole domain $\Omega_\alpha$, which proves the assertion of Theorem~\ref{thm:main3}.

\end{proof}

\begin{proof}[Proof of Theorem~\ref{theorem_zero_counting_measureB}]
We turn now to Theorem~\ref{theorem_zero_counting_measureB}.  

With the definition of $H$ in Proposition~\ref{PropdefmuB}, by  \eqref{asymptoticsBsummary} and Theorem \ref{thm:main3} we conclude that
\begin{align*} 
\lim_{N\to \infty}   \log  \left|B_{n,m}(z)\right|^{1/N}  & = H(z) \nonumber \\ 
&  = -U^{\mu_B}(z) +\Re \left( r_3- r_1 \right), \quad z\in \C\setminus E_\alpha,
\end{align*}
and this is enough to conclude Theorem \ref{theorem_zero_counting_measureB}.
\end{proof}

\section{Numerical experiments and empirical observations} \label{sec:numerics}

Figures \ref{figure_zeros_N=30}--\ref{figure_zeros_traj_AB_2} are results of numerical experiments that we explain and further explore next.

The supports of the measures $\mu_1$, $\mu_2$, $\mu_3$ and $\mu_B$ coincide with critical trajectories of the canonical quadratic differential $\varpi$ in \eqref{defQqd}. The numerical evaluation of these trajectories is explored and explained in our previous work \cite[Appendix~C]{martinez_silva_critical_measures}.

To compute the zeros of the polynomials $A_{n,m},B_{n,m}$ and $P_{n,m}$ we proceed as follows. Given the moments \eqref{explicitcoeff}, we create the matrix of mixed moments
$$
H_{n,m}=
\begin{pmatrix}
\frak f_1^{(0)} & \cdots & \frak f_1^{(n-1)} & \frak f_2^{(0)} & \cdots & \frak f_2^{(m-1)} \\ 
\frak f_1^{(1)} & \cdots & \frak f_1^{(n)} & \frak f_2^{(1)} & \cdots & \frak f_2^{(m)} \\
& \vdots & & & \vdots &   \\
\frak f_1^{(m+n-1)} & \cdots & \frak f_1^{(m+2n-2)} & \frak f_2^{(m+n-1)} & \cdots & \frak f_2^{(2m+n-2)} 
\end{pmatrix}
$$

We put the coefficients of $A_{n,m}$ and $B_{n,m}$ into column vectors $\vec A_{n,m}$ and $\vec B_{n,m}$ in increasing order. Then $\vec A_{n,m}$ and $\vec B_{n,m}$ solve the linear system \cite[Section~23.1]{ismail_book}
\begin{equation}\label{linear_system_coefficients}
H_{n,m} 
\begin{pmatrix}
\vec A_{n,m} \\ \vec B_{n,m}
\end{pmatrix}
=
\begin{pmatrix}
0 \\ \vdots \\ 0 \\ 1
\end{pmatrix}
\end{equation}

Similarly, the coefficients of $P_{n,m}$ also solve a linear system, but with coefficients matrix given by the transpose of $H_{n,m}$.

Because the moments \eqref{explicitcoeff} are given explicitly, the linear system above can be solved numerically with high precision using any standard numerical solver. For our figures, we used the linear solver provided by Mathematica, working with 400 digits precision. 

In Figures~\ref{figure_zeros_1-5}--\ref{figure_zeros_3-7} we display the zeros of $P_{n,m}$, $A_{n,m}$ and $B_{n,m}$ for the values $\alpha=1/5<\alpha_c$, $\alpha =3/10\in (\alpha_c,\alpha_2)$ and $\alpha =3/8>\alpha_2$, and various choices of $n$ and $m$. 

The same approach carries over to a more general situation: given an integer $K>2$, consider the type II monic multiple orthogonal polynomial $Q_{n,m}$ of degree $n+m$ defined through the orthogonality conditions
\begin{align*}
& \int_{\gamma_{\ell_1} \cup \gamma^-_{\kappa_1}} Q_{n,m}(z) z^{j} e^{-z^K}dz=0,\quad j=0,\hdots, n-1,\\
& \int_{\gamma_{\ell_2} \cup \gamma^-_{\kappa_2}} Q_{n,m}(z) z^{j} e^{-z^K}dz=0,\quad j=0,\hdots, m-1,
\end{align*}
where $\gamma_\kappa$ is the oriented ray from $0$ to $\infty e^{2\pi i \kappa/K}$, $\kappa = 0,\hdots, K-1$, and $\gamma_\kappa^-$ is the result of reversing the orientation of $\gamma_\kappa$. 

Similarly, we define the type I MOPs $C_{n,m}$ and $D_{n,m}$ as polynomials of degrees at most $n-1$ and $m-1$, respectively, that satisfy
\begin{align*}
& \int_{\gamma_{\ell_1} \cup \gamma^-_{\kappa_1}} C_{n,m}(z) z^{j} e^{-z^K}dz+ \int_{\gamma_{\ell_2} \cup \gamma^-_{\kappa_2}} D_{n,m}(z) z^{j} e^{-z^K}dz=0,\quad j\leq n+m-2, \\
& \int_{\gamma_{\ell_1} \cup \gamma^-_{\kappa_1}} C_{n,m}(z) z^{j} e^{-z^K}dz+ \int_{\gamma_{\ell_2} \cup \gamma^-_{\kappa_2}} D_{n,m}(z) z^{j} e^{-z^K}dz=1,\quad j= n+m-1. \\
\end{align*}

In the formulas above, we assume $\ell_j\neq \kappa_j$, $j=1,2$, and $\{\ell_1,\kappa_1\}\neq \{\ell_2,\kappa_2\}$ to ensure non-triviality. From
$$
\int_0^\infty x^j e^{-x^K}dx=\frac{1}{K}\, \Gamma\left( \frac{j+1}{K} \right),
$$
where $\Gamma(\cdot)$ is the Gamma function, and we can actually compute
$$
\int_{\gamma_{\ell} \cup \gamma^-_{\kappa}} Q_{n,m}(z) z^{j} e^{-z^K}dz=\left(e^{\frac{2\pi i\ell(j+1)}{K}}-e^{\frac{2\pi i\kappa(j+1)}{K}}\right)\frac{1}{K}\Gamma\left( \frac{j+1}{K} \right).
$$

Using these moments and the linear system equivalent to \eqref{linear_system_coefficients}, in Figures \ref{figure_zeros_quintic_symmetric}--\ref{figure_zeros_seventh_asymmetric2} we plot the appropriately rescaled zeros of the polynomials $C_{n,m}$, $D_{n,m}$ and $Q_{n,m}$, for several values of $n$ and $m$, weights $e^{-z^{5}}$ and $e^{-z^{7}}$ and various choices of contours $\gamma_{l}\cup \gamma_k^{-}$. We devote the last paragraphs of this paper to some empirical discussion motivated by the results of these experiments. 

According to Theorem~\ref{theorem_zero_counting_measure}, the limiting density of zeros for $A_{n,m}$ and $P_{n,m}$ coincide on $\Delta_1$. But Figures~\ref{figure_zeros_1-5}--\ref{figure_zeros_3-7} allow us to conjecture that  a stronger statement is true: for any $n,m$, the zeros of $A_{n,m}$ and $P_{n,m}$ on $\Delta_1$ interlace. This same interlacing appears for the real zeros of $Q_{n,m}$ and $C_{n,m}$ displayed in Figures~\ref{figure_zeros_quintic_symmetric}--\ref{figure_zeros_seventh_symmetric}.

As one can see in Figures \ref{figure_zeros_1-5}--\ref{figure_zeros_3-5}, the zeros of $B_{n,m}$ are not on the real line. Nevertheless, they seem to display a ``quasi-interlacing'' pattern with the zeros of $P_{n,m}$ on the set where $\mu_B$ and $\mu_2$ coincide. A similar phenomenon can be observed in Figures \ref{figure_zeros_quintic_symmetric}--\ref{figure_zeros_seventh_asymmetric2}

The zeros of $P_{n,m}$ outside the real line seem to quickly approach on $\supp\mu_2$, even for very small values of $n$ and $m$. The same is true for the zeros of $B_{n,m}$ on $\supp\mu_B\cap \supp\mu_2$, but its zeros on $\supp\mu_B\cap \H_-$ seem to converge slower to their limiting support. Furthermore, the plots indicate that the measures $\mu_3$ and $\mu_B$ have very small mass near $a_1$ for $\alpha>\alpha_c$. 

The contours of orthogonality in Figures~\ref{figure_zeros_quintic_symmetric}--\ref{figure_zeros_seventh_symmetric} were chosen in such a way that the polynomial $P_{n,m}$ is real and has some real zeros. As one can observe in these figures, for such choice of contours the zeros display some transitions similar to the ones rigorously described in this paper for the weight $e^{-z^3}$. 

In Figures \ref{figure_zeros_quintic_symmetric}--\ref{figure_zeros_seventh_asymmetric2} indicate that for small values of $\alpha$ a portion of zeros of $P_{n,m}$ coincide, in the large degree limit, with the zeros of $C_{n,m}$, and the remaining zeros of $P_{n,m}$ coincide with the zeros of $B_{n,m}$. However, when $\alpha$ increases, zeros of $C_{n,m}$ and $D_{n,m}$ get closer to each other, and they seem to display some sort of repulsion. At this time, some of the zeros of $P_{n,m}$ do not follow zeros of $C_{n,m}$ or $D_{n,m}$ anymore.


\section*{Acknowledgements}

The first author was partially supported by the Spanish Government together with the European Regional Development Fund (ERDF) under 
grants MTM2014-53963-P and MTM2017-89941-P (from MINECO), by Junta de Andaluc\'{\i}a (the Excellence Grant P11-FQM-7276 and the research group FQM-229), and by Campus 
de Excelencia Internacional del Mar (CEIMAR) of the University of Almer\'{\i}a.

\newpage

\begin{figure}[p!]
\begin{subfigure}{.5\textwidth}
\centering
\begin{overpic}[scale=.45]{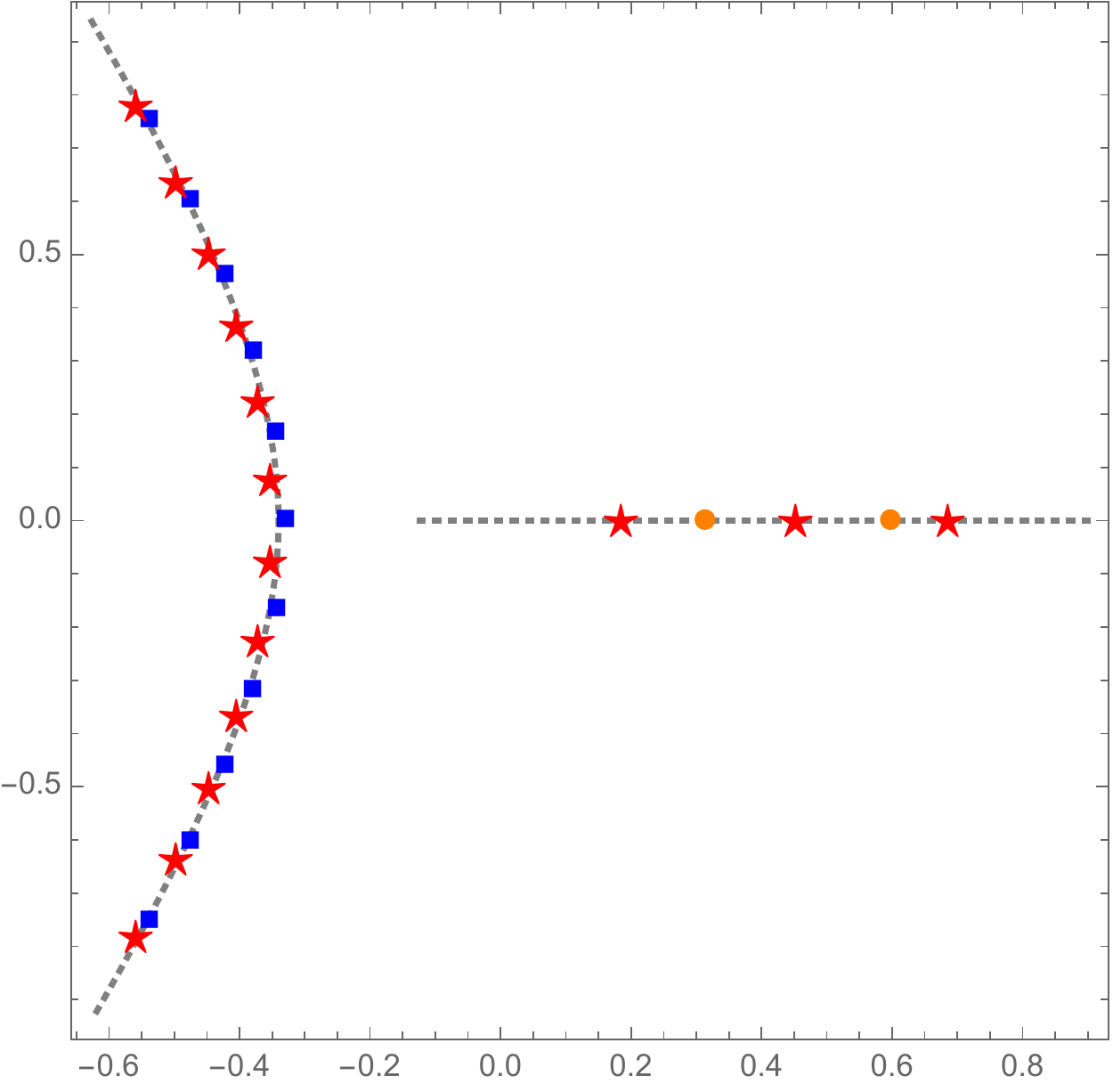}
\end{overpic}
\end{subfigure}%
\begin{subfigure}{.5\textwidth}
\centering
\begin{overpic}[scale=.45]{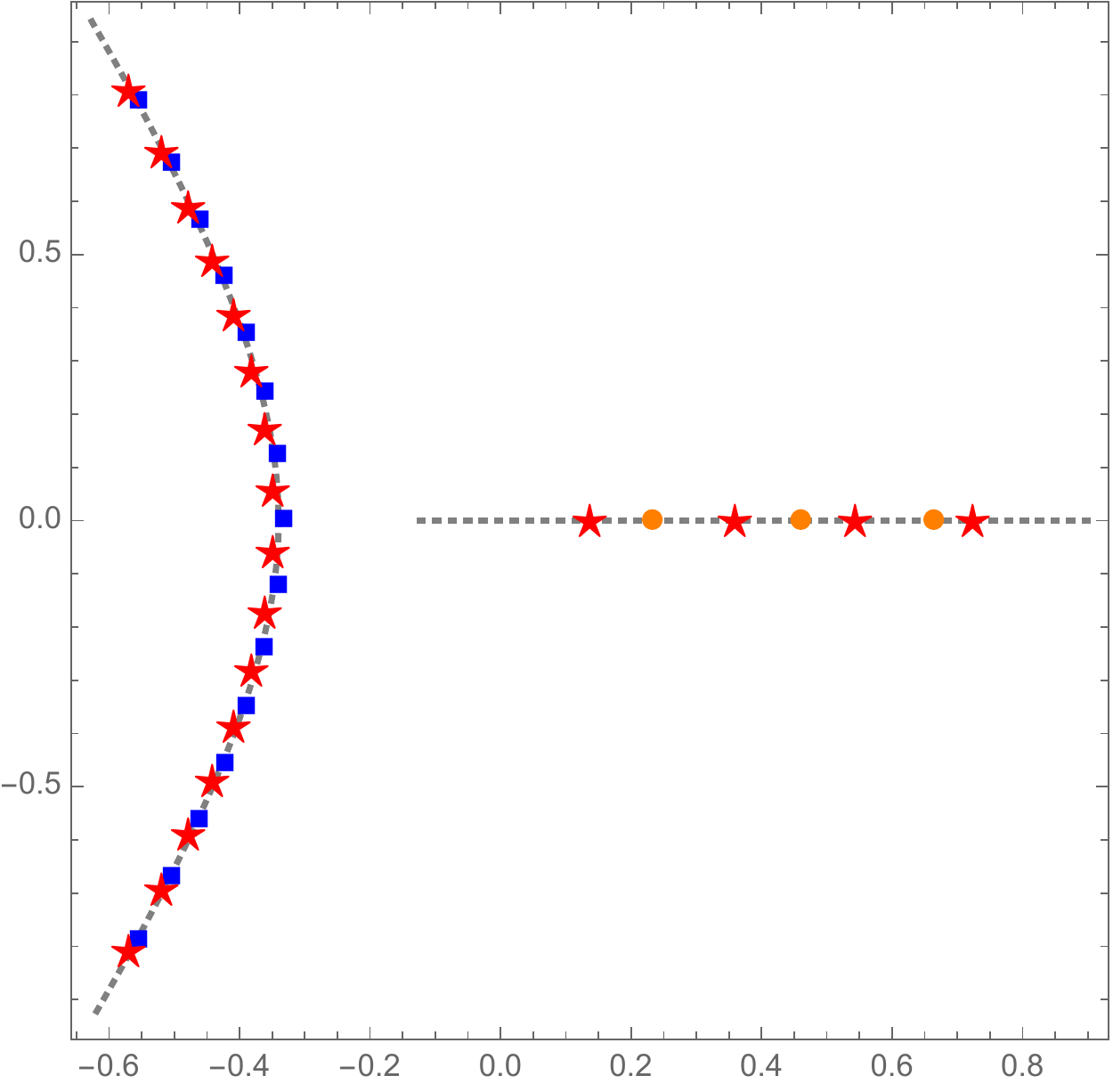}
\end{overpic}
\end{subfigure}\\
\begin{subfigure}{.5\textwidth}
\centering
\begin{overpic}[scale=.45]{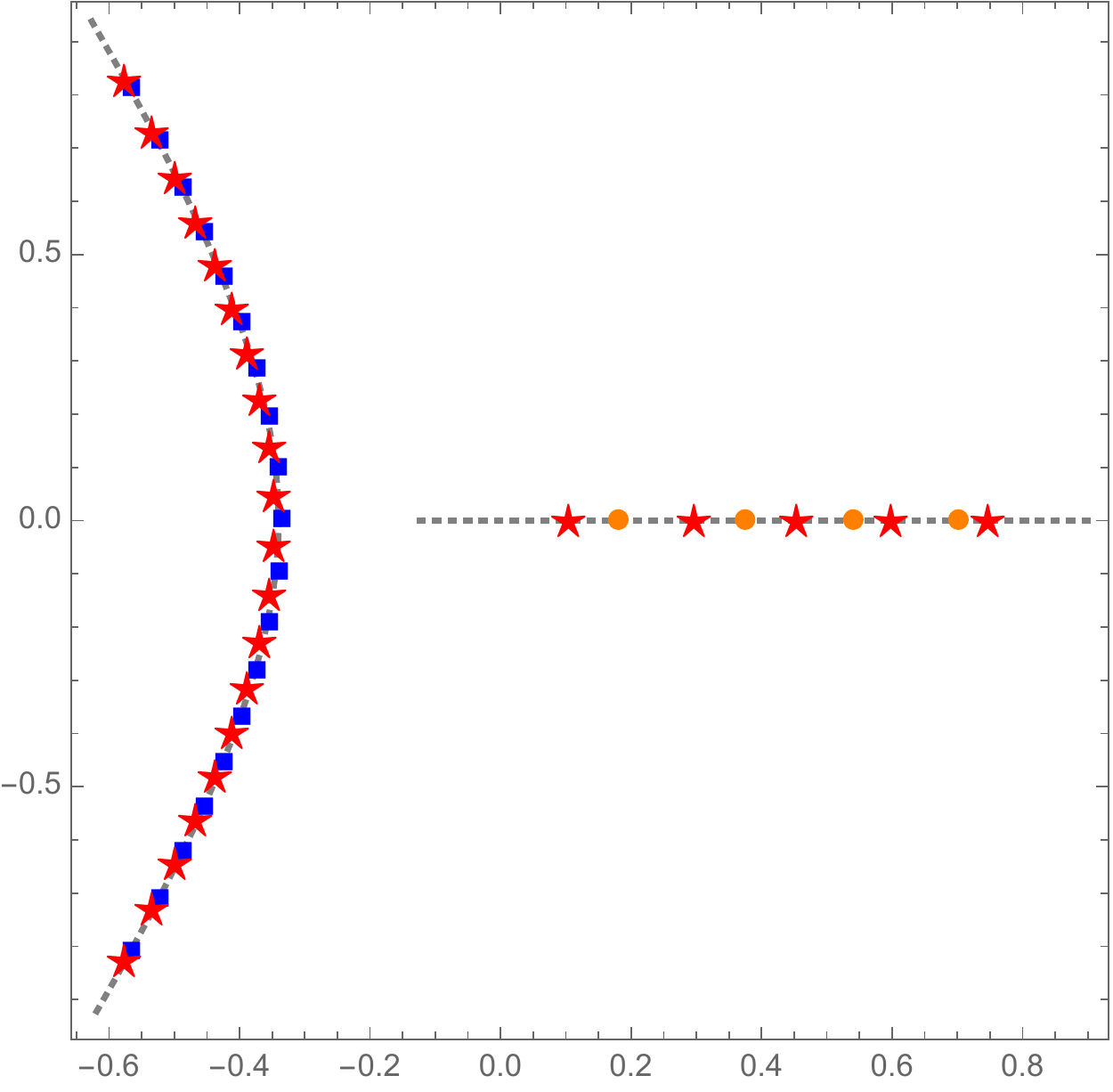}
\end{overpic}
\end{subfigure}%
\begin{subfigure}{.5\textwidth}
\centering
\begin{overpic}[scale=.45]{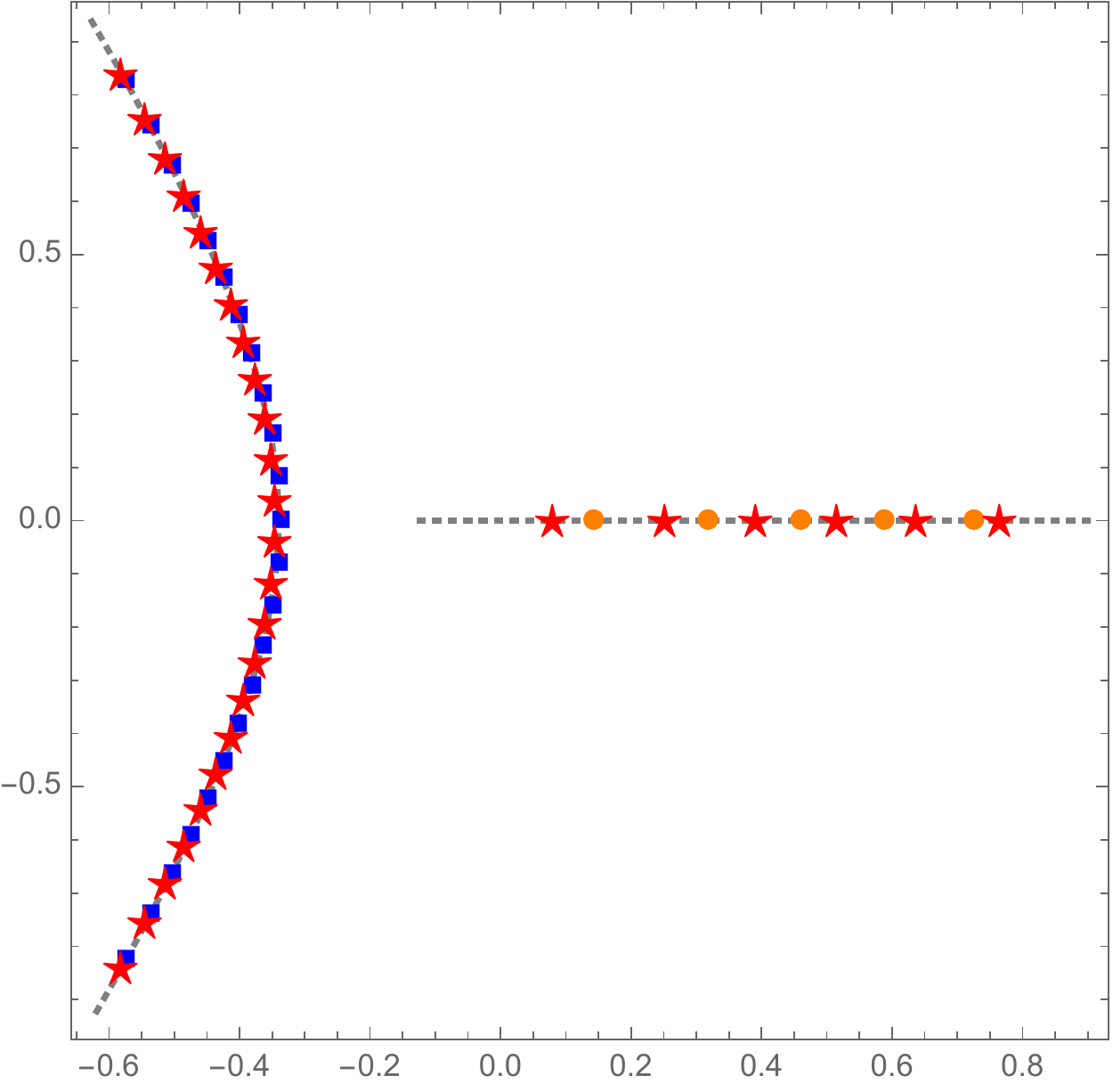}
\end{overpic}
\end{subfigure}\\
\begin{subfigure}{.5\textwidth}
\centering
\begin{overpic}[scale=.45]{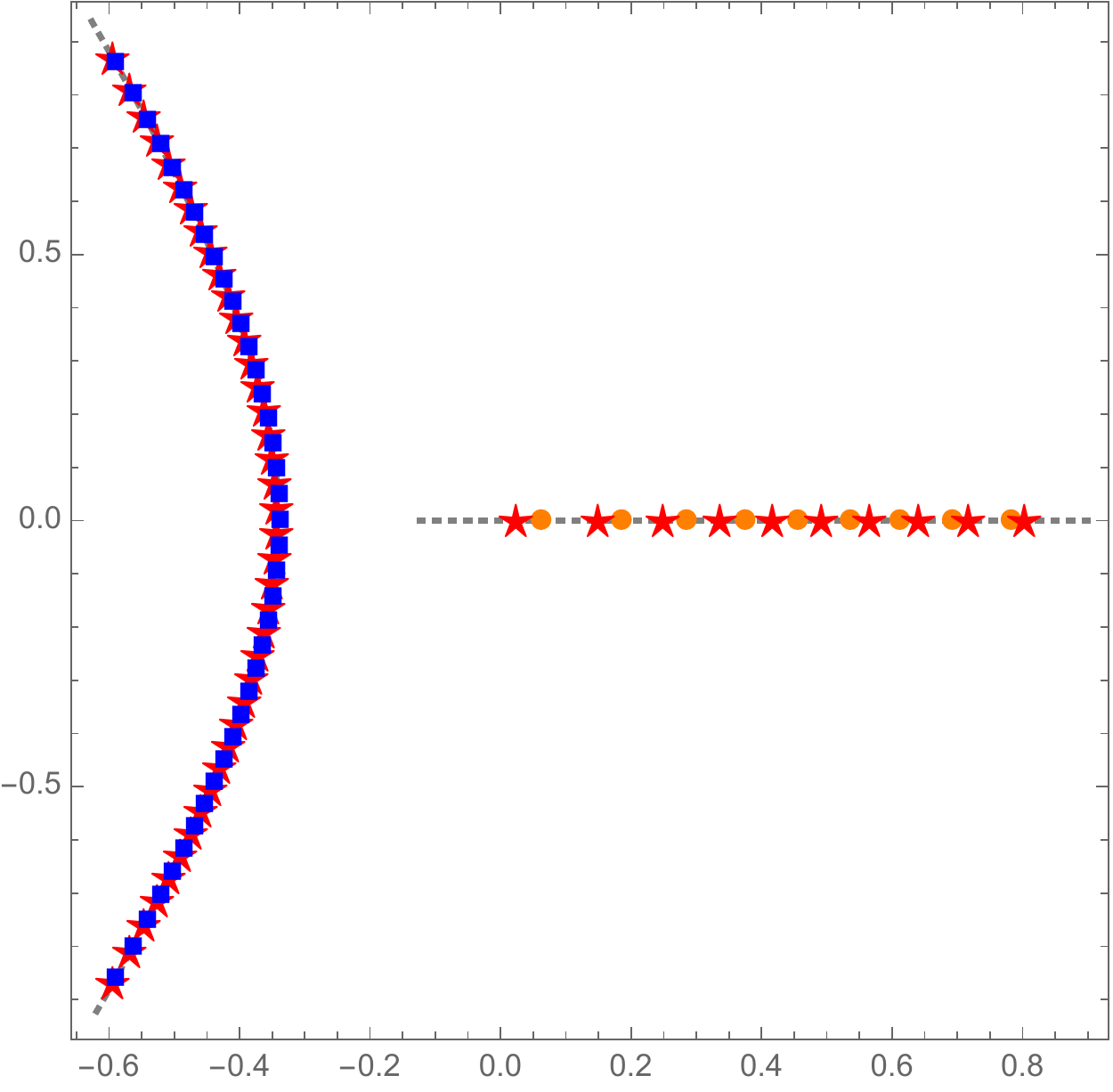}
\end{overpic}
\end{subfigure}%
\begin{subfigure}{.5\textwidth}
\centering
\begin{overpic}[scale=.45]{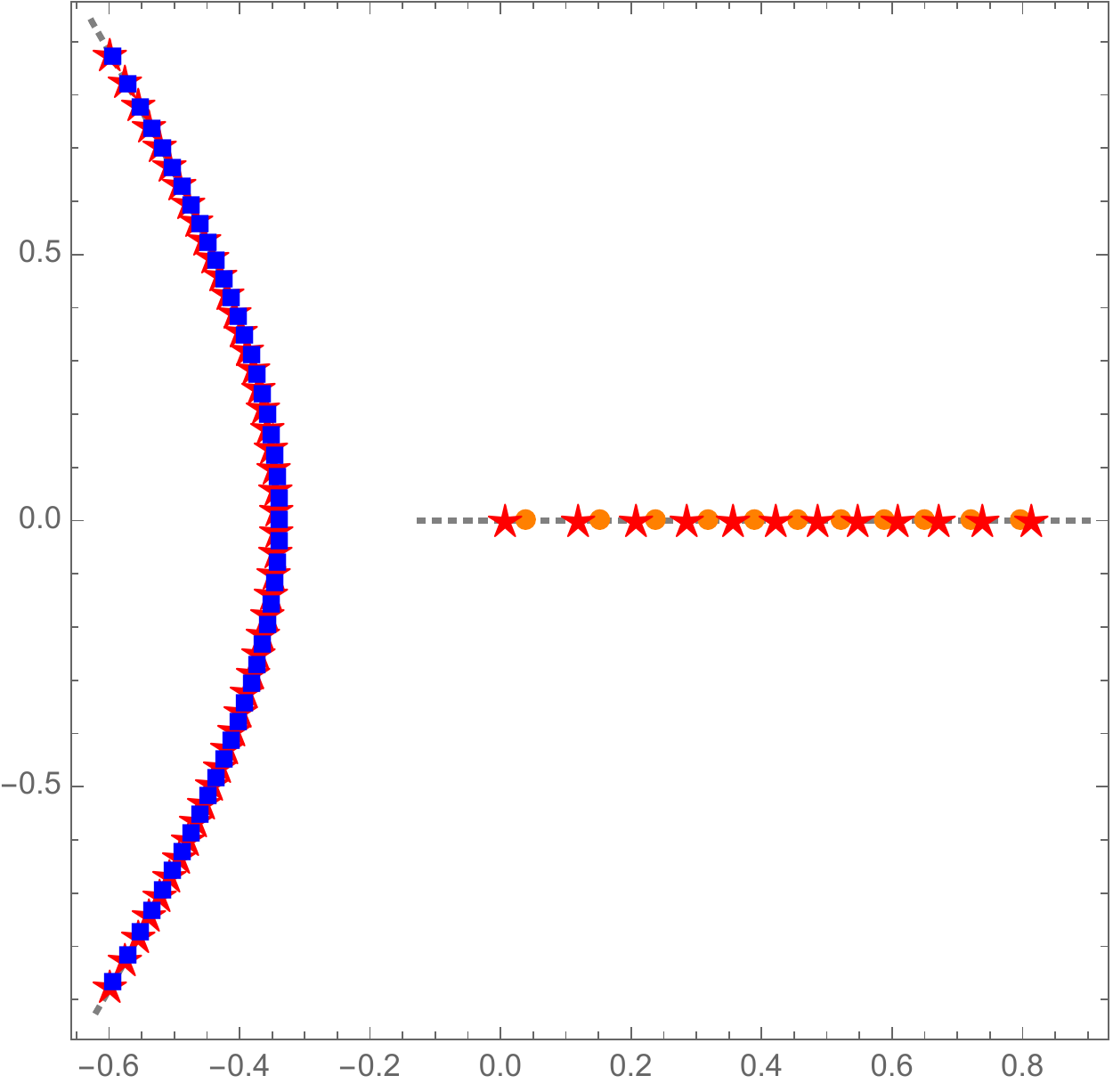}
\end{overpic}
\end{subfigure}
\caption{Zeros of $P_{n,m}$ (stars), $A_{n,m}$ (dots) and $B_{n,m}$ (squares), all corresponding to the same value $\alpha=1/5<\alpha_c$. The dashed lines are the supports of $\mu_1$ and $\mu_2=\mu_B$. From left to right, top to bottom: $(n,m)=(3,12),(4,16),(5,20),(6,24),(10,24),(12,48)$.}\label{figure_zeros_1-5}
\end{figure}

\begin{figure}[p!]
\begin{subfigure}{.5\textwidth}
\centering
\begin{overpic}[scale=.45]{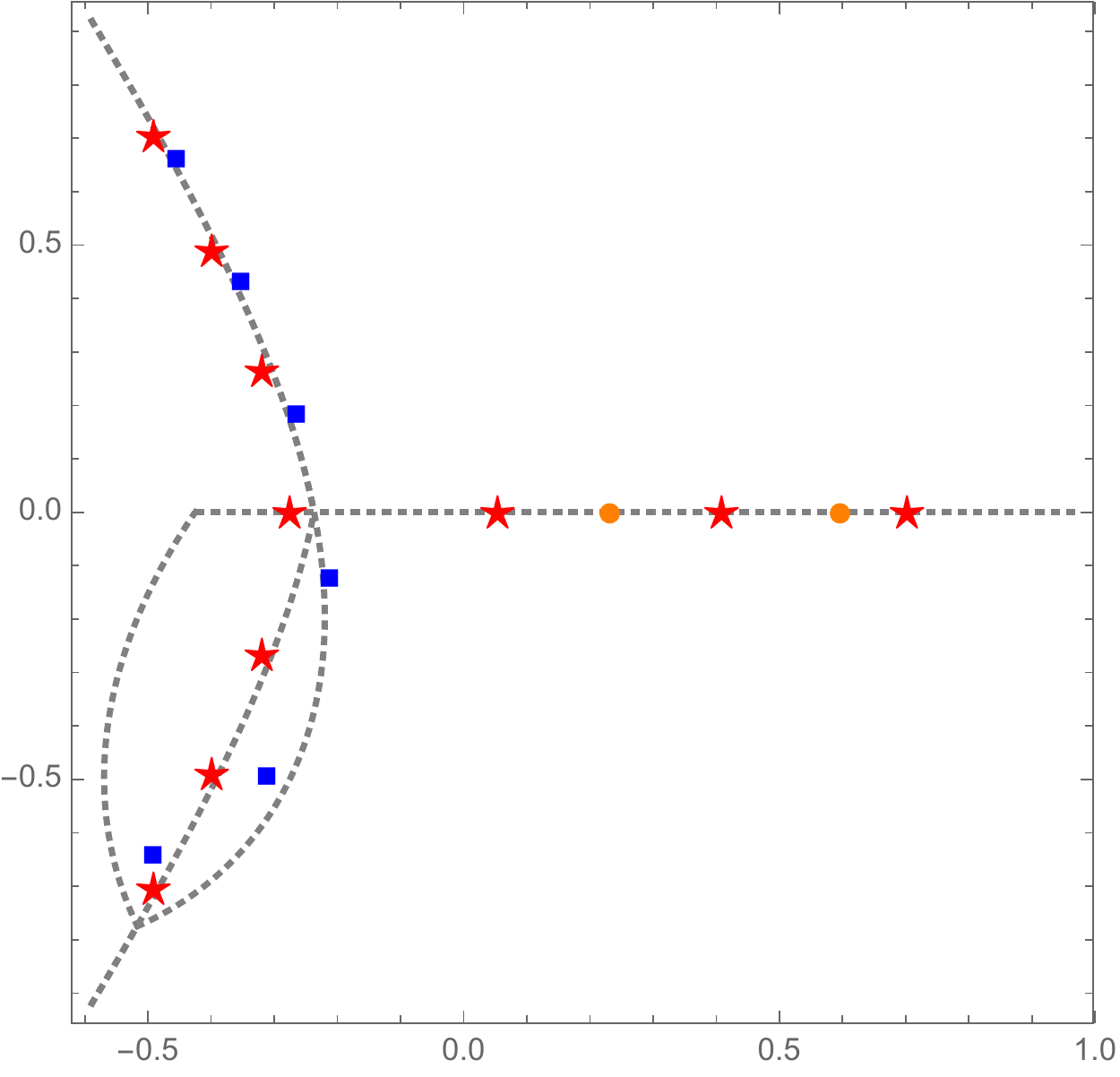}
\end{overpic}
\end{subfigure}%
\begin{subfigure}{.5\textwidth}
\centering
\begin{overpic}[scale=.45]{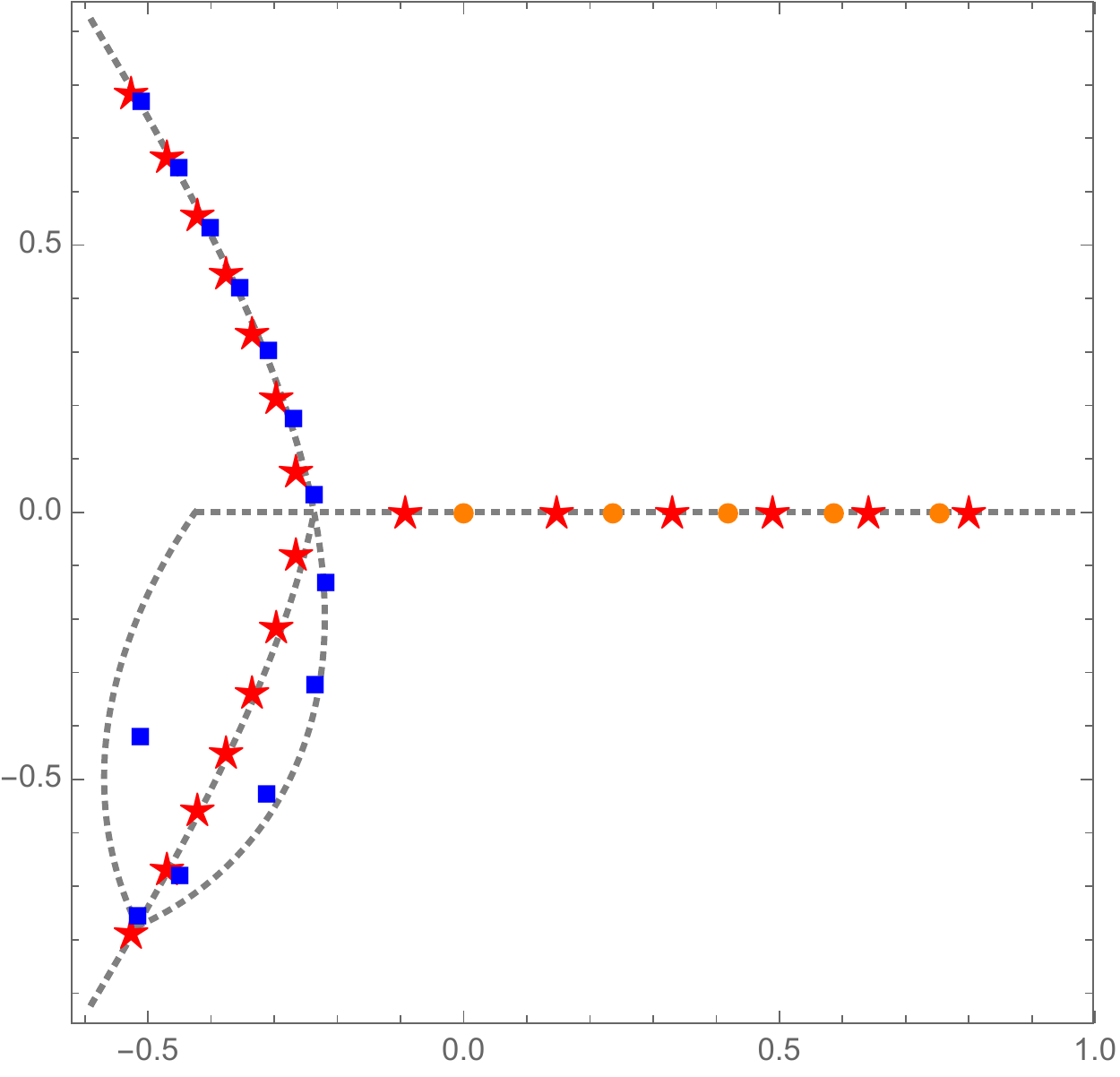}
\end{overpic}
\end{subfigure}\\
\begin{subfigure}{.5\textwidth}
\centering
\begin{overpic}[scale=.45]{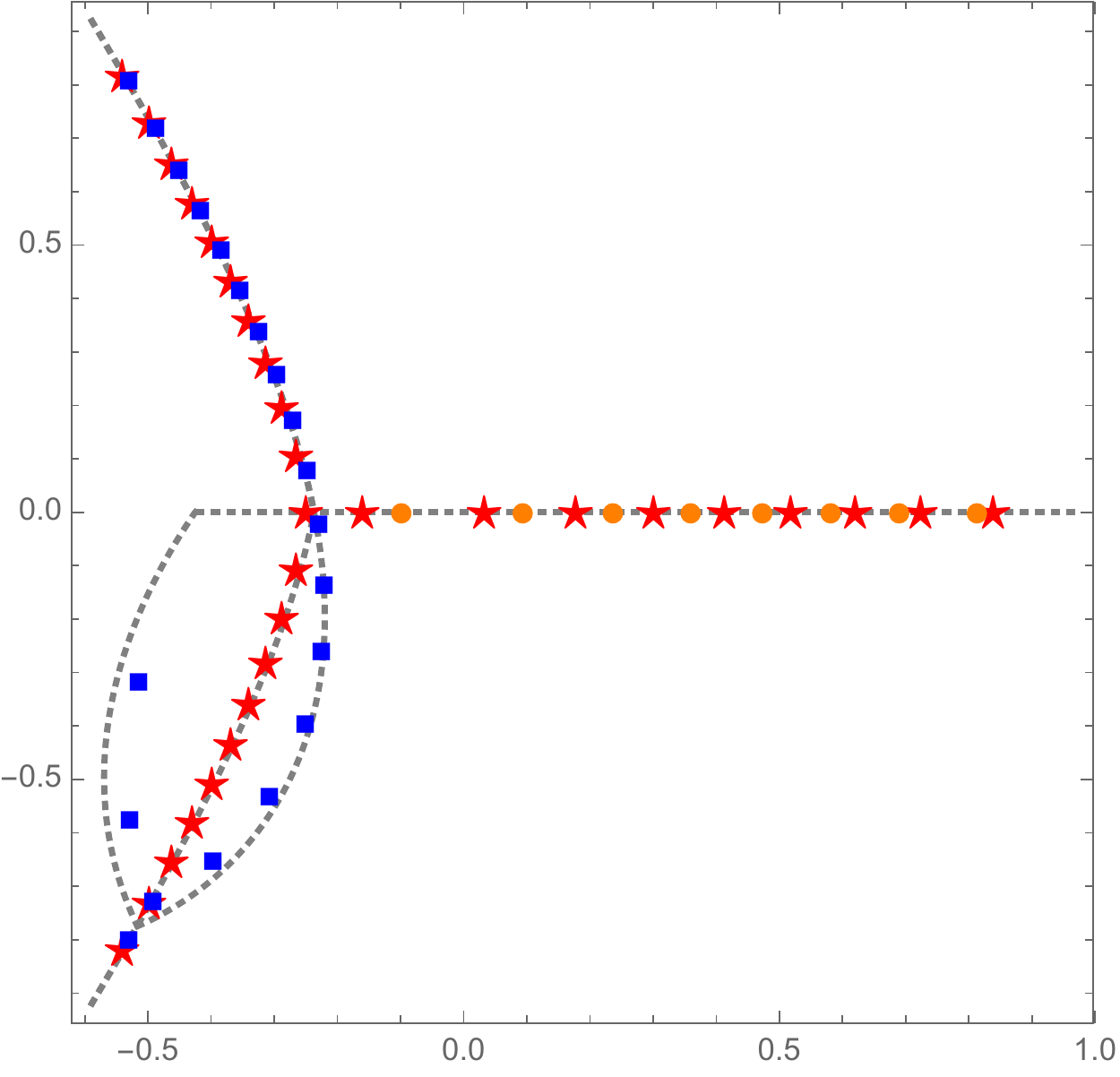}
\end{overpic}
\end{subfigure}%
\begin{subfigure}{.5\textwidth}
\centering
\begin{overpic}[scale=.45]{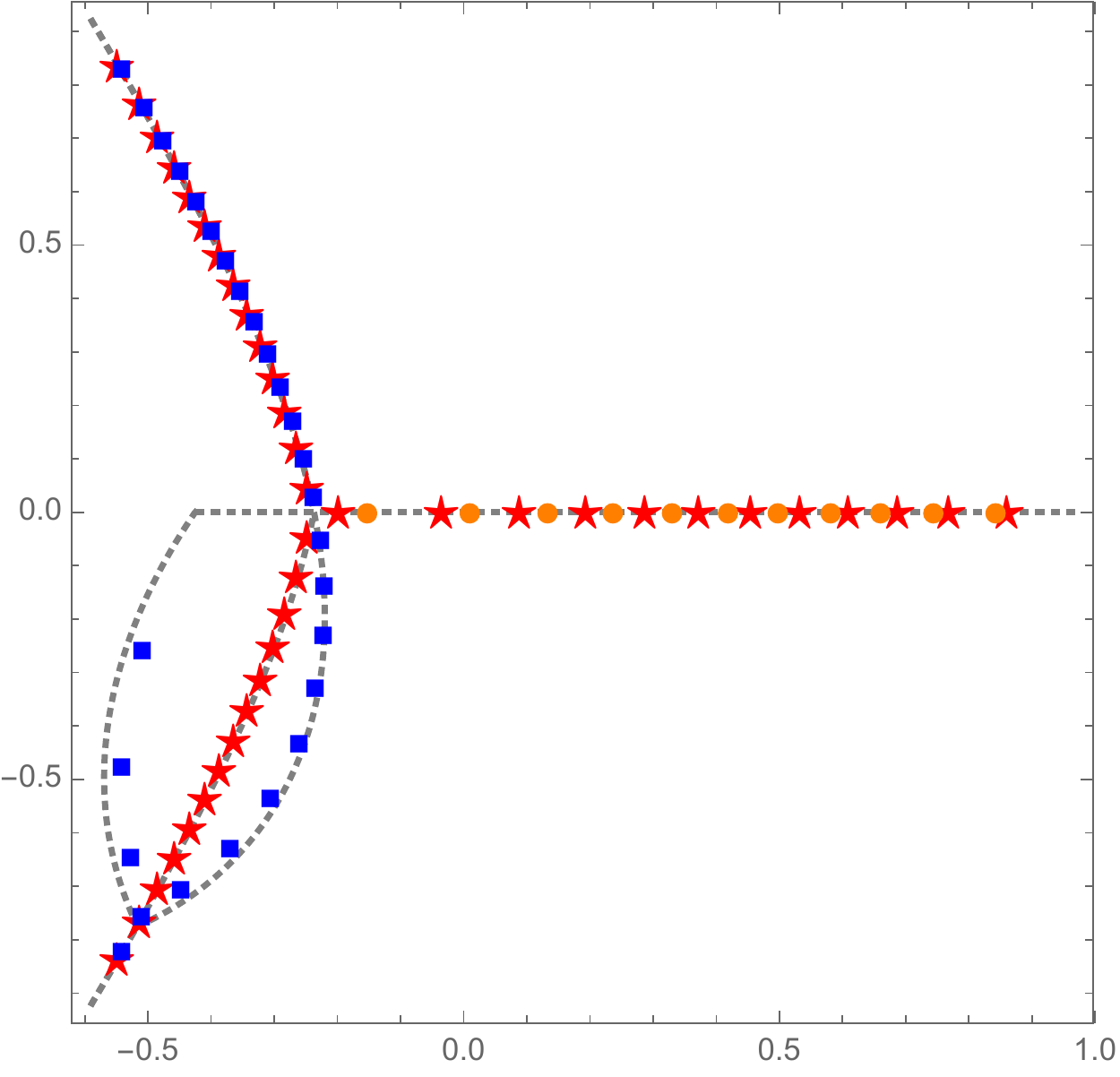}
\end{overpic}
\end{subfigure}\\
\begin{subfigure}{.5\textwidth}
\centering
\begin{overpic}[scale=.45]{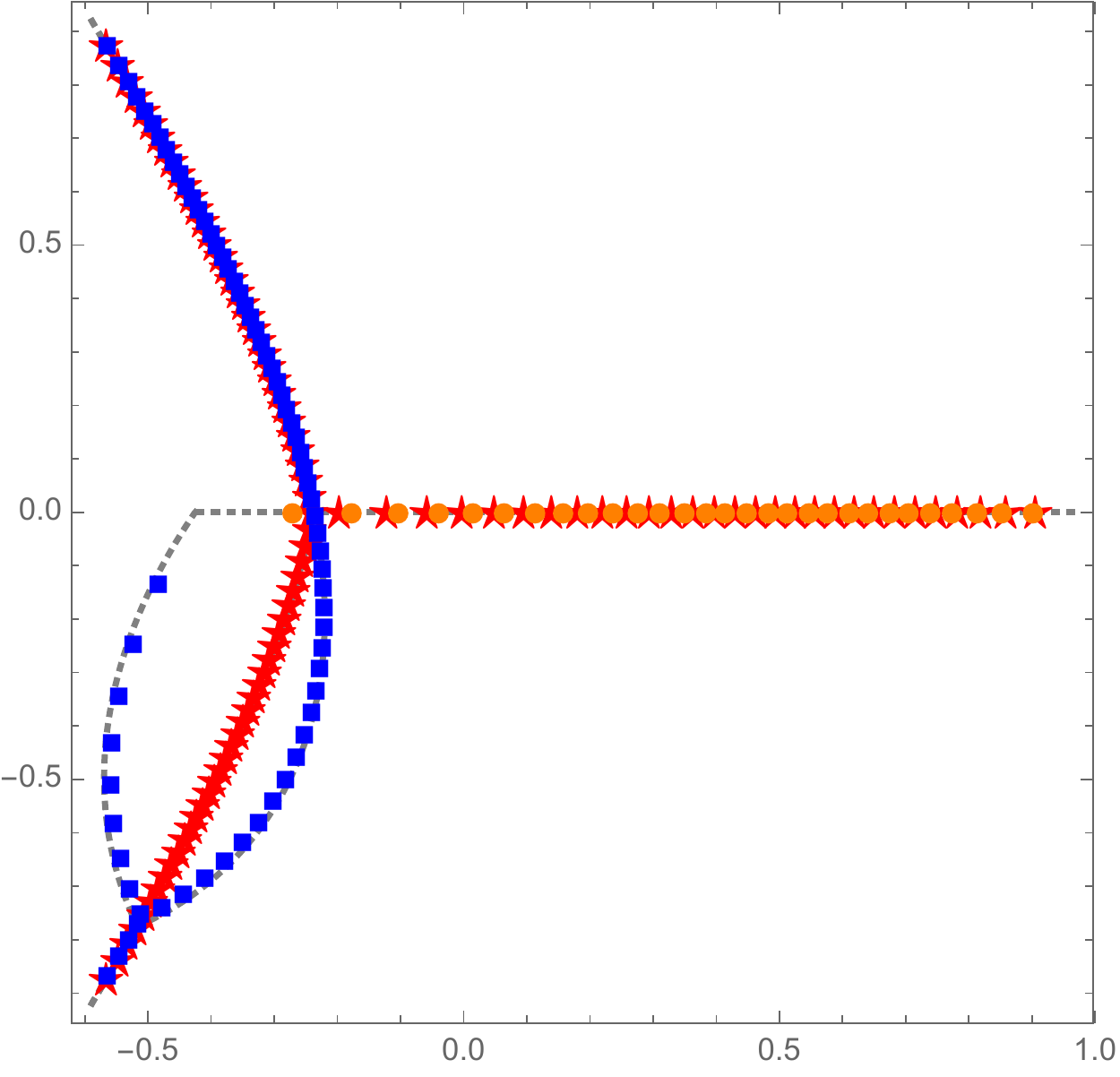}
\end{overpic}
\end{subfigure}%
\begin{subfigure}{.5\textwidth}
\centering
\begin{overpic}[scale=.45]{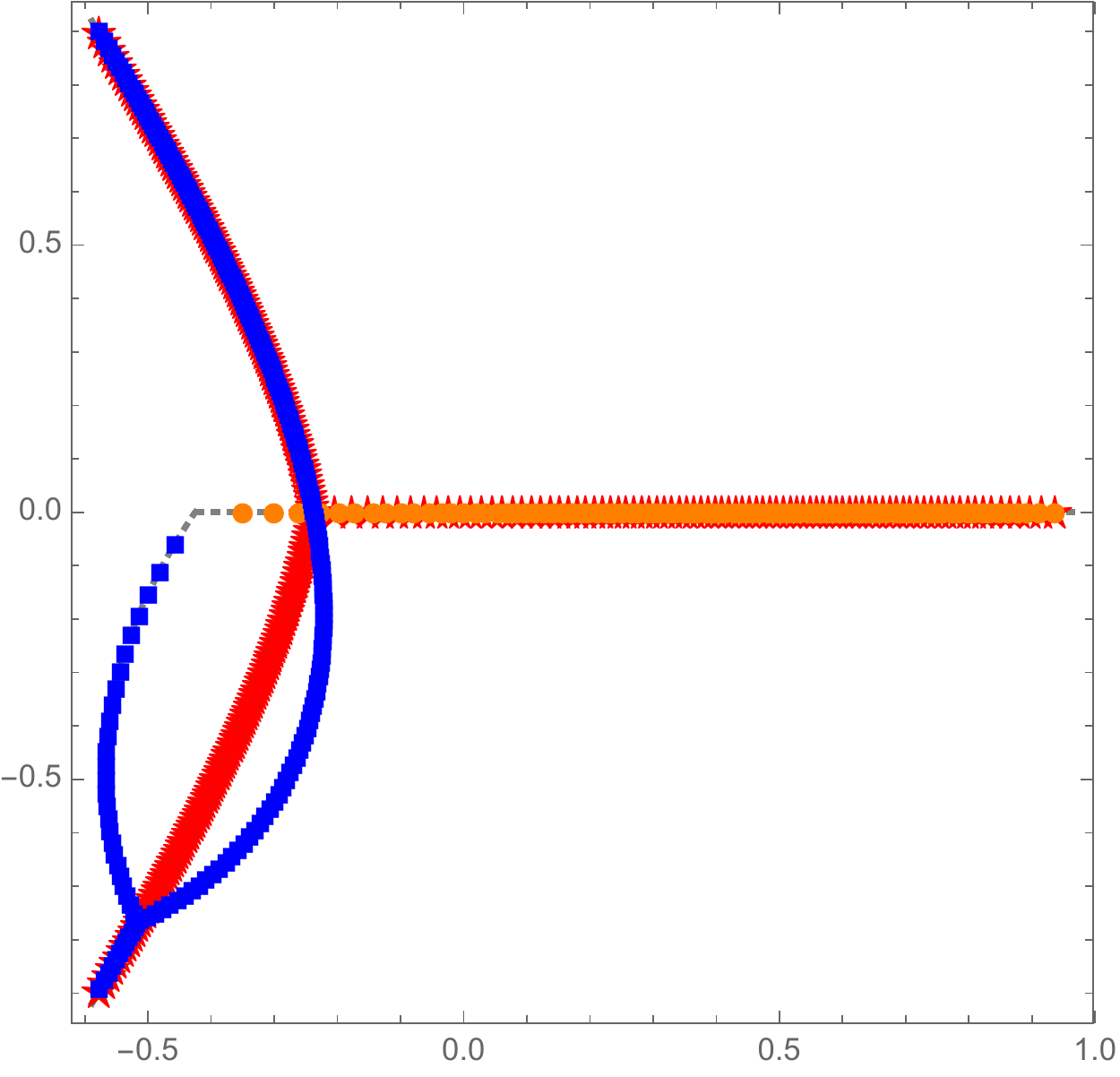}
\end{overpic}
\end{subfigure}
\caption{Zeros of $P_{n,m}$ (stars), $A_{n,m}$ (dots) and $B_{n,m}$ (squares), all corresponding to the same value $\alpha=3/10\in (\alpha_c,\alpha_2)$. The dashed lines are the supports of $\mu_1,\mu_2,\mu_3$ and $\mu_B$. From left to right, top to bottom: $(n,m)=(3,7),(6,14),(9,21),(12,28),(30,70),(90,210)$.}\label{figure_zeros_3-7}
\end{figure}

\begin{figure}[p!]
\begin{subfigure}{.5\textwidth}
\centering
\begin{overpic}[scale=.45]{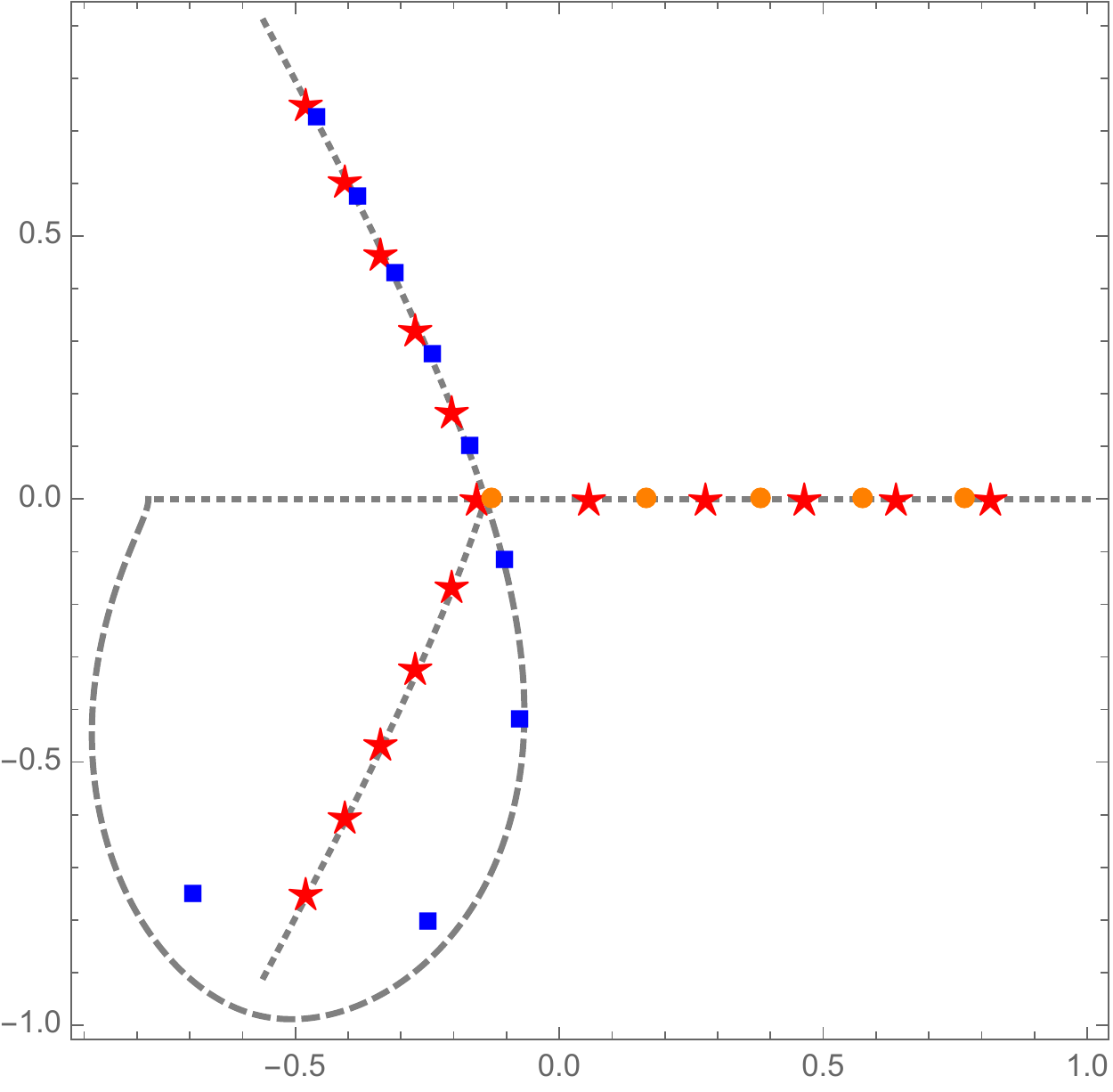}
\end{overpic}
\end{subfigure}%
\begin{subfigure}{.5\textwidth}
\centering
\begin{overpic}[scale=.45]{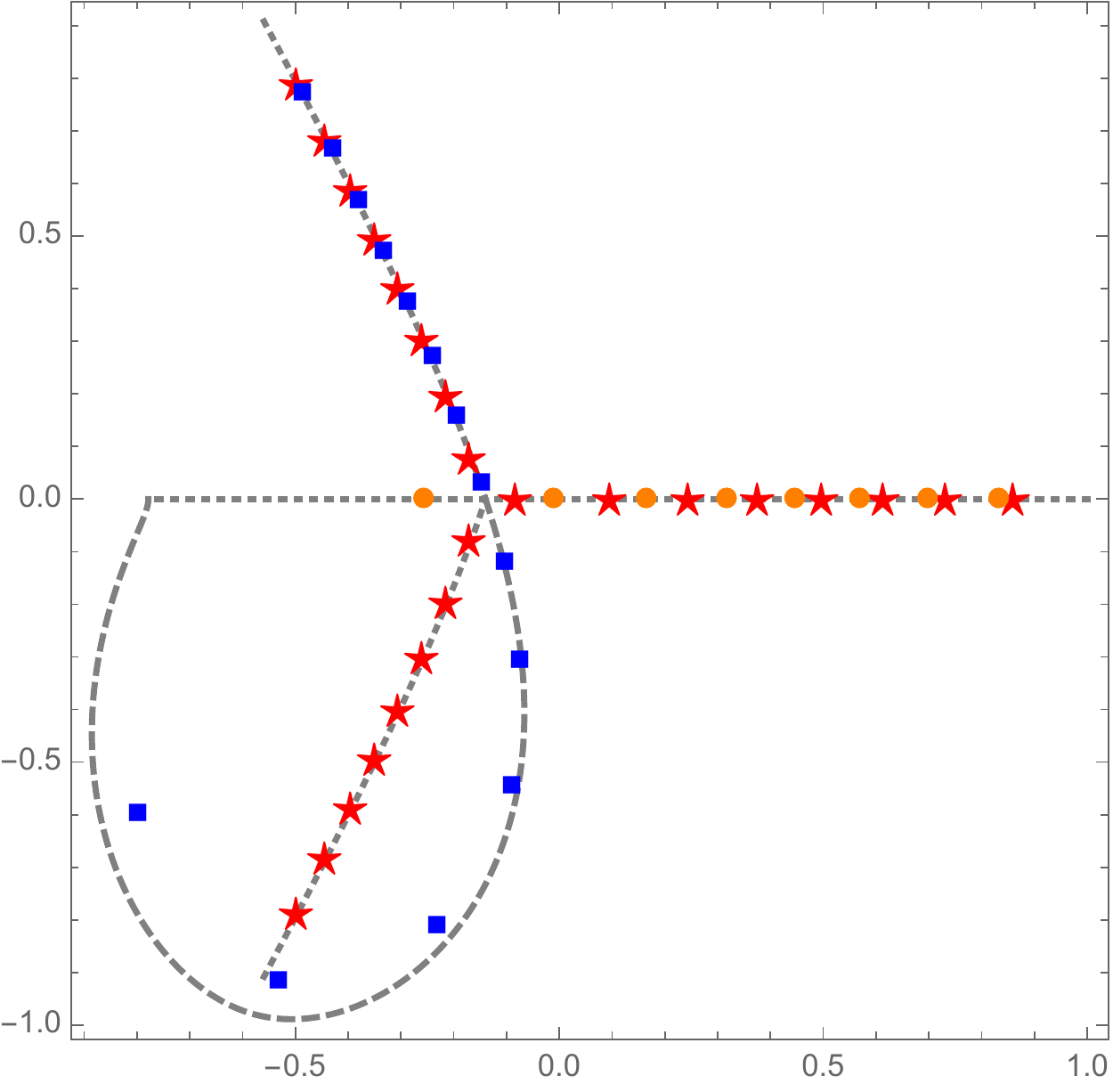}
\end{overpic}
\end{subfigure}\\
\begin{subfigure}{.5\textwidth}
\centering
\begin{overpic}[scale=.45]{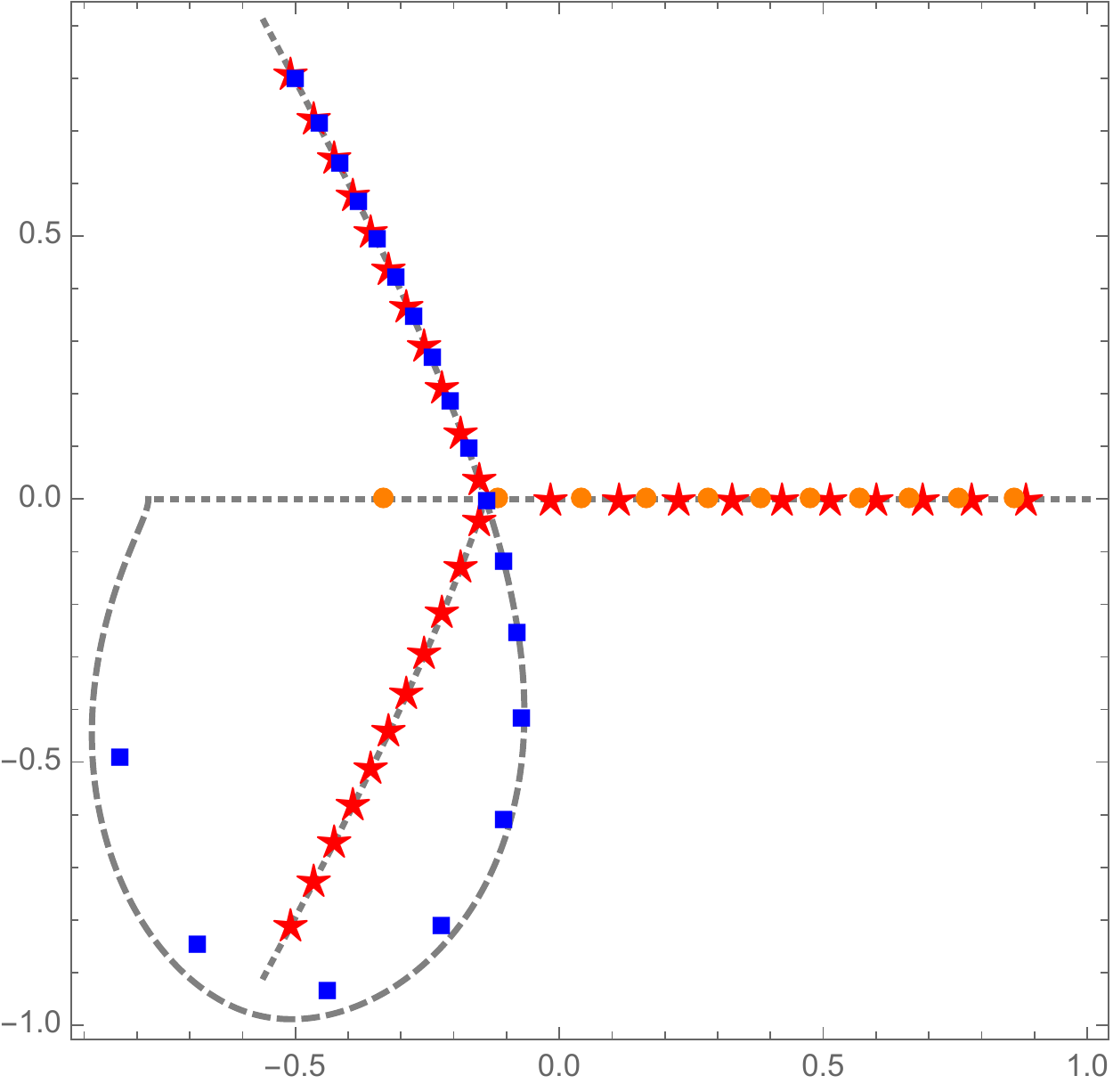}
\end{overpic}
\end{subfigure}%
\begin{subfigure}{.5\textwidth}
\centering
\begin{overpic}[scale=.45]{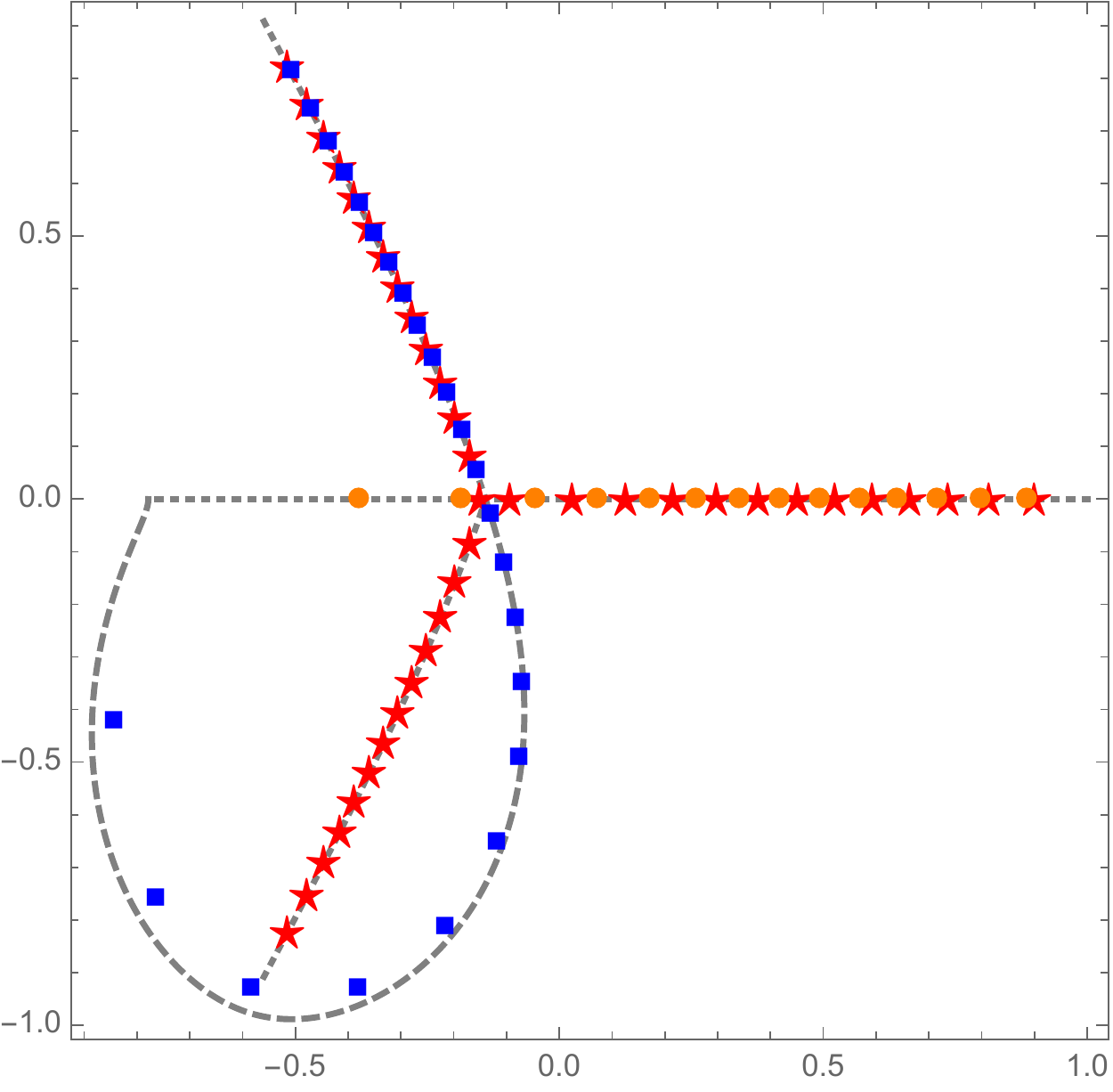}
\end{overpic}
\end{subfigure}\\
\begin{subfigure}{.5\textwidth}
\centering
\begin{overpic}[scale=.45]{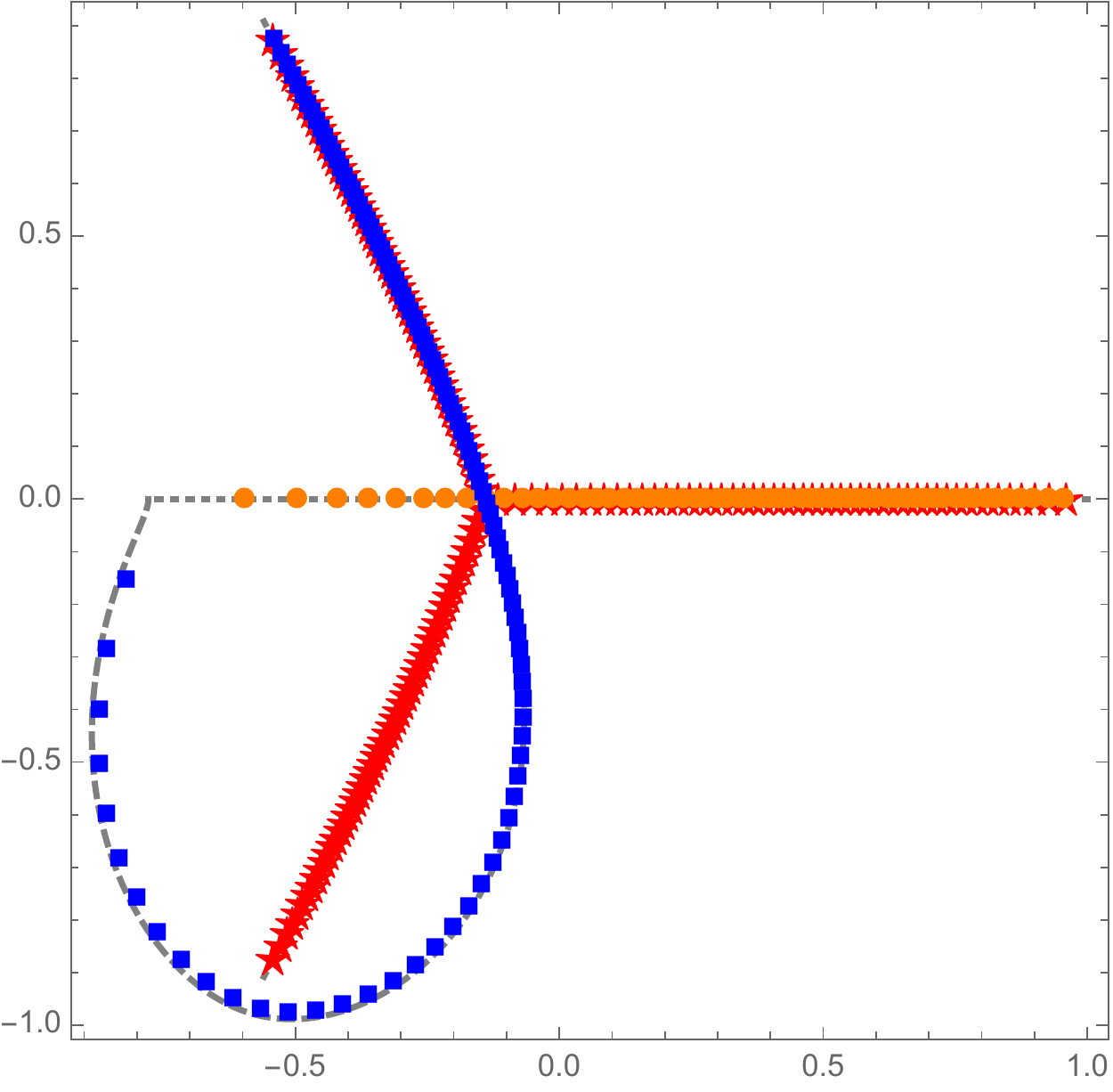}
\end{overpic}
\end{subfigure}%
\begin{subfigure}{.5\textwidth}
\centering
\begin{overpic}[scale=.45]{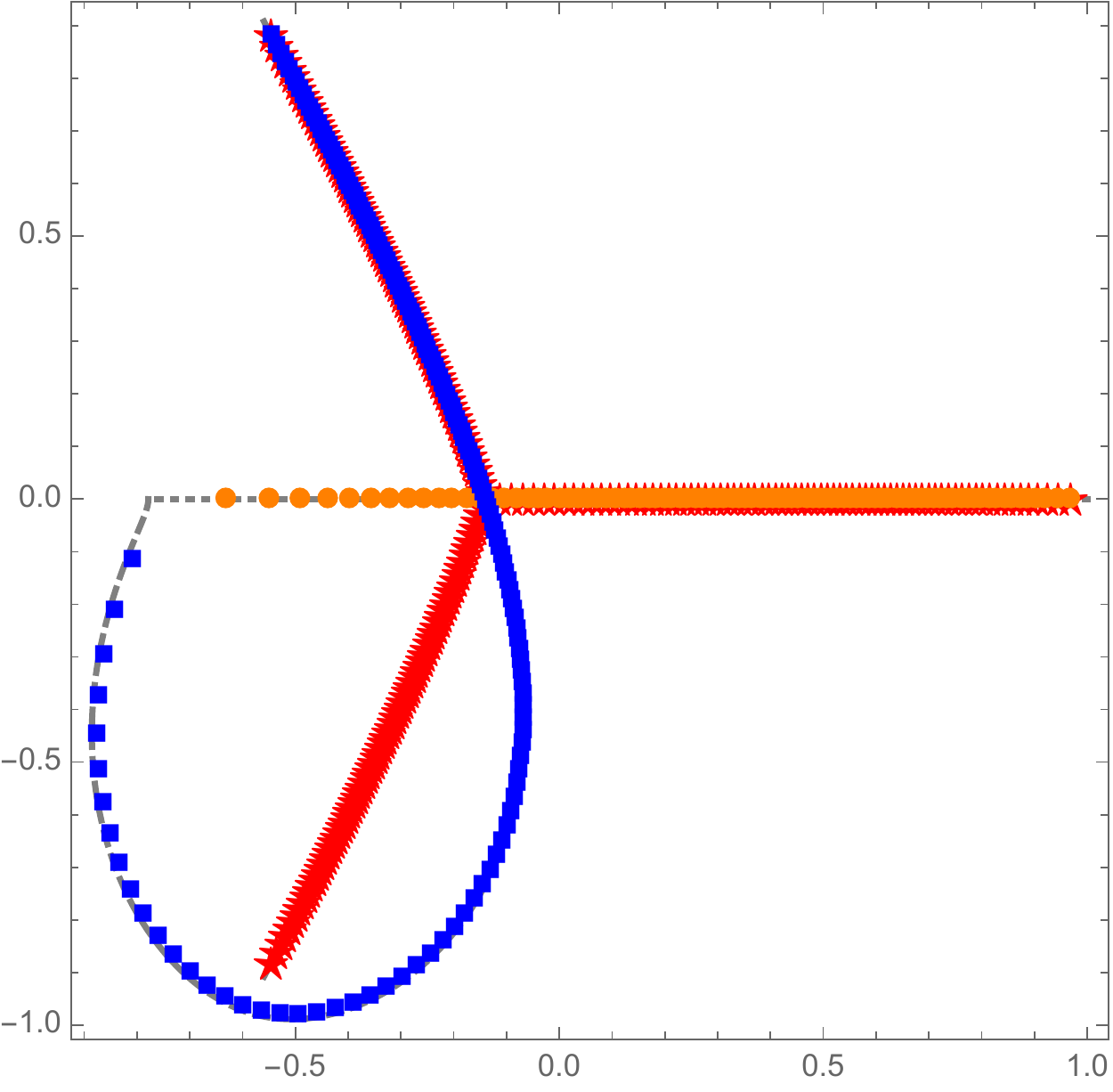}
\end{overpic}
\end{subfigure}
\caption{Zeros of $P_{n,m}$ (stars), $A_{n,m}$ (dots) and $B_{n,m}$ (squares), all corresponding to the same value $\alpha=3/8>\alpha_2$. The dashed lines are the supports of $\mu_1,\mu_2,\mu_3$ and $\mu_B$. From left to right, top to bottom: $(n,m)=(6,10),(9,15),(12,20),(15,25),(60,100),(90,150)$.}\label{figure_zeros_3-5}
\end{figure}

\begin{figure}[p!]
\begin{subfigure}{.5\textwidth}
\centering
\begin{overpic}[scale=.45]{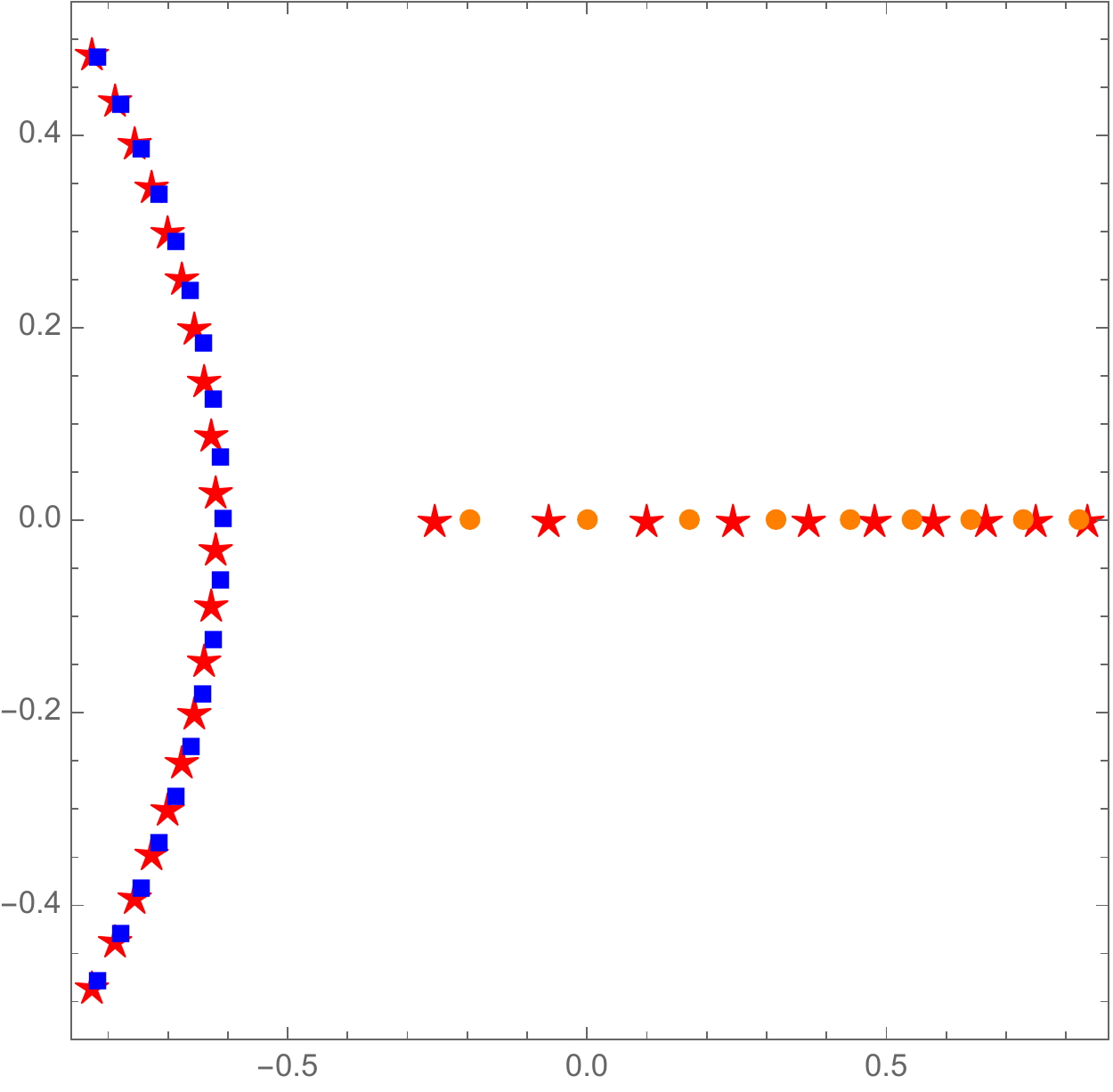}
\end{overpic}
\end{subfigure}%
\begin{subfigure}{.5\textwidth}
\centering
\begin{overpic}[scale=.45]{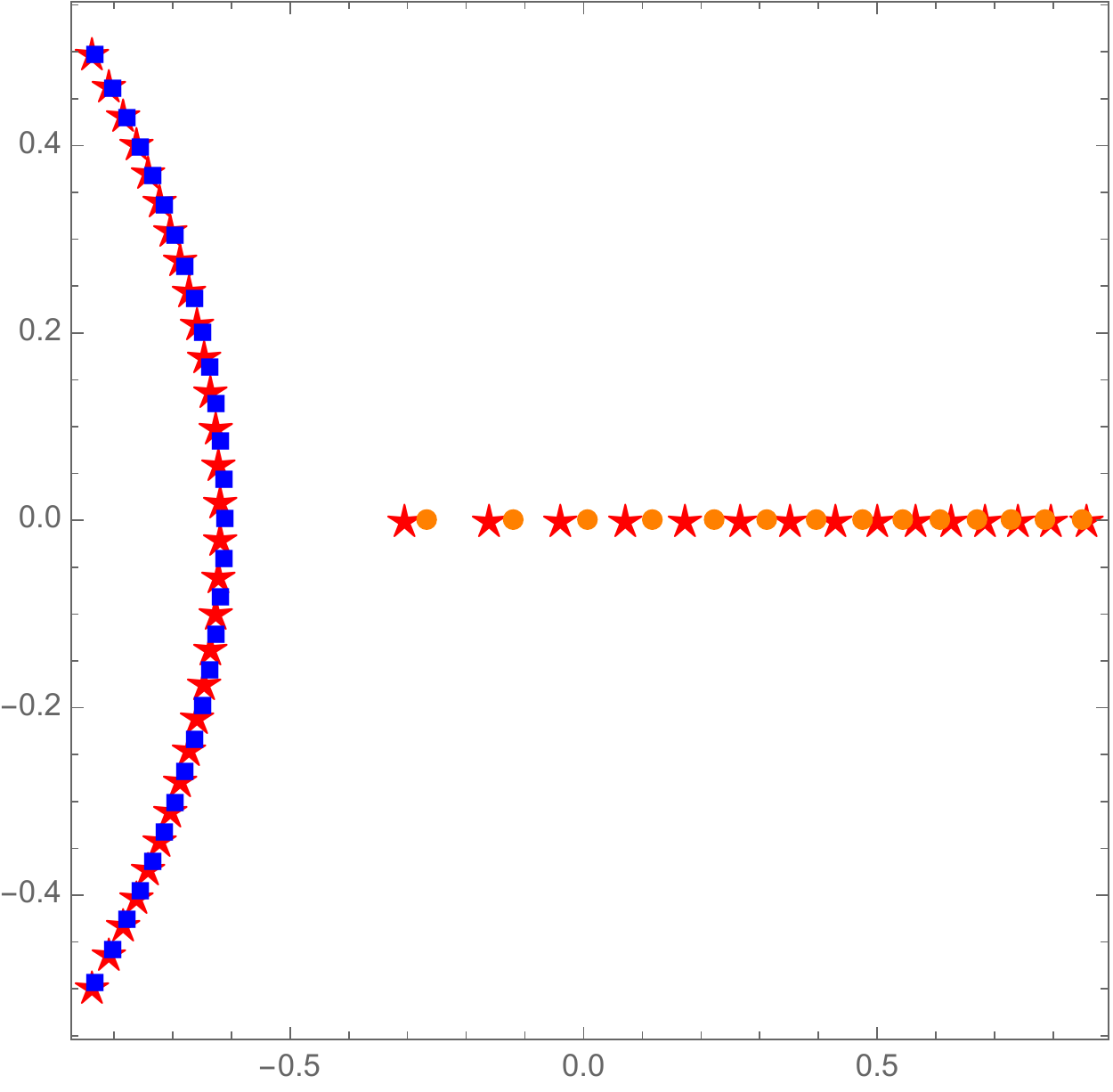}
\end{overpic}
\end{subfigure}\\
\begin{subfigure}{.5\textwidth}
\centering
\begin{overpic}[scale=.45]{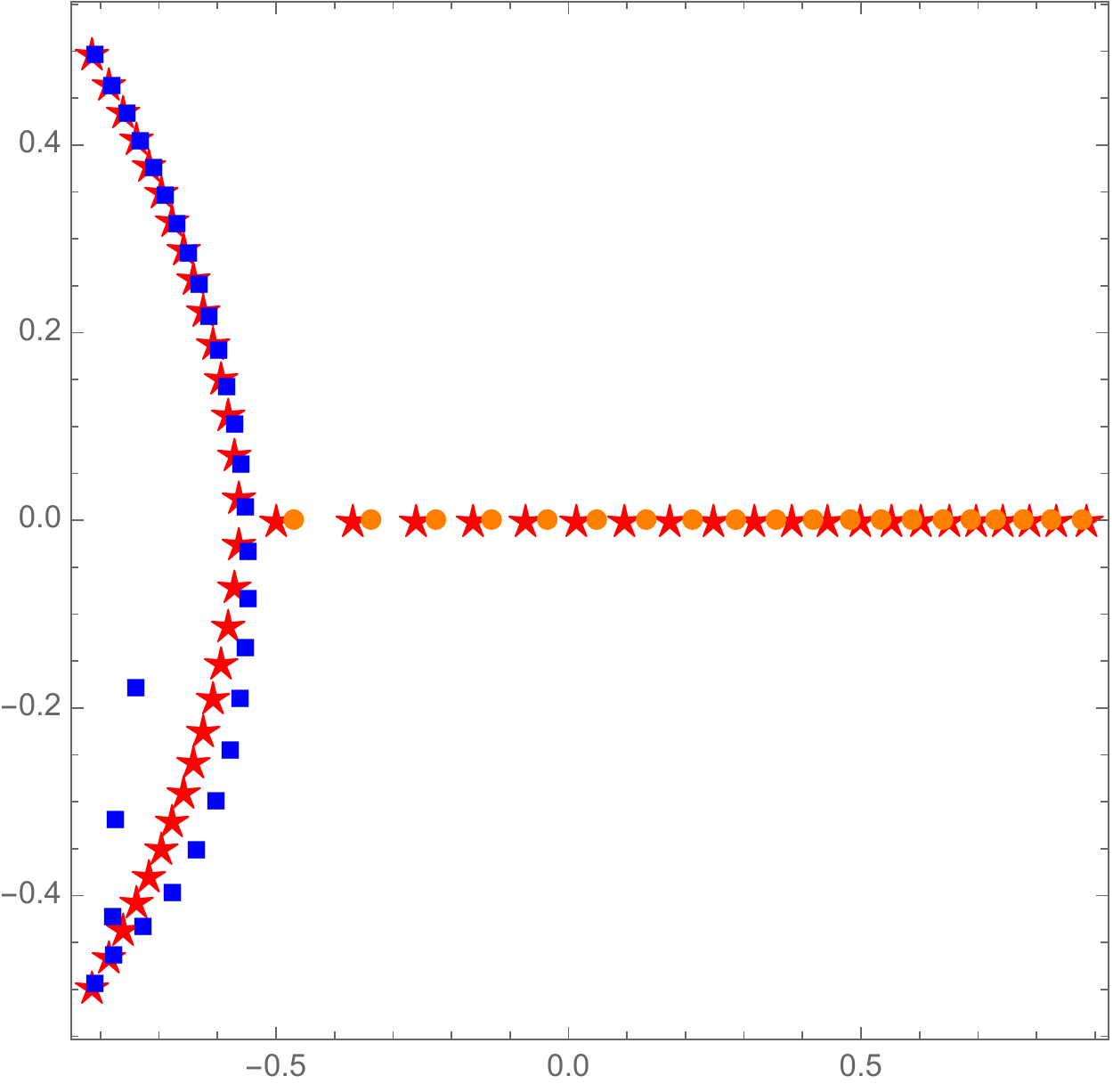}
\end{overpic}
\end{subfigure}%
\begin{subfigure}{.5\textwidth}
\centering
\begin{overpic}[scale=.45]{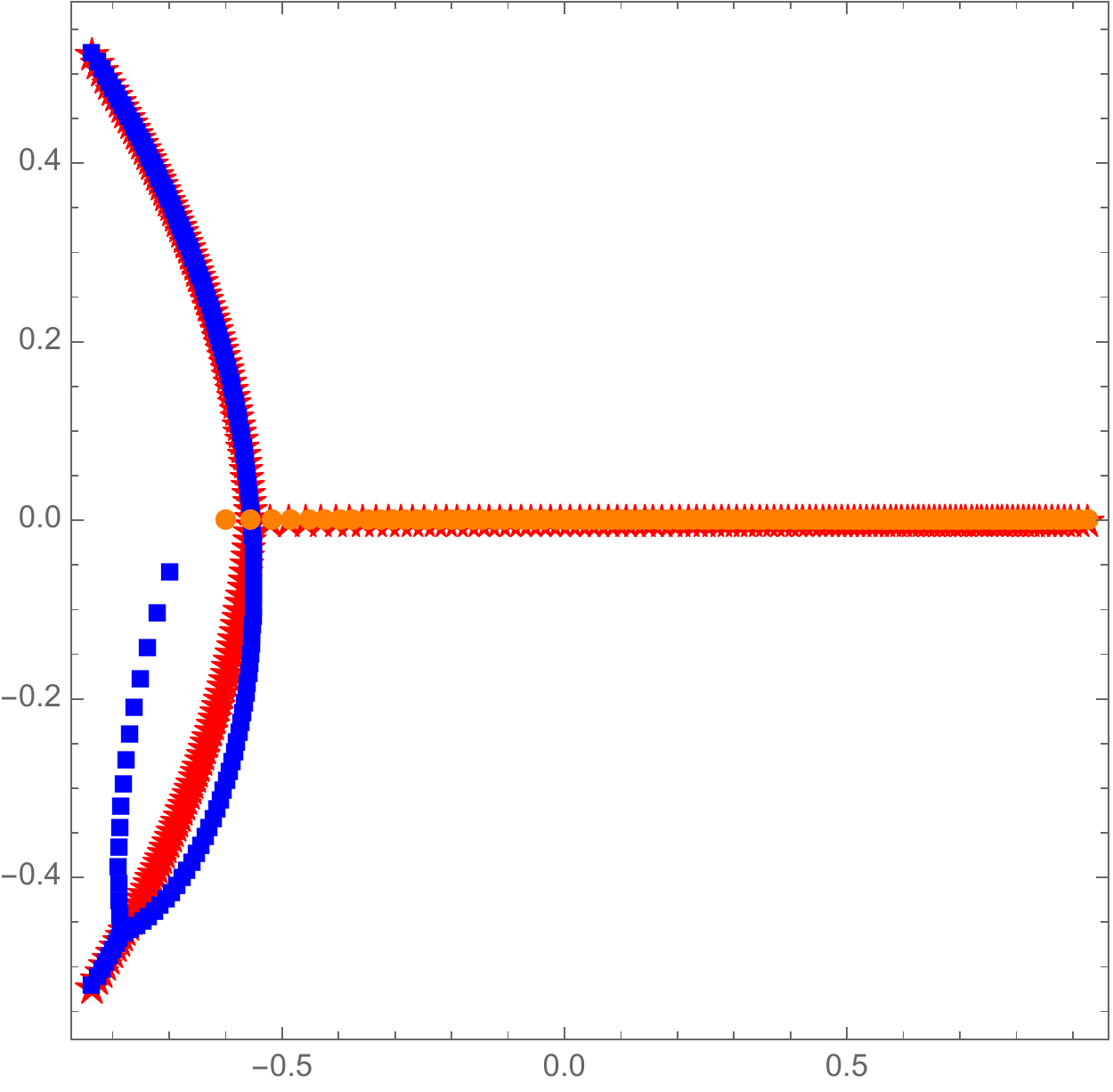}
\end{overpic}
\end{subfigure}\\
\begin{subfigure}{.5\textwidth}
\centering
\begin{overpic}[scale=.45]{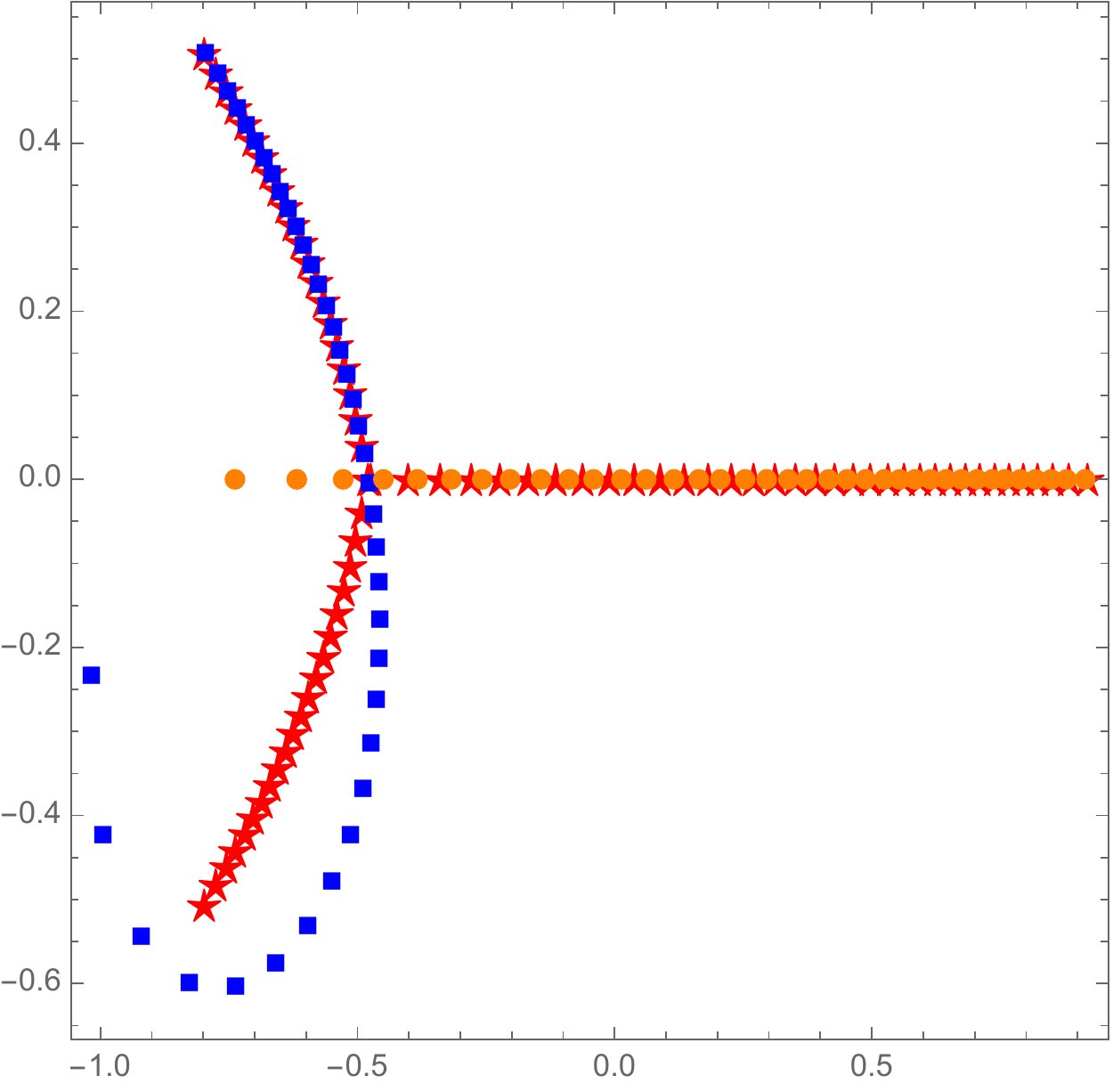}
\end{overpic}
\end{subfigure}%
\begin{subfigure}{.5\textwidth}
\centering
\begin{overpic}[scale=.45]{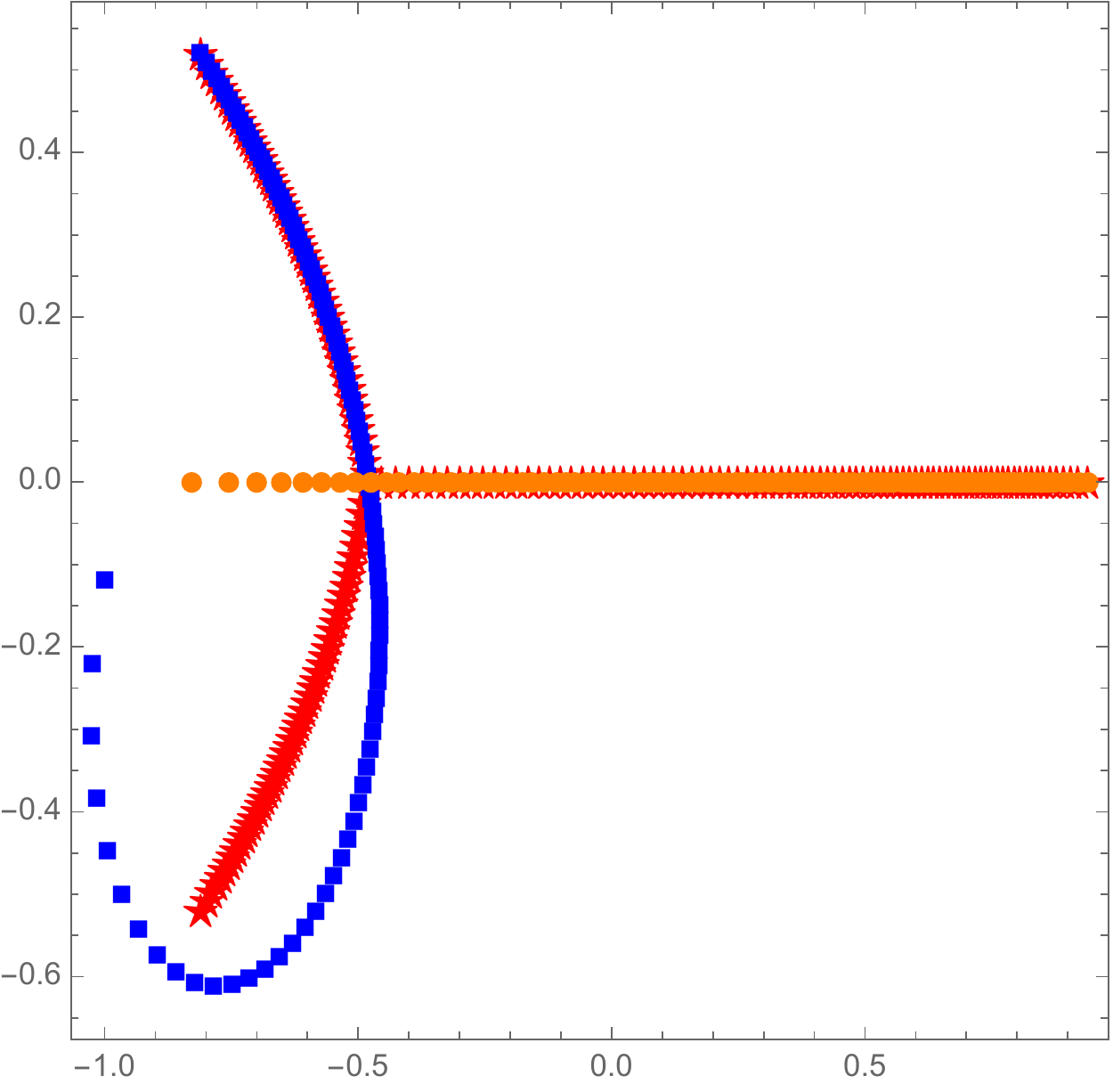}
\end{overpic}
\end{subfigure}
\caption{Rescaled zeros of $Q_{n,m}$ (stars), $C_{n,m}$ (dots) and $D_{n,m}$ (squares) when the weight is $e^{-z^5}$. 
In all the cases the orthogonality contour for $n$ is the union of rays $(e^{-\frac{4\pi i}{5}}\infty,0]\cup[0,+\infty)$, and for $m$ is the union $(e^{-\frac{4\pi i}{5}}\infty,0]\cup[0,e^{-\frac{4\pi i}{5}}\infty)$.
From left to right, top to bottom: $(n,m)=(10,20)$ and $(15,30)$ ($\alpha=1/3$); $(n,m)=(21,30)$ and $(105,150)$ ($\alpha=7/17$) ; $(n,m)=(38,40)$ and $(95,100)$ ($\alpha=19/39$).}\label{figure_zeros_quintic_symmetric}
\end{figure}

\begin{figure}[p!]
\begin{subfigure}{.5\textwidth}
\centering
\begin{overpic}[scale=.45]{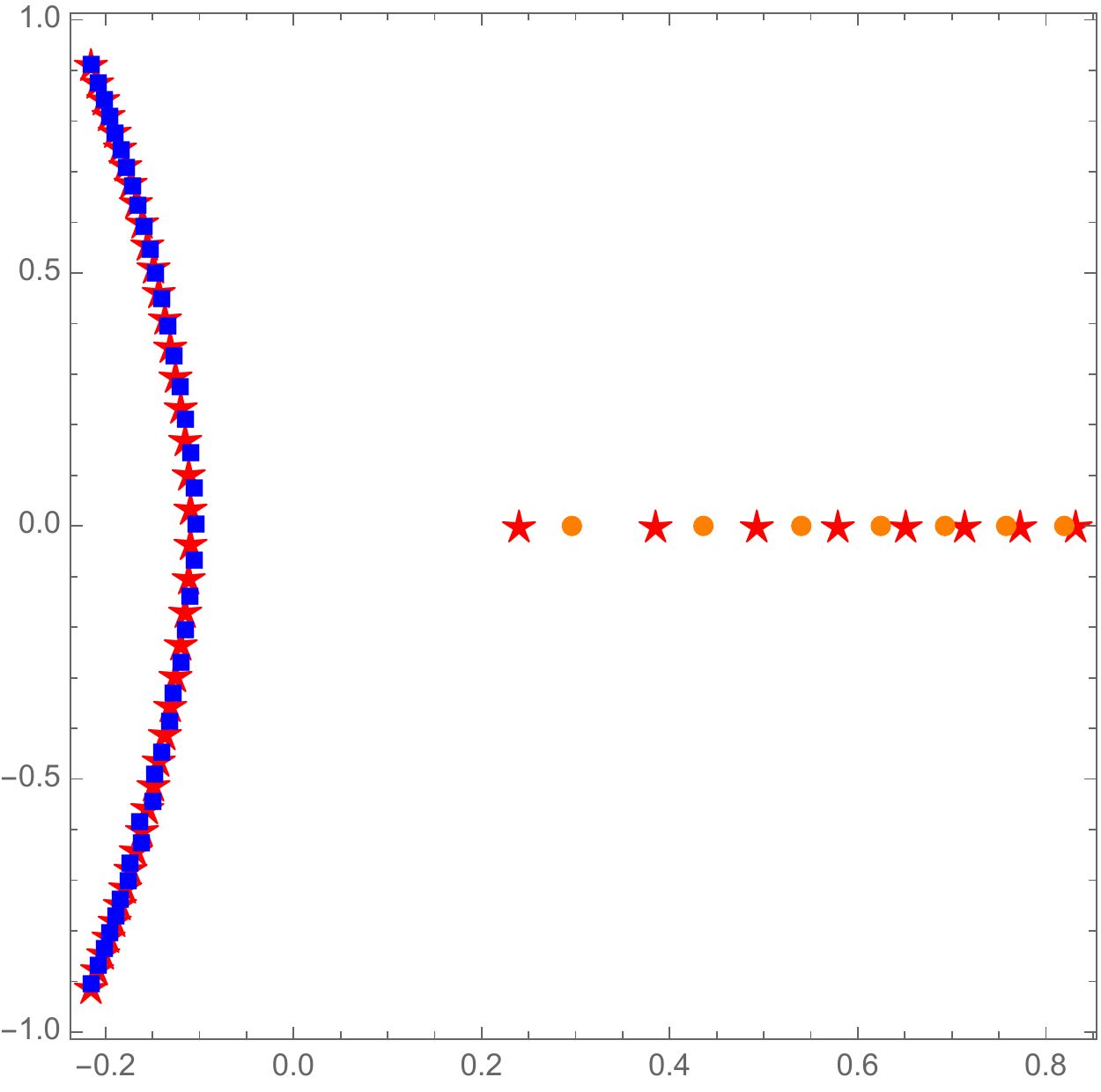}
\end{overpic}
\end{subfigure}%
\begin{subfigure}{.5\textwidth}
\centering
\begin{overpic}[scale=.45]{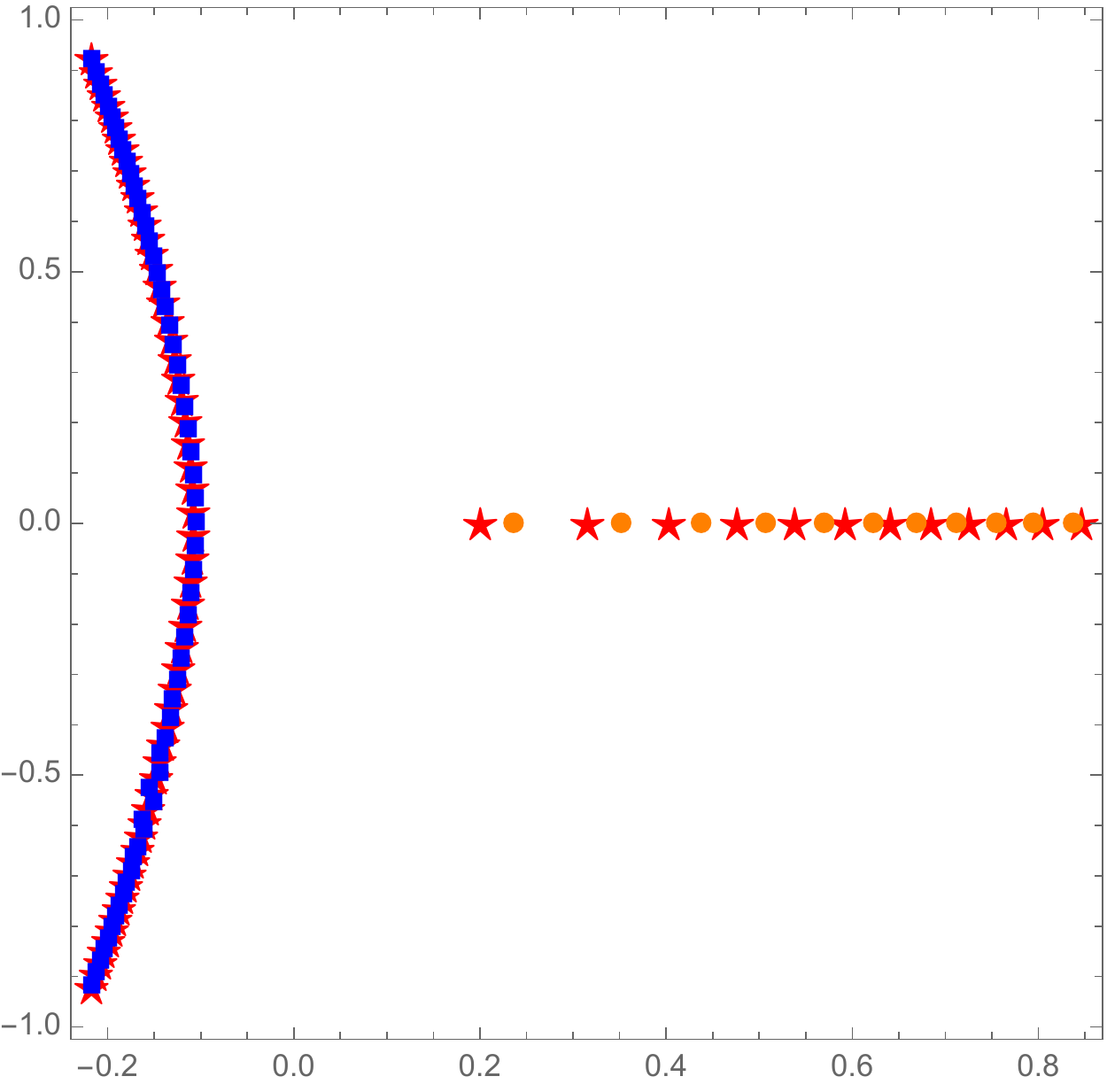}
\end{overpic}
\end{subfigure}\\
\begin{subfigure}{.5\textwidth}
\centering
\begin{overpic}[scale=.45]{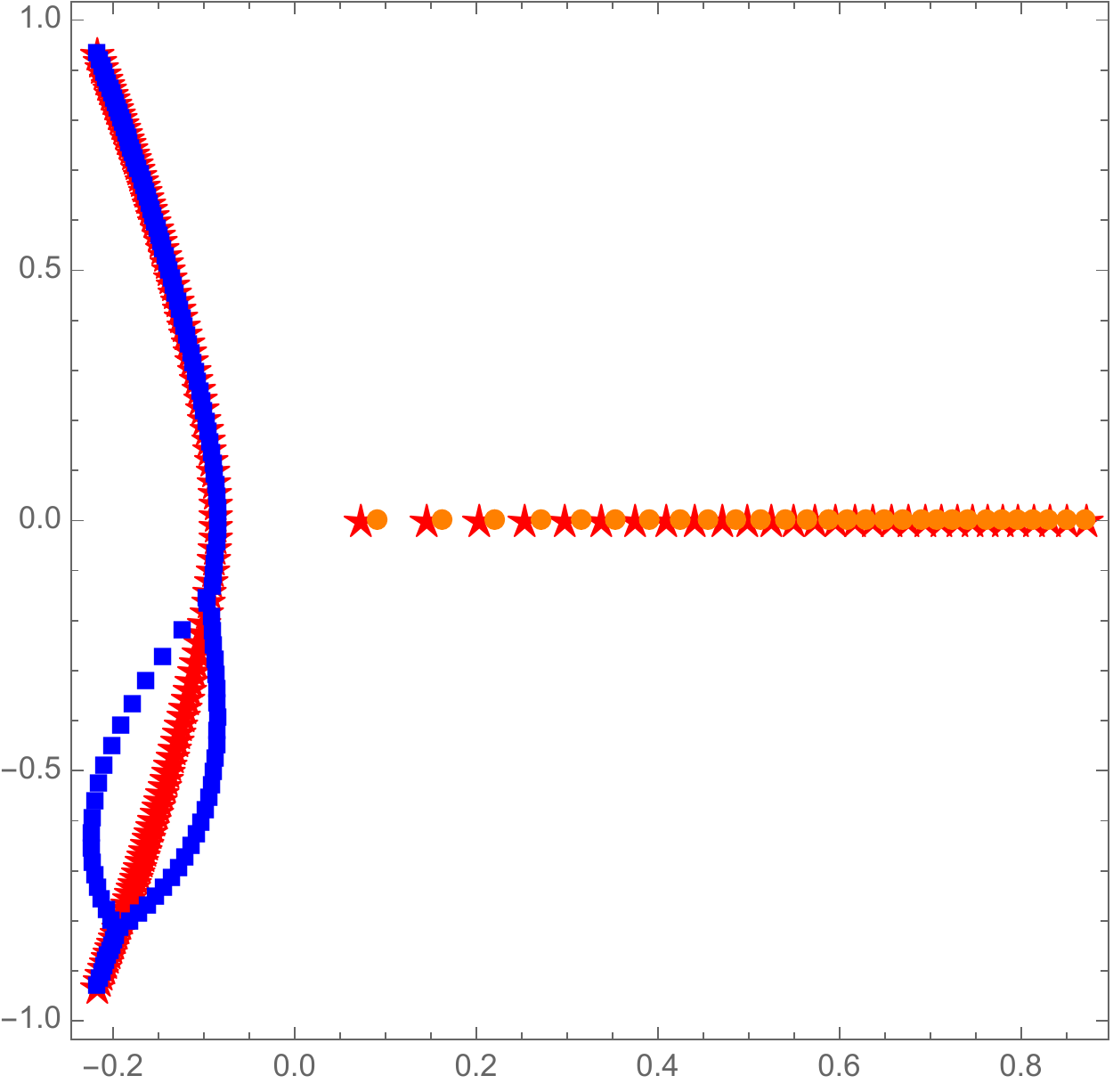}
\end{overpic}
\end{subfigure}%
\begin{subfigure}{.5\textwidth}
\centering
\begin{overpic}[scale=.45]{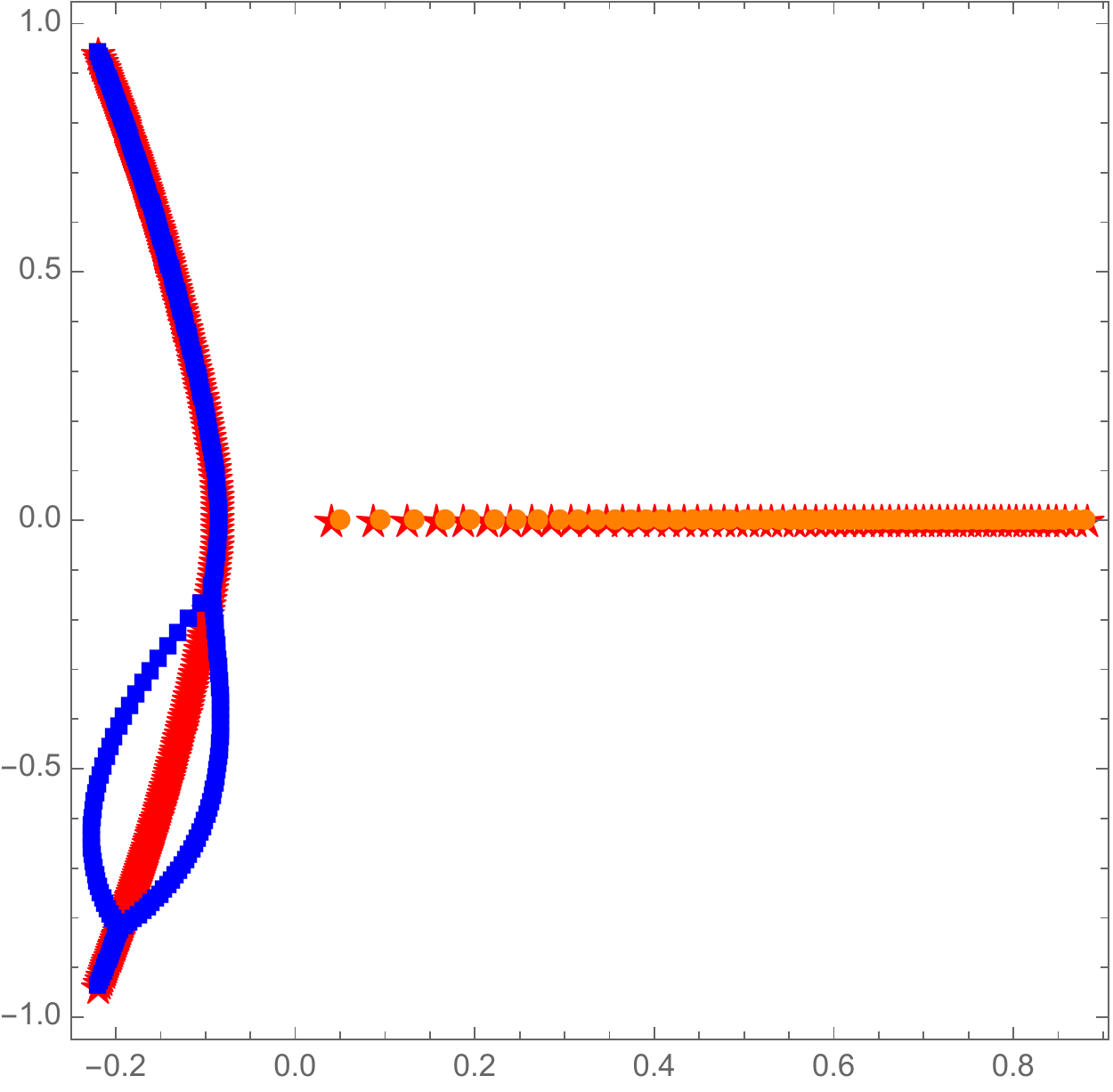}
\end{overpic}
\end{subfigure}\\
\begin{subfigure}{.5\textwidth}
\centering
\begin{overpic}[scale=.45]{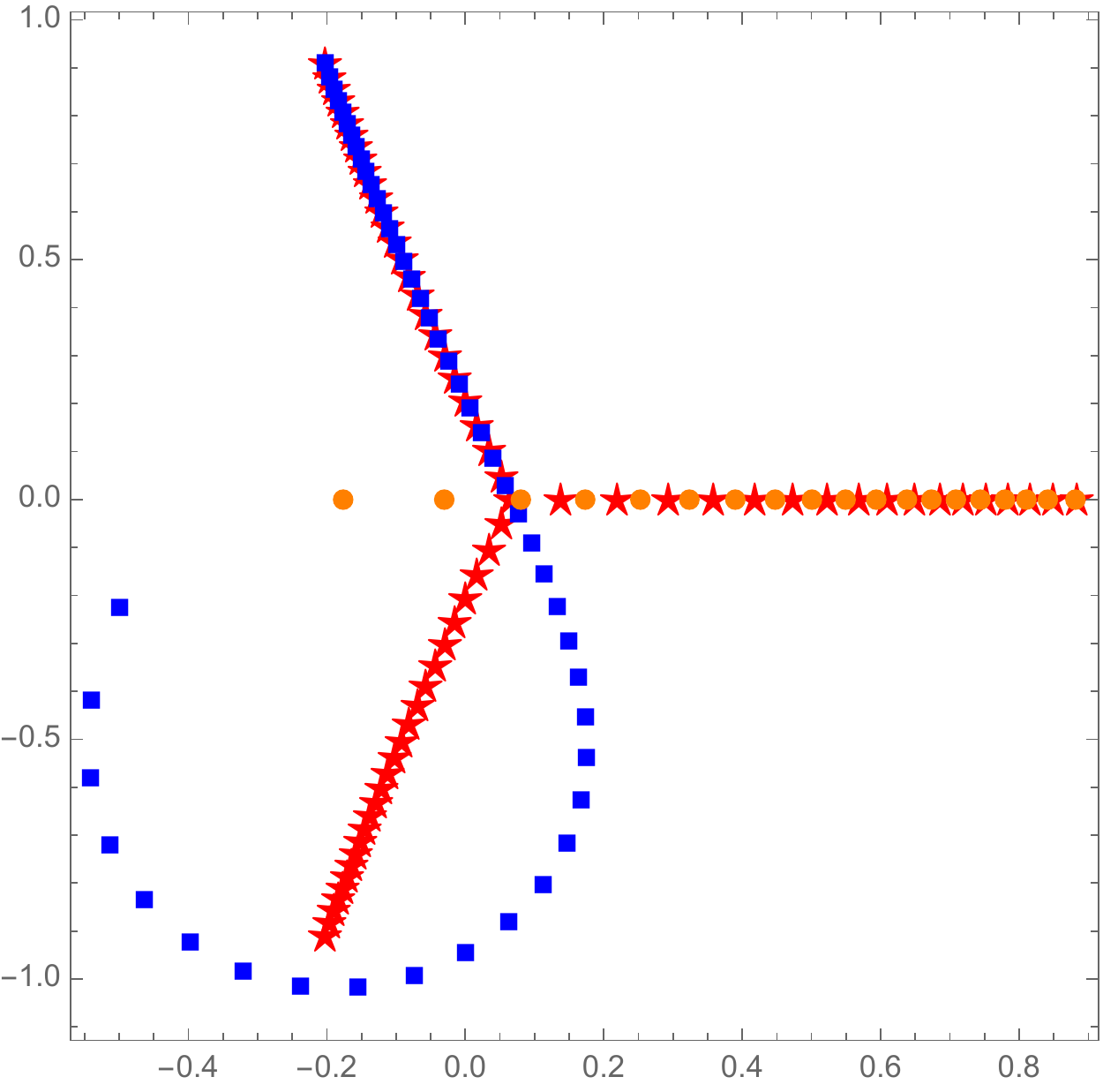}
\end{overpic}
\end{subfigure}%
\begin{subfigure}{.5\textwidth}
\centering
\begin{overpic}[scale=.45]{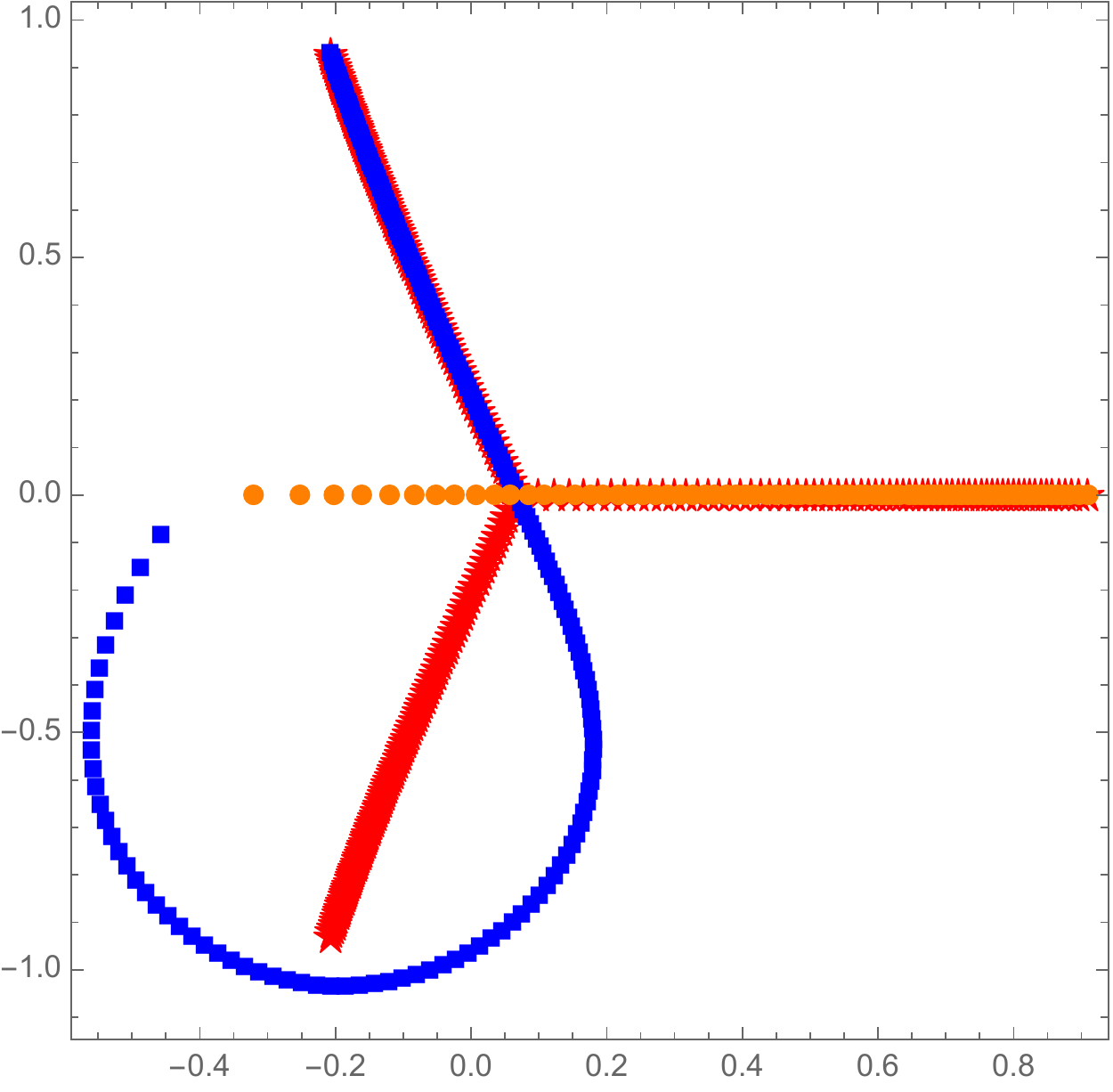}
\end{overpic}
\end{subfigure}
\caption{Rescaled zeros of $Q_{n,m}$ (stars), $C_{n,m}$ (dots) and $D_{n,m}$ (squares) when the weight is $e^{-z^5}$.
In all the cases the orthogonality contour for $n$ is the union of rays $(e^{-\frac{4\pi i}{7}}\infty,0]\cup[0,+\infty)$, and for $m$ is the union $(e^{-\frac{4\pi i}{7}}\infty,0]\cup[0,e^{-\frac{4\pi i}{7}}\infty)$.
From left to right, top to bottom: $(n,m)=(8,40)$ and $(12,60)$ ($\alpha=1/6$); $(n,m)=(30,130)$ and $(60,260)$ ($\alpha=3/16$) ; $(n,m)=(20,50)$ and $(80,200)$ ($\alpha=1/4$).}\label{figure_zeros_seventh_symmetric}
\end{figure}

\begin{figure}[p!]
\begin{subfigure}{.5\textwidth}
\centering
\begin{overpic}[scale=.45]{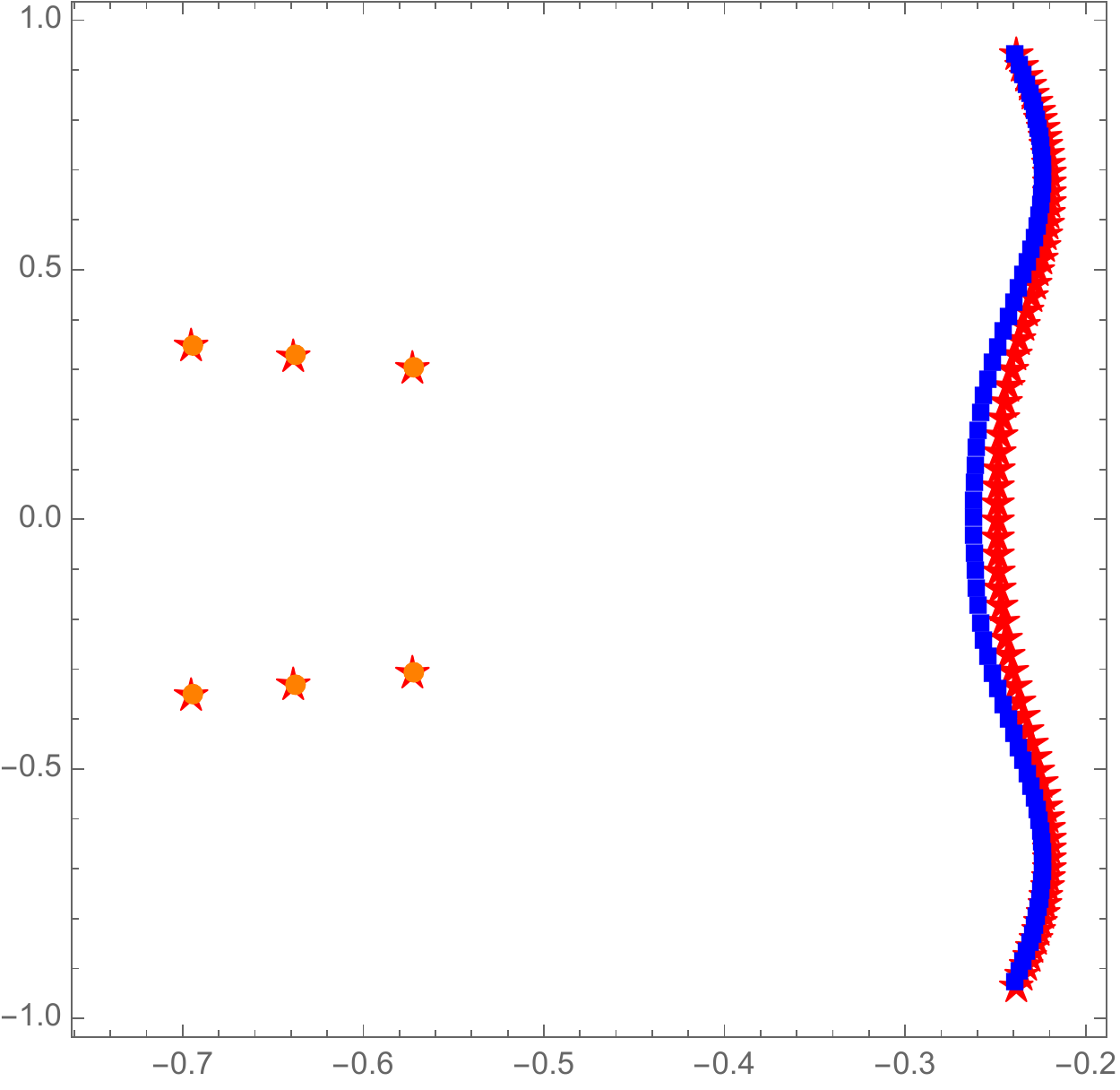}
\end{overpic}
\end{subfigure}%
\begin{subfigure}{.5\textwidth}
\centering
\begin{overpic}[scale=.45]{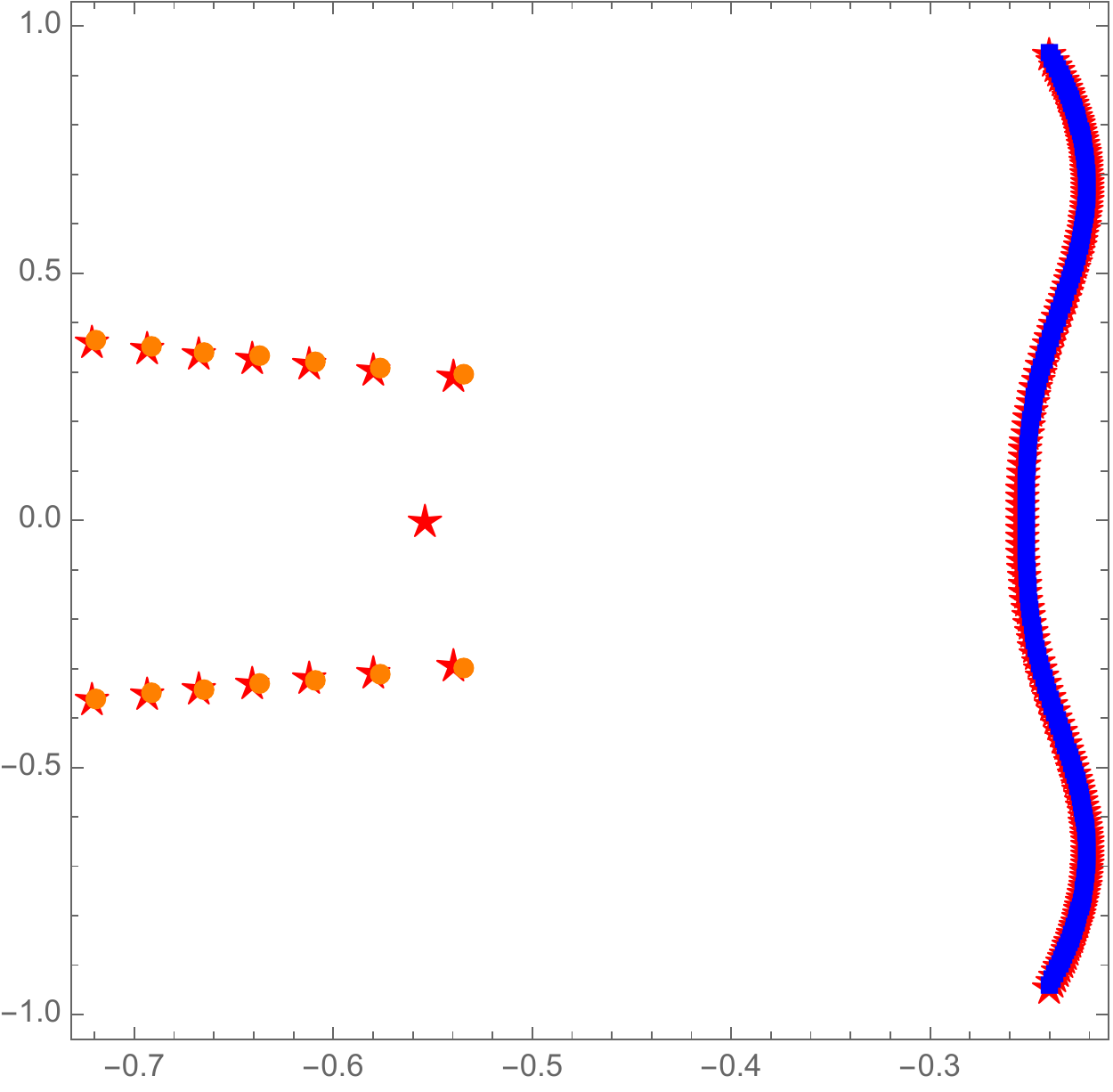}
\end{overpic}
\end{subfigure}\\
\begin{subfigure}{.5\textwidth}
\centering
\begin{overpic}[scale=.45]{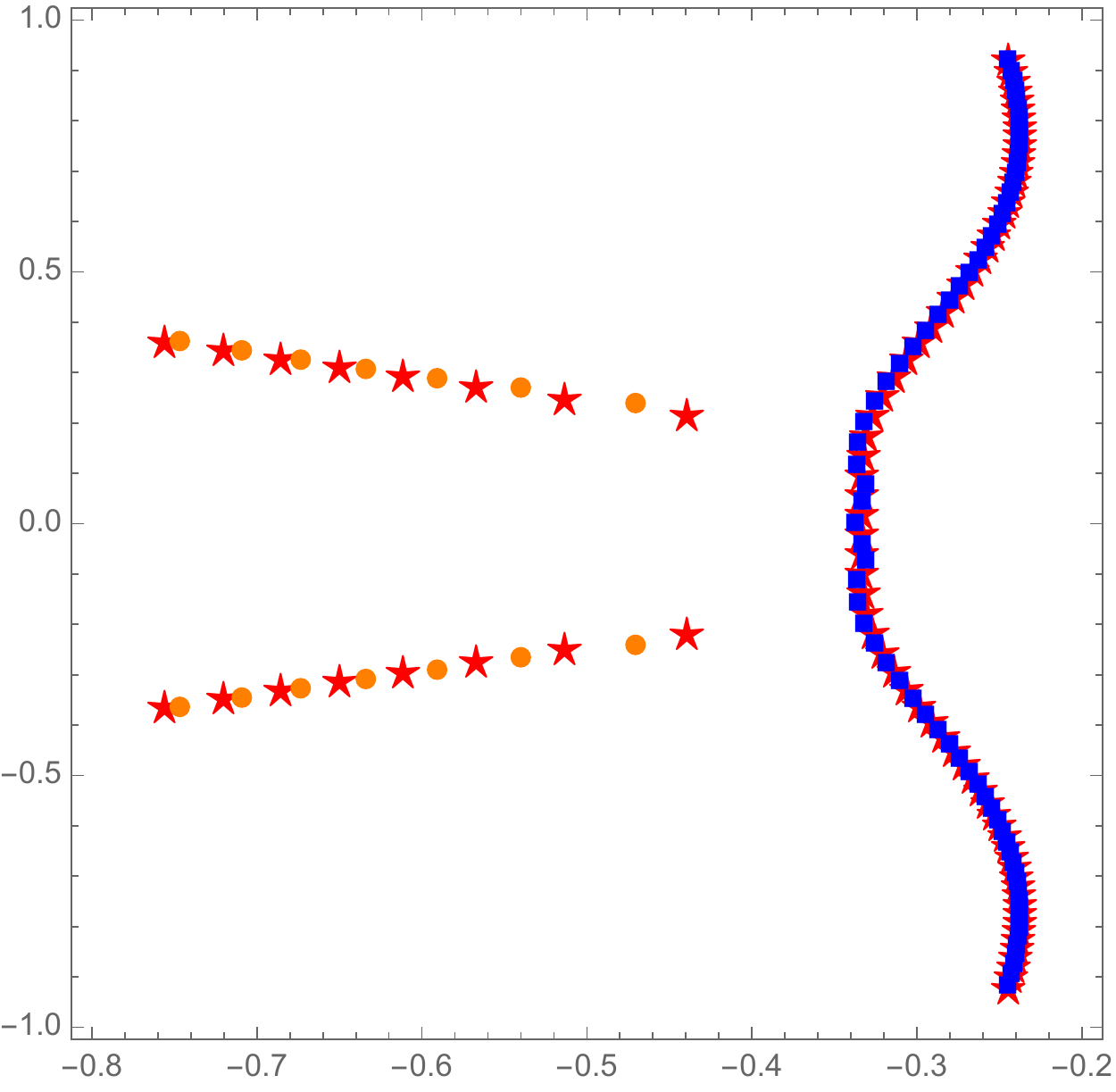}
\end{overpic}
\end{subfigure}%
\begin{subfigure}{.5\textwidth}
\centering
\begin{overpic}[scale=.45]{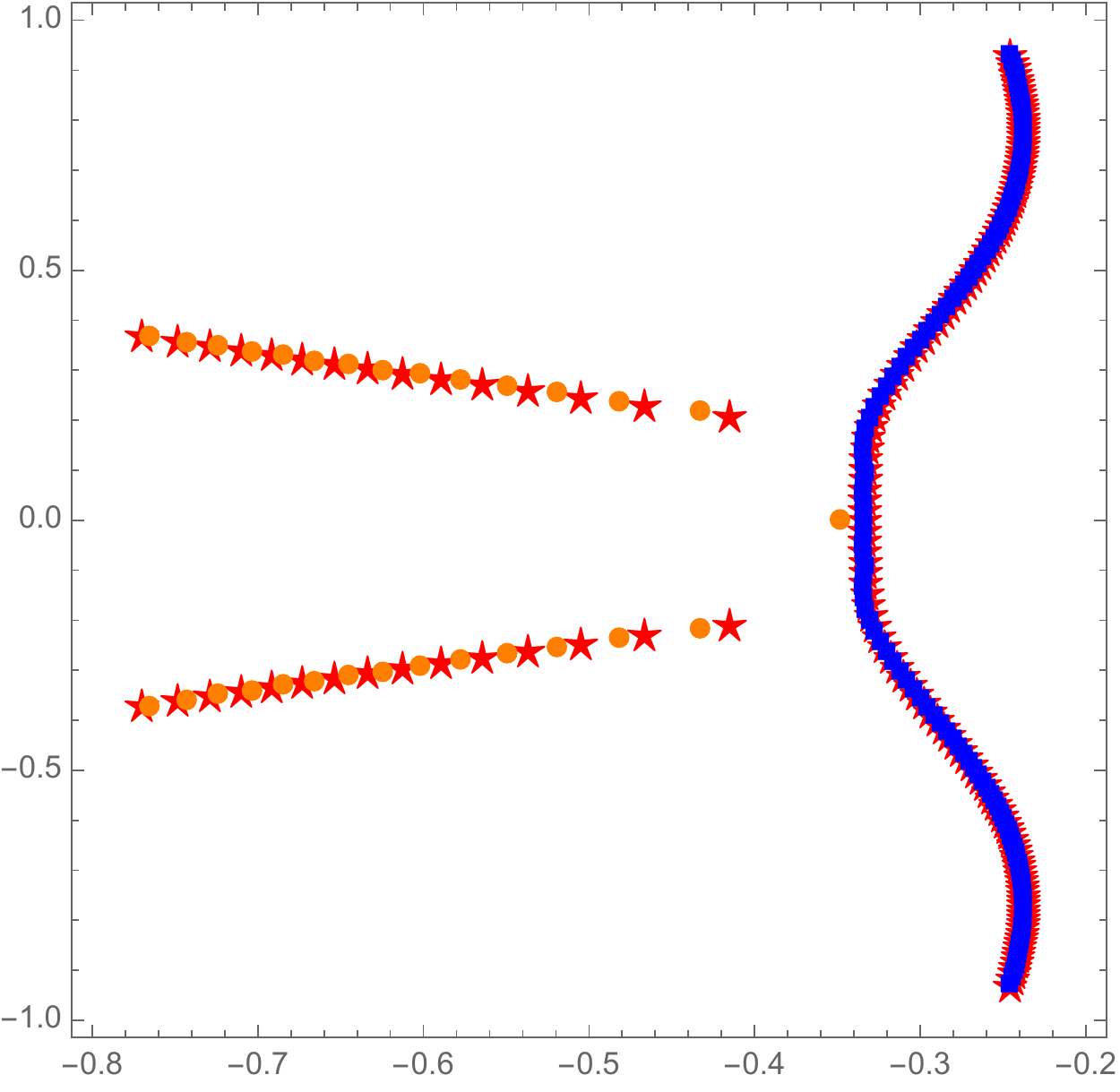}
\end{overpic}
\end{subfigure}\\
\begin{subfigure}{.5\textwidth}
\centering
\begin{overpic}[scale=.45]{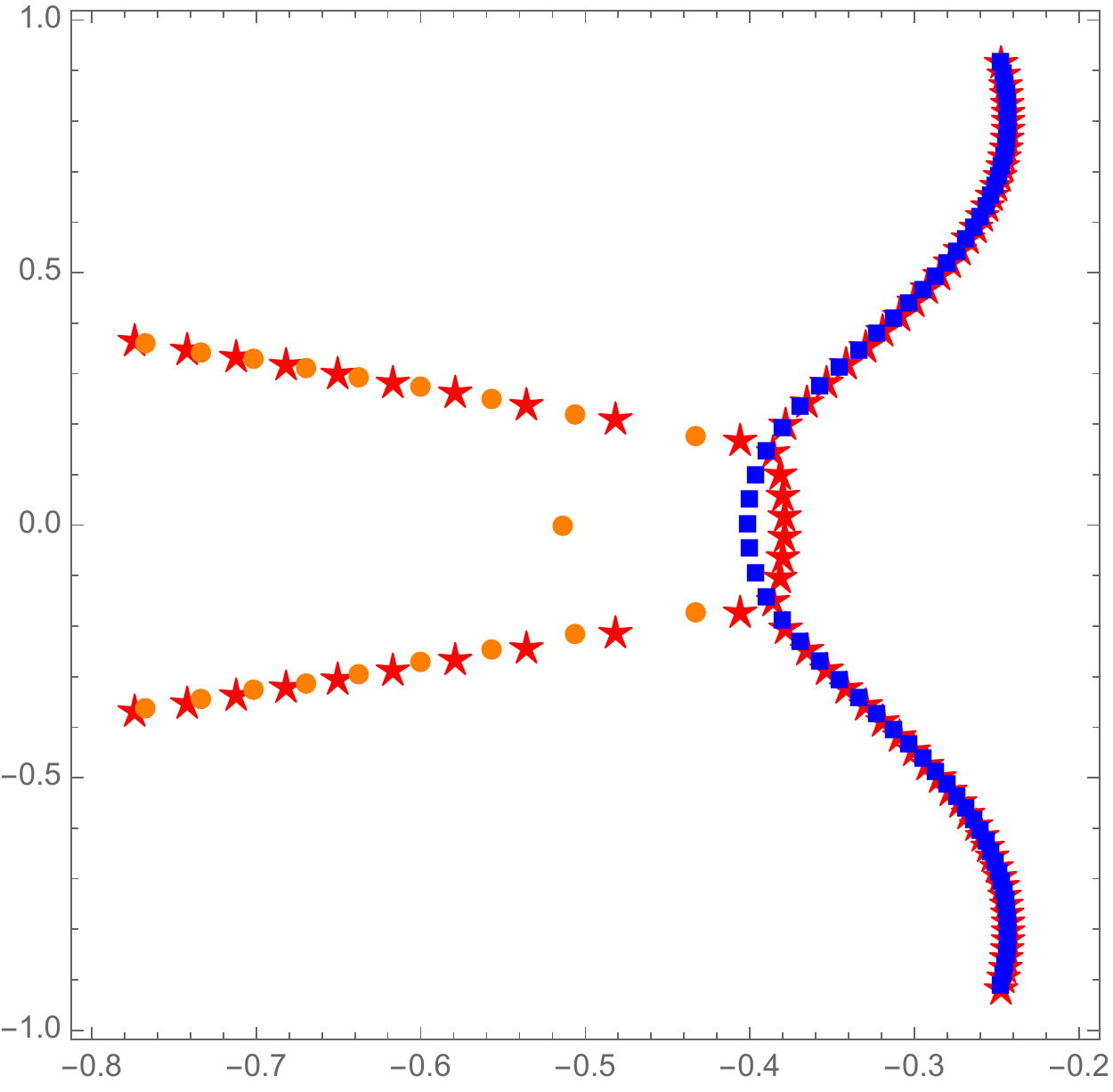}
\end{overpic}
\end{subfigure}%
\begin{subfigure}{.5\textwidth}
\centering
\begin{overpic}[scale=.45]{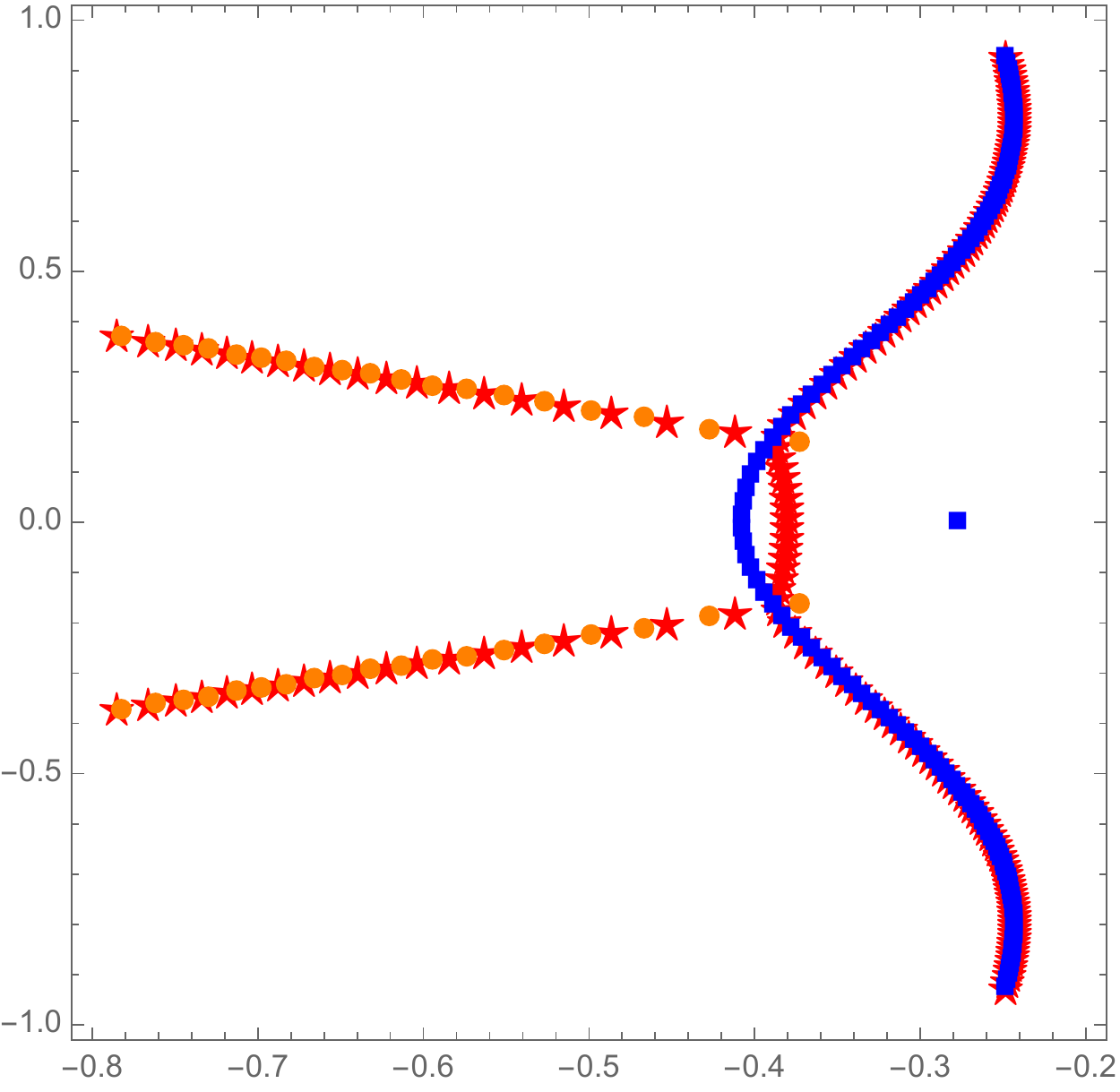}
\end{overpic}
\end{subfigure}
\caption{Rescaled zeros of $Q_{n,m}$ (stars), $C_{n,m}$ (dots) and $D_{n,m}$ (squares) when the weight is $e^{-z^7}$.
In all the cases the orthogonality contour for $n$ is the union of rays $(e^{-\frac{6\pi i}{7}}\infty,0]\cup[0,e^{\frac{6\pi i}{7}}\infty)$, and for $m$ is the union $(e^{-\frac{4\pi i}{7}}\infty,0]\cup[0,e^{\frac{4\pi i}{7}}\infty)$.
From left to right, top to bottom: $(n,m)=(7,77)$ and $(15,165)$ ($\alpha=1/12$); $(n,m)=(16,72)$ and $(30,135)$ ($\alpha=2/11$) ; $(n,m)=(20,70)$ and $(40,140)$ ($\alpha=2/9$).}\label{figure_zeros_seventh_asymmetric1}
\end{figure}

\begin{figure}[p!]
\begin{subfigure}{.5\textwidth}
\centering
\begin{overpic}[scale=.45]{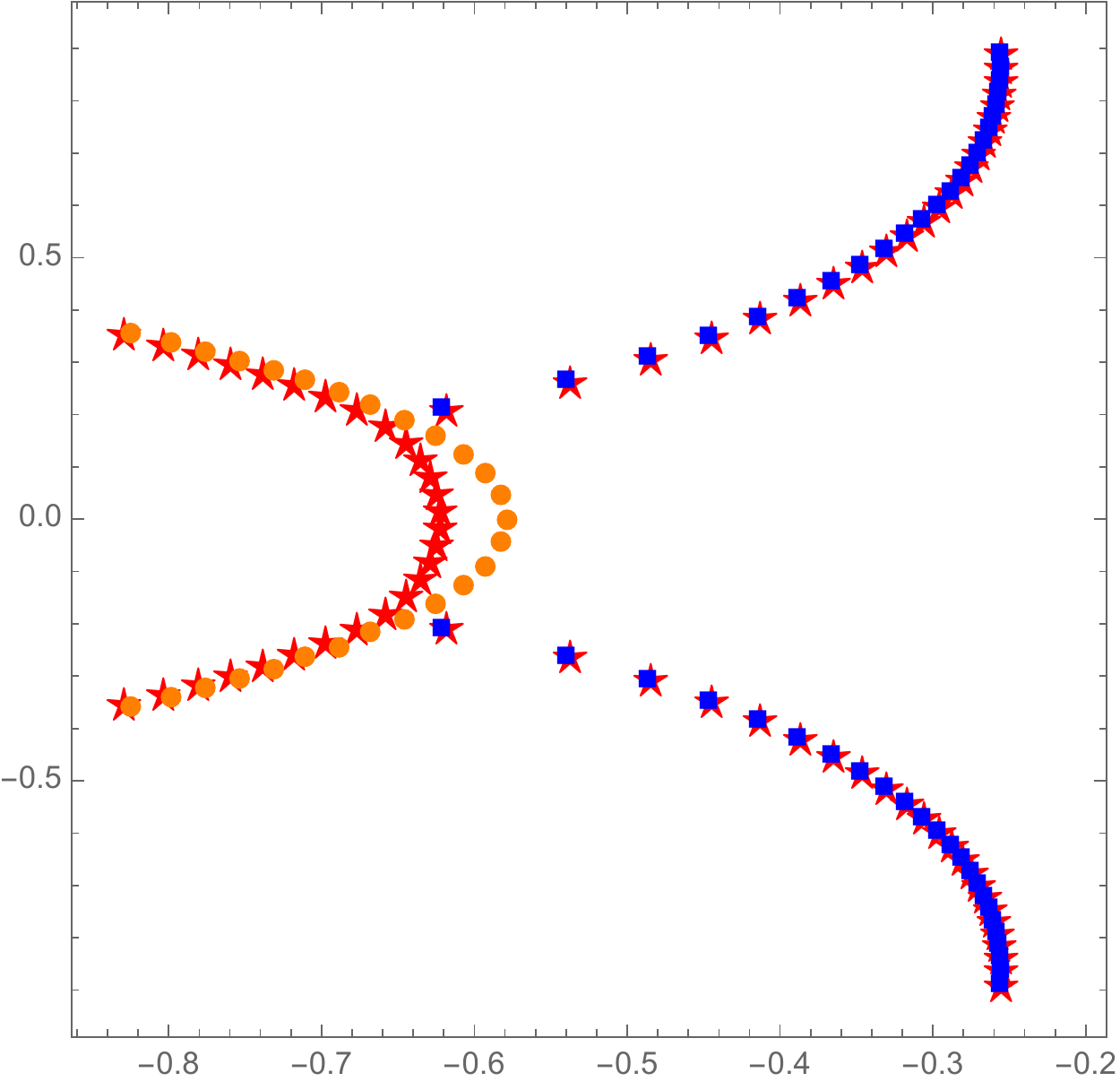}
\end{overpic}
\end{subfigure}%
\begin{subfigure}{.5\textwidth}
\centering
\begin{overpic}[scale=.45]{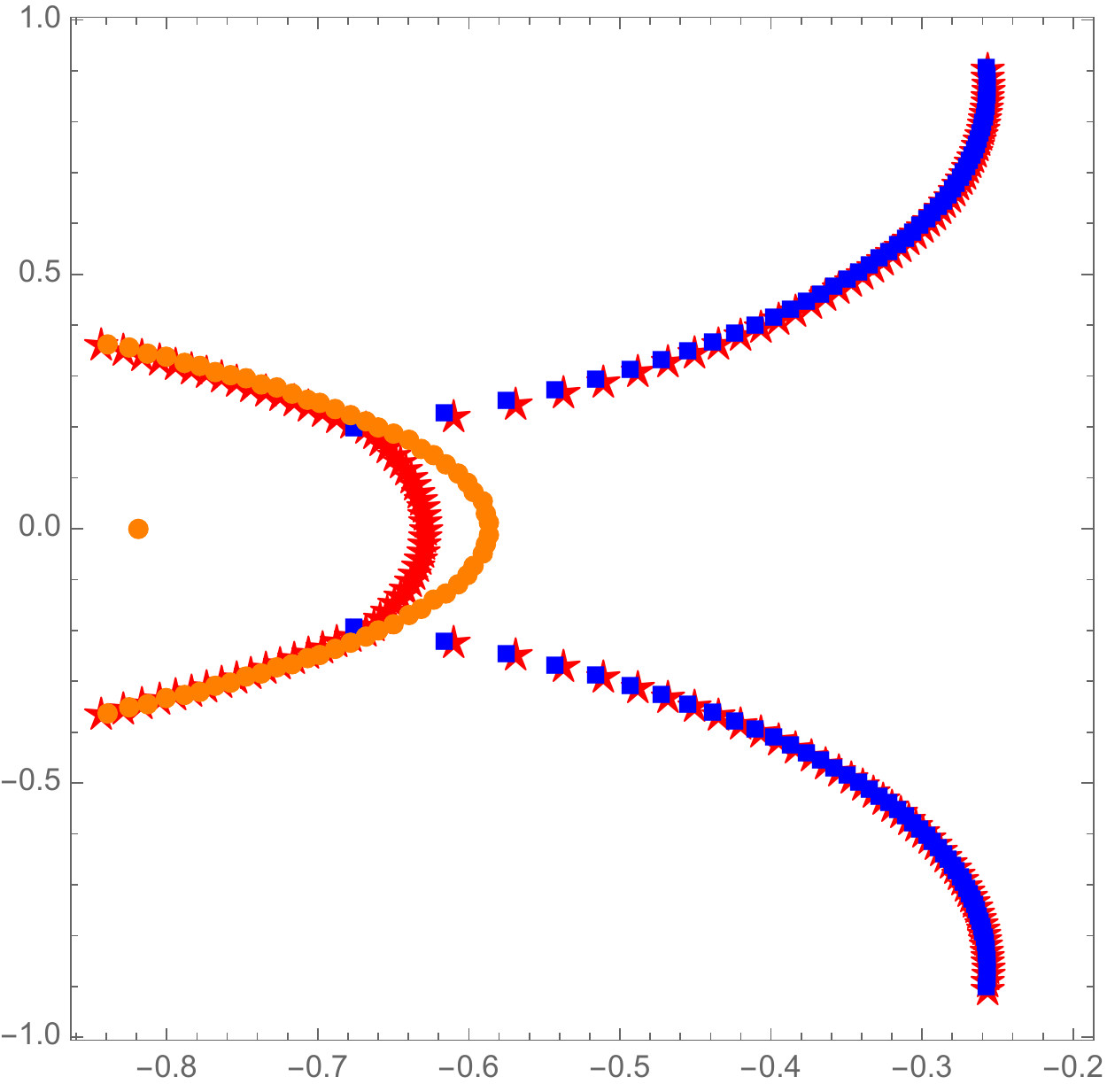}
\end{overpic}
\end{subfigure}\\
\begin{subfigure}{.5\textwidth}
\centering
\begin{overpic}[scale=.45]{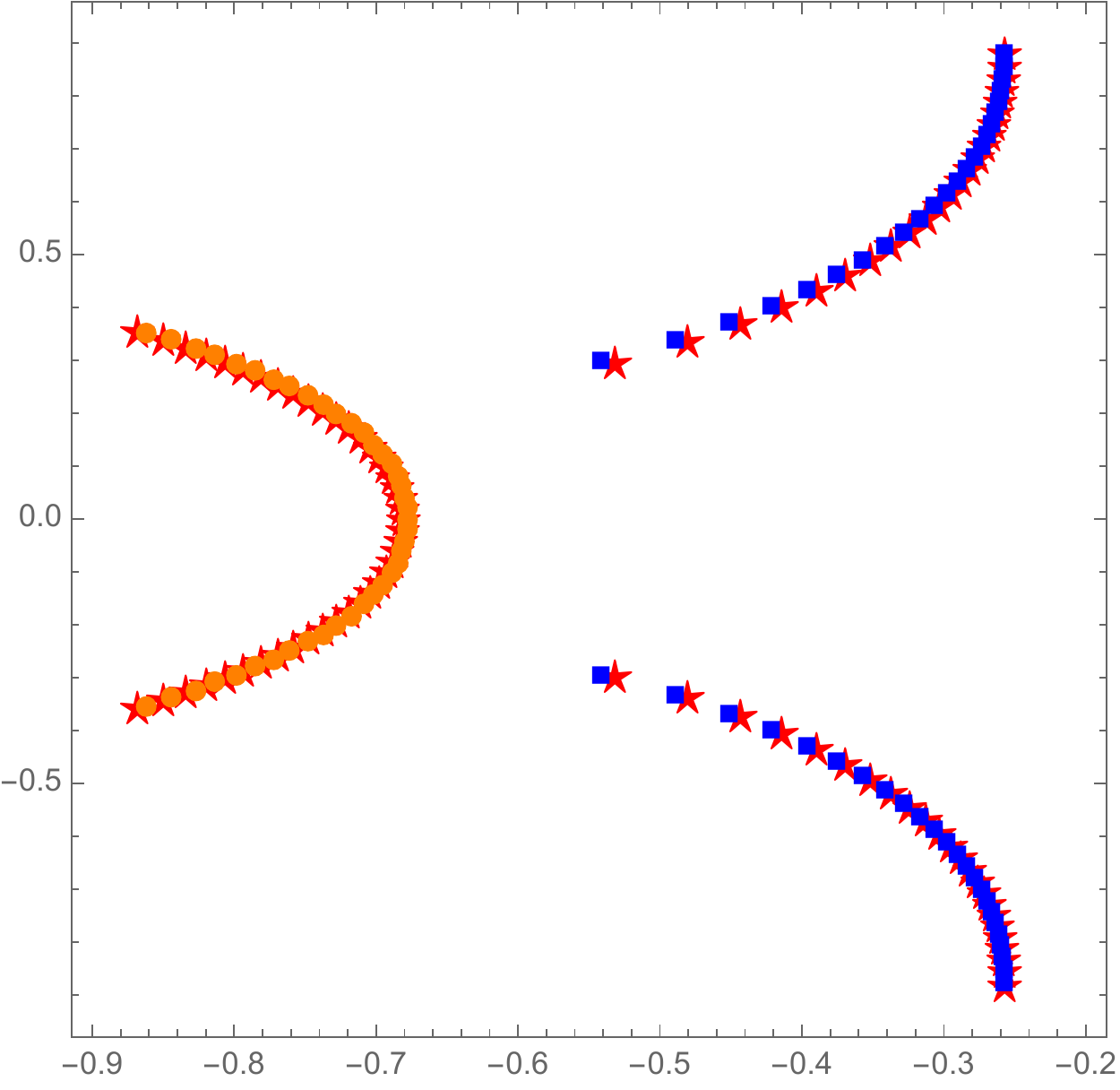}
\end{overpic}
\end{subfigure}%
\begin{subfigure}{.5\textwidth}
\centering
\begin{overpic}[scale=.45]{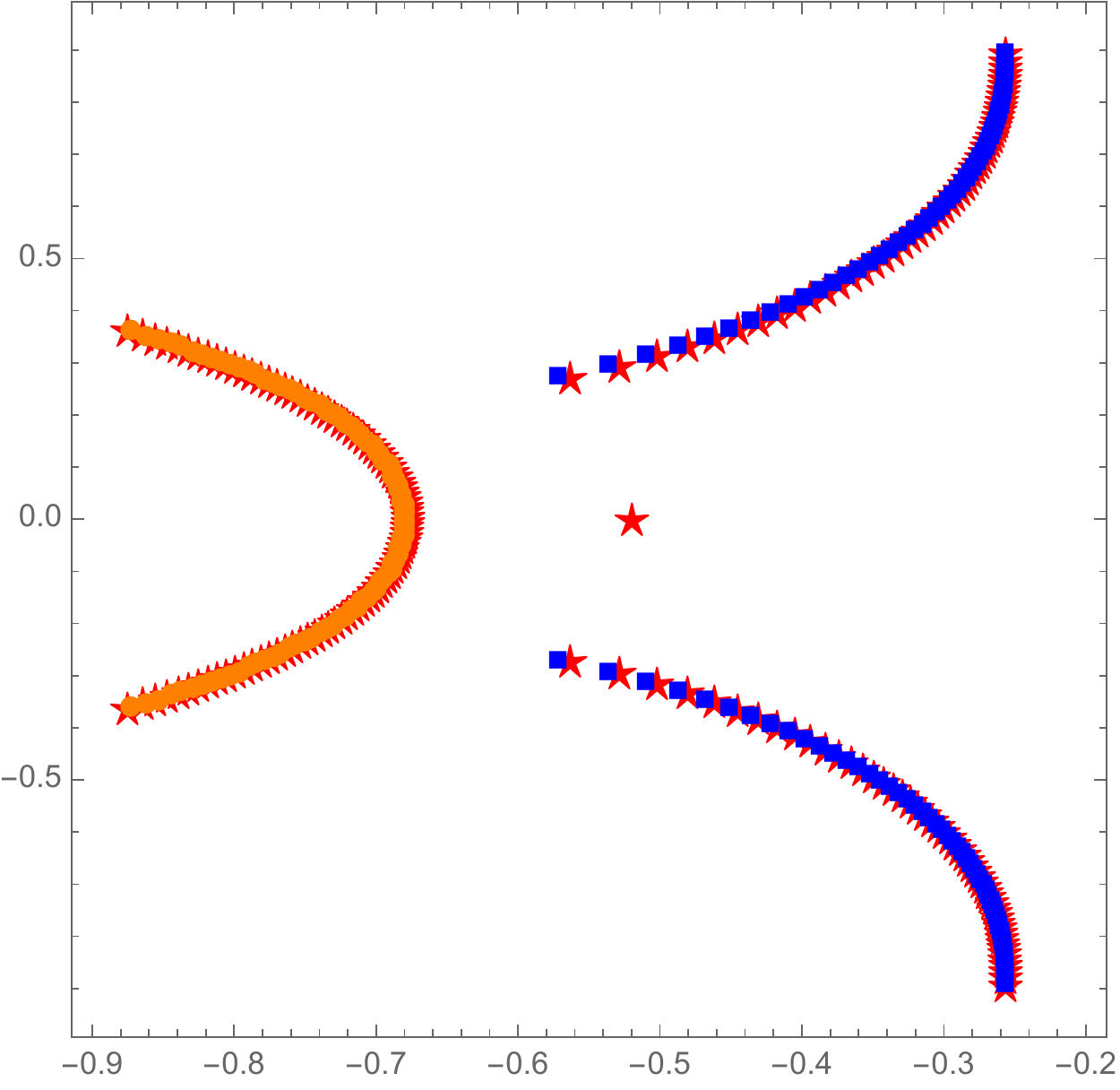}
\end{overpic}
\end{subfigure}
\caption{Rescaled zeros of $Q_{n,m}$ (stars), $C_{n,m}$ (dots) and $D_{n,m}$ (squares) when the weight is $e^{-z^7}$.
In all the cases the orthogonality contour for $n$ is the union of rays $(e^{-\frac{6\pi i}{7}}\infty,0]\cup[0,e^{\frac{6\pi i}{7}}\infty)$, and for $m$ is the union $(e^{-\frac{4\pi i}{7}}\infty,0]\cup[0,e^{\frac{4\pi i}{7}}\infty)$.
From left to right, top to bottom: $(n,m)=(28,49)$ and $(60,105)$ ($\alpha=4/11$); $(n,m)=(42,49)$ and $(90,105)$ ($\alpha=6/13$).}\label{figure_zeros_seventh_asymmetric2}
\end{figure}

\clearpage



\begin{thebibliography}{90}
	
	\bibitem{alvarez_alonso_medina_1}
	G.~Álvarez, L.~Mart{\'{\i}}nez-Alonso and E.~Medina
	\emph{Determination of {$S$}-curves with applications to the theory
		of non-{H}ermitian orthogonal polynomials}, 
	J. Stat. Mech. Theory Exp. (2013), no.~6, P06006, 28.
	
	\bibitem{alvarez_alonso_medina_2}
	G.~Álvarez, L.~Mart{\'{\i}}nez-Alonso and E.~Medina
	\emph{Partition functions and the continuum limit in {P}enner matrix
		models}, 
	J. Phys. A \textbf{47} (2014), no.~31, 315205, 29.
	
	\bibitem{MR2475084}
	A.~I.~Aptekarev, \emph{Asymptotics of {H}ermite-{P}ad\'e approximants for a
		pair of functions with branch points}, Dokl. Akad. Nauk \textbf{422} (2008),
	no.~4, 443--445.
	
	\bibitem{MR2963451}
	A.~I.~Aptekarev, V.~I.~Buslaev, A.~Mart\'{\i}nez-Finkelshtein and
	S.~P.~Suetin, \emph{Pad\'e approximants, continued fractions, and orthogonal
		polynomials}, Uspekhi Mat. Nauk \textbf{66} (2011), no.~6(402), 37--122,
	translation in Russian Math. Surveys 66 (2011), no. 6, 1049–1131.
	
	\bibitem{MR2963452}
	A.~I.~Aptekarev and A.~B.~{\`E}.~Ko{\u\i}{\`e}laars, \emph{Hermite-{P}ad\'e
		approximations and ensembles of multiple orthogonal polynomials}, Uspekhi
	Mat. Nauk \textbf{66} (2011), no.~6(402), 123--190.
	
	\bibitem{aptekarev_kuijlaars_vanassche_hermite_pade_genus_0}
	A.~I.~Aptekarev, A.~B.~J.~Kuijlaars and W.~Van~Assche,
	\emph{Asymptotics of {H}ermite-{P}ad\'e rational approximants for two
		analytic functions with separated pairs of branch points (case of genus 0)},
	Int. Math. Res. Pap. IMRP (2008), Art. ID rpm007, 128. 
	
	\bibitem{aptekarev_lysov_tulyakov}
	A.~I.~Aptekarev, V.~G.~Lysov and D.~N.~Tulyakov,
	\emph{Random matrices with an external source and the asymptotics of multiple orthogonal polynomials},
  	Mat. Sb. (2011), no.~2(202), 3--56.
	
	\bibitem{aptekarev_lysov_tulyakov_2}
	A.~I.~Aptekarev, V.~G.~Lysov and D.~N.~Tulyakov,
	\emph{The global eigenvalue distribution regime of random matrices with an anharmonic potential and an external source},
  	Teoret. Mat. Fiz. (2009), no.~1(159), 34--57.
	
	
	\bibitem{aptekarev_vanassche_yatsselev}
	A.~I.~Aptekarev, W.~Van~Assche and M.~L.~Yatsselev, \emph{Hermite-Padé approximants
		for a pair of Cauchy transforms with overlapping symmetric supports},
	Commun. Pure Appl. Math. \textbf{70} (2017), no. 3, 444--510.
	
	
%
%
%
%
%
	
	\bibitem{bertola_gekhtman_szmigielski_cauchy_two_matrix_model}
	M.~Bertola, M.~Gekhtman and J.~Szmigielski, \emph{Strong asymptotics for
		{C}auchy biorthogonal polynomials with application to the {C}auchy two-matrix
		model}, J. Math. Phys. \textbf{54} (2013), no.~4, 043517, 25.
	
	
%
%
	
	\bibitem{bleher_delvaux_kuijlaars_external_source}
	P.~M.~Bleher, S.~Delvaux and A.~B.~J.~Kuijlaars, \emph{Random matrix model with
		external source and a constrained vector equilibrium problem}, Comm. Pure
	Appl. Math. \textbf{64} (2011), no.~1, 116--160. 
	
	\bibitem{kuijlaars_bleher_external_source_gaussian_I}
	P.~M.~Bleher and A.~B.~J.~Kuijlaars, \emph{Large {$n$} limit of {G}aussian random
		matrices with external source.{I}}, Comm. Math. Phys. \textbf{252} (2004),
	no.~1-3, 43--76. 
	
	\bibitem{kuijlaars_bleher_external_source_multiple_orthogonal}
	P.~M. Bleher and A.~B.~J. Kuijlaars, \emph{Random matrices with external source
		and multiple orthogonal polynomials}, Int. Math. Res. Not. (2004), no.~3,
	109--129. 
	
%
	
	\bibitem{bleher_kuijlaars_normal_matrix_model}
	P.~M. Bleher and A.~B.~J. Kuijlaars, \emph{Orthogonal polynomials in the
		normal matrix model with a cubic potential}, Adv. Math. \textbf{230} (2012),
	no.~3, 1272--1321.
	
	\bibitem{deano_kuijlaars_huybrechs_complex_orthogonal_polynomials}
	A.~Dea{\~n}o, D.~Huybrechs and A.~B.~J. Kuijlaars, \emph{Asymptotic zero
		distribution of complex orthogonal polynomials associated with {G}aussian
		quadrature}, J. Approx. Theory \textbf{162} (2010), no.~12, 2202--2224.
	
	
%
%
%
	
	\bibitem{duits_geudens_kuijlaars}
	M.~Duits, D.~Geudens and A.~B.~J. Kuijlaars, \emph{A vector
		equilibrium problem for the two-matrix model in the quartic/quadratic case},
	Nonlinearity \textbf{24} (2011), no.~3, 951--993. 
	
	\bibitem{duits_kuijlaars_two_matrix_model}
	M.~Duits and A.~B.~J. Kuijlaars, \emph{Universality in the two-matrix
		model: a {R}iemann-{H}ilbert steepest-descent analysis}, Comm. Pure Appl.
	Math. \textbf{62} (2009), no.~8, 1076--1153. 
	
	\bibitem{duits_kuijlaars_mo}
	M.~Duits, A.~B.~J. Kuijlaars and M.~Y. Mo, \emph{The {H}ermitian two
		matrix model with an even quartic potential}, Mem. Amer. Math. Soc.
	\textbf{217} (2012), no.~1022, v+105.
	
	\bibitem{fidalgo}
	U.~Fidalgo Prieto and G.~L\'opez Lagomasino, \emph{Nikishin systems are perfect}, Constr. Approx. \textbf{34} (2011), no. 3, 297--356.
	
	\bibitem{filipuk_vanassche_zhang}
	G.~Filipuk, W.~Van~Assche and L.~Zhang, \emph{Multiple orthogonal polynomials
		associated with an exponential cubic weight}, J. Approx. Theory \textbf{190} (2015), 1--37. 
	
	
	\bibitem{Assche01}
	J.~S. Geronimo, A.~B.~J. Kuijlaars and W.~Van~Assche,
	\emph{Riemann-{H}ilbert problems for multiple orthogonal polynomials},
	Special functions 2000: current perspective and future directions (Tempe,
	AZ), NATO Sci. Ser. II Math. Phys. Chem., vol.~30, Kluwer Acad. Publ.,
	Dordrecht, 2001, pp.~23--59.
	
	\bibitem{gonchar_rakhmanov_1981_pade}
	A.~A. Gonchar and E.~A. Rakhmanov, \emph{On the convergence of simultaneous
		{P}ad\'e approximants for systems of functions of {M}arkov type}, Trudy Mat.
	Inst. Steklov. \textbf{157} (1981), 31--48, 234, Number theory, mathematical
	analysis and their applications. 
	
	\bibitem{MR807734}
	A.~A. Gonchar and E.~A. Rakhmanov, \emph{The equilibrium problem for vector potentials}, Uspekhi Mat.
	Nauk \textbf{40} (1985), no.~4(244), 155--156. 
	
	\bibitem{gonchar_rakhmanov_rato_rational_approximation}
	A.~A. Gonchar and E.~A. Rakhmanov, \emph{Equilibrium distributions and the rate of rational
		approximation of analytic functions}, Mat. Sb. (N.S.) \textbf{134(176)}
	(1987), no.~3, 306--352, 447.  
	
	
\bibitem{Hermite1873}
Ch.~Hermite, \emph{Sur la fonction exponentielle}, C. R. Acad. Sci. Paris \textbf{77} (1873), 18--24, 74--79, 226--233, 285--293.
	
	
		\bibitem{suetin13}
	N.~R. Ikonomov, R.~K. Kovacheva, and S.~P. Suetin, \emph{On the limit zero distribution of type I Hermite-Pad\'e polynomials}, preprint  arXiv:1506.08031.
	
	\bibitem{suetin14}
	N.~R. Ikonomov, R.~K. Kovacheva, and S.~P. Suetin, \emph{Zero Distribution of Hermite-Pad\'e polynomials and convergence properties of Hermite approximants for multivalued analytic functions}, preprint arXiv:1603.03314.
	
%
%
%
%
	
	\bibitem{kuijlaars_martinez-finkelshtein_wielonsky_bessel_paths}
	A.~B.~J. Kuijlaars, A.~Mart{\'{\i}}nez-Finkelshtein and F.~Wielonsky,
	\emph{Non-intersecting squared {B}essel paths and multiple orthogonal
		polynomials for modified {B}essel weights}, Comm. Math. Phys. \textbf{286}
	(2009), no.~1, 217--275.
	
%
	
	\bibitem{kuijlaars_vanassche_wielonsky_hermite_pade}
	A.~B.~J. Kuijlaars, W.~Van~Assche and F.~Wielonsky, \emph{Quadratic
		{H}ermite-{P}ad\'e approximation to the exponential function: a
		{R}iemann-{H}ilbert approach}, Constr. Approx. \textbf{21} (2005), no.~3,
	351--412.  
	
	\bibitem{kuijlaars_lopez_normal_matrix_model}
	A.~B.~J. Kuijlaars and A.~L{\'o}pez-Garc{\'{\i}}a, \emph{The normal matrix
		model with a monomial potential, a vector equilibrium problem, and multiple
		orthogonal polynomials on a star}, Nonlinearity \textbf{28} (2015), no.~2,
	347--406. 
	
		\bibitem{kuijlaars_mo}
	A.~B.~J. Kuijlaars and M.~Y.~Mo, \emph{The global parametrix in the Riemann--Hilbert steepets descent analysis for orthogonal polynomials}, Comput. Methdos Funct. Theory \textbf{11} (2011), no.~1,
	161--178. 
	
	\bibitem{kuijlaars_silva}
	A.~B.~J. Kuijlaars and G.~L.~F. Silva, \emph{S-curves in polynomial
		external fields}, J. Approx. Theory \textbf{191} (2015), 1--37. 
	
	
%
	
\bibitem{leurs_vanassche}
M.~Leurs and W.~Van~Assche	, \emph{Jacobi-Angelesco multiple orthogonal polynomials on an r-star}, preprint arXiv:1804.07512.
	
	\bibitem{Markov1895}
	A. Markoff, \emph{Deux d\'emonstrations de la convergence de certaines fractions continues}, Acta Math. \textbf{19}:1 (1895), 93--104.
	
	
	\bibitem{martinez_rakhmanov}
	A.~Mart{\'{\i}}nez-Finkelshtein and E.~A. Rakhmanov, \emph{Critical measures,
		quadratic differentials, and weak limits of zerosof {S}tieltjes polynomials},
	Comm. Math. Phys. \textbf{302} (2011), no.~1, 53--111. 
	
	\bibitem{martinez-saff}
	A.~Mart{\'{\i}}nez-Finkelshtein and E.~B. Saff, \emph{Asymptotic properties of Heine-Stieltjes and Van Vleck polynomials}, 
	J. Approx. Theory \textbf{118} (2002), no.~1, 131--151.
	
	\bibitem{martinez_silva_critical_measures}
	A.~Mart{\'{\i}}nez-Finkelshtein and G.~L.~F. Silva, \emph{Critical
		measures for vector energy: Global structure of trajectories of quadratic differentials}, 
	Adv. Math. \textbf{302} (2016), 1137--1232.
	
	\bibitem{MR2647571}
	A.~Mart{\'{\i}}nez-Finkelshtein and E.~A. Rakhmanov, \emph{On
		asymptotic behavior of {H}eine-{S}tieltjes and {V}an {V}leck polynomials},
	Recent trends in orthogonal polynomials and approximation theory, Contemp.
	Math., vol. 507, Amer. Math. Soc., Providence, RI, 2010, pp.~209--232.
	
%
	
	\bibitem{ismail_book}
	I.~Mourad, \emph{Classical and Quantum Orthogonal Polynomials in One Variable}, Encyclopedia of Mathematics and its Applications, vol.~98, Cambridge University Press, Cambridge, UK, 2005.
	
	\bibitem{nikishin_sorokin_book}
	E.~M. Nikishin and V.~N. Sorokin, \emph{Rational approximations and
		orthogonality}, Translations of Mathematical Monographs, vol.~92, American
	Mathematical Society, Providence, RI, 1991, Translated from the Russian by
	R.~P. Boas. 
	
%
%
%
	
	\bibitem{MR2796829}
	E.~A. Rakhmanov, \emph{On the asymptotics of {H}ermite-{P}ad\'e polynomials for
		two {M}arkov functions}, Mat. Sb. \textbf{202} (2011), no.~1, 133--140.
	
	
	\bibitem{rakhmanov_orthogonal_s_curves}
	E.~A. Rakhmanov, \emph{Orthogonal polynomials and {$S$}-curves}, Contemp. Math., vol.
	578, Amer. Math. Soc., Providence, RI, 2012. 
	
	\bibitem{rakhmanov_hermite_pade}
	E.~A. Rakhmanov, \emph{The asymptotics of {H}ermite-{P}ad\'e polynomials
		for two {M}arkov-type functions}, Sbornik: Mathematics \textbf{202} (2011),
	no.~1, 127.
	
	\bibitem{saff_totik_book}
	E.~B.~Saff and V.~Totik, \emph{Logarithmic potentials with external fields}, 
	Springer-Verlag, Berlin, 1997.
	
%
%
%
	
	\bibitem{stahl_orthogonal_polynomials_complex_weight_function}
	H.~Stahl, \emph{Orthogonal polynomials with complex-valued weight
		function. {I},{II}}, Constr. Approx. \textbf{2} (1986), no.~3, 225--240,
	241--251. 
	
	\bibitem{stahl_orthogonal_polynomials_complex_measures}
	H.~Stahl, \emph{Orthogonal polynomials with respect to complex-valued measures},
	Orthogonal polynomials and their applications ({E}rice, 1990), IMACS Ann.
	Comput. Appl. Math., vol.~9, Baltzer, Basel, 1991, pp.~139--154.  
	
%
	
%
%
%
	
	
\end{thebibliography}
%

\end{document}